\documentclass[reqno,11pt, centertags]{amsart}
\usepackage[left=1in,right=1in,bottom=1in,top=1in]{geometry}
\usepackage{amsmath,amsthm,amscd,amssymb,color,eucal,enumitem,latexsym,lscape,graphicx,multirow,mathrsfs,mathtools}
\mathtoolsset{showonlyrefs}

\usepackage{tikz}
\usepackage{amsmath}
\usepackage{verbatim}
\usetikzlibrary{arrows,shapes}

\newcommand{\bbox}[1]{\vspace{20pt}\fbox{\parbox{450pt}{{\bf #1}}}\vspace{20pt}}
\makeatletter
\def\theequation{\@arabic\c@equation}

\def\namedlabel#1#2{\begingroup
	#2%
	\def\@currentlabel{#2}%
	\phantomsection\label{#1}\endgroup
}

\newcommand{\tA}{\widetilde{A}}

\newcommand{\gaD}{\gamma_{{}_D}}
\newcommand{\gr}{{\text{graph}}}

\newcommand{\gaN}{\gamma_{{}_N}}
\newcommand{\tN}{\tau_{{}_N}}
\newcommand{\Om}{\Omega}
\newcommand{\dOm}{{\partial\Omega}}

\newcommand{\Sp}{\operatorname{Spec}}
\newcommand{\e}{\hbox{\rm e}}
\newcommand{\bd}{\hbox{\rm d}}
\newcommand{\bp}{{\mathbf{p}}}
\newcommand{\bq}{{\mathbf{q}}}

\newcommand{\bbN}{{\mathbb{N}}}
\newcommand{\bbR}{{\mathbb{R}}}

\newcommand{\bbZ}{{\mathbb{Z}}}

\newcommand{\C}{{\mathbb{C}}}

\newcommand{\bbK}{{\mathbb{K}}}

\newcommand{\bbC}{{\mathbb{C}}}

\newcommand{\cA}{{\mathcal A}}
\newcommand{\cB}{{\mathcal B}}

\newcommand{\cD}{{\mathcal D}}
\newcommand{\cE}{{\mathcal E}}
\newcommand{\cg}{{\mathcal G}}
\newcommand{\cG}{{\mathcal F}}
\newcommand{\cH}{{\mathcal H}}
\newcommand{\cI}{{\mathcal I}}
\newcommand{\cJ}{{\mathcal J}}
\newcommand{\cK}{{\mathcal K}}
\newcommand{\cL}{{\mathcal L}}
\newcommand{\cM}{{\mathcal M}}

\newcommand{\cO}{{\mathcal O}}
\newcommand{\cP}{{\mathcal P}}

\newcommand{\cR}{{\mathcal R}}
\newcommand{\cS}{{\mathcal S}}
\newcommand{\cT}{{\mathcal T}}
\newcommand{\cU}{{\mathcal U}}
\newcommand{\cV}{{\mathcal V}}
\newcommand{\cW}{{\mathcal W}}
\newcommand{\cX}{{\mathcal X}}
\newcommand{\cY}{{\mathcal Y}}

\newcommand{\bfi}{{\bf i}}

\newcommand{\no}{\nonumber}
\newcommand{\lb}{\label}

\newcommand{\ol}{\overline}

\newcommand{\wti}{\widetilde  }

\newcommand{\dL}{\prescript{d\!}{}L}

\newcommand{\ran}{\text{\rm{ran}}}

\newcommand{\hatt}{\widehat}

\newcommand{\dist}{\operatorname{dist}}

\newcommand{\rnohs}{\rangle_{H^{-1/2}(\partial\Omega)}}
\newcommand{\lnoh}{_{H^{1/2}(\partial\Omega)}\langle}

\newcommand{\gd}{\widehat{\gamma}_{{}_D}}
\newcommand{\gn}{\widehat{\gamma}_{{}_N}}

\newcommand{\bndr}{H^{{1}/{2}}(\partial \Omega)\times H^{-{1}/{2}}(\partial \Omega)}

\newcommand{\bndra}{\mathfrak{H}\times\mathfrak{H}}

\numberwithin{equation}{section}
\newcommand{\sel}[1]{}
\renewcommand{\div}{\operatorname{div}}
\renewcommand{\det}{\operatorname{det}}

\newcommand{\dom}{\operatorname{dom}}

\newcommand{\tr}{\mathrm{T}}

\newcommand{\spec}{\operatorname{Spec}}

\renewcommand{\Re}{\operatorname{Re }}
\renewcommand\Im{\operatorname{Im}}
\renewcommand{\ker}{\operatorname{ker}}

\newtheorem{theorem}{Theorem}[section]
\newtheorem{hypothesis}[theorem]{Hypothesis}
\newtheorem{lemma}[theorem]{Lemma}
\newtheorem{corollary}[theorem]{Corollary}
\newtheorem{proposition}[theorem]{Proposition}

\theoremstyle{definition}
\newtheorem{definition}[theorem]{Definition}

\newtheorem{example}[theorem]{Example}

\newtheorem{remark}[theorem]{Remark}

\begin{document}

\begin{abstract}
	This work offers a new prospective on asymptotic perturbation theory for varying self-adjoint extensions of symmetric operators. Employing symplectic formulation of self-adjointness we use a version of resolvent difference identity for two arbitrary self-adjoint extensions which facilitates asymptotic analysis of resolvent operators via first order expansion for the family of Lagrangian planes associated with perturbed operators. Specifically, we derive a Riccati-type differential  equation and the first order asymptotic expansion for resolvents of self-adjoint extensions determined by smooth one-parameter families  of Lagrangian planes.  This asymptotic perturbation theory yields a symplectic version of the abstract Kato selection theorem and  Hadamard-Rellich-type variational formula for slopes of multiple eigenvalue curves bifurcating from an eigenvalue of the unperturbed operator. The latter, in turn, gives a general infinitesimal version of the celebrated formula equating the spectral flow of a path of self-adjoint extensions and the Maslov index of the corresponding path of Lagrangian planes.  Applications are given to quantum graphs, periodic Kronig-Penney model, elliptic second order partial differential operators with Robin boundary conditions, and physically relevant heat equations with thermal conductivity. 
\end{abstract}

\allowdisplaybreaks

\title[First-order asymptotic perturbation theory 
]{First-order asymptotic perturbation theory for 
extensions of symmetric operators}
\author[Y. Latushkin]{Yuri Latushkin}
\address{Department of Mathematics,
The University of Missouri, Columbia, MO 65211, USA}
\email{latushkiny@missouri.edu}
\author[S. Sukhtaiev]{Selim Sukhtaiev	}
\address{Department of Mathematics and Statistics,
Auburn University, Auburn, AL 36849, USA}
\email{szs0266@auburn.edu}
\date{\today}
\keywords{Hadamard variational formula, Lagrangian planes, Kato selection theorem, Maslov index, self-adjoint extensions, boundary triplets}
\maketitle

{\scriptsize{\tableofcontents}}

\section{Introduction}
\subsection{Overview} This work concerns first order asymptotic expansions for resolvents and eigenvalues of self-adjoint extensions of symmetric operators subject to small perturbations of their operator theoretic domains. In the context of elliptic partial differential operators, for instance, the perturbations that we discuss model small variations of the boundary conditions, the spatial domains, and the lower order terms of differential expressions. Our main motivations stem from the Arnold--Keller--Maslov index theory, cf.  \cite{MR0211415, MR820079, BZ1, CLM, MR99207, Mas, RoSa95},  for self-adjoint  elliptic differential operators and from the classical  Hadamard--Rayleigh--Rellich \cite{H,R,MR0240668} variation formulas for their eigenvalues.  Our main new technical tool is a strikingly simple formula for the difference of resolvents of two arbitrary self-adjoint extensions of a symmetric operator derived in the context of abstract boundary triplets \cite{Behrndt_2020, BL2012,  DHM2017, DHM2020, DHM2022, DHMS2006, DHMS2012, DMbook, Schm} and inspired in part by a recent progress in description of all self-adjoint extensions of the Laplacian \cite{GLMZ05,GM08,GM10}. 
This approach gives a powerful addition to the perturbation theory via quadratic forms as it allows one to control the resolvents and spectral projections of operators with varying domains.

In this paper we study one-parameter {\em families} of self-adjoint extensions of densely defined symmetric operators.  The main results of this work are twofold.  First, we obtain new and quite general asymptotic expansion formulas for resolvents of self-adjoint operators determined by one-parameter differentiable families of Lagrangian planes, and derive a  Riccati-type differential equation for the resolvents. From this, we derive a new abstract variational Hadamard-type formula for the slopes of eigenvalue curves bifurcating from a multiple discrete isolated eigenvalue of the unperturbed operator. Motivated by closely related Hadamard variation formulas for partial differential operators on varying domains, we use the term {\em Hadamard-type} for  formulas giving $t$-derivatives of the eigenvalues of abstract and differential $t$-dependent operators treated in this paper. Our second major set of results uses the Hadamard-type formulas to bridge the celebrated Atiyah--Patodi--Singer theory and the Maslov index theory as they relate the spectral flow of a family of self-adjoint extensions to the Maslov index of the corresponding path of Lagrangian planes. We give a proof of an infinitesimal version of this relation in a very general abstract setting where all three objects may vary: the domains of the self-adjoint extensions, the boundary traces, and the operators per se.  On a more technical level, we systematically use a version of the formula for the difference of resolvent operators of two arbitrary  self-adjoint extensions of a given symmetric operator. Specifically, we express this difference in terms of orthogonal projections onto  Lagrangian planes uniquely associated with  the self-adjoint extensions in question and thus offer a novel point of view on the resolvent difference formulas through the prism of symplectic functional analysis. 
 
The asymptotic perturbation theory is a gem of classical mathematical physics \cite[Chapter VIII]{K80}. Given a family of, generally, unbounded operators $H_t=H_{t_0}+H^{(1)}_{t_0}(t-t_0)+\dots$ depending on a parameter $t\in[0,1]$ and considered as perturbations of a fixed operator $H_{t_0}$, the theory provides, for $t$ near $t_0$, formulas for the resolvent operators of $H_t$, for the Riesz projections on a group of isolated eigenvalues of $H_t$, as well as the asymptotic expansions of the type  $\lambda_j(t)=\lambda+\lambda^{(1)}_j(t-t_0)+\dots$ for the semi-simple eigenvalues $\lambda_j(t)$, $1\le j\le m$, of $H_t$ bifurcating from an eigenvalue $\lambda=\lambda(t_0)$ of $H_{t_0}$ of multiplicity  $m$. 
Of course it is not always the case that $H_t$ is an additive perturbation of $H_{t_0}$; a simple example being  the Neumann Laplacian considered as a perturbation of the Dirichlet Laplacian posted on the same open set $\Omega\subset\bbR^n$. Operator-theoretical domains of the two operators are given by the Neumann and Dirichlet boundary traces. The difference of the two operators on the intersection of their domains is zero, and thus neither of them is an additive perturbation of the other. When the operators are posted on a $t$-dependent family of open sets $\Omega_t$ and, in addition, are subject to perturbations by a family of $t$-dependent potentials, we are facing the situation when all three objects (the boundary traces, the boundary conditions prescribing the domains of the operators, and the operators per se) are being perturbed. And yet the fundamental questions remain of how to relate their resolvent operators, eigenvalues, etc. 

To answer the questions, we employ the extension theory for symmetric operators that goes back to M.\ Birman \cite{MR0080271}, M.\ Krein \cite{MR0024574, MR0024575}, and M.\ Vishik \cite{MR0052655}, see also \cite{AlSi, DMbook, GG, Schm}, and that has been an exceptionally active area of research \cite{AGW, A12, BL, BM, BtE15, BGW, DM, GG, MM02, Pa} culminating in the comprehensive monograph \cite{Behrndt_2020}. 
Unlike the classical sesquilinear  forms-based approach utilized in analytic perturbation theory, see, e.g., \cite[Section VII.6.5]{K80}, the foundational for the current paper result is a very simple formula for the difference of the resolvents of any two self-adjoint extensions of a symmetric operator. The classical Krein's formula going back to \cite{MR0024574, MR0024575} expresses the difference of the resolvents of a special, ``Dirichlet-type'', self-adjoint extension and yet another, arbitrary, self-adjoint extension of a symmetric operator  via the $\gamma$-field and the abstract Weyl $M$-function function.
Given any two arbitrary self-adjoint extensions, the classical Krein's formula is a powerful tool that has been used to prove, for example, that the difference of the resolvents of the two extensions belongs to the appropriate Schatten-von Neumann class, cf., e.g., Theorem 2 and Corollary 4 in \cite{DM}. 

In the current paper we give a very elementary and direct proof (without using the Krein's formula) of the resolvent difference formula of any two arbitrary self-adjoint extensions that we 
were not able to find in the literature. Unlike Krein's resolvent formula, the  resolvent difference formula that we offer does not contain the $\gamma$-fields nor the Weyl function, and thus is of much lower level than the celebrated Krein's resolvent formula. However, it appears to be a perfect tool for studying  {\em families} of self-adjoint extensions constructed by means of {\em families} of Lagrangian planes and {\em families} of trace operators, which is the main objective of our work. Indeed, variation formulas for eigenvalues of differential operators posted on a one-parameter family of domains are typically obtained for differential operators defined via Dirichlet forms, see, e.g., \cite[Section VII. 6.5]{K80}, \cite{MR943117}, which essentially restricts the set of admissible boundary conditions to Dirichlet, Neumann, and Robin. We drop this restriction by avoiding the quadratic form approach and, instead, dealing with perturbations of self-adjoint extensions through our new symplectic version  of the resolvent difference formula thus deriving the Hadamard-type eigenvalue formulas in a quite general setting.

The Hadamard-type formulas are instrumental in applications of spectral theory to differential operators. For example, they recently played a pivotal role in the works of G. Berkolaiko, P. Kuchment, U. Smilansky \cite{MR3000497} and  G. Cox, C. Jones, J. Marzuola \cite{CJM1, CJM2} on nodal count for eigenfunctions of Schr\"odinger operators and in the work of A. Hassell \cite{MR2630052} on ergodic billiard systems which are not quantum uniquely ergodic. The formulas are also central to the applications that we give, in particular, to our treatment, discussed in more details below, of the periodic Kronig--Penney model, spectral flow formulas for one-parameter families of Robin Laplacians leading to a unified approach to Friedlander's and Rohleder's inequalities, of the heat equation posted on bounded domains, and of one-parameter families of quantum graphs. 

\subsection{Description of abstract results} We consider  self-adjoint extensions of a closed densely defined symmetric operator $A$ acting in a Hilbert space $\cH$. The extensions in question are defined  by Lagrangian planes in an auxiliary (boundary) Hilbert space $\mathfrak H\times \mathfrak H$ by means of a two component trace map $\tr=[\Gamma_0, \Gamma_1]^{\top}: \dom(\tr)\subset \cH\rightarrow \bndra$ with dense range and satisfying the abstract Green identity 
\begin{align}\lb{int1}
	\langle A^*u,v\rangle_{\cH}-\langle u,A^*v\rangle_{\cH}=\left\langle
J \tr u,\tr v\right\rangle_{\mathfrak{H}\times\mathfrak{H}}, u,v \in\dom(\tr),\ 
J:=\begin{bmatrix}
0 & I_{\mathfrak{H}} \\
-I_{\mathfrak{H}}& 0
\end{bmatrix}. 
\end{align}
The trace operator $\tr$, geared to facilitate abstract integration by parts arguments, is a central object in our setting. 

A typical realization of this setup is given by the Laplace operator  $A:=-\Delta$ with domain $\dom(A)=H^2_0(\Omega)$ acting in $\cH:=L^2(\Omega)$ and the trace map $\tr u=(u\upharpoonright_{\partial\Omega},\sel{ -\Phi} \partial_{\nu} u\upharpoonright_{\partial\Omega})$\footnote{\sel{where $\Phi$ denotes natural Riesz isomorphism $\Phi\in\cB (H^{-1/2}(\partial \Omega), H^{1/2}(\partial \Omega))$ as defined in \eqref{aub27}}} defined on $\dom(\tr)=\{u\in H^1(\Omega): \Delta u\in L^2(\Omega)\}$. In this case $A^*=-\Delta$ with the domain $\dom(A^*)=\{u\in L^2(\Omega): \Delta u\in L^2(\Omega)\}$, the boundary space  $\mathfrak H= H^{1/2}(\partial \Omega)$, and \eqref{int1} is the standard Green identity. Equipping $\cH_+:=\dom(A^*)$ with the graph norm of the Laplacian and $\cD:=\dom(T)$ with the norm $(\|u\|^2_{H^1(\Omega)}+\|\Delta u\|^2_{L^2(\Omega)})^{1/2}$, we get a crucial dense embedding $\cD\hookrightarrow\cH_+$. This embedding becomes equality in the one-dimensional setting when $\Omega=[a,b]\subset \bbR$; in fact, one has $\cH_+=\cD=H^2([a,b])$.  

 Motivated by this example and returning to the abstract setting, we equip $\cD=\dom(\tr)$ with an abstract Banach norm $\|\cdot\|_{\cD}$, the space $\cH_+=\dom(A^*)$ with the graph norm of  $A^*$, and assume  that the embedding $\cD\hookrightarrow\cH_+$ is dense and bounded. Drawing further parallels between the abstract and the PDE/ODE settings, throughout this work we distinguish between the strict inclusion $\cD\subsetneq \cH_+$ and the equality $\cD=\cH_+$.  The case when $\cD$ is strictly contained in $\cH_+$ is closely related the setting considered \sel{in the pioneering paper by V. Derkach, M. Malamud \cite{DM95}, where the concept of generalized triplet was originally introduced,} and to  the notion of quasi-boundary triplets extensively studied  \sel{ in the work of} J. Behrndt and M. Langer \cite{BL, BL2012}, 
J. Behrndt and T. Micheler \cite{BM}, V.\ Derkach, S.\ Hassi, M.\ Malamud and H.\ de Snoo \cite{DHM2017, DHM2020, DHM2022, DHMS2006, DHMS2012}. In case when $\cD=\cH_+$ the triplet $(\mathfrak{H}, \Gamma_0,\Gamma_1)$ is called the
{\em \sel{ordinary} boundary triplet}. This case is understood much better and was developed, in particular, in the classical work  by V. Gorbachuk and M. Gorbachuk \cite{GG} and A.\ Kochubej, by V. Derkach and M. Malamud \cite{DM}
and many others, see, e.g., \cite{Behrndt_2020, BL2012, DHMS2012, DMbook, Schm} and the extensive bibliography therein. The main reason why we consider a non-surjective embedding $\cD\hookrightarrow \cH_+$ is that, when applied to elliptic operators, it allows one to use the standard Dirichlet and Neumann trace operators as components of $\tr$ and therefore discuss physically relevant boundary value problems (e.g., heat equation on bounded domains). The disadvantage of the condition $\cD\subsetneq \cH_+$, however, is that it restricts the class of admissible self-adjoint extensions of $A$ to those with domains containing in $\cD$. We refer to \cite{Calkin1939, DHMS2006, DHMS2012, DHM2017, DHM2020, DHM2022,HS2012,W2012a} for an in-depth study of unbounded traces and stress that abstract results of this type are not the main focus of the current work. On the other hand, the case of ordinary boundary triplets $\cD=\cH_+$ covers all possible self-adjoint extensions at the expense of dealing with the trace map $\tr$ which, when considered in the context of second order elliptic partial differential operators, is a non-local first order operator on the boundary of the spatial domain. The trace maps of this type have been studied, in particular, by G. Grubb \cite{Gr}, H. Abels,  G. Grubb, and I. Wood \cite{AGW}, F. Gesztesy and M. Mitrea \cite{GM08}, \cite{GMjde09}, \cite{GM10}. 

The ordinary boundary triplets are particularly well suited for ordinary differential operators and quantum graphs; we will exploit this in Section \ref{abt}.  Our approach allows one to obtain some new results that are not reachable or very hard to obtain using other methods such as the quadratic forms. This includes our arguably new Riccati-type differential equations for the resolvents, our ability to handle quite general boundary conditions for quantum graphs where the form method results are not known, our new and convenient formulas for the slopes of the eigenvalue curves for both quantum graphs with general boundary conditions and the PDE operators, as well as our ability to handle non-local boundary conditions (even of generalized Robbin type but also such as those that appear in describing Krein's self-adjoint extensions of PDE operators).  

Having introduced the notion of an abstract trace map and Green identity \eqref{int1}, we switch to a symplectic version of the resolvent difference formula. We note that the right-hand side of \eqref{int1} can be written as $\omega(\tr u, \tr w)$, where $\omega(\cdot,\, \cdot)=\langle J \cdot,\,\cdot\rangle_{\mathfrak{H}}$ is the natural symplectic form. It is well known that self-adjoint extensions of $A$ in $\cH$ can be described by Lagrangian planes in various symplectic Hilbert boundary spaces. W.\ N.\ Everitt and W.\ N.\ Markus \cite{MR1955204}, B.\ Booss-Bavnbek and K.\ Furutani \cite{BbF95}, for example, relate self-adjoint extensions to Lagrangian subspaces of the symplectic quotient space $\dom(A^*)/\dom(A)$, while J.\ Behrndt and M.\ Langer \cite{BM}, K.\ Pankrashkin \cite{Pa}, K.\ Schm\"{u}dgen \cite[Chapter 14]{Schm} and \cite{Behrndt_2020}, on the other hand, discuss self-adjointness in terms of linear relations. Closely following these works, we utilize the abstract Green identity \eqref{int1} assuming (possibly, non-surjective) embedding $\cD\hookrightarrow \cH_+$, and associate self-adjoint extensions $\cA$ of $A$  to Lagrangian planes $\cG\subset \bndra$ via the mapping $\dom(\cA)\mapsto \cG:=\overline{\tr(\dom(\cA))}$, see Theorems \ref{LLSA}, \ref{SALL}  and Corollary \ref{LLSASALL} for more details on this correspondence. This observation brings us one step closer to the perturbation theory for self-adjoint extensions with continuously varying domains of self-adjointness as it allows us to recast this non-additive perturbation problem in terms of the perturbation of Lagrangian planes, or more specifically, in terms of perturbation of the orthogonal projections onto the planes. 

A major issue in perturbation theory for unbounded operators with varying domains is that their difference could be defined  on a potentially very small subspace, e.g. on the zero subspace. This issue is not as severe when one talks about self-adjoint extensions $\cA_1, \cA_2$ of the same operator $A$, since $\dom(A)\subset \dom(\cA_1)\cap \dom(\cA_2)$ but there is still a caveat: the difference $\cA_1-\cA_2$ could be the zero operator, hence, $\cA_1$, $\cA_2$ could be trivial additive perturbations of one another (again, think about the Dirichlet and Neumann realizations of the second derivative on a segmen). To deal with this issue,  one considers instead of $\cA_1-\cA_2$ the difference of the {\em resolvents} $(\cA_1- \zeta)^{-1}-(\cA_2-\zeta)^{-1}$ and, classically, expresses it in terms of the abstract Weyl $M$-function, see Appendix \ref{Sec4.1} and, in particular, Proposition \ref{prop:KrNaF} for a brief reminder of this topic. Such an expression is called the {\em Krein (or Krein-Naimark) resolvent formula}; we refer to
\cite{MR0024574, MR0024575} and \cite{KL1973, LT1977}.

This foundational result in spectral theory has been studied and derived in various settings by many authors; we refer to the texts \cite{AhkGlazman, Behrndt_2020,Schm} where one can find a detailed historical account and further bibliography. Without even attempting to give a review of the vast literature on this subject, we mention here the work 
by  H. Abels, G. Grubb and I. Wood \cite{AGW}, W.O Amrein and D.B. Pearson \cite{MR2077193}, S. Albeverio and K.  Pankrashkin \cite{MR2148628}, J. Behrndt and M. Langer \cite{BL}, S. Clark, F. Gesztesy, R. Nichols, and M. Zinchenko \cite{ClarkGesNickZinch}, V. Derkcach and M. Malamud \cite{DMbook}, \cite{DM}, F. Gesztesy and M. Mitrea \cite{GM08}, \cite{GMjde09}, \cite{GM10}, G. Grubb \cite{MR2537622}, A. Posilicano \cite{MR2468877}, A. Posilicano and L. Raimondi \cite{MR2525266}.  We specifically mention important contribution for the case of quasi-boundary triplets in \cite[Theorem 5.1]{BL} and in more complete form  in Theorem 6.16 and Corollary 6.17 of \cite{BL2012}; for generalized boundary triplets of bounded type in Theorem 7.26 and Proposition 7.27 of the paper \cite{DHMS2012} by V.\ Derkach, S.\ Hassi and M.\ Malamud; for so-called AB-generalized boundary triplets (which covers the previous two cases) in Theorem 4.12, Remark 4.13, Corollary 4.14 of \cite{DHM2020}. In addition, in a recent paper \cite{DHM2022} by V.\ Derkach, S.\ Hassi and M.\ Malamud (see also \cite{DHM2017}) the authors studied boundary triplets and gave an analytic characterization of their Weyl functions as form domain invariant Nevanlinna functions. These papers contain applications of boundary triplets techniques closely related to the results in Sections \ref{ssLapLip} and \ref{SS5.1} of the present paper. 
Most closely related to our work is the Krein formula for two arbitrary self-adjoint extensions of the Lapalce operator expressing the resolvent difference in terms of an operator valued Herglotz function that has been obtained in \cite{GM10}, see also \cite{GLMZ05,GM08,Nakam01}. 

However, all above-mentioned Krein-type formulas are not quite suited for the purposes of the current paper as they do not capture quantitatively the perturbations of operator-theoretic domains of the self-adjoint extensions in the form that we need. One of the main objectives of the current  work is to address this issue. 
 Specifically, we propose to use a very elementary new {\em resolvent difference
 formula} expressing the difference of the resolvents of two arbitrary self-adjoint extensions of a given symmetric operator in terms of the {\em projections onto the Lagrangian planes} determining the domains of  the extensions. As far as we can see this simple but extremely handy version of the formula was not widely used in the literature in the generality that we offer, see, however, already mentioned Theorem 2 and Corollary 4 in \cite{DM}.

Indeed, for arbitrary self-adjoint extensions $\cA_1, \cA_2$ of a symmetric operator $A$, we obtain the following symplectic version of the formula for the difference of  resolvents $R_1(\zeta)=(\cA_1- \zeta)^{-1}$ and  $R_2(\zeta)=(\cA_2- \zeta)^{-1}$,
\begin{align}\lb{int2}	
R_1(\zeta)-R_2(\zeta)&=\big(\tr R_{2}(\overline{\zeta})\big)^*\, Q_2JQ_1(\tr R_{1}({\zeta})), 
\end{align}
where  $\zeta\not\in\spec(\cA_1)\cup\spec(\cA_2)$, $J$ is the symplectic matrix from \eqref{int1}, $Q_1, Q_1\in \cB(\bndra)$ are the orthogonal projections onto the Lagrangian planes $\cG_1, \cG_2\subset \bndra$ defining the self-adjoint extensions $\cA_1, \cA_2$ via $\cG_1=\overline{\tr(\dom(\cA_1))}$, $\cG_2=\overline{\tr(\dom(\cA_2))}$.  In particular,  using the property $Q_1JQ_1=0$, a key property of projections onto Lagrangian planes, formula \eqref{int2} yields
\begin{align}\lb{int4}
R_2({\zeta})-R_1({\zeta})&=\big(\tr R_{2}(\overline{\zeta})\big)^*\, (Q_2-Q_1)JQ_1(\tr R_{1}({\zeta})),
\end{align}
which indicates that  $\|R_2({\zeta})-R_1({\zeta})\|_{\cB(\cH)}\rightarrow0$ whenever $\|Q_2-Q_1\|_{\cB(\bndra)}\rightarrow0$, see Theorem \ref{thm1.7}. Also, we rewrite the resolvent difference formula \eqref{int4} in terms of bounded operators $X_k, Y_k\in\cB(\mathfrak{H})$ chosen such that $\cG_k=\ker[X_k, Y_k]$, $k=1,2$, see \eqref{nn5.14a}. 

Relying on the resolvent difference formula \eqref{int4}, we investigate differentiability properties and obtain asymptotic expansion for resolvent operators as functions of a scalar parameter $t\in[0,1]$ parametrizing sufficiently smooth paths of Lagrangian planes $t\mapsto \cG_t$, additive bounded self-adjoint perturbations $t\mapsto V_t\in\cB(\cH)$, and trace maps $t\mapsto \tr_t$ satisfying Green identity \eqref{int1}. That is, we develop a full scale first order asymptotic theory for a one parameter family of self-adjoint operators $H_t:=\cA_t+V_t$, with $\cA_t$ being a self-adjoint extension of $A$ associated with the Lagrangian plane  $\cG_t$ via the relation $\overline{\tr_t(\dom(\cA_t))}=\cG_t$. First, we prove that, respectively, continuity, Lipschitz continuity, and differentiability at $t_0\in[0,1]$ of the paths of Lagrangian planes, bounded perturbations, and trace maps, yields continuity, Lipschitz continuity, and differentiability, respectively, of the path of resolvent operators $t\mapsto R_t(\zeta):= (H_t-\zeta)^{-1}$, $\zeta\not \in\spec(H_{t_0})$. At the first glance such results should seemingly follow from the resolvent difference formula \eqref{int4} as it suggests that $R_t({\zeta})-R_{t_0}({\zeta})$ and $Q_t-Q_{t_0}$ are of the same order. It turns out, however, that the boundedness of the appropriate norm of $TR_t(\zeta)$ for $t$ near $t_0$ could be a subtle issue depending on whether we are dealing with the strict inclusion $\cD\subsetneq \cH_+$ or the equality $\cD= \cH_+$. 

Let us elaborate on this in more detail. First, the operator  $TR_t(\zeta)$ is bounded as a linear mapping from $\cH$ to $\bndra$, i.e.   $TR_t(\zeta)\in \cB(\cH, \bndra)$ even without assuming that $\cD=\dom(\tr)$ is equipped with its own Banach norm, see Proposition \ref{prop1.3}. When it is, however, we claim more: $\tr\in \cB(\cD, \bndra)$ and $R_t(\zeta)\in\cB(\cH, \cD)$, see Proposition \ref{prop2.4}.
The main issue is that in the abstract setting one does not have a good quantitative control of the norm  $\|R_t(\zeta)\|_{\cB(\cH, \cD)}$ as a function of $t$. We, therefore, impose the assumption
\begin{equation}\lb{int10}
\|R_t(\zeta)\|_{\cB(\cH, \cD)}\underset{t\rightarrow t_0}{=}\cO(1). 
\end{equation}
That being said,  condition \eqref{int10} is automatically satisfied when the strict inclusion $\cD\subsetneq \cH_{+}$
 is replaced by the equality $\cD=\cH_{+}$, in which case we show not only boundedness \eqref{int10} but also continuity of the reslovent operators
\begin{equation}\lb{int15}
\|R_t(\zeta)-R_{t_0}(\zeta)\|_{\cB(\cH, \cD)}\underset{t\rightarrow t_0}{=}o(1), 
\end{equation}
see Proposition \ref{prop2.4h}. We stress that \eqref{int10} is a natural assumption for the case when  $\cD\subsetneq \cH_{+}$. This assumption is satisfied, although not trivially, in many PDE contexts of interest as its proof essentially boils down to controlling $L^2(\Omega)$ to $H^1(\Omega)$ norm of the reslovent of a second order elliptic operator for $t$ near $t_0$, see Section \ref{ssKKFREO} where we check it for elliptic operators subject to Robin boundary conditions. To sum up, the resolvent difference formula \eqref{int4} together with hypothesis \eqref{int10} yield continuity of the resolvent operators $t\mapsto R_t(\zeta)$. The differentiability requires not only \eqref{int10} but actually \eqref{int15} that we impose as an assumption when  $\cD\subset \cH_{+}$. As we already pointed out \eqref{int15} holds automatically if $\cD=\cH_+$ and it holds in most standard PDE realizations of a more general situation $\cD\subsetneq\cH_+$.

Having discussed differentiability of the mapping $t\mapsto R_t(\zeta)$ we now switch to first order asymptotic expansions of the resolvents.  The main goal of this part of the paper is to derive an Hadamard-type formula\footnote{As we have already noted above, we borrow the term {\em Hadamard-type formula} from the PDE literature on geometric perturbations of spatial domains and use it for general formulas for  derivatives of eigenvalues} for derivatives of the eigenvalues curves of $H_t$.  As a first step, we derive in Theorem \ref{prop1.8new}  the following asymptotic expansion  for the resolvent,
\begin{align}
\begin{split}\lb{int20}
R_t(\zeta)\underset{t\rightarrow t_0}{=}R_{t_0}(\zeta)&+
\big(-R_{t_0}(\zeta)\dot V_{t_0}R_{t_0}(\zeta)+(\tr_{t_0}  R_{t_0}(\overline{\zeta}))^*\dot Q_{t_0}J \tr_{t_0}    R_{t_0}(\zeta)\\
&+(\tr_{t_0}  R_{t_0}(\overline{\zeta}))^*J\dot \tr_{t_0}  R_{t_0}(\zeta)\big)(t-t_0)+o(t-t_0),\ \text{\ in\ }\cB(\cH);
\end{split}
\end{align}
here and throughout the paper $\frac{d}{dt}$ is abbreviated by the dot, e.g., $\dot V_{t_0}=\frac{dV}{dt}|_{t=t_0}$. In particular, we deduce a new Riccati-type differential equation for the resolvents,
\begin{equation}
\begin{split}
\dot{R}_{t_0}(\zeta)=-R_{t_0}(\zeta)\dot V_{t_0}R_{t_0}(\zeta)&+(\tr_{t_0}  R_{t_0}(\overline{\zeta}))^*\dot Q_{t_0}J \tr_{t_0}    R_{t_0}(\zeta)\\
&+(\tr_{t_0}  R_{t_0}(\overline{\zeta}))^*J\dot \tr_{t_0}  R_{t_0}(\zeta).
\end{split}
\end{equation}

Next, we compute the slopes of eigenvalue curves $\{\lambda_j(t)\}_{j=1}^m$ bifurcating from an isolated eigenvalue $\lambda\in\spec (H_{t_0})$ of multiplicity $m\geq 1$. Our strategy is to integrate \eqref{int20} over a contour $\gamma\subset \bbC$ enclosing the eigenvalues $\{\lambda_j(t)\}_{j=1}^m$ for $t$ near $t_0$, obtain an asymptotic expansion for the $m-$dimensional operator $P(t)H_tP(t)$, where $P(t)$ is the Riesz  projector onto the spectral subspace $\ran(P(t))=\bigoplus_{j=1}^m\ker(H_t-\lambda_j(t))$, and reduce matters to asymptotic perturbation techniques for finite dimensional self-adjoint operators. Specifically,  we employ the body of finite dimensional results from Theorem II.5.4 and Theorem II.6.8 of \cite{K80}.
In the literature on Maslov index and spectral flow these results are called the {\em Kato selection theorem}, cf.\ \cite[Theorem 4.28]{RoSa95}, 
as they allow one to properly choose the $m$ branches of the eigenvalue curves for $P(t)H_tP(t)$ and compute their slopes. 
A subtle issue in this scheme, though, is that the finite dimensional operators  $P(t)H_tP(t)$ are defined on varying $t-$dependent spaces $\ran(P(t))$. As in \cite{LS17},  we remedy this by introducing a differentiable family of unitary operators  $t\mapsto U_t$, cf. \eqref{int27}, \eqref{int28}, mapping  $\ran(P(t_0))$ onto $\ran(P(t))$ and obtain the first order expansion for unitarily equivalent to $P(t)H_tP(t)$ operators acting in a fixed finite-dimensional space $\ran(P(t_0))$, see Lemma \ref{theorem2.2}. Finally, utilizing this expansion and the Kato selection theorem we show that there is a proper labelling of the eigenvalues $\{\lambda_j(t)\}_{j=1}^m$ of $H_t$ for $t$ near $t_0$ and an orthonormal basis $\{u_j\}_{j=1}^m\subset\ker( H_{t_0}-\lambda)$ such that the following  Hadamard-type formula holds, 
\begin{equation}\lb{int30}
\dot\lambda_j(t_0)=\langle \dot V_{t_0}u_j, u_j\rangle_{\cH}+\omega(\dot Q_{t_0}\tr_{t_0} u_j, \tr_{t_0} u_j)+\omega( \tr_{t_0}u_j, \dot \tr_{t_0} u_j), 1\leq j\leq m,
\end{equation}
where $\omega(f,g)=\left\langle
J f,g\right\rangle_{\mathfrak{H}\times\mathfrak{H}}$, $f,g\in\bndra$ is the symplectic form. This quite general result is one of the major points of the paper; we apply it in several particular situations.

Also, we use this computation to give an infinitesimal version of 
a general abstract analogue of
the classical formula, cf.\ \cite{BbF95, BZ1,CLM}, relating the following two quantities: (1) the Maslov index of the path $t\mapsto \cG_t\oplus \tr \big(\ker\big(A^*+V_t-\lambda\big)\big)$ relative to the diagonal plane in $\mathfrak{H}\times\mathfrak{H}$, and (2) the spectral flow of the family $t\mapsto H_t$ through $\lambda$ for $t$ near $t_0$.  Heuristically, the latter quantity  is given by the difference between the number of monotonically increasing and decreasing eigenvalue curves of $H_t$ bifurcating from $\lambda$. The former quantity is equal to the signature of the Maslov form which is a certain bilinear form defined on  $\tr\big( \ker(H_{t_0}-\lambda)\big)$, see Sections \ref{mastriples} and \ref{section5.4}. In order to relate the two, we prove by computation
that, in fact, the value of the Maslov crossing form coincides with the right-hand side of \eqref{int30}, cf.  Theorem \ref{HadmardSimple} and Proposition \ref{HadmardSimple5}. Similar relations have been established, in particular, by G. Cox, C.K.R.T. Jones, J. Marzuola  in \cite{CJM1, CJM2}, B. Boo\ss -Bavnbek, C. Zhu \cite{BZ1},  B. Boo\ss -Bavnbek, K. Furutani \cite{BbF95},P. Howard, A. Sukhtayev \cite{MR3437596, MR4043663}. The computational and applied aspects of the Maslov index theory have recently been considered by F. Chardard, F. Dias, T. J.  Bridges \cite{MR2277067, MR2598511, MR2524837, MR2831770}

In a later part of the paper we also give a  generalization of the resolvent difference formula to the case of 
{\em adjoint pair} of operators,
 see, e.g., \cite{AGW, BGW, BMNW} and the literature cited therein. Important contributions to the theory of adjoint pairs can be found in \cite{A12, MHMNW, MM02}. It allows one to describe non-selfadjoint extensions for an adjoint pair of densely defined closed (but not necessarily symmetric) operators. A typical example of the adjoint pair is given by a non-symmetric elliptic second order partial differential operator and its formal adjoint; this example is also discussed in the paper.

\subsection{Summary of applications} 
Our applications  are given in Sections \ref{abt} and \ref{subsec1.1}. In Section \ref{abt}
we collected all results pertaining  the \sel{ordinary} boundary triplets (covering the case of metric graphs, and ``rough'' PDE traces). This section also provides more applications of the asymptotic expansions of resolvents in the context of ordinary boundary triplets obtained by the authors in \cite{LSTaylor2}.  In Section \ref{subsec1.1}  we deal with more general case of densely defined not surjective traces (which covers the ``weak'' PDE traces). Our main applications are to spectral count for Robin Laplacians on bounded domains, periodic Kronig--Penney models,  Hadamard-type formulas for Schr\"odinger operators on metric graphs, and heat equation posted on bounded Lipschitz domains. Let us succinctly describe relevant results.

$\bullet$ We prove that for Baire almost every periodic sequence of coupling constants $\alpha=\{\alpha_k\}_{k=1}^{\infty}\in \ell^{\infty}(\bbZ, \bbR)$ the spectrum of the Schr\"odinger operator $H_{\alpha}$ acting in $L^2(\bbR)$ and given by 
\begin{align}
&H_{{\alpha}}:=-\frac{\bd^2}{\bd x^2}+\sum_{k\in\bbZ}\alpha_k \delta(x-k),
\end{align}
has no closed gaps, see Section \ref{ssPerKrPen}. The analogous assertion for Schr\"odinger operators $H_V=-\frac{d^2}{dx^2}+V$ for periodic $V\in C^{\infty}(\bbR)$ (due to B. Simon \cite{MR473321}) and their discrete versions have been instrumental in the works of D. Damanik, J. Fillman and M. Lukic \cite{MR3737889} and A. Avila \cite{MR2504859}, correspondingly, on Cantor spectra for generic limit-periodic Schr\"odinger operators. As in  \cite{MR473321}, we prove this statement by perturbation arguments applied to the Hill equation on a finite interval associated with $H_{\alpha}$ (an alternative approach covering a wide class second-order differential operators is proposed in the work of D. Damanik, J. Fillman and the second author). 

$\bullet$    For a general elliptic second order operator $\cL:=-\div(\mathtt{A}\nabla)+ \mathtt{a}\cdot \nabla -\nabla \cdot \mathtt{a} +\mathtt{q}$ posted on a bounded Lipschitz domain $\Omega\subset \bbR^d, d\geq 2$, see Section \ref{SS5.1}, and subject to a one parameter family of Robin conditions $\partial_{\nu}u=\Theta_t u$ on $\partial\Omega$, we derive Hadamard- and resolvent difference formulas, see Theorem \ref{theorem6.2}, and use these results to discuss in Section \ref{ssKKFREO} a unified approach to L.\ Friedlander's and J.\ Rohleder's  inequalities via a spectral flow argument, see \cite{Fri91,Rohl14} and \cite{CJM2}.  

$\bullet$ For an arbitrary compact metric graph $\cg$ and the Schr\"odinger operator $H_t=-\frac{d^2}{dx^2}+V$ subject to parameter dependent vertex conditions $X_t u+Y_t\partial_{n}u=0$ (here $\partial_{n}u$ is the derivative of $u$ taken in the inward direction along each edge), we derive the following Hadamard-type formula for the slopes of eigenvalue curves $\{\lambda_j(t)\}_{j=1}^m$ bifurcating from an eigenvalue of $H_{t_0}$ of multiplicity $m\geq 1$,
\begin{equation}\lb{5.10newnew24a}
\dot{\lambda}_j(t_0)= \langle \dot V_{t_0}u_j, u_j\rangle_{L^2(\cg)}+\big\langle (X_{t_0}\dot{Y}_{t_0}^*-Y_{t_0}\dot{X}_{t_0}^*) \phi _j, \phi_j\big\rangle_{L^2(\partial\cg)}, 
\end{equation}
where  $ \{u_j\}_{j=1}^m$ is a certain orthonormal basis of  $\ker(H_{t_0}-\lambda(t_0))$,  $\phi_{j}$ is a unique vector in $L^2(\partial\cg)$ satisfying $ u_j=-Y_{t_0}^*\phi_j$ and $\partial_n u_j=X_{t_0}^*\phi_j$, $1\leq j\leq m$, see Section \ref{ssQuaGr}.  In the theory of quantum graphs, Hadamard-type formulas are often derived on a case-by-case basis for simple eigenvalue curves, see, for example a classical monograph by G. Berkolaiko and P. Kuchment \cite[Section 3.1.4.]{BK}; \eqref{5.10newnew24a} closes this gap in the literature. In addition, we derive a resolvent difference formula expressing the difference of two arbitrary self-adjoint realizations of the Schr\"odinger operator in terms of the vertex matrices $X_j, Y_j$, $j=1,2$.

$\bullet$ For the heat equation 
\begin{equation} 
\begin{cases}
&u_\mathtt{t}(\mathtt{t},x)=\kappa \rho(x) \Delta_x u (\mathtt{t},x), x\in\Omega,
\mathtt{t}\ge0,\\
&-\kappa \frac{\partial  u}{\partial n}= u,\text{\ on\ }\partial \Omega,
\end{cases}
\end{equation}
describing the temperature $u$ of a material in the region $\Omega\subset \bbR^3$ with thermal conductivity $\kappa$ immersed in a surrounding medium of zero temperature (here $1/\rho(x)$ is the product of the density of the material times its heat capacity), we give a new proof of continuous dependence of $u$ on $\kappa$ with respect to $L^2(\Omega)$ norm, see Section \ref{ssHeatEq}. 

 The symplectic (Lagrangian) point of view on self-adjoint extensions and boundary triplets systematically used in this paper (and a more general approach via Krein spaces, cf.\ \cite{DHMS2006}) is a quite powerful tool that of course brings up many new and unresolved issues. Among the open questions we mention: finding a symplectic interpretation of the abstract Weyl's function; describing exit-space extensions using symplectic approach; studying (in the context of self-adjoint extensions) so called lateral perturbations introduced in \cite{BK22}; relating Hadamard-type formulas to the secular equations \cite{BK} for quantum graphs.

{\bf Organization of the paper.} In Section \ref{KreinFormulas} we begin with basic setup, discuss properties of the trace operators and their composition with the resolvents for the general case when the imbedding $\cD\hookrightarrow\cH_+$ is not surjective.
The most general symplectic resolvent difference formula for the difference of resolvents of any two self-adjoint extensions is proven in Theorem \ref{thm1.7}. 
In Section \ref{Sec3.1} we discuss our main setup and assumptions on one-parameter families of traces, self-adjoint extensions, and bounded perturbations, and provide typical examples when our assumptions are
satisfied. The examples include: Schr\"odinger operators with Robin-type boundary conditions on families of star-shaped domains, second order operators on infinite cylinders with variable multidimensional cross-sections, operators arising as Floquet--Bloch decomposition of periodic Hamiltonians, and first order elliptic operators of Cauchy-Riemann type on cylinders. In Section \ref{Sec3.2} we obtain general resolvent expansions and derive the Riccati equations for the resolvent operators. The variational Hadamard-type formula for the eigenvalue curves is proven in Section \ref{Sec3.3}. This section also contains resolvent difference formulas for families of self-adjoint extensions given by either families of projections in the boundary space $\mathfrak{H}\times\mathfrak{H}$ or as kernels of the bounded  row-operators $[X_t, Y_t]$. In Section \ref{Sec4.2} we formulate our major results for the case $\cD=\cH_+$, that is, for the \sel{ordinary} boundary triplets. As an example, we treat the ODE case of Robin boundary conditions on a segment. In Section \ref{ssLapLip}  we study Robin Laplacian on multidimensional domains in the framework of the boundary triplets which requires the use of the ``rough'' traces. Section \ref{ssQuaGr} is devoted to applications to quantum graphs, here, in particular, we derive Hadamard-type formula \eqref{5.10newnew24a}. The periodic Kronig-Penney model is considered in Section \ref{ssPerKrPen}. In Section \ref{mastriples} we begin discussion on connections to the Maslov index and prove a general result relating the value of the Maslov crossing form and the slope of the eigenvalue curves for \sel{ordinary} boundary triplets. In Section \ref{SS5.1} we switch to the second order elliptic operators, return back to the case $\cD\subsetneq\cH_+$, and use weak boundary traces. Hadamard-type and resolvent difference  formulas for  Robin realizations, Friedlander's and Rohleder's theorems are discussed in Section \ref{ssKKFREO}. Applications to the heat equation are given in Section \ref{ssHeatEq}.
In Section \ref{HFSSD} we derive from our general results the classical Hadamard--Rellich formula for the eigenvalues of the Schr\"odinger operator posted on a family of star-shaped domains. The Maslov crossing form for elliptic operators defined by means of the weak solutions is studied in Section \ref{section5.4}. In Section \ref{sec:dualpairs} we provide generalizations of the resolvent difference formula to the case of an {\em adjoin pair} of operators. This results are applied to the example of an elliptic second order partial differential operator and its formal adjoint. In Appendix \ref{AppA} we give a detailed discussion of the correspondence between the Lagrangian planes in the boundary space $\mathfrak{H}\times\mathfrak{H}$ and the domains of the self-adjoint extensions. We introduce and study the notion of  {\em aligned subspaces}  and show that for these the correspondence is a bijection. Appendix \ref{Sec4.1} shows how to derive the classical Krein's formulas involving the $M$-function from the new symplectic version of the resolvent difference formula that we offered in the paper. Finally, Appendix \ref{appA} contains some well known material regarding PDE boundary traces.

{\bf Notation.} We denote the space of bounded linear operators acting between two Banach spaces $\cX$ and $\cY$ by  $\cB(\cX, \cY)$ and let $\cB(\cX):=\cB(\cX, \cX)$. The closure of an operator $T: \cX\rightarrow \cY$ is denoted by $\overline{T}$. We denote by $\spec(T)$ the spectrum, by $\spec_{\rm disc}(T)$ the set of isolated eigenvalues of finite algebraic multiplicity, and by $\spec_{\rm ess}(T)=\spec(T)\setminus\spec_{\rm disc}(T)$ the essential spectrum of $T$. The scalar product 
  (linear with respect to the {\it first} argument) and the norm on a Hilbert space $\cH$ are denoted by $\langle \cdot, \cdot\rangle_{\cH}$ and $\|\cdot\|_{\cH}$ respectively. When $\cH$ is a Hilbert space, we denote the space of bounded {\it linear} functionals on $\cH$ by $\cH^*$ and define a {\it conjugate-linear} Riesz isomorphism by $\Phi: \cH^*\mapsto \cH$, $\cH^*\ni\psi\mapsto \Phi_{\psi}\in\cH$ so that $_{\cH}\langle f, \psi\rangle_{\cH^*}:=\psi(f)=_{}\langle f, \Phi_{\psi} \rangle_{\cH}, f\in \cH$.  In the special case of Sobolev spaces $\cH=H^{1/2}(\partial\Omega)$ we set $\cH^*=H^{-1/2}(\partial\Omega)$ and denote $\langle f, \psi\rangle_{-1/2}:=_{H^{1/2}(\partial\Omega)}\langle f, \psi\rangle_{H^{-1/2}(\partial\Omega)} $ for $f\in H^{1/2}(\partial\Omega)$, $\psi\in H^{-1/2}(\partial\Omega)$. The closure of a subspace $S\subset \cH$ with respect to $\|\cdot\|_{\cH}$ is denoted by $\overline{S}^{\cH}$ while its orthogonal complement by $S^{\perp_\cH}$.  For operators $A,B\in\cB(\cX, \cY)$, we let $[A,B]\in\cB(\cX\times\cX, \cY)$, $[A,B](h_1, h_2)^{\top}:=Ah_1+Bh_2$, $h_1, h_2\in\cX$ and $[A,B]^{\top}\in\cB(\cX, \cY\times\cY)$, $[A,B]^{\top}(h):=(Ah, Bh)^\top$, $h\in\cX$, where $\top$ stands for transposition. 
  We denote by $\Lambda(\cX\times\cX)$ the set of Lagrangian subspaces in $\cX\times\cX$ equipped with the symplectic form $\omega$ induced by the operator $J=\left[\begin{smallmatrix}0 & I_\cX\\-I_\cX&0\end{smallmatrix}\right]\in\cB(\cX\times\cX)$. Given an operator valued function $f:\bbR\rightarrow\cB(\cX)$, we write $f(t)=o((t-t_0)^n)$ as $t\rightarrow t_0$ if $\|f(t)\|_{\cB(\cX)}|t-t_0|^{-n}\rightarrow 0$ as  $t\rightarrow t_0$. Similarly, $f(t)=\cO((t-t_0)^n)$ as $t\rightarrow t_0$ whenever $\|f(t)\|_{\cB(\cX)}|t-t_0|^{-n}\leq c$ for some $c>0$ and all $t\not = t_0$ in some open interval containing $t_0$. We denote by $\mathbb{B}_r(\zeta)$ the disc in 
  $\bbC$ of radius $r$ centered at $\zeta$ and by $\mathbb{B}^n_r$ the ball in $\bbR^n$ of radius $r$ centered at zero.

{\bf Acknowledgments.}\ YL was supported by the NSF grants
	DMS-1710989, DMS-2106157, and would like to thank the Courant Institute of Mathematical Sciences and especially Prof.\ Lai-Sang Young for the opportunity to visit CIMS. \sel{SS were supported in part by NSF grant DMS-2243027, Simons Foundation grant MP-TSM-00002897, and by the Office of the Vice President for Research \& Economic Development (OVPRED) at Auburn University through the Research Support Program grant}. Both authors gratefully acknowledge support from the Simons Center for Geometry and Physics, Stony Brook University, where a part of this research was completed during the workshop ``Ergodic Operators and Quantum Graphs". 
	
	 We are thankful to 
	M. Ashbaugh, V. Derkach, C. Gal, F. Gesztesy, R. Schnaubelt for pointing out very useful sources in the literature. 
	\sel{We are grateful to the referees of  this paper for their insightful suggestions and deep mathematical remarks. We are especially grateful for the suggestion incorporated as Remark  \ref{rem:ref1}.}

\section{A symplectic resolvent difference formula}\lb{KreinFormulas}
Let $\cH, \mathfrak{H}$ be complex, separable Hilbert spaces. Let $A$ be a densely defined, closed, symmetric operator acting in $\cH$ and having equal (possibly infinite) deficiency indices, that is, 
\begin{equation}
\dim\ker(A^*-\bfi)=\dim\ker(A^*+\bfi).
\end{equation}
We denote $\cH_+=\dom(A^*)$ and equip this Hilbert space
with the graph scalar product 
\begin{equation}
\langle u, v\rangle_{\cH_+}:=\langle u, v\rangle_{\cH}+\langle A^*u, A^*u\rangle_{\cH}, \ u,v\in\dom(A^*).
\end{equation}
 Let $\cH_-=(\cH_+)^*$ denote the space adjoint to $\cH_+$ with
\begin{equation}\lb{aub100}
\cH_+\hookrightarrow\cH\hookrightarrow\cH_-,
\end{equation}
where the first embedding is given by $\cH_+\ni u\mapsto u\in\cH$, and the second embedding is given by $\cH\ni v\mapsto \langle\cdot, v\rangle_{\cH}$. Let $\Phi^{-1}:\cH_+\to\cH_-$ be the Riesz isomorphism such that \[{}_{\cH_+}\langle u,\Phi^{-1}w\rangle_{\cH_-}=\langle u,w\rangle_{\cH_+}=
\langle u,w\rangle_{\cH}+\langle A^*u,A^*w\rangle_{\cH}, u,w\in\cH_+.\]The following hypothesis will be assumed throughout the rest of the paper. 
\begin{hypothesis}\lb{hyp3.6} We assume that $A$ is a densely defined, closed, symmetric operator acting in $\cH$ and having equal (possibly infinite) deficiency indices. Suppose that $\cD$ is a core for $A^*$, that is, $\cD$ is a dense subspace of $\cH_+$ with respect to the graph norm of $A^*$, and assume that $\dom(A)\subset \cD$. Consider a linear operator
\begin{equation}\lb{eq1}
\tr:= [\Gamma_0, \Gamma_1]^\top: \cH_+\rightarrow \mathfrak{H}\times\mathfrak{H} \text{ such that $ \dom(\tr)=\cD$, $\overline{\ran(\tr)}=\bndra$}
\end{equation}	
	 called the {\em trace operator}. Assume that $\tr$  satisfies the following abstract Green identity,
	\begin{equation}\lb{3.61}
	\langle A^*u,v\rangle_{\cH}-\langle u,A^*v\rangle_{\cH}=\langle\Gamma_1u,\Gamma_0v\rangle_{\mathfrak{H}}-\langle\Gamma_0u,  \Gamma_1v\rangle_{\mathfrak{H}} \text{ for all $u,v\in\cD$}.
	\end{equation}
\end{hypothesis}
A simple but very important setting satisfying Hypothesis \ref{hyp3.6} is given by \sel{ordinary} boundary triplets, cf., e.g., \sel{\cite{Behrndt_2020, DM95}}, in which case one lets $\cD=\dom(A^*)=\cH_+$ and one can always define a Hilbert space $\mathfrak{H}$ and a trace operator $\tr$ satisfying \eqref{3.61}. This scenario is discussed in Section \ref{abt} below. Yet more elaborate setting, which is more suitable for PDEs,  is discussed in Section \ref{subsec1.1} where Hypothesis \ref{hyp3.6} holds with $\cD\subsetneq\dom(A^*)$ being a proper subset of $\cH_+$. 

\sel{
\begin{remark}
The notion of ordinary boundary triplets has been modified and generalized in several (similar but not equivalent) directions and applied to elliptic differential operators by multiple authors. The pioneering paper \cite{DM95} offered the first such generalization where $\Gamma_0$ was assumed to be surjective and the operator $A^*|_{\ker\Gamma_0}$ self-adjoint, see also  \cite{BL, BL2012, BM, DHM2017, DHM2020, DHM2022, DHMS2006, DHMS2012, DM}.
\end{remark}
}

In the following propositions we collect some elementary properties of the operator $\tr$ and its composition with the resolvent $R(\zeta,\cA)=(\cA-\zeta)^{-1}$ of a self-adjoint extension $\cA$ of $A$. 

\begin{proposition}\lb{remark2.2} Under Hypothesis \ref{hyp3.6}  the following assertions hold. 
	
	(1) \, $\dom(A)=\ker(\tr)$.
	
	(2) \, The operator $\tr:\cD\subset\cH_+\rightarrow\bndra $ defined in \eqref{eq1} is closable.

\end{proposition}
\begin{proof}
	(1)\, Identity \eqref{3.61} yields $\dom(A)\subseteq\ker(\tr)$. Indeed, pick an arbitrary $u\in\dom(A)$.  Since $\ran(\tr)$ is dense in $\bndra$, there is a sequence $v_n\in\cD$ such that $\tr v_n\rightarrow (\Gamma_1u, -\Gamma_0u)$. Using \eqref{3.61} and $u\in\dom(A)$, we infer $\langle\Gamma_1u,\Gamma_0v_n\rangle_{\mathfrak{H}}-\langle\Gamma_0u,  \Gamma_1v_n\rangle_{\mathfrak{H}}=0$. Passing to the limit yields $\|\Gamma_1u\|^2_{\mathfrak H}+\|\Gamma_0u\|^2_{\mathfrak H}=0$, hence $u\in\ker(\tr)$. The inclusion $\ker(\tr)\subseteq \dom(A)$ follows from \eqref{3.61}, density of $\cD$ in $\cH_+$, and the fact that $A^{**}=A$ (since $A$ is closed). 
	
	(2)\,   Suppose that a sequence $\{u_n\}_{n\in\bbN}$ converges to $0$ in $\cH_+$ while  \[\{(\Gamma_0u_n, \Gamma_1u_n)^\top \}_{n\in\bbN}\] converges to some $(f,g)^\top$ in $\bndra$. Then for all $v\in\cD$ one has
	\begin{align}
		\langle f,\Gamma_0v\rangle_{\mathfrak{H}}-\langle g, \Gamma_1v\rangle_{\mathfrak{H}}&=\lim\limits_{n\rightarrow\infty}\langle \Gamma_1 u_n,\Gamma_0v\rangle_{\mathfrak{H}}-\langle \Gamma_0 u_n, \Gamma_1v\rangle_{\mathfrak{H}}\\
		&=\lim\limits_{n\rightarrow\infty}\langle A^*u_n,v\rangle_{\cH}-\langle u_n,A^*v\rangle_{\cH}=0. 
	\end{align}
	Hence, by density of $\ran(\tr)$ in $\bndra$, we have $\langle f, h_1 \rangle_{\mathfrak{H}}-\langle g, h_2 \rangle_{\mathfrak{H}}=0$ for all $h_1, h_2\in\mathfrak H$. Setting, $h_1=f$, $h_2=-g$ we get $f=g=0$.
\end{proof}

\begin{proposition}\lb{prop1.3} Assume Hypothesis \ref{hyp3.6} and assume that there exists a self-adjoint extension   $\cA$  of  $A$ satisfying $\dom(\cA)\subset \cD$. Then the resolvent operator $R(\zeta, \cA):=(\cA-\zeta)^{-1}\in\cB(\cH)$, $\zeta\in\bbC\setminus\Sp(\cA)$, can be viewed as a bounded operator from $\cH$ to $\cH_ +$. Furthermore,
	\begin{equation}\lb{3221}
		\tr R(\zeta, \cA)\in \cB(\cH, \bndra).
	\end{equation}
\end{proposition}
\begin{proof}
	For all $u\in\cH$ one has
	\begin{equation}\begin{split}
			\| R(\zeta, \cA)u\|_{\cH_+}^2&
			=\|R(\zeta, \cA)u\|_\cH^2+\|A^*R(\zeta, \cA)u\|_\cH^2\\
			&\le \|R(\zeta, \cA)u\|_\cH^2+\big(\|(A^*-\zeta)R(\zeta, \cA)u\|_\cH+|\zeta|\big\|R(\zeta, \cA)u\|_\cH)^2,
	\end{split}\end{equation}
	that is, 
	\begin{equation}\lb{1.13}
		\| R(\zeta, \cA)\|_{\cB(\cH,\cH_+)}^2\leq \|R(\zeta, \cA)\|_{\cB(\cH)}^2+(1+|\zeta|\|R(\zeta, \cA)\|_{\cB(\cH)})^2,
	\end{equation}
	hence $ R(\zeta, \cA)\in\cB(\cH,\cH_+)$. Since $\dom(\cA)\subset \cD=\dom(\tr)$, the operator $\tr R(\zeta, \cA)$ is defined on all of $\cH$. Using this and that  $\overline{\tr}$ is closed as an operator from $\cH_+$ to $\bndra$ by Proposition \ref{remark2.2}(2), we note that 
	$\overline{\tr}  R(\zeta, \cA)\in \cB(\cH, \bndra)$ as a closed everywhere defined operator acting between Hilbert spaces.	Furthermore, since $\ran( R(\zeta, \cA))=\dom (\cA)\subset\cD=\dom(\tr)$, we have $\overline{\tr} R(\zeta, \cA)={\tr}R(\zeta, \cA)$ which  proves the assertion. 
\end{proof} 

\begin{remark}\label{rem:exist}  In Proposition \ref{prop1.3} (and everywhere when needed below), in addition to Hypothesis \ref{hyp3.6}  we {\em assume} the existence of a self-adjoint extension $\cA$ of $A$ with $\dom(\cA)\subset \cD$. The question of the existence of such a self-adjoint extension under merely Hypothesis \ref{hyp3.6} is a subtle one. The nontrivial issue of whether or not, and under which additional minimal assumptions, this indeed happens is beyond the scope of this paper. (We refer interested readers to \cite{Behrndt_2020, DHM2017, DHMS2006, DHMS2012} where closely related questions are discussed and relevant bibliography is provided. In this regard, we highlight an ingenious relevant work \cite{Calkin1939} that was re-discovered and further developed in \cite{HS2012, W2012a}.) That said, the condition $\dom(\cA)\subset \cD$ is indeed prevalent in the settings related to elliptic partial differential operators, ordinary differential operators, and quantum graphs covering our principal applications, see Sections \ref{ssLapLip}, \ref{ssQuaGr}, \ref{SS5.1} -- \ref{HFSSD} where relevant PDE models satisfying all abstract assumptions are discussed in detail. 

We stress that the main objective of our work is to develop first order asymptotic perturbation theory for {\it given} one parametric families of self-adjoint extensions $t\mapsto\cA_t$ of the operator $A$ with the additional property $\dom(\cA_t)\subset\cD$. In the current paper, the operator theoretic setting described by Hypothesis \ref{hyp3.6} and the condition $\dom(\cA_t)\subset\cD$ mainly serves as the vehicle for unifying several important classes of partial differential elliptic operators and ordinary differential operators on metric graphs.
\hfill $\Diamond$ \end{remark}

As it is well-known, the domains of self-adjoint extensions of $A$  are closely related to Lagrangian planes in $\bndra$, see, e.g., \cite[Theorem 3.1.6]{GG}, \cite{Harmer}, \cite{Pa}, \cite[Proposition 14.7]{Schm} and Theorems \ref{LLSA}, \ref{SALL} below. The main results of this section is a resolvent difference formula for two given extensions corresponding to two arbitrary Lagrangian planes, see Theorem \ref{thm1.7}. 
To proceed, we will need to recall some basic definitions from symplectic functional analysis. First, we note that the abstract Green identity \eqref{3.61} gives rise to a symplectic form $\omega$ defined by
\begin{align}
\begin{split}\lb{5.3}
\omega\big((f_1,f_2)^{\top}, (g_1,g_2)^{\top}\big):&=\langle f_2,g_1\rangle_{\mathfrak{H}}-\langle f_1, g_2\rangle_{\mathfrak{H}}\\
&=\left\langle
J (f_1, f_2)^{\top},(g_1,g_2)^{\top}
\right\rangle_{\mathfrak{H}\times\mathfrak{H}},\ 
J:=\begin{bmatrix}
0 & I_{\mathfrak{H}} \\
-I_{\mathfrak{H}}& 0
\end{bmatrix},
\end{split}
\end{align}
$f_k,g_k\in\mathfrak H, k=1,2$. Indeed, using this notation \eqref{3.61} can be re-written as follows,
\begin{equation}\label{3.61new}
\langle A^*u,v\rangle_{\cH}-\langle u,A^*v\rangle_{\cH}=\omega(\tr u, \tr v) \text{ for all $u,v\in\cD$}.
\end{equation} We denote the annihilator of a subspace $\cG\subset\mathfrak{H}\times\mathfrak{H}$ by
\begin{equation}\lb{anig}
\cG^{\circ}:=\{ (f_1,f_2)^{\top}\in \mathfrak{H}\times\mathfrak{H}: \omega\big((f_1,f_2)^{\top}, (g_1,g_2)^{\top}\big)=0 \text{\ for all}\ (g_1,g_2)^{\top}\in \cG\},
\end{equation}
and recall that the subspace $\cG$ is called {\it Lagrangian} if $\cG=\cG^{\circ}$. We denote by $\Lambda(\mathfrak{H}\times\mathfrak{H})$ the metric space of Lagrangian subspaces of $\mathfrak{H}\times\mathfrak{H}$  equipped with the metric
\begin{equation}
d(\cG_1, \cG_2):=\|Q_1-Q_2\|_{\cB(\mathfrak{H}\times\mathfrak{H})},\ \cG_1, \cG_2\in \Lambda(\mathfrak{H}\times\mathfrak{H}), 
\end{equation}
where  $Q_j$ is the orthogonal projection onto $\cG_j$ acting in $\mathfrak{H}\times\mathfrak{H}$, $j=1,2$. 


Next, we recall \sel{a well-known fact  (originally due to Rofe-Beketov, see \cite{MR0244808} and also \cite[Proposition 4(b)]{Pa}\footnote{\cite{Pa} refers to Lagrangian planes as {\em self-adjoint linear relations} (s.a.l.r.), see \cite[Remark 1]{Pa} and describes $\cG$ by means of the equation $Xf_1=Yf_2$ rather than $Xf_1+Yf_2=0$ used in \eqref{3241}. We choose the latter to be consistent with \cite[Theorem 1.4.4 A]{BK}.}}) that any Lagrangian plane $\cG\in\Lambda(\bndra)$ can be written as follows 
\begin{equation}\lb{3241}
\cG=\{(f_1,f_2)^\top\in\mathfrak{H}\times\mathfrak{H}: Xf_1+Yf_2=0\}=\ker ([X,Y]),
\end{equation}
where $[X,Y]$ is a $(1\times 2)$ block operator matrix with $X,Y$ satisfying 
\begin{align}
& X Y^*=YX^*,\quad  X,Y\in\cB(\mathfrak{H}), \lb{com}\\ 
&0\not\in \spec(M^{X,Y}) \text{ for the operator block-matrix  $M^{X,Y}:=\begin{bmatrix}
	X&Y\\
	-Y&X
	\end{bmatrix}$}.\lb{inv}
\end{align}
We note that
\begin{equation}
M^{X,Y} (M^{X,Y})^*=(XX^*+YY^*)\oplus(XX^*+YY^*).
\end{equation}
In particular, $0\not\in\spec(M^{X,Y})$ if and only if $0\not\in\spec(XX^*+YY^*)$. Using this observation we write the orthogonal projection $Q$ onto $\cG$ from \eqref{3241} as follows,
\begin{align}\lb{orproj}
Q=\begin{bmatrix}
-Y^*\\X^*
\end{bmatrix}
\left(XX^*+YY^*\right)^{-1}[-Y, X]=[-Y^*, X^*]^\top W(X,Y).
\end{align} 
Here and below, for brevity, for any $X,Y, X_j, Y_j\in\cB(\mathfrak{H})$, $j=1,2$,
we use notation $W$ and $Z_{1,2}$ for the operators
\begin{equation}\lb{defw}
\begin{split}
W(X,Y)&:= \left(XX^*+YY^*\right)^{-1}[-Y, X],\quad  W(X,Y)\in\cB(\bndra, \mathfrak{H}),\\
Z_{2,1}&:=(W(X_2,Y_2))^*(X_2Y_1^*-Y_2X_1^*)W(X_1,Y_1),\quad Z_{2,1}\in\cB(\mathfrak{H}\times \mathfrak{H}).
\end{split}
\end{equation}



We are ready to formulate the principal result of this section -- a symplectic resolvent difference formula for any two arbitrary self-adjoint extensions of $A$.  We refer to Appendix \ref{AppA}
 for connections of the self-adjoint properties of the extensions and Lagrangian properties of the traces of their domains. Also, we refer to Appendix \ref{Sec4.1} and, in particular, to Proposition \ref{prop:KrNaF} for the classical Krein--Naimark formula,\sel{ cf.\ \cite{AhkGlazman}, \cite[Theorem 2.6.1]{Behrndt_2020}, \cite[Chapter 7]{DMbook},  \cite{DM}, \cite[Theorem 14.18]{Schm}}. Finally, a more general version of the symplectic resolvent difference formula that holds for  {\em adjoint pairs} of operators is given in Theorem \ref{APthm1.7} below. 
 
In the next theorem we assume the existence of two self-adjoint extensions of $A$ with domains in $\cD$. As we have pointed out in Remark \ref{rem:exist}, this assumption is nontrivial in the abstract setting of Hypothesis \ref{hyp3.6} but holds for many PDE and quantum graph scenarios, as discussed in Sections \ref{ssLapLip}, \ref{ssQuaGr}, \ref{SS5.1}, \ref{ssKKFREO}, \ref{ssHeatEq}, \ref{HFSSD} below.  
 
 \begin{theorem}\lb{thm1.7}
	Assume Hypothesis \ref{hyp3.6} and suppose there exist two self-adjoint extensions $\cA_1$ and $\cA_2$ of $A$ with domains containing in $\cD$. Then for any $\zeta\not\in(\spec(\cA_1)\cup\spec(\cA_2))$  we have
	\begin{align}
	\lb{5.14}
	R_2({\zeta})-R_1({\zeta})&= \big(\Gamma_0 R_2(\overline{\zeta})\big)^*\Gamma_1 R_1({\zeta})-\big(\Gamma_1 R_2(\overline{\zeta})\big)^*\Gamma_0 R_1({\zeta}),\\
	R_2({\zeta})-R_1({\zeta})&=\big(\tr R_{2}(\overline{\zeta})\big)^*\, J\tr R_{1}({\zeta}),\lb{5.14aJ}
	\end{align}
	where $R_j(\zeta):= (\cA_j-\zeta)^{-1}$
	and $\tr R_j(\overline{\zeta})= \big(\Gamma_0 R_j(\overline{\zeta}), \Gamma_1 R_j(\overline{\zeta})\big)$ is considered as an operator in $\cB(\cH, \bndra)$, $j=1,2$. 
	
	Assume, further, that $\overline{\tr(\dom\cA_j)}$ is a Lagrangian plane in $\bndra$ and \[\overline{\tr(\dom\cA_j)}=\ker([X_j, Y_j])\] with $X_j, Y_j$ satisfying \eqref{com} and \eqref{inv},  and let $Q_j$ denote the orthogonal projection onto $\overline{\tr(\dom\cA_j)}$ for $j=1,2$. Then 
	\begin{align}	
	R_2({\zeta})-R_1({\zeta})&=\big(\tr R_{2}(\overline{\zeta})\big)^*\, Q_2JQ_1\tr R_{1}({\zeta}),
	\lb{5.14a}\\
	R_2({\zeta})-R_1({\zeta})&=\big(\tr R_{2}(\overline{\zeta})\big)^*\, Z_{2,1}\tr\, R_{1}({\zeta}),\lb{nn5.14a}
	\end{align}
	where the operators $Z_{2,1}=(W(X_2, Y_2) )^*(X_2Y_1^*-Y_2X^*_1)W(X_1, Y_1)$ and $W(X_j, Y_j)$  are defined in \eqref{defw}.
\end{theorem}
\begin{proof}
	By Proposition \ref{prop1.3} we have $\Gamma_0 R_2(\overline{\zeta}), \Gamma_1 R_2(\overline{\zeta})\in\cB(\cH, \mathfrak{H})$. 
	In particular, the adjoint operators appearing in \eqref{5.14} are also bounded. Next, using 
	$(\cA_j-\zeta)R_j(\zeta)=(A^*-\zeta)R_j(\zeta)$, $\cA_2=\cA_2^*$ , and the Green identity \eqref{3.61}, for  arbitrary  $u,v\in\cH$  we infer,
	\begin{align}
	\begin{split}\no
	\langle  R_2({\zeta})u&- R_1({\zeta})u, v\rangle_{\cH}=\langle  R_2({\zeta})u- R_1({\zeta})u, ( \cA_2-\overline{\zeta})R_2(\overline{\zeta})v\rangle_{\cH}\\
	&=\langle  ( \cA_2-{\zeta})R_2({\zeta})u, R_2(\overline{\zeta})v\rangle_{\cH}-\langle R_1({\zeta})u, ( A^*-\overline{\zeta} )R_2(\overline{\zeta})v\rangle_{\cH}\\
	&=\langle  u, R_2(\overline{\zeta})v\rangle_{\cH}-\langle (A^*-{\zeta})R_1({\zeta})u, 
	R_2(\overline{\zeta})v\rangle_{\cH}\\
	&\quad+\langle \Gamma_1 R_1({\zeta})u, \Gamma_0 R_2(\overline{\zeta})v\rangle_{\mathfrak{H}}-\langle \Gamma_0 R_1({\zeta})u, \Gamma_1 R_2(\overline{\zeta})v\rangle_{\mathfrak{H}}\\
	&=\langle \Gamma_1 R_1({\zeta})u, \Gamma_0   R_2(\overline{\zeta})v\rangle_{\mathfrak{H}}-\langle \Gamma_0 R_1({\zeta})u, \Gamma_1   R_2(\overline{\zeta})v\rangle_{\mathfrak{H}}\\
	&= \left\langle \big( (\Gamma_0  R_2(\overline{\zeta}))^*\Gamma_1 R_1({\zeta})-(\Gamma_1  R_2(\overline{\zeta}))^*\Gamma_0 R_1({\zeta})\big)u, v\right\rangle_{\cH}.
	\end{split}
	\end{align}
	This yields \eqref{5.14}. Rewriting \eqref{5.14} using  $J$ introduced in \eqref{5.3} yields \eqref{5.14aJ}.
	For all $u\in\cH$ we have $\tr R_j({\zeta})u\in\tr(\dom\cA_j)$
	and thus $Q_{j}\tr  R_{j}({\zeta})=\tr R_{j}({\zeta})$; so,  equation \eqref{5.14aJ} implies \eqref{5.14a} since $Q_2^*=Q_2$. Equation \eqref{nn5.14a} follows from \eqref{orproj}, \eqref{defw} and \eqref{5.14a}. 
\end{proof}

\begin{remark}\lb{remark2.9} As it is easy to see from the proof of Theorem \ref{thm1.7},  the symplectic resolvent difference formulas \eqref{5.14aJ}, \eqref{5.14a} hold even if $\cA_1$ is  a non self-adjoint restriction of $A^*$;  the only assertion used was $\dom(\cA_j)\subset\dom(\tr)$, $j=1,2$. We further recall that the classical Krein's resolvent formula, see, e.g., \cite{Behrndt_2020, Schm} and Appendix \ref{Sec4.1}, gives an expression of the difference of the resolvents of an {\em arbitrary} self-adjoint extension $\cA$ of $A$ and a {\em special}, ``Dirichlet"-type extension $\cA_0$ whose domain is $\ker(\Gamma_0)$. The difference of the resolvents of the two extensions is expressed in terms of the $\gamma$-field and the abstract Weyl's function; we recall this in Proposition \ref{prop:KrNaF}. The symplectic resolvent difference formula offered in Theorem \ref{thm1.7} does not contain of course that much information as Krein's resolvent formula as it does not involve, e.g.,  the Weyl function. We stress, however, that Theorem \ref{thm1.7} works for any two {\em arbitrary} self-adjoint extensions $\cA_1$ and $\cA_2$; the domains of neither of them should be the kernels of $\Gamma_0$ or $\Gamma_1$.  Also, as we will see below in Section \ref{section1}, the symplectic resolvent difference formula in Theorem \ref{thm1.7} appears to be very useful, for instance, in establishing continuity and differentiability properties of the resolvents of {\em families} of self-adjoint extensions. Clearly, the resolvent difference formula in Theorem \ref{thm1.7} can be easily obtained by applying the classical Krein's formula, first, to $\cA_1$ and $\cA_0$ and, next, to  $\cA_2$ and $\cA_0$ and then by subtracting the two formulas, cf.\ Remark \ref{Rem4.4}. This way of computing the difference of resolvents of two arbitrary extensions was often used since very classical work to show, for instance, that the difference belongs to the Schatten-von Neumann ideal, see, e.g., Theorem 2 and Corollary 4 in \cite{DM}.  Finally, as we demonstrate in the proof of Proposition \ref{prop:KrNaF}, the resolvent difference formula can be also used as the first step in proving the classical Krein's formula (of course, several more steps are required for the proof to dig out the wealth of information that the classical formula contains).
	\hfill$\Diamond$  \end{remark}

We conclude this section with a series of auxiliary assertions aiming  
to place the above results in the vast literature on the theory of boundary relations in Krein spaces and discuss further the adjoint operators $(\tr R_j(\zeta))^*$ appearing in \eqref{5.14a}, \eqref{nn5.14a}. Although the assertions could of independent interests, they are not being used in the remainder of the paper.

\begin{remark}\lb{rem2.3}
	We now briefly mention how to recast Hypothesis \ref{hyp3.6} using Krein's spaces in the context of boundary triplets as discussed in the inspirational paper \cite{DHMS2006} whose authors are dealing with very general but still closely related to our setting.  Let $J_{\mathfrak H}=\bfi J$, cf. \eqref{5.3}, and define in $\bndra$ an indefinite scalar product 
	\begin{align}\no
	\langle\langle(f_1, f_2)^{\top},(g_1, g_2)^{\top} \rangle\rangle_{\bndra}:&=\langle J_{\mathfrak H} (f_1, f_2)^{\top}, (g_1, g_2)^{\top}\rangle_{\bndra}\\
	&=\bfi \omega((f_1, f_2)^{\top},(g_1, g_2)^{\top}),\  f_1, f_2, g_1, g_2\in\mathfrak{H}.\no
\end{align}	
Let $J_{\cH}$ be an analogous operator in $\cH\times\cH$ yielding the corresponding indefinite scalar product $\langle\langle (u_1, u_2)^{\top},(v_1, v_2)^{\top}\rangle\rangle_{\cH}$. Then $(\cH\times\cH, \langle\langle \cdot, \cdot\rangle\rangle_{\cH\times\cH})$ and $(\bndra, \langle\langle \cdot, \cdot\rangle\rangle_{\bndra}$) are Krein spaces and the operator $\tr$ induces an isometry between them. To define the latter in precise terms,  let $G(A^*)$ denote the graph of $A^*$ in $\cH\times\cH$ and introduce an operator 
	  \begin{equation}
	  \cT: \cH\times\cH\rightarrow\bndra, \dom(\cT)=\{(u, A^*u)^\top: u\in\cD \}\subset G(A^*),\,  \cT(u, A^*u)^{\top}:=(\Gamma_0 u, \Gamma_1 u)^{\top}.
	  \end{equation}
Then Green's identity \eqref{3.61} yields 
	  \begin{equation}
		\langle\langle (u, A^*u)^{\top},(v, A^* v)^{\top}\rangle\rangle_{\cH}=\langle\langle \cT(u, A^*u)^{\top},\cT(v, A^* v)^{\top}\rangle\rangle_{\mathfrak{H}}
		\,\text{ for all $u,v\in\cD$},	  \end{equation}
	    and so $\cT$ is an isometry between the Krein spaces $(\cH\times\cH, \langle\langle \cdot, \cdot\rangle\rangle_{\cH})$ and $(\bndra, \langle\langle \cdot, \cdot\rangle\rangle_{\mathfrak H}$). Following \cite{DHMS2006} we will identify the graph of $\cT$ with $\cT$ and treat it as a linear relation in $\cH\times\cH\times\bndra$, see \cite{Behrndt_2020} for a comprehensive introduction into spectral theory of linear relations. In particular, $\cT^{-1}\subset \cT^{[*]}$, where the inverse is understood in the sense of relations and $\cT^{[*]}$ denotes the adjoint relation with respect to the Krein inner products. An important question is whether $\cT$ is unitary, that is, $\cT^{-1}=\cT^{[*]}$.  Proposition 2.5 in \cite{DHMS2006} gives sufficient conditions for an isometric map $\cT$ to be unitary. The conditions are: (i) $G(A^*)^{[\perp]}\subset G(A^*)$, (ii) $(\ran \cT)^{[\perp]}\subset \text{mul}(\cT)$ (here $^{[\perp]}$ denotes the orthogonal complement in the Krein space, and mul is the multi-valued part of the relation), and (iii) $\dom \cT^{[*]}\subset \ran(\cT)$. We note that (i), (ii) follow from Hypothesis \ref{hyp3.6} while (iii) does not (in general), even in the more restrictive setting of quasi-boundary triples studied in \cite{BL}. A deep characterization of the equality $\cT^{-1}=\cT^{[*]}$ in terms of the Nevanlinna property of the Weyl function is given in \cite[Theorem 3.9]{DHMS2006}, see also Theorem 7.57 and Corollary 7.58 in \cite{DHMS2012}. We stress that Hypothesis \ref{hyp3.6} alone is not sufficient for $\cT$ being unitary! To further compare the setting of \cite{DHMS2006} with that given by Hypothesis \ref{hyp3.6} we note that the latter deals with {\it densely} defined symmetric operator $A$ and the linear relation $\cT$ with dense range. These density assumptions model elliptic differential operators on bounded domains and ordinary differential operators on metric graphs, and, at the same time, yield natural relations between self-adjoint extensions of $A$ and Lagrangian planes in $\bndra$ as described in Theorems \ref{LLSA} and \ref{SALL}. In the more general setting of \cite{DHMS2006} these relations do not always take place, cf.\ Remark \ref{rem:A3}. 
\hfill$\Diamond$ \end{remark}

\begin{remark}\label{LagKrein}
	We choose to use Lagrangian (symplectic) language throughout the paper. Alternatively, Lagrangian plains are called {\em self-adjoint linear relations}, and we refer to \cite{Behrndt_2020, Schm} for a detailed account of the topic, see also \cite{Pa}. Another way to describe the same object is to involve the Krein spaces introduced in Remark \ref{rem2.3}. We notice that $\cG^\circ$ defined in \eqref{anig} is just $\cG^{[\bot]}$, the $\langle\langle\cdot\,,\cdot \rangle\rangle_{{}_{\mathfrak{H}}}$-orthogonal to $\cG$ subspace of $\mathfrak{H}\times\mathfrak{H}$, and $\cG$ is Lagrangian if and only if $\cG=\cG^{[\bot]}$.
	\hfill$\Diamond$  \end{remark}

Next, we discuss the operator $(\tr R_2(\zeta))^*$ appearing in Theorem \ref{thm1.7}. Let us first record the following useful fact about $\tr^*$. 
	
	\begin{proposition}\lb{remark2.2new}
		The domain of the adjoint operator $\tr^*:\dom(\tr^*)\subset \bndra\rightarrow \cH_-$, cf.\ \eqref{aub100}, satisfies  $J(\tr(\cD))\subseteq\dom(\tr^*)$.
	\end{proposition}
	\begin{proof}
By the general definition of adjoint operator, $\dom((\tr)^*)$ is the set of $h\in\mathfrak{H}\times\mathfrak{H}$ such that there exists a $w\in\cH_+$ so that for all $u\in\cD=\dom(\tr)$ one has
\begin{equation}\label{adjtrR2}
	\langle\tr u,h\rangle_{{}_{\mathfrak{H}\times\mathfrak{H}}}=
	{}_{\cH_+}\langle u,\Phi^{-1}w\rangle_{\cH_-}=\langle u,w\rangle_{\cH_+}=
	\langle u,w\rangle_{\cH}+\langle A^*u,A^*w\rangle_{\cH};
\end{equation}
if this is the case then $(\tr)^*h:=\Phi^{-1}w$. We recall the orthogonal direct sum decomposition $\cH_+=\dom(A)\dot{+}(\dom(A))^{\perp_{\cH_+}}$ where, by \cite[Lemma 3.1(a)]{BbF95},
\begin{equation}\label{adjtrR3}
	(\dom(A))^{\perp_{\cH_+}}=\big\{v\in\cH_+: A^*v\in\cH_+ \text{ and } v=-A^*(A^*v)\big\}.\end{equation}
Since $\dom(A)\subset\cD$ and $\ker(\tr)=\dom(A)$ by part (1) of the proposition, we have
\[\tr(\cD)=\tr\big((\dom(A))^{\perp_{\cH_+}}\cap\cD\big).\] If $h:=(h_1,h_2)^\top=J\tr v$ for some $v\in(\dom(A))^{\perp_{\cH_+}}\cap\cD$ then
\begin{align*}
	\langle\tr u,h\rangle_{{}_{\mathfrak{H}\times\mathfrak{H}}}&=
	\langle\Gamma_0u,h_1\rangle_{{}_{\mathfrak{H}}}+
	\langle\Gamma_1u,h_2\rangle_{{}_{\mathfrak{H}}}=\langle\Gamma_0u,\Gamma_1v\rangle_{{}_{\mathfrak{H}}}-
	\langle\Gamma_1u,\Gamma_0v\rangle_{{}_{\mathfrak{H}}}\\
	&=\langle u, A^*v\rangle_\cH-\langle A^*u,v\rangle_{\cH}
\end{align*}
by the Green identity \eqref{3.61}. Letting $w=A^*v$ we derive \eqref{adjtrR2} from \eqref{adjtrR3} and thus $J(\tr(\cD))\subseteq\dom((\tr)^*)$.\end{proof}

 It is tempting to re-write the pre-factor $(\tr R_2(\overline{\zeta}))^*$ in the right-hand side of \eqref{5.14aJ} in terms of the product of the operators adjoint to $T$ and $R_2(\overline{\zeta})$. To that end, we first prove an auxiliary result about the product of the adjoints.
\begin{proposition}\label{rem:adjtrR} Assume Hypothesis \ref{hyp3.6}
		and recall \eqref{aub100}. Assume there exists a self-adjoint extension  $\cA$ of  $A$ satisfying $\dom(\cA)\subset \cD$ and denote $R(\zeta, \cA):=(\cA-\zeta)^{-1}\in\cB(\cH)$ for all $\zeta\in\bbC\setminus\Sp(\cA)$. The operator $R( \overline{\zeta}, \cA)\in\cB(\cH)$ can be uniquely extended to a bounded linear operator in $\cB(\cH_-,\cH)$ that we will denote by $\cR( \overline{\zeta}, \cA)$. This extension is given by the operator $(R({\zeta}, \cA))^*\in\cB(\cH_-,\cH)$ adjoint to $R({\zeta}, \cA)\in\cB(\cH,\cH_+)$. With this notational conventions, the operator $\big(\tr R(\zeta, \cA)\big)^*\in\cB(\bndra,\cH)$ can be written as
		\begin{equation}\label{adjtrR1}
			(\tr R(\zeta,\cA))^*h= \cR(\overline{\zeta}, \cA)(\tr)^* h
			\text{ for all  $h\in J(\tr(\cD))$}. 
		\end{equation}
	\end{proposition}
	\begin{proof} For the sake of the proof we will denote by $\widehat{R}(\zeta, \cA)\in\cB(\cH, \cH_+)$ the resolvent operator $R(\zeta, \cA)$ viewed as an operator acting from $\cH$ to $\cH_+$, cf.\ \eqref{1.13};  thus $(\widehat{R}(\zeta, \cA))^*\in\cB(\cH_-, \cH)$. We let $i\in\cB(\cH_+,\cH)$ denote the first imbedding $i:w\mapsto w$ in \eqref{aub100} and let $j=(i)^*\in\cB(\cH,\cH_-)$ denote the second imbedding in \eqref{aub100} so that $\langle iu,w\rangle_{\cH}={}_{\cH_+}\langle u,jw\rangle_{\cH_-}$ for all $u\in\cH_+\hookrightarrow\cH$ and $w\in\cH\hookrightarrow\cH_-$. In this notation $i\widehat{R}(\zeta,\cA)=R(\zeta,\cA)$,
		and, in order to prove the first part of the proposition,  we have to show that 
		\begin{equation}\label{HHH}
			(\widehat{R}(\zeta, \cA))^*jw=R(\overline{\zeta}, \cA)w \text{ for all $w\in\cH$},
		\end{equation}
		and so $\cR(\overline{\zeta},\cA):=(\widehat{R}(\zeta, \cA))^*\in\cB(\cH_-,\cH)$ is indeed a bounded extension to $\cH_-$ of $R(\overline{\zeta}, \cA)\in\cB(\cH)$. For any $u\in\cH_+$ and $w\in\cH$ we infer,
		\begin{align*}
			\langle iu, (\widehat{R}(\zeta, \cA))^*jw\rangle_\cH&=
			{}_{\cH_+}\langle u, j(\widehat{R}(\zeta, \cA))^*jw\rangle_{\cH_-} & \text{ (because $i^*=j$)}
			\\&=
			\langle\big(j(\widehat{R}(\zeta, \cA))^*j\big)^*u,w\rangle_\cH & \text{(because $j(\widehat{R}(\zeta, \cA))^*j\in\cB(\cH,\cH_-)$)}\\&=
			\langle i\widehat{R}(\zeta, \cA) iu,w\rangle_\cH & \text{ (because $i^*=j$)}\\&=
			\langle R(\zeta, \cA) iu,w\rangle_\cH & \text{ (because $i\widehat{R}(\zeta, \cA)=R(\zeta,\cA)$)}\\&=
			\langle  iu, R(\overline{\zeta}, \cA) w\rangle_\cH & \text{ (because $\cA=\cA^*$ in $\cH$).}
		\end{align*}
		Since $\ran(i)$ is dense in $\cH$ we have \eqref{HHH}.
		
		It remains to prove \eqref{adjtrR1}, that is, in the notation of the current proof, that
		\begin{equation}\label{adjtrR12}
			(\tr \widehat{R}(\zeta,\cA))^*h=(\widehat{R}(\zeta, \cA))^*(\tr)^* h
			\text{ for all  $h\in J(\tr(\cD))$}. 
		\end{equation}
		By \cite[Problem III.5.26]{K80}, we have 
		$(\widehat{R}(\zeta, \cA))^*(\tr)^* \subseteq (\tr \widehat{R}(\zeta, \cA))^*$,
		where the domain of the product $(\widehat{R}(\zeta, \cA))^*(\tr)^*$
		is set to be equal to $\dom(\tr^*)$.
		Since  $J(\tr(\cD))\subseteq\dom(\tr^*)$ by Proposition \ref{remark2.2new}  we infer \eqref{adjtrR12}.
	\end{proof}

\begin{corollary}\lb{Rem2.6}
	Resolvent difference formula formulas \eqref{5.14}, \eqref{5.14aJ} can be also rewritten as
	\begin{equation}\lb{aub102}
		R_2({\zeta})-R_1({\zeta})= \cR_{2}({\zeta})\tr^* J\tr R_{1}({\zeta}),
	\end{equation}
	where the operator $\cR_2(\zeta)$ in the right-hand side is viewed as a unique extension of the resolvent $R_2(\zeta)\in\cB(\cH)$ to an element of $\cB(\cH_-,\cH)$ as in Proposition \ref{rem:adjtrR} and, in fact, is given by $(R_2(\overline\zeta))^*\in \cB(\cH_-,\cH)$. Indeed,   \eqref{aub102} follows from \eqref{5.14aJ}, \eqref{adjtrR1}, and the fact that $\ran \big(J\tr R_{1}({\zeta})\big)\subseteq J(\tr(\cD))\subseteq\dom(\tr^*)$, by Proposition \ref{remark2.2new} (3). 
	\hfill$\Diamond$  \end{corollary}

\begin{remark}\label{rem:ref1} 
We conclude this preliminary section with a slight generalization\footnote{We thank the referee of an earlier version of the paper for suggesting this generalization}, see \eqref{newKF} below, of the resolvent difference formula in Theorem \ref{thm1.7}. To formulate it, we will freely use elementary facts on (linear) relations as nicely described in \cite[Chapter 1]{Behrndt_2020}. In particular, we will identify the operators on a Hilbert space $\cH$ with their graphs in $\cH\times\cH$. In this remark (and only in this remark) instead of Hypothesis \ref{hyp3.6} we will impose the following assumptions. Let $A\subset A^*$ be a symmetric relation in $\cH\times\cH$ (not necessarily densely defined), and $\tr=[\Gamma_0, \Gamma_1]^\top:\dom(\tr)\subseteq A^*\to\mathfrak{H}\times\mathfrak{H}$ be a linear operator (possibly unbounded) with a dense in $A^*$ domain and such that the following abstract Green's identity holds,
\[ \langle u_2, v_1\rangle_\cH- \langle u_1, v_2\rangle_\cH=
\langle\Gamma_1 \widehat{u}, \Gamma_0 \widehat{v}\rangle_\mathfrak{H}-\langle\Gamma_0\widehat{u},\Gamma_1\widehat{v}\rangle_\mathfrak{H} \text{
for all $\widehat{u}=(u_1,u_2), \widehat{v}=(v_1,v_2)\in\dom(\tr)$.}\]
(Clearly, if $A^*$ is an operator then $u_2=A^*u_1$, $v_2=A^*v_1$ and so this becomes 
 \eqref{3.61} upon setting $\Gamma_0u_1=\Gamma_0\widehat{u}$ and $\Gamma_1u_1=\Gamma_1\widehat{u}$, cf.\ \cite[Section 2.1]{Behrndt_2020}).
Furthermore, let $\cA_1$ and $\cA_2$ be two relations such that $A\subset\cA_j\subset A^*$ and $\rho(\cA_j)\neq\emptyset$, $j=1,2$, and assume that $\cA_1\subset\dom(\tr)$ and $\cA_2^*\subset\dom(\tr)$. Let us fix $\zeta\in\rho(\cA_1)\cap\rho(\cA_2)$ and use the resolvents $R_1(\zeta)$ and $R_2(\zeta)^*=(\cA_2^*-\overline{\zeta})^{-1}$ of $\cA_1$ and $\cA_2^*$ to write the relations $\cA_1$ and $\cA_2^*$ as follows,
\begin{equation}\label{newKF0}\begin{split}
 \cA_1&=\big\{\widehat{u}:= \big(R_1(\zeta)u, (I_\cH+\zeta R_1(\zeta))u\big): u\in\cH\big\}, \\ 
\cA_2^*&=\big\{ \widehat{v}:= \big(R_2(\zeta)^*v, (I_\cH+\overline{\zeta} R_2(\zeta)^*)v\big): v\in\cH\big\}.\end{split}\end{equation}
Using Green's identity then yields 
\[
\langle\Gamma_1 \widehat{u}, \Gamma_0 \widehat{v}\rangle_\mathfrak{H}-\langle\Gamma_0\widehat{u},\Gamma_1\widehat{v}\rangle_\mathfrak{H}=
\langle (I_\cH+\zeta R_1(\zeta))u, R_2(\zeta)^*v\rangle_\cH-\langle R_1(\zeta)u,
(I_\cH+\overline{\zeta} R_2(\zeta)^*)v\rangle_\cH\]
and so re-arranging the right hand side of the last formula gives the desired generalization of the resolvent difference formula,
\begin{equation}\label{newKF}
\langle \big(R_2(\zeta)-R_1(\zeta)\big)u,v\rangle_\cH=\langle\Gamma_1 \widehat{u}, \Gamma_0 \widehat{v}\rangle_\mathfrak{H}-\langle\Gamma_0\widehat{u},\Gamma_1\widehat{v}\rangle_\mathfrak{H} \text{ for all $u,v\in\cH$}
\end{equation}
and $\widehat{u}, \widehat{v}$ as defined in \eqref{newKF0}. (Clearly, when $A$, $\cA_1$, $\cA_2$ are operators the resolvent difference equation \eqref{newKF} becomes \eqref{5.14}). 
\hfill$\Diamond$  \end{remark}

\section{Riccati equation for resolvents and Hadamard-type formulas for eigenvalues}\lb{section1}
In this section we consider a one-parameter family of self-adjoint extensions of a given symmetric operator perturbed by a family of bounded operators. In turn, the extensions are constructed using families of Lagrangian subspaces in a boundary space and boundary traces that also depend on the parameter. Our final objective is to derive a differential (Riccati-type) equation for the resolvents of the perturbed operators and formulas for the derivatives of their isolated eigenvalues with respect to the parameter.  The latter abstract formulas generalize, on one side, the classical perturbation results from the case of additive perturbations, see, e.g., \cite[Section II.5]{K80}, and, on another, the Rayleigh--Hadamard-type variational formulas for eigenvalues of partial differential operators depending on a parameter, see, e.g., \cite{Gri,Henry}.

\subsection{Parametric families of operators}\label{Sec3.1} We continue to assume that $A$ is a densely defined closed symmetric operator with equal (possibly infinite) deficiency indices, that  $\cH_+=\dom(A^*)$ is equipped with graph norm of $A^*$, and that $\cD$, the domain of the trace operator, is a dense subspace of $\cH_+$. 
The following hypothesis will be assumed throughout this section.
\begin{hypothesis}\lb{hyp2.2}  We assume that Hypothesis \ref{hyp3.6} holds for the trace operator $\tr$ and a subspace $\cD\subset\cH_+$ with $\dom(\tr)=\cD$, 
and, in addition, 
we assume that the subspace 
$\cD$ of $\cH_+$ is equipped with a Banach norm $\|\cdot\|_\cD$ such that the (injective) imbedding  $\jmath$ of $\cD$ into $\cH_+$ is continuous with respect to this norm, i.e. $\jmath\in\cB(\cD, \cH_+)$.  
\end{hypothesis}
A typical example that we have in mind is  the Laplacian $A=-\Delta$ on $L^2(\Omega)$ with $\dom(A)=H^2_0(\Omega)$ for an open bounded $\Omega\subset\bbR^n$
with smooth boundary. In this case, we have
\begin{equation}
A^*=-\Delta,\ \cH_+=\dom(A^*):=\{u\in L^2(\Omega) : \Delta u\in L^2(\Omega)\},
\end{equation}
$\cD:=\cD^1(\Omega)$, where the space
\[\cD^1(\Omega):=\{u\in H^1(\Omega) : \Delta u\in L^2(\Omega)\}\]  is equipped with the norm  $\|u\|_\cD:=(\|u\|_{H^1(\Omega)}^2+\|\Delta u\|_{L^2(\Omega)}^2)^{1/2}.$

 For $u\in\cD$ the trace operator is given by \[\tr u=[\gaD u, -\Phi \gaN u]^\top\in\mathfrak{H}\times\mathfrak{H} \text{ with } \mathfrak{H}:= H^{1/2}(\partial\Omega),\] here $\gaD$ is the Dirichlet and $\gaN=\nu\cdot\gaD\nabla u$ is the (weak) Neumann trace maps\footnote{see Appendix \ref{appA} for a discussion of trace maps}, and $\Phi$ is the Riesz isomorphism between $H^{-1/2}(\partial\Omega)=(H^{1/2}(\partial\Omega))^*$ and $H^{1/2}(\partial\Omega)$, cf.\ \eqref{aub27} below.

\begin{proposition}\lb{prop2.4} Under Hypothesis \ref{hyp2.2}
one has  $\tr\in\cB(\cD, \bndra)$.  In addition, if $\cA$ is a self-adjoint extension of $A$ with $\dom(\cA)\subset\cD$ then there exist  $c,C>0$ such that
	\begin{equation}\lb{ner}
	c\|u\|_{\cH_+}\leq \|u\|_{\cD}\leq C \|u\|_{\cH_+} \text{ for all  $u\in\dom(\cA)$.}
	\end{equation}
	In other words, the norms in $\cH_+$ and $\cD$ are equivalent on $\dom(\cA)$ for any self-adjoint extension $\cA$ of $A$ with $\dom(\cA)\subset\cD$. Furthermore, if $V=V^*\in\cB(\cH)$ and $\zeta\not\in\Sp(\cA+V)$ then
	\begin{equation}\lb{2.3}
(\cA+V-\zeta)^{-1}\in\cB(\cH, \cD).
	\end{equation}
\end{proposition}
\begin{proof}
	The operator $\tr$ is bounded as an everywhere defined on the Banach space $\cD$ closable operator (see Proposition \ref{prop1.3}).
	We claim that $\dom(\cA)$ is a $\|\cdot\|_\cD$-closed subspace of the Banach space $\cD$. Indeed, suppose that $u_n\rightarrow u$ in $\cD$ for some $u_n\in \dom(\cA)$. Since $\cD$ is continuously embedded into $\cH_+$, the sequence $\{u_n\}_{n\in\bbN}$ is Cauchy in $\cH_+$, that is, it is Cauchy with respect to the graph norm of $A^*$. Hence, $\{u_n\}$ is convergent to $u$ in $\cH$ and the sequence of vectors $A^*u_n=\cA u_n$ converges in $\cH$. Since $\cA$ is a closed operator,  we conclude that $u\in\dom(\cA)$, as claimed. Now, we will consider $\jmath$ as a mapping from the Banach space $(\dom(\cA), \|\cdot\|_{\cD})$ into the Banach space $(\dom(\cA), \|\cdot\|_{\cH_+})$. This mapping is bounded and bijective, hence its inverse is also bounded yielding \eqref{ner}. 
Assertion \eqref{2.3} follows from \eqref{1.13} and \eqref{ner}.
\end{proof}

\begin{remark}\lb{remark2.7}
It is worth comparing Propositions \ref{prop1.3} and \ref{prop2.4}: indeed,  \eqref{3221} says that the product $\tr R(\zeta, \cA)$ is a bounded operator while Proposition \ref{prop2.4} gives that each factor in this product is bounded. The latter fact will be used in the proof of Theorem \ref{prop1.8new} below (specifically, see \eqref{new1.19n}) and it comes at the expense of assuming  Hypothesis \ref{hyp2.2}. 
\hfill$\Diamond$  \end{remark}

\begin{hypothesis}\lb{hyp1.3}  
 We assume that \[\tr:[0,1]\rightarrow\cB(\cD, \mathfrak{H}\times\mathfrak{H}):t\mapsto\tr_t\] is a one-parameter family of trace operators and $\cD\subset\cH_+$ is a  $t$-independent subspace such that $\tr_t$ and $\cD=\dom(\tr_t)$ satisfy Hypothesis \ref{hyp2.2} (and thus, in particular, Hypothesis \ref{hyp3.6}) for each $t\in[0,1]$.   
 Let  $Q: [0,1]\rightarrow \cB(\mathfrak{H}\times\mathfrak{H}), t\mapsto Q_t$, be a one-parameter family of orthogonal projections. We assume that $\ran  (Q_t) \in\Lambda(\mathfrak{H}\times\mathfrak{H})$ is a Lagrangian plane for each $t\in[0,1]$. We further assume that there exists a family $\cA_t$, $t\in[0,1]$, of self-adjoint extensions of $A$ satisfying 
 \begin{align}
 &\dom(\cA_t)\subset \cD, \lb{3272}\\
 &\overline{\tr_t\big( \dom(\cA_{t})\big)}=\ran (Q_t).\lb{3272new}
 \end{align}
 Let 
 $V: [0,1]\rightarrow \cB(\cH)$, $t\mapsto V_t$ be a one-parameter family of	self-adjoint bounded operators. We denote $H_t:= \cA_t+V_t$ and $R_t(\zeta):= (H_t-\zeta)^{-1}\in\cB(\cH)$ for $\zeta\not\in \Sp(H_t)$ and $t\in[0,1]$. 
\end{hypothesis}

Hypothesis \ref{hyp1.3} gives a rather general setup for boundary value problems parameterized by a one dimensional variable. We briefly list 
several families of operators for which the operators per se, their domains, and  respective traces depend on a given parameter. Our immediate  objective is just to give a glimpse of the typical situations of the setup described in Hypothesis \ref{hyp1.3}.
More examples with detailed analysis are given below, see Subsections  \ref{ssLapLip}, \ref{ssQuaGr}, \ref{ssPerKrPen}, \ref{ssKKFREO}, \ref{ssHeatEq}, and  \ref{HFSSD}.

\begin{example}\label{parametricSchr}
A well studied model which fits Hypothesis \ref{hyp1.3} is the family of Schr\"odinger operators equipped with Robin-type boundary conditions considered on a family of subdomains $\Omega_t\subset\Omega$ obtained by linear shrinking of a bounded star-shaped domain $\Omega\subset \bbR^n$ to its center. 
The   linear rescaling of $\Omega_t$ back to $\Omega$ leads to a one-parameter family of Schr\"odinger operators $H_t:=-\Delta_t+V$ in $L^2(\Omega)$ subject to Robin boundary conditions $(\theta_tu-t^{-1}\frac{\partial u}{\partial \nu})\upharpoonright_{\partial\Omega}=0$, where $\theta_t\in L^{\infty}(\partial\Omega, \bbR)$ is the rescaled boundary function. In this case, the minimal symmetric operator is given by the Laplacian considered on  $H^2_0(\Omega)$, its self-adjoint extensions $-\Delta_t$ are determined by the boundary condition $(\theta_tu-t^{-1}\frac{\partial u}{\partial \nu})\upharpoonright_{\partial\Omega}=0$ which in turn corresponds to the Lagrangian planes $\{(f,g)^\top\in H^{1/2}(\partial\Omega)\times H^{1/2}(\partial\Omega): \theta_t f=g \}$ in $H^{1/2}(\partial\Omega)\times H^{1/2}(\partial\Omega)$. That is, we have
\begin{align}
&\cH:=L^2(\Omega),  \mathfrak{H}:=H^{1/2}(\partial\Omega), \tr_t:=[\gaD, -t^{-1}\Phi\gaN]^\top,\\
&A:=-\Delta, \dom(A)= H^2_0(\Omega), \cD=\cD^1(\Omega):=\{u\in H^1(\Omega): \Delta u\in L^2(\Omega)\},\\
& \dom(\cA_t):=\{u\in \cD^1(\Omega): \theta_t\gaD u=t^{-1}\gaN u\}, \\
&\ran(Q_t):=\{(f,g)^\top\in H^{1/2}(\partial\Omega)\times H^{1/2}(\partial\Omega): \theta_t f=g \},
\end{align} 
here $\gaD$ and $\gaN$ denote the Dirichlet and (weak) Neumann traces, see Appendix \ref{appA}, and $\Phi:H^{-1/2}(\partial\Omega)\to H^{1/2}(\partial\Omega)$ denotes the Riesz isomorphism, see \eqref{aub27}. Similar models are systematically studied in \cite{CJLS, CJM1, DJ11} and discussed in some details in a more general setting  in Section \ref{HFSSD} below.
\hfill$\Diamond$  \end{example}

\begin{example}\lb{ex3.9} Our next example is a matrix second order  operator  posted on a multidimensional infinite cylinder with variable cross sections. We denote by $t\in\bbR$ the axial and by $x$ the transversal variables, that is, we set
	\[\Omega:=\big\{ (t,x)\in\bbR^{n+1}: t\in\bbR, x\in{\mathbb B}_{r(t)}^n\big\} \subset\bbR^{n+1},\] 
	 where, for instance, $r(t)=1+t/(1+t^2)$, and ${\mathbb B}_{r}^n$ is the ball in $\bbR^n$ of radius $r$ centered at zero. Denoting $\Delta_{(t,x)}=\partial^2_t+\Delta_x$ and $\Delta_x=\sum_{j=1}^n\partial^2_{x_j}$, we will consider in $L^2(\Omega; \bbC^N)$ the Schr\"odinger operator 
	\[-\Delta_{(t,x)}+V=-\partial_t^2+B_t, \text{ where $B_t=-\Delta_x(t)+V$ and $V=V(t,x)$}\]
	is a smooth bounded  $(N\times N)$-matrix valued potential taking symmetric values while the $x$-Laplace operator $-\Delta_x(t)$
	is acting in $L^2({\mathbb B}_{r(t)}^n; \bbC^N)$ and equipped with the following domain,
\[\dom(-\Delta_x(t)):=\big\{ u\in \cD^1({\mathbb B}_{r(t)}^n): \tr u:=
	(\gamma_{{}_{D,\partial{\mathbb B}_{r(t)}^n}}u, -\Phi\gamma_{{}_{N,\partial{\mathbb B}_{r(t)}^n}}u)\in\cG_t\big\}, \]
	where $\cG:t\mapsto\cG_t$ is a given smooth family of Lagrangian subspaces in the boundary space $H^{1/2}(\partial{\mathbb B}_{r(t)}^n)\times H^{1/2}(\partial{\mathbb B}_{r(t)}^n)$. We note parenthatically   that the spectral flow  of the family $\{B_t\}_{t=-\infty}^\infty$ of the self-adjoint operators $B_t$ is of interest as it is related to the spectrum of the Schr\"odinger operator $-\Delta_{(t,x)}+V$ in $L^2(\Omega; \bbC^N)$; this relation could be established using spatial dynamics, cf. \cite{LatPog, SS, SS1}, via a connection to a first order differential operator, cf. \cite{LatTom} and \cite{GLMST}. 
	 Rescaling $x\mapsto z=x/r(t)$ of ${\mathbb B}_{r(t)}^n$ onto ${\mathbb B}_{1}^n$ gives rise to a family of operators $H_t$ defined analogously to $B_t$ by 
	\begin{equation}\label{dfnht}
	H_t=-(r(t))^{-2}\Delta_z(t)+V_t,\, \text{ where $z\in{\mathbb B}_{1}^n$, $V_t(z)=V(t, r(t)z)$},\end{equation}
	the Lagrangian subspace $\widehat{\cG}_t$ is obtained from $\cG_t$ by rescaling as well, and the $z$-Laplacian $-\Delta_z(t)$ acting in $L^2({\mathbb B}_{1}^n; \bbC^N)$ is equipped with the domain 
	\begin{equation}\label{dfndm}
	\dom(-\Delta_z(t)):=\big\{ w\in \cD^1({\mathbb B}_{1}^n): \tr_t w:=
	(\gamma_{{}_{D,\partial{\mathbb B}_{1}^n}}w, -(r(t))^{-1}\Phi\gamma_{{}_{N,\partial{\mathbb B}_{1}^n}}w)\in\widehat{\cG}_t\big\}. \end{equation}
	The family of operators $H_t$ can be considered within the setting of Hypothesis \ref{hyp1.3} with $\tr_t$ given in \eqref{dfndm}, $V_t$ given in \eqref{dfnht}, and $Q_t$ being the projection onto $\widehat{\cG}_t$.
\hfill$\Diamond$  \end{example} 

\begin{example}
The next example is given by a one-parameter family of operators arising in Floquet--Bloch decomposition of  periodic Hamiltonians on $\bbR$, see \cite[Theorem XII.88]{RS78} and Example \ref{ex47} below. We consider the Schr\"odinger operator $\displaystyle{ A:=-\frac{\bd^2}{\bd x^2}+V}$ on $(0,1)$ with domain $H^2_0(0,1)$ and its sefl-adjoint extensions determined by the following boundary conditions
$u(1)=e^{\bfi t}u(0), u'(1)=e^{\bfi t}u'(0), t\in[0,2\pi).$  In this case the setup described  in Hypothesis \ref{hyp1.3} is as follows,
\begin{align}
&\cH:=L^2(0,1),  \mathfrak{H}:= \C^2, \Gamma_0 u =(u(0), u(1)), \Gamma_1 u =(u'(0),- u'(1)),\\
&A:=-\frac{\bd^2}{\bd x^2}, \dom(A)= H^2_0(0,1), \cD=H^2(0,1);\\
& \dom(\cA_t):=\{u\in H^2(\Omega): u(1)=e^{\bfi t}u(0), u'(1)=e^{\bfi t}u'(0)\}, \\
&\ran(Q_t):=\{(z_1, z_2, z_3, z_4)\in \C^4: z_2=e^{\bfi t}z_1, z_3=-e^{\bfi t} z_4 \}.
\end{align} 
\hfill$\Diamond$  \end{example}

\begin{example}\label{Ca-Ri} This example concerns a {\em first order} operator related to the perturbed Cauchy--Riemann operator on a two-dimensional infinite cylinder, cf.\ \cite[Section 7]{RoSa95}. Let $a,b:\bbR\to\bbR$ be smooth functions having limits $a_\pm<b_\pm$ at $\pm\infty$ and such that  $a(t)<b(t)$ for all $t\in\bbR$, and consider the two-dimensional cylinder \[\Omega=\{(t,x)\in\bbR^2:  a(t)<x<b(t), t\in\bbR\}.\] For $N\ge1$ we consider the perturbed Cauchy-Riemann operator $\bar{\partial}_{S,\cG}=\partial_t+B_t$ acting in the space $L^2(\Omega; \bbR^{2N})$ of real vector valued functions, where
	\[B_t=-J_N\partial_x(t)+S,  t\in\bbR, J_N=\begin{bmatrix}0&I_{\bbR^N}\\-I_{\bbR^N}&0\end{bmatrix}\,, \]
	 and  $S=S(t,x)\in\bbR^{2N\times 2N}$ is a given smooth bounded matrix valued function taking symmetric values and having limits $S_\pm(x)$ as $t\to\pm\infty$. Here and below for each $t\in\bbR$ we denote by 
	$\partial_x(t)$ the operator of $x$-differentiation in $L^2\big((a(t),b(t));\bbR^{2N}\big)$ with the 
	domain \begin{equation}\label{dfndmmm}
	\dom(\partial_x(t))=\big\{u\in H^1\big((a(t),b(t));\bbR^{2N}\big): \tr_tu:=(u(a(t)), u(b(t)))\in\cG_t\big\},\end{equation}
	where $\cG:t\mapsto\cG_t\in\Lambda(2N)$ is a given smooth family of Lagrangian subspaces in $\bbR^{4N}$ having limits $\cG_\pm$ as $t\to\pm\infty$. Again, we note that the spectral flow of the family $\{B_t\}_{t=-\infty}^{+\infty}$ of  the self-adjoint  operators $B_t$ is of interest since, in particular, it is equal (see, e.g., \cite{GLMST,LatTom}) to the Fredholm index of the Cauchy-Riemann operator  $\bar{\partial}_{S,\cG}$, see a detailed discussion and various implications of this fact in \cite[Section 7]{RoSa95}. Rescaling $u(t,x)\mapsto w(t,z):=u(t, z(b(t)-a(t))+a(t))$, $z\in(0,1)$, gives rise to an analogous to $B_t$ operator $H_t$ acting in $L^2([0,1];\bbR^{2N})$ as
	\begin{equation}\label{dfnhttt}
	H_t=-J_N\partial_z(t)+V_t, t\in\bbR, z\in(0,1), \text{where $V_t(z)=S(t, (b(t)-a(t))z+a(t))$}\end{equation}
	and $\partial_z(t)=(b(t)-a(t))\frac{\partial}{\partial z}$ is the operator in $L^2([0,1];\bbR^{2N})$ with the domain
	\[\dom(\partial_z(t))=\big\{w\in H^1([0,1];\bbR^{2N})): \tr w:=(w(0), w(1))\in\cG_t\big\}.\]
The family of operators $H_t$ can be considered within the setting of Hypothesis \ref{hyp1.3}  with the trace given in \eqref{dfndmmm}, with $Q_t$ being the projection onto $\cG_t$, and $V_t$ given in \eqref{dfnhttt}.
\hfill$\Diamond$  \end{example}

\begin{example} Parameter depended Hamiltonians satisfying Hypothesis \ref{hyp1.3} play an important role in the theory of quantum graphs. For example, the well-known eigenvalue bracketing, see \cite[Section 3.1.6]{BK}, is established by  studying the dependence of eigenvalues  of the $\delta$-type graph  Laplacian on the coupling constant. We refer the reader to Section \ref{ssQuaGr} for an in-depth discussion of parameter depended quantum graphs satisfying Hypothesis \ref{hyp1.3}.
	
\hfill$\Diamond$  \end{example}

\begin{remark}
	Hypothesis \ref{hyp1.3} is satisfied, for example, when $\ran(Q_t)\in\Lambda(\bndra)$ is $(\cD, \tr_t)$ aligned, cf. Definition \ref{defASSOC}, and $\cA_t$ is the operator associated with $\ran(Q_t)$ and $\dom(\cA_t)\subset \cD$, $t\in[0,1]$, see Theorem \ref{LLSA}. Conversely, if $\cA_t$ is a self-adjoint extension of $A$ with $\dom(A_t)\subset \cD$, $t\in[0,1]$, which is $(\cD, \tr_t)$ aligned and  $\ran (Q_t)$ is a subspace associated with $\cA_t$ then $\ran(Q_t)\in\Lambda(\bndra)$, $t\in[0,1]$, see Theorem \ref{SALL}.
	\hfill$\Diamond$  \end{remark}

\subsection{Resolvent expansion}\label{Sec3.2}
Our first major result in the setting of Hypothesis \ref{hyp1.3} is a symplectic  formula for the difference of the resolvents $R_t(\zeta)=(H_t-\zeta)^{-1}$
of the operators $H_t$ at different values of $t$.

\begin{theorem}\lb{thm5.5} Assume Hypothesis \ref{hyp1.3} and let 
	$t, s, \tau \in [0,1],$ $\zeta\not\in\Sp(H_{t})\cup \Sp(H_{s})$. Then for $R_t(\zeta):= (H_t-\zeta)^{-1}$  and $H_t=\cA_t+V_t$
	one has 
	\begin{align}
	R_{t}&(\zeta)- R_{s}(\zeta)=  R_t({\zeta})(V_s-V_t)R_s(\zeta)+(\tr_{\tau}  R_{t}(\overline{\zeta}))^* J\tr_{\tau}    R_{s}(\zeta)\label{new5.14} \\
	&=R_t({\zeta})(V_{s}-V_t)R_{s}(\zeta)+(\tr_{t}  R_t(\overline{\zeta}))^*(Q_t-Q_{s})J \tr_{s}   R_{s}(\zeta)\nonumber\\
	&\hskip5cm
	+(\tr_{t}  R_t(\overline{\zeta}))^*J (\tr_t-\tr_{s})   R_{s}(\zeta).\lb{new5.14a}
	\end{align}
	\end{theorem}
	The operators whose adjoints enter \eqref{new5.14}, \eqref{new5.14a} are being considered as elements of $\cB(\cH, \bndra)$ (cf. Proposition \ref{prop2.4})  and thus their adjoints are elements of $\cB(\bndra, \cH)$.

\begin{proof}
	As in the proof of Theorem \ref{thm1.7} for arbitrary  $u,v\in\cH$ and $\tr_{\tau}=[\Gamma_0, \Gamma_1]^\top$ one has
	\begin{align}
	\begin{split}\no
	\langle  &R_t(\zeta)u- R_s(\zeta)u, v\rangle_{\cH}=\langle  R_t(\zeta)u- R_s(\zeta)u, ( H_t-\overline{\zeta})R_t(\overline\zeta)v\rangle_{\cH}\\
	&=\langle  ( H_t-\zeta)R_t(\zeta)u, R_t(\overline{\zeta})v\rangle_{\cH}-\langle R_s(\zeta)u, ( A^*+V_t-\overline{\zeta})R_t(\overline{\zeta})v\rangle_{\cH}\\
	&=\langle  u, R_t(\overline{\zeta})v\rangle_{\cH}+\langle R_s(\zeta)u, (V_s-V_t)R_t(\overline{\zeta})v\rangle_{\cH}-\langle (A^*+V_s-\zeta )R_s(\zeta)u, R_t(\overline\zeta)v\rangle_{\cH}\\
	&\qquad+\langle \Gamma_1 R_s(\zeta)u, \Gamma_0 R_t(\overline\zeta)v\rangle_{\mathfrak{H}}-\langle \Gamma_0 R_s(\zeta)u, \Gamma_1 R_t(\overline\zeta)v\rangle_{\mathfrak{H}}\\
	&=\langle R_s(\zeta)u, (V_s-V_t)R_t(\overline{\zeta})v\rangle_{\cH}+\langle \Gamma_1 R_s(\zeta)u, \Gamma_0  R_t(\overline\zeta)v\rangle_{\mathfrak{H}}-\langle \Gamma_0 R_s(\zeta)u, \Gamma_1  R_t(\overline\zeta)v\rangle_{\mathfrak{H}}\\
	&= \left\langle \big(R_t({\zeta})(V_s-V_t)R_s(\zeta)+(\Gamma_0 R_t(\overline{\zeta}))^*\Gamma_1 R_s-(\Gamma_1 R_t(\overline{\zeta}))^*\Gamma_0 R_s\big)u, v\right\rangle_{\cH}.
	\end{split}
	\end{align}
	Thus
	\begin{equation*}\label{1.15}
	R_t(\zeta)-  R_s(\zeta)=R_t({\zeta})(V_s-V_t)R_s(\zeta)+(\Gamma_0 R_t(\overline{\zeta}))^*\Gamma_1 R_s(\zeta)-(\Gamma_1 R_t(\overline{\zeta}))^*\Gamma_0 R_s(\zeta),
	\end{equation*}
 yielding \eqref{new5.14}. 
	In order to prove \eqref{new5.14a} we note that 
	\[\tr_{s}   R_{s}(\zeta)=Q_{s} \tr_{s}   R_{s}(\zeta) \text{ and }\, \tr_{t}   R_{t}(\zeta)=Q_{t} \tr_{t}   R_{t}(\zeta).\]
	In addition, we have $Q_sJQ_s=0$ since $\ran(Q_s)$ is Lagrangian. This implies
	\begin{align}
	(\tr_{t} & R_t(\overline{\zeta}))^*J \tr_{t}   R_s(\zeta)
	=(\tr_{t}  R_t(\overline{\zeta}))^*J \tr_{s}   R_s(\zeta)+(\tr_{t}  R_t(\overline{\zeta}))^*J (\tr_{t}-\tr_s)   R_s(\zeta)\\
	&=(\tr_{t}  R_t(\overline{\zeta}))^*Q_tJQ_s \tr_{s}   R_s(\zeta)+(\tr_{t} R_t(\overline{\zeta}))^*J (\tr_{t}-\tr_s)   R_s(\zeta)\\
	&=(\tr_{t}  R_t(\overline{\zeta}))^*(Q_t-Q_s) J\tr_{s}   R_s(\zeta)+(\tr_{t}  R_t(\overline{\zeta}))^*J (\tr_{t}-\tr_s)   R_s(\zeta).
	\end{align}
	Utilizing this and letting $\tau=t$ in \eqref{new5.14} yields \eqref{new5.14a}.
\end{proof}
\begin{remark}We note that \eqref{new5.14} holds even if $\cA_s$ is a non self-adjoint restriction of $A$.
\hfill$\Diamond$  \end{remark}

 Next, given the one-parameter families of self-adjoint extensions $\cA_t$, traces $\tr_t$ and operators $V_t$ described in Hypothesis \ref{hyp1.3}, we will show that the resolvent operators for $H_t=\cA_t+V_t$ are continuous (differentiable) at a given point $t=t_0$ whenever the mappings $t\mapsto Q_t$, $t\mapsto \tr_t$, $t\mapsto V_t$  are continuous (differentiable) at $t_0$. 
 
 To introduce appropriate assumptions we recall from Proposition \ref{prop2.4} (replacing $\dom(\cA)$ by $\dom(\cA_t)$) that under Hypothesis \ref{hyp2.2}  the norms in $\cD$ and $\cH_+$ are equivalent on $\dom(\cA_t)$ for each $t\in[0,1]$, cf.\ \eqref{ner}, but with the constant $c$ that might depend of $t$. We will need a uniform for $t$ near $t_0$ version of this assertion: In addition to Hypothesis \ref{hyp1.3} we will often assume that, for a given $t_0\in[0,1]$, there are constants $C,c >0$ such that
\begin{equation} \label{hyp1.3i2bis}
 c\|u\|_{\cH_+}\le\|u\|_\cD\le C\|u\|_{\cH_+} \text{ for all  $u\in\dom(\cA_t)$ and $t$ near $t_0$.}
\end{equation}
These inequalities are equivalent to uniform with respect to the parameter $t$ boundedness of the norms of resolvents of $\cA_t$ as operators from $\cH$ to $\cD$, see  Proposition \ref{propEq} below. We stress that \eqref{hyp1.3i2bis} does {\em not} mean that the norms $\|\cdot\|_{\cH_+}$ and $\|\cdot\|_\cD$ are equivalent on $\cD$; they are equivalent only on the domains of the extensions $\cA_t$ of $A$ but uniformly for $t$ near $t_0$.

\begin{hypothesis}      \lb{hyp1.3i2}
 In addition to Hypotheses \ref{hyp2.2} and \ref{hyp1.3} we assume, for a given $t_0\in[0,1]$, that
        \begin{equation}\label{2.2b}
        \|(\cA_t-\bfi)^{-1}\|_{\cB(\cH, \cD)}=\cO(1) \text{ as $ t\rightarrow t_0$}.
        \end{equation}
\end{hypothesis}
\begin{remark}\lb{rem3.14}Suppose that $V_t$ form Hypothesis \ref{hyp1.3} satisfies $V_t=\cO(1)$,  $t\rightarrow t_0$ and that $\zeta\in\C\setminus\bbR$. Then \eqref{2.2b} is equivalent to 
	\begin{equation}
	\|(\cA_t+V_t-\zeta)^{-1}\|_{\cB(\cH, \cD)}=\cO(1) \text{ as $ t\rightarrow t_0$}.
	\end{equation}
	Indeed, we have
	\begin{equation}\lb{aub31}
	(\cA_t+V_t-\zeta)^{-1}=(\cA_t-\bfi)^{-1}+(\cA_t-\bfi)^{-1}(\bfi-\zeta+V_t)(\cA_t+V_t-\zeta)^{-1}.
	\end{equation}
Considering $(\cA_t-\bfi)^{-1}$ as a mapping from $\cH$ to $\cD$, $ (\cA_t+V_t-\zeta)^{-1}$ as a mapping from $\cH$ to itself, and using the bound $\|(\cA_t+V_t-\zeta)^{-1}\|_{\cB(\cH)}\leq (|\Im\zeta|)^{-1}$, we infer the claim. 
\hfill$\Diamond$  \end{remark}

The equivalence of Hypothesis \ref{hyp1.3i2} and assertion \eqref{hyp1.3i2bis} 
  is proven next.
\begin{proposition}\label{propEq} Assume Hypothesis \ref{hyp2.2}. Then \eqref{hyp1.3i2bis} is equivalent to \eqref{2.2b}.
\end{proposition}
\begin{proof} If \eqref{2.2b} holds then for any $u\in\dom(\cA_t)$ and $t$ near $t_0$ one has
\begin{align*}
\|u\|_\cD&=\|(\cA_t-\bfi)^{-1}(\cA_t-\bfi)u\|_\cD\le c\|(\cA_t-\bfi)u\|_\cH\\
&\le c(\|\cA_tu\|_\cH+\|u\|_\cH)\le\sqrt{2}c\|u\|_{\cH_+},
\end{align*}
thus proving \eqref{hyp1.3i2bis}, as $\|u\|_{\cH_+}\leq c\|u\|_{\cD}$ by Hypothesis \ref{hyp2.2}.

Conversely, using \eqref{hyp1.3i2bis}, for all $t$ near $t_0$ and any $v\in\cH$ one has
\begin{align*}
\|(\cA_t-\bfi)^{-1}v\|_\cD&\le C\|(\cA_t-\bfi)^{-1}v\|_{\cH_+}\\
&=C\big(\|(\cA_t-\bfi)^{-1}v\|_{\cH}^2+\|\cA_t(\cA_t-\bfi)^{-1}v\|_{\cH}^2\big)^{1/2}\\&\le C\big(\|(\cA_t-\bfi)^{-1}\|_{\cB(\cH)} \|v\|_{\cH}^2+(\|v\|_\cH+\|(\cA_t-\bfi)^{-1}v\|_{\cH})^2\big)^{1/2}\\&\le \sqrt{5}C\|v\|_\cH,
\end{align*}
since $\cA_t$ is self-adjoint, thus proving \eqref{2.2b}.
\end{proof}
Assuming that the families $Q_t$, $\tr_t$ are continuous at $t=t_0$, under Hypothesis \ref{hyp1.3i2} the resolvent difference formula formula  \eqref{new5.14a} with $V_t=0$ shows (as in the proof of Theorem \ref{prop1.8new} (\ref{18i1}) below) that 
\begin{align}\lb{2.2boNew}
        \begin{split}
        &\big\|(\cA_t-\bfi)^{-1}-(\cA_{t_0}-\bfi)^{-1}\big\|_{\cB(\cH)}\underset{t\rightarrow t_0}{=}o(1),\\
        &\big\|(\cA_t-\bfi)^{-1}-(\cA_{t_0}-\bfi)^{-1}\big\|_{\cB(\cH, \cH_+)}\underset{t\rightarrow t_0}{=}o(1).
        \end{split}
        \end{align}
        In the proof of differentiability of the resolvent of $H_t$ we will need, however, a somewhat stronger continuity assumption, given next, regarding the resolvents of $\cA_t$ considered as operators from $\cH$ to $\cD$. As we will demonstrate in Sections \ref{abt} and \ref{subsec1.1} below, the stronger assumption does hold in the case of  boundary triplets and for Robin-type elliptic partial differential operators on bounded domains. 
\begin{hypothesis}\lb{hyp1.3i3}  In addition to Hypotheses \ref{hyp2.2} and \ref{hyp1.3} we assume that  for a given $t_0\in[0,1]$ one has
        \begin{equation}\lb{2.2bo}
        \big\|(\cA_t-\bfi)^{-1}-(\cA_{t_0}-\bfi)^{-1}\big\|_{\cB(\cH, \cD)}=o(1),\ t\rightarrow t_0.
        \end{equation}
\end{hypothesis}
\begin{remark}
Suppose that $V_t$ from Hypothesis \ref{hyp1.3} satisfies $(V_t-V_{t_0})=o(1)$, $t\rightarrow t_0$ and that $\zeta\in\C\setminus\bbR$. Then \eqref{2.2bo} is equivalent to 
\begin{equation}
\|(\cA_t+V_t-\zeta)^{-1}-(\cA_{t_0}+V_{t_0}-\zeta)^{-1}\|_{\cB(\cH, \cD)}=o(1) \text{ as $ t\rightarrow t_0$}.
\end{equation}
The proof is similar to the proof of Remark \ref{rem3.14} We also note that \eqref{2.2bo} implies \eqref{2.2b}. 
\hfill$\Diamond$  \end{remark}

After these preliminaries we are ready to present the main result of this subsection.
\begin{theorem}\lb{prop1.8new}
        We fix $t_0\in[0,1]$, $\zeta_0\not\in\Sp(H_{t_0})$ and define \[\cU_\epsilon=\{(t,\zeta)\in[0,1]\times\bbC: |t-t_0|\le\epsilon, |\zeta-\zeta_0|\le\epsilon\} \text{ for $\epsilon>0$}.\]
        \begin{enumerate}
        \item\lb{18i1}Assume Hypothesis \ref{hyp1.3i2} and suppose that the mappings $t\mapsto \tr_t$, $t\mapsto V_t$, $t\mapsto Q_t$ are continuous at $t_0$. Then there exists an $\varepsilon>0$ such that if $(t,\zeta)\in\cU_\varepsilon$ then $\zeta\not\in\Sp(H_t)$ and
        the operator valued function $t\mapsto R_t(\zeta)= (H_t-\zeta)^{-1}$ is continuous at $t_0$ uniformly for $|\zeta-\zeta_0|<\varepsilon$.
        \item\lb{18i2}Assume Hypothesis \ref{hyp1.3i2} and suppose that the mappings $t\mapsto \tr_t$, $t\mapsto V_t$, $t\mapsto Q_t$ are Lipschitz continuous at $t_0$. Then there exists a constant $c>0$ such that for all $(t,\zeta)\in\cU_\varepsilon$ one has
        \begin{align}\lb{new1.29n}
        &\| R_{t}(\zeta)- R_{t_0}(\zeta) \|_{\cB(\cH)}\le c|t-t_0|.
        \end{align}
        \item\lb{18i3}
        Assume Hypothesis \ref{hyp1.3i3} and suppose that the mappings $t\mapsto \tr_t$, $t\mapsto V_t$, $t\mapsto Q_t$ are differentiable at $t_0$.  Then for some $\varepsilon>0$ the following asymptotic expansion holds uniformly for $|\zeta-\zeta_0|<\varepsilon$,
        \begin{align}
        \begin{split}\lb{new1.43nn}
        R_t(\zeta)\underset{t\rightarrow t_0}{=}R_{t_0}(\zeta)&+
        \big(-R_{t_0}(\zeta)\dot V_{t_0}R_{t_0}(\zeta)+(\tr_{t_0}  R_{t_0}(\overline{\zeta}))^*\dot Q_{t_0}J \tr_{t_0}    R_{t_0}(\zeta)\\
        &+(\tr_{t_0}  R_{t_0}(\overline{\zeta}))^*J\dot \tr_{t_0}  R_{t_0}(\zeta)\big)(t-t_0)+o(t-t_0),\ \text{\ in\ }\cB(\cH).
        \end{split}
        \end{align}
        In particular, the function $t\mapsto R_t(\zeta_0)=(H_t-{\zeta_0})^{-1}$ is differentiable at $t=t_0$ and satisfies the following Riccati equation
        \begin{align}
        \begin{split}\lb{derR}
        \dot{R}_{t_0}(\zeta_0)=-R_{t_0}(\zeta_0)\dot V_{t_0}R_{t_0}(\zeta_0)&+(\tr_{t_0}  R_{t_0}(\overline{\zeta_0}))^*\dot Q_{t_0}J \tr_{t_0}    R_{t_0}(\zeta_0)\\
        &+(\tr_{t_0}  R_{t_0}(\overline{\zeta_0}))^*J\dot \tr_{t_0}  R_{t_0}(\zeta_0).
        \end{split}
        \end{align}
        \end{enumerate}
        \end{theorem}
The operators whose adjoints enter \eqref{new1.43nn}, \eqref{derR} are considered as elements of $\cB(\cH, \bndra)$, cf.\ Proposition \ref{prop2.4},  and their adjoints are elements of $\cB(\bndra, \cH)$, the dot denotes the derivative with respect to $t$ evaluated at $t_0$.  We emphasize the generality of formulas \eqref{new1.43nn}--\eqref{derR} where all three objects may vary: the domain of the extension, the trace operator, and the ``lower order terms'' of the operator itself. In Theorem \ref{prop5.1} we will give analogous results using a slightly different description of the domains of the self-adjoint extensions. Also, see Theorem \ref{theorem4.5} for the case when the trace operator  is $t$-independent. We refer to Remark \ref{rem1.8} below for somewhat more symmetric versions of the RHS of  \eqref{new1.43nn} and \eqref{derR} and to Remark \ref{rem:cond} for a comment on the continuity and differentiability conditions in the theorem.

\begin{proof}
	First,  we prove that the mapping $t\mapsto R_t(\bfi)\in\cB(\cH)$ is continuous at $t_0$.  Hypothesis  \ref{hyp1.3i2} by Remark \ref{rem3.14} yields
	\begin{equation}\lb{2.4a}
	\| R_t(\bfi)\|_{\cB(\cH, \cD)}=\cO(1), t\rightarrow t_0.
	\end{equation} 
	 Using \eqref{new5.14a} with $\zeta=\bfi$, $s=t_0$, and \eqref{2.4a} we get
	\begin{align}
	\begin{split}\lb{prop1.8}
	R_t(\bfi)-  R_{t_0}(\bfi)&=R_t({\bfi})(V_{t_0}-V_t)R_{t_0}(\bfi)\\
	&\quad+(\tr_t   R_t(-{\bfi}))^*(Q_t-Q_{t_0})JQ_{t_0} \tr_{t_0}    R_{t_0}(\bfi)\\
	&\qquad+(\tr_t   R_t(-{\bfi}))^*J (\tr_t-\tr_{t_0})    R_{t_0}(\bfi)\underset{t\rightarrow t_0}{=}o(1).
	\end{split}
	\end{align}
	
	{\it Proof of (\ref{18i1}),(\ref{18i2})}.  Fix $\varepsilon_0>0$ such that $\mathbb{B}_{\varepsilon_0}(\zeta_0)\subset \C\setminus \Sp(H_{t_0})$. Then  by  \eqref{prop1.8} and \cite[Theorem VIII.23]{RS1} we have
	$\mathbb{B}_{\varepsilon_0}(\zeta_0)\cap\Sp(H_{t})=\emptyset$
	for $t$ sufficiently close to $t_0$. Hence, 
	\begin{equation}\lb{new1.17}
	\sup\{\| R_t(\zeta)\|_{\cB(\cH)}: (t,\zeta)\in \cU_{\varepsilon}\}<\infty
	\end{equation}
	for a sufficiently small $\varepsilon>0$.  We claim that yet a smaller choice of $\varepsilon>0$ gives
	\begin{equation}\lb{2.11a}
	\sup\{\|  R_t(\zeta)\|_{\cB(\cH, \cD)}: (t,\zeta)\in \cU_{\varepsilon}\}<\infty.
	\end{equation}
	Indeed, by the resolvent identity one has
	\begin{equation}
	  R_t(\zeta)=	  R_t(\bfi)-(\bfi-\zeta)  R_t(\bfi)R_t(\zeta).
	\end{equation}
	Using this and \eqref{2.4a}, we see that \eqref{new1.17} yields \eqref{2.11a}. Next, by \eqref{new5.14a} and \eqref{2.11a} we infer
	\begin{align}
	\begin{split}\lb{new1.19}
	R_t(\zeta)-   R_{t_0}(\zeta) &=R_t({\zeta})(V_{t_0}-V_t)R_{t_0}(\zeta)\\
	&\quad+(\tr_t   R_t(\overline{\zeta}))^*(Q_t-Q_{t_0})JQ_{t_0} \tr_{t_0}    R_{t_0}(\zeta)\\
	&\qquad+(\tr_t   R_t(\overline{\zeta}))^*J (\tr_t-\tr_{t_0})    R_{t_0}(\zeta)\\
	&\hspace{-1cm}\leq c\max\{\|Q_{t}-Q_{t_0}\|_{\cB(\mathfrak{H}\times\mathfrak{H})}, \|\tr_{t}-\tr_{t_0}\|_{\cB(\cH_+, \mathfrak{H}\times\mathfrak{H})}, \|V_{t}-V_{t_0}\|_{\cB(\cH)} \} 
	\end{split}
	\end{align}
	for some $c>0$ and all $(t,\zeta)\in\cU_{\varepsilon}$; here we used the inequality
	\begin{equation}\lb{new1.19n}
	\|\tr_t   R_t(\overline{\zeta})\|_{\cB(\cH, \bndra)}\leq \|\tr_t\|_{\cB(\cD, \bndra)}\|R_t(\overline{\zeta})\|_{\cB(\cH, \cD)}, 
	\end{equation}
	see Proposition \ref{prop2.4} and Remark \ref{remark2.7}. Now both assertions {\it (\ref{18i1}),(\ref{18i2})} follow from \eqref{new1.19}. 
	
	{\it Proof of (\ref{18i3})}. First, we notice that \eqref{2.2bo} and the resolvent identity give 
	\begin{equation}\lb{2.13}
	\|  R_{t}(\zeta)-  R_{t_0}(\zeta) \|_{\cB(\cH,\cD)}\rightarrow 0,\ t\rightarrow0, 
	\end{equation} 
	uniformly for $|\zeta-\zeta_0|<\varepsilon$, with $\varepsilon>0$ as above.
	Next, by assumptions we have
	\begin{align}
	\begin{split}\lb{new1.20n}
	&Q_t\underset{t\rightarrow t_0}{=}Q_{t_0}+\dot Q_{t_0}(t-t_0)+o(t-t_0),\\
	&V_t\underset{t\rightarrow t_0}{=}V_{t_0}+\dot V_{t_0}(t-t_0)+o(t-t_0),\\
	&\tr_t\underset{t\rightarrow t_0}{=}\tr_{t_0}+\dot \tr_{t_0}(t-t_0)+o(t-t_0).
	\end{split}
	\end{align}
	Combining these expansions, \eqref{new5.14a}, \eqref{new1.29n}, and \eqref{2.13} we see that
	\begin{align}
	\begin{split}\no
	&R_t(\zeta)- R_{t_0}(\zeta)\underset{t\rightarrow t_0}{=}( R_{t_0}(\zeta)+\cO(t-t_0))(-\dot V_{t_0}(t-t_0)+o(t-t_0))R_{t_0}(\zeta) \\
	&+\big((\tr_{t_0}+\cO(t-t_0)) ( R_{t_0}(\overline{\zeta})+O_{\|\cdot\|_{\cB(\cH, \cD)}}(1) )\big)^*\times\\&\hskip3cm\times(\dot Q_{t_0}(t-t_0)+o(t-t_0)) JQ_{t_0} \tr_{t_0}   R_{t_0}(\zeta)\\
	&+\big((\tr_{t_0}+\cO(t-t_0)) ( R_{t_0}(\overline{\zeta})+O_{\|\cdot\|_{\cB(\cH, \cD)}}(1))\big)^*\times\\&\hskip3cm\times J (\dot \tr_{t_0}(t-t_0)+o(t-t_0))   R_{t_0}(\zeta)\\
	&\underset{t\rightarrow t_0}{=}\big(-R_{t_0}(\zeta)\dot V_{t_0}R_{t_0}(\zeta)+(\tr_{t_0} R_{t_0}(\overline{\zeta}))^*\dot Q_{t_0}J \tr_{t_0}   R_{t_0}(\zeta)\\
	&\hskip3cm+(\tr_{t_0}R_{t_0})^*J\dot \tr_{t_0} R_{t_0}(\zeta)\big)(t-t_0)+o(t-t_0),
	\end{split}
	\end{align}
	 in $\cB(\cH)$ uniformly for $|\zeta-\zeta_0|<\varepsilon$. This shows \eqref{new1.43nn} which implies \eqref{derR}.
\end{proof}
\begin{remark}\lb{rem1.8}
	The operator $\dot Q_{t_0}J\in\cB(\mathfrak{H}\times\mathfrak{H})$ in \eqref{new1.43nn}, \eqref{derR} is self-adjoint. Indeed, 
	since $\ran(Q_t)$ is Lagrangian, we have $J=JQ_{t}+Q_{t}J$ which implies the assertion upon differentiating with respect to $t$.  Since $\dot{Q}_tJ=-J\dot{Q}$ we can re-write the term $\dot{Q}_{t_0}J$ in \eqref{new1.43nn} and \eqref{derR} in a more symmetric fashion as \[\dot{Q}_{t_0}J=\frac12\big(\dot{Q}_{t_0}J-J\dot{Q}_{t_0}).\] Furthermore, the identity $Q_tJQ_t=0$ yields 
\[\big(\tr_t R_{t_0}(\overline{\zeta})\big)^*J\tr_tR_{t_0}(\zeta)=
\big(Q_t\tr_t R_{t_0}(\overline{\zeta})\big)^*JQ_t\tr_tR_{t_0}(\zeta)=0.\] Differentiating this identity at $t=t_0$ shows that the respective terms in the RHS of \eqref{new1.43nn} and \eqref{derR} could be also re-written as
	\begin{align}
&(\tr_{t_0}  R_{t_0}(\overline{\zeta_0}))^*J\dot \tr_{t_0}  R_{t_0}(\zeta_0)=\frac12\Big((\tr_{t_0}  R_{t_0}(\overline{\zeta_0}))^*J\dot \tr_{t_0}  R_{t_0}(\zeta_0)\\&\hskip5cm-(\dot \tr_{t_0}  R_{t_0}(\overline{\zeta_0}))^*J \tr_{t_0}  R_{t_0}(\zeta_0)\Big).
	\end{align}
\hfill$\Diamond$  \end{remark}
\begin{remark}\label{rem:cond}
 The assumptions of continuity and differentiability of the families $\tr, V$ and $Q$ are imposed at a fixed point $t_0\in[0,1]$. For many interesting examples these assumptions hold for all $t_0\in[0,1]$; a typical situation of this type is described in Example \ref{ODEexm}. However, these assumptions might fail for some points in $[0,1]$. A typical example of the latter situation is furnished by the classical Hadamard formula setting for star-shaped domains described in Section \ref{HFSSD} where the trace operator is singular at $t_0=0$ but is differentiable for each $t_0\in(0,1]$.
\hfill$\Diamond$  \end{remark}

\begin{remark} Discontinuities of the path $t\mapsto Q_t$ in general result in discontinuities of the eigenvalues curves. To give an example, let $\cA_t=-\Delta$ be the realization of the Laplacian on a bounded Lipschitz domain $\Omega\subset \bbR^n, n\geq 2$, subject to the boundary conditions $\chi_{[0,1/2]}(t)\gaD u+\chi_{(1/2, 1]}(t)\gaN u=0$; here $\gaD, \gaN$ are Dirichlet and Neumann traces, cf. Appendix \ref{appA}, and $\chi$ is the characteristic function.  That is, $\cA_t$ is the Dirichlet Laplacian for $t\in[0,1/2]$ and the Neumann Laplacian for $t\in(1/2,1]$. The corresponding path of Lagrangian planes is piece-wise constant with a jump at $t=1/2$. At this point the boundary conditions change from Dirichlet to Neumann and, due to the celebrated inequality of L. Friedlander \cite{Fri91}, this produces a non-trivial shift in the spectrum, which, in turn, shows the discontinuities of the eigenvalues. We revisit Friedlander's Inequality in Example \ref{FREX} below and provide a symplectic proof thereof, cf.\ \cite{CJM2}. 
\end{remark}
\subsection{Hadamard-type variational formulas}\label{Sec3.3}
In this section, we derive the first order expansion formula for the mapping $t\mapsto P(t)H_tP(t)$ near $t=t_0$. Here, the operator $H_t=\cA_t+V_t$ is as in Hypothesis \ref{hyp1.3i3} and $P(t)$ is a spectral projection of $H_t$ which corresponds to the $\lambda$-group, cf.\ \cite[Section II.5.1]{K80}, consisting of $m$ isolated eigenvalues of $H_t$ bifurcating from the eigenvalue $\lambda=\lambda_{t_0}$ of multiplicity $m$ of the operator $H_{t_0}$, see Hypothesis \ref{hyp3.1o} below. A subtlety is presented by the fact that  the operators $P(t)H_tP(t)$ act in varying finite-dimensional spaces $\ran(P(t))$; we rectify this by means of a unitary mapping $U: \ran(P(t_0))\rightarrow \ran(P_t)$, as in, e.g., \cite[Section I.4.6]{K80}. After this we use the first order perturbation theory for finite-dimensional operators, cf.\ \cite[Section II.5.4]{K80},  to deduce a formula for the derivative of the eigenvalue curves which we call the {\em Hadamard-type variational formula}, see \eqref{2.36}. This terminology stems from a classical Rayleigh--Hadamard--Rellich formulas for derivatives of the eigenvalues of Laplacian posted on a parameter-dependent family of domains, cf.\ Section \ref{HFSSD} below for details of this particular situation. We note that the approach adopted in this section was originally carried out in  \cite{LS17} for a specific PDE situation of the one-parameter family of Schr\"odinger operators with Robin boundary conditions on star-shaped domains mentioned in Example \ref{parametricSchr}.
\begin{hypothesis}\lb{hyp3.1o}  
For a given  $t_0\in[0,1]$, we assume that $\lambda=\lambda(t_0)$ is an isolated eigenvalue of $H_{t_0}$ with finite multiplicity $m\in\bbN$. Let 
	\begin{equation}\no
	\gamma:=\big\{z\in\bbC: 2|z-\lambda|=\dist\big(\lambda, \Sp(H_{t_0})\setminus\{\lambda\}\big)\big\},
	\end{equation}
	and let $B\subset\bbC$ denote the disc enclosed by $\gamma$ such that $\Sp(H_{t_0})\cap B=\{\lambda\}$. 
\end{hypothesis}

Throughout this section we assume Hypothesis \ref{hyp1.3i2}, and that the maps $t\mapsto\tr_t, V_t, Q_t$ are continuous at a given $t_0\in[0,1]$. By Theorem \ref{prop1.8new}, there exists $\varepsilon>0$ such that $\gamma$ encloses $m$ eigenvalues (not necessarily distinct) of the operator $H_t$ whenever $|t-t_0|<\varepsilon$ and $\varepsilon>0$ is sufficiently small.  For such $t$ we let $P(t)$ denote the Riesz projection
\begin{equation}\lb{3.1}
P(t):=\frac{-1}{2\pi \bfi}\int_{\gamma} R_t(\zeta)d\zeta, R_t(\zeta)=(H_t-\zeta)^{-1}
\end{equation} 
and recall the reduced resolvent given by 
\begin{equation}\lb{3.1a}
S:=\frac{1}{2\pi \bfi}\int_{\gamma}(\zeta-\lambda)^{-1} R_{t_0}(\zeta)d\zeta						
\end{equation}	 
and the identity $P(t_0)R_{t_0}(\zeta)=(\lambda-\zeta)^{-1}P(t_0)$.
\begin{remark}\lb{prop3.2}
	The Riemann sums defining integrals in \eqref{3.1}, \eqref{3.1a} converge not only in $\cB(\cH)$ but also in $\cB(\cH, \cD)$. Consequently,  $ P(t), S\in\cB(\cH, \cD)$. In addition, one has
	\begin{align}\lb{3231}
	\frac{1}{2\pi \bfi}	\int_{\gamma} \tr_{t} \big((\zeta-\lambda)^{-1}R_{t}(\zeta)\big)d\zeta&=	 \tr_{t}\frac{1}{2\pi \bfi}	\int_{\gamma} \big((\zeta-\lambda)^{-1}R_{t}(\zeta)\big)d\zeta=\tr_t S,\\
	&(\tr_t P(t))\in\cB(\cH, \bndra). \lb{3232}
	\end{align}
This follows from continuity of the mapping $\bbC\ni\zeta\mapsto R_t(\zeta)\in\cB(\cH, \cD)$ for every $t\in[0,1]$ which can be inferred from
	$
	 R_t(\zeta)- R_t(\zeta_0)=(\zeta-\zeta_0) R_t(\zeta)R_t(\zeta_0),
$
	(cf.  \eqref{new1.17}, \eqref{2.11a}), and  $\tr_t\in\cB(\cD, \bndra)$.
	\hfill$\Diamond$  \end{remark}

Next we derive an asymptotic expansion of $P(t)H_tP(t)$ for $t$ near $t_0$. To that end, we introduce the operator $D(t):=P(t)-P(t_0)$ satisfying $\|D(t) \|_{\cB(\cH)}\underset{t\rightarrow t_0}{=}o(1)$, which follows from \eqref{new1.29n}, \eqref{3.1}. In particular, for $t$ near $t_0$ the following operators are well defined
\begin{align}
\begin{split}\lb{int27}
&U(t):=(I-D^2(t))^{-1/2}((I-P(t))(I-P(t_0))+P(t)P(t_0)),\\
&U(t)^{-1}=((I-P(t_0))(I-P(t))+P(t_0)P(t))(I-D^2(t))^{-1/2},
\end{split}
\end{align}
moreover, as in  \cite[Section I.4.6]{K80}, \cite[Proposition 2.18]{F04}, we note  that
\begin{equation}\lb{int28}
U(t)P(t_0)=P(t)U(t),
\end{equation}
and  that $U(t)$ maps $\ran(P(t_0))$ onto $\ran(P(t))$ unitarily (for $t$ near $t_0$). Given this auxiliary operators we are ready to expand $P(t)H_tP(t)$, which is an $m$-dimensional operator, for $t$ near $t_0$.   
\begin{lemma}\lb{theorem2.2}  For a given $t_0\in[0,1]$ we assume 
that the mappings $t\mapsto \tr_t$, $t\mapsto V_t$, $t\mapsto Q_t$ are differentiable at $t_0$ and that Hypotheses \ref{hyp1.3i3} and \ref{hyp3.1o} hold. Then one has
	\begin{align}
	\begin{split}
	\lb{1.42n}
	&P(t_0)U(t)^{-1}H_tP(t)U(t)P(t_0)\underset{t\rightarrow t_0}{=}\lambda P(t_0)+\Big(P(t_0)\dot V_{t_0}P(t_0)\\
	&\quad-(\tr   P(t_0))^*\dot Q_{t_{0}}J\tr   P(t_0)-(\tr_{t_0}  P(t_0))^*J\dot \tr_{t_0}   P(t_0)\Big)(t-t_0)+o(t-t_0).
	\end{split}
	\end{align}
\end{lemma}

\begin{proof} Our strategy is to expand the left-hand side of \eqref{1.42n} using  \eqref{new1.43nn}. Multiplying \eqref{new1.43nn} by $P(t_0)$ from the right and using identity
\begin{equation}\label{RPR}
R_{t_0}(\zeta)P(t_0)=P(t_0)R_{t_0}(\zeta)=(\lambda-\zeta)^{-1}P(t_0),
\end{equation} 
where $R_t(\zeta)=(H_t-\zeta)^{-1}$, we get	
		\begin{align}
		\begin{split}\lb{1.45new}
		R_t&(\zeta)P(t_0)\underset{t\rightarrow t_0}{= } {(\lambda-\zeta)^{-1}}P(t_0)+(\lambda-\zeta)^{-1}
\Big(-R_{t_0}(\zeta)\dot V_{t_0}P(t_0)\\&+\big(\tr_{t_0} R_{t_0}(\overline{\zeta})\big)^*\, \dot Q_{t_0}J\tr_{t_0} P(t_0)
		+\big(\tr_{t_0}  R_{t_0}(\overline{\zeta})\big)^*J\dot \tr_{t_0}  P({t_0})\Big)(t-t_0)\\
		&+o(t-t_0).
		\end{split}
		\end{align}
The proof is split in several steps.

	{\bf Step 1.} One has 
	\begin{align}\lb{1.44old}
	&P(t_0)P(t)P(t_0)\underset{t\rightarrow t_0}{=}P(t_0)+o(t-t_0).
	\end{align}
	\begin{proof}
	For any continuous $F:\gamma\to\cB(\mathfrak{H}\times\mathfrak{H}, \cH)$ we have
	\[\Big(\int_\gamma F(\zeta)\, d\zeta\Big)^*=-\int_\gamma (F(\overline{\zeta}))^*\, d\zeta.\]
	 Applying this to 
	$F(\zeta)=\frac{1}{2\pi\bfi}(\lambda-{\zeta})^{-1}\tr_{t_0}R_{t_0}({\zeta})$ and using 
	\eqref{3.1a}, \eqref{3231} yields
	\[\int_\gamma\big(\frac{1}{2\pi\bfi}(\lambda-\overline{\zeta})^{-1}\tr_{t_0}R_{t_0}(\overline{\zeta})\big)^*\, d\zeta=\big(-\int_\gamma\frac{1}{2\pi\bfi}(\lambda-\zeta)^{-1}\tr_{t_0}R_{t_0}(\zeta)\, d\zeta\big)^*=(\tr_{t_0}S)^*.	\]
We use this, multiply both sides  of \eqref{1.45new} by $-\frac{1}{2\pi \bfi }$ and integrate over $\gamma$ to obtain the following,
		\begin{align}
		\begin{split}\lb{1.44}
		P(t)P(t_0)&\underset{t\rightarrow t_0}{=}  P(t_0)+\Big(-S\dot V_{t_0}P(t_0)+\left(\tr_{t_0}   S\right)^*\, \dot Q_{t_0}J\tr_{t_0}  P(t_0)\\
		&\qquad+\left(\tr_{t_0}   S\right)^*J\dot \tr_{t_0}  P({t_0})\Big)(t-t_0)+o(t-t_0).
		\end{split}
		\end{align}
	Taking adjoints we get
		\begin{align}
		\begin{split}\lb{1.45}
		P(t_0)P(t)&\underset{t\rightarrow t_0}{=}  P(t_0)+\Big(-P(t_0)\dot V_{t_0}S+(\tr_{t_0}
		   P(t_0))^*\, \dot Q_{t_0}J\tr_{t_0}  S\\
		&\qquad+\big(\dot\tr_{t_0}  P({t_0})\big)^*J \tr_{t_0}   S\Big)(t-t_0)+o(t-t_0).
		\end{split}
		\end{align}
		Multiplying this by $P(t_0)$ from the right and using $SP(t_0)=0$  we arrive at \eqref{1.44old}.
	\end{proof}
	{\bf Step 2.} One has
	\begin{align}
	P(t_0)U(t)P(t_0)&=(P(t_0)U^{-1}(t)P(t_0))^*\underset{t\rightarrow t_0}{=}P(t_0)+o(t-t_0),\lb{1.48new}\\
	(I-P(t_0))U(t)P(t_0)&= (P(t_0)U(t)^{-1}(I-P(t_0)))^*\\
	& \underset{t\rightarrow t_0}{=}(I-P(t_0)) \Big(-S\dot V_{t_0}P(t_0)+\left(\tr_{t_0}   S\right)^*\, \dot Q_{t_0}J\tr_{t_0}    P(t_0)\\
	&\quad\qquad+\left(\tr_{t_0}   S\right)^*J\dot \tr_{t_0}  P({t_0})\Big)(t-t_0)+o(t-t_0).\lb{1.50new}
	\end{align}
	\begin{proof}
First we note an auxiliary expansion
$D(t)\underset{t\rightarrow t_0}{=}\cO(t-t_0)$ which follows
from \eqref{new1.29n}, \eqref{3.1} and formula $D(t)=P(t)-P(t_0)$.
Thus 
\[(I-D^2(t))^{-1/2}\underset{t\rightarrow t_0}{=}I+\cO(|t-t_0|^2)\]
and then
		\begin{align}
		\begin{split}\lb{1.51}
		U(t)&=(I-D^2(t))^{-1/2}((I-P(t))(I-P(t_0))+P(t)P(t_0))\\
		&\underset{t\rightarrow t_0}{=}((I-P(t))(I-P(t_0))+P(t)P(t_0))+o(t-t_0).
		\end{split}
		\end{align}
		Using this and \eqref{1.44old} we obtain
		\begin{equation}
		P(t_0)U(t)P(t_0)\underset{t\rightarrow t_0}{=}P(t_0)P(t)P(t_0)+o(t-t_0)\underset{t\rightarrow t_0}{=}P(t_0)+o(t-t_0).
		\end{equation}
		Similarly, employing  \eqref{1.51} one infers 
		\begin{align}
		&(I-P(t_0))U(t)P(t_0)\underset{t\rightarrow t_0}{=}(I-P(t_0))P(t)P(t_0)+o(t-t_0)
		\end{align}
and thus \eqref{1.50new} follows by multiplying  \eqref{1.44} 
by $I-P(t_0)$ from the left.
	\end{proof}
	
	{\bf Step 3. } One has
	\begin{align}
	\begin{split}\lb{1.57}
	P(t_0)&U^{-1}(t)R_t(\zeta)U(t)P(t_0)\underset{t\rightarrow t_0}{=}(\lambda-\zeta)^{-1}P(t_0)\\
	&+(\lambda-\zeta)^{-2}\Big(-P(t_0)\dot V_{t_0}P(t_0)+\big(\tr_{t_0}   P(t_0)\big)^*\dot Q_{t_0}J\tr_{t_0}   P(t_0)\\
	&+\big(\tr_{t_0}  P(t_0)\big)^*J\dot \tr_{t_0}  P({t_0})\Big)(t-t_0)+o(t-t_0).
	\end{split}
	\end{align}
	\begin{proof} First, we sandwich the middle term in the left-hand side, $R_t(\zeta)$, by $P(t_0)+(I-P(t_0))$ and write
		\begin{align}
		&P(t_0)U^{-1}(t)R_t(\zeta)U(t)P(t_0)=I+II+III+IV.
		\end{align}
		Let us treat each term individually, starting with
		\begin{align}
		I:=P(t_0)U^{-1}(t)(I-P(t_0))&\times(I-P(t_0))R_t(\zeta)P(t_0)\\
		&\times P(t_0)U(t)P(t_0)\underset{t\rightarrow t_0}{=} o(t-t_0),
		\end{align}
		by  \eqref{1.45new},  \eqref{1.48new}, \eqref{1.50new}
		as the main terms in the RHS of \eqref{1.45new} and \eqref{1.50new} both contain the factor $(t-t_0)$.  Similarly, we infer
		\begin{align}
		II:=P(t_0)U^{-1}(t)P(t_0)&\times P(t_0)R_t(\zeta)(I-P(t_0))\\
		&\times(I-P(t_0))U(t)P(t_0)\underset{t\rightarrow t_0}{=} o(t-t_0),
		\end{align}
		by  \eqref{1.45new}, \eqref{1.48new}, \eqref{1.50new}, and 
		\begin{align}
		III:= &P(t_0)U^{-1}(t)(I-P(t_0))\times R_t(\zeta)\\
		&\times (I-P(t_0))U(t)P(t_0)\underset{t\rightarrow t_0}{=} o(t-t_0),
		\end{align}
		by \eqref{1.50new}. The last term admits the required in \eqref{1.57} expansion because 
		\begin{align}
		IV:=P(t_0)U^{-1}(t)P(t_0)\times P(t_0)R_t(\zeta)P(t_0)\times P(t_0)U(t)P(t_0)
		\end{align}
and  we can use  \eqref{1.45new}, identity  \eqref{RPR} and (twice)\eqref{1.48new}.
	\end{proof}
	
	{\bf Step 4.} Recalling the identities 
	\begin{equation}
	H_tP(t):=\frac{-1}{2\pi \bfi}\int_{\gamma} \zeta R_t(\zeta)d\zeta, \quad \frac{1}{2\pi \bfi}\int_{\gamma} \zeta (\lambda-\zeta)^{-2}d\zeta=1,
	\end{equation}	
	multiplying \eqref{1.57} by $-\zeta/2\pi \bfi$ and then integrating  over $\gamma$ we arrive at \eqref{1.42n}
\end{proof}

We are ready to present the main result of this section that gives a formula for the slopes of the appropriately chosen branches of the eigenvalues curves bifurcating from an isolated eigenvalue of finite multiplicity. We recall that our assumptions on differentiability of $\tr, V$ and $Q$ is imposed at a particular point $t_0$ where $\lambda=\lambda(t_0)$ is the  isolated eigenvalue, cf.\  Remark \ref{rem:cond}. To avoid confusions we also refer to the classical Rellich's example \cite[Example V.4.14]{K80} recalled as Example \ref{exm:rell} below to emphasize that we are not claiming global differentiability of all  eigenvalue curves. Indeed, in this example there is a point $t_0$ where one eigenvalue curve has a singularity and so our assumptions do not hold while all others curves are differentiable as claimed in the theorem. 

\begin{theorem}\lb{thm2.2} Assume Hypotheses \ref{hyp1.3i3} and \ref{hyp3.1o} and suppose that  the mappings $t\mapsto \tr_t$, $t\mapsto V_t$, $t\mapsto Q_t$ are differentiable at $t_0$. We introduce the operator
	\begin{equation}
	T^{(1)}:=P(t_0)\dot V_{t_0}P(t_0)-(\tr_{t_0}   P(t_0))^*\dot Q_{t_{0}}J\tr_{t_0}   P(t_0)-(\tr_{t_0}  P(t_0))^*J\dot \tr_{t_0}  P(t_0),
	\end{equation}
	and denote the eigenvalues and the orthonormal eigenvectors of this  $m$ dimensional operator by $\{\lambda^{(1)}_j\}_{j=1}^m$ and $\{u_j\}_{j=1}^m\subset\ran(P(t_0))=\ker(H_{t_0}-\lambda)$ correspondingly\footnote{We stress that  $u_j$ are eigenvectors of $H_{t_0}$ corresponding to its eigenvalue $\lambda=\lambda(t_0)$.}. Then there exists a labeling of the eigenvalues $\{\lambda_j(t)\}_{j=1}^m$ of $H_t$, for $t$ near $t_0$,  satisfying the asymptotic formula
	\begin{equation}
	\lambda_j(t)\underset{t\rightarrow t_0}{=} \lambda+\lambda_j^{(1)}(t-t_0)+o(t-t_0),\lb{2.35}
	\end{equation}
	moreover, one has
	\begin{equation}
	\dot\lambda_j(t_0)=\langle \dot V_{t_0}u_j, u_j\rangle_{\cH}+\omega(\dot Q_{t_0}\tr_{t_0} u_j, \tr_{t_0} u_j)+\omega( \tr_{t_0}u_j, \dot \tr_{t_0} u_j),\lb{2.36}
	\end{equation}
	for each $1 \leq j\leq m$.
\end{theorem}
\begin{proof}
	Recalling that $U(t)$ is a unitary map between $\ran(P(t_0))$ and $\ran(P(t))$, see  \cite[Section I.4.6]{K80}, \cite[Proposition 2.18]{F04}, we note that  $H_t\upharpoonright_{\ran(P(t))}$ is similar to \[P(t_0)U(t)^{-1}H_tP(t)U(t)P(t_0)\upharpoonright_{\ran(P(t_0))}\] for $t$ near $t_0$.  In particular the eigevalues of these operators coincide and it is sufficient to expand the eigenvalues of the latter.  To that end we utilize the expansion \eqref{1.42n}  together with the finite dimensional first order perturbation theory, specifically,  \cite[Theorem II.5.11]{K80}, to deduce \eqref{2.35}. Next, we have
	\begin{align}
	\begin{split}\lb{aub30}
		&\dot\lambda_j(t_0)=	\lambda_j^{(1)}=\langle T^{(1)} u_j,  u_j \rangle_{\cH}\\
	&=\langle\big( P(t_0)\dot V_{t_0}P(t_0)-(\tr_{t_0}   P(t_0))^*\dot Q_{t_{0}}J\tr_{t_0}   P(t_0)-(\tr_{t_0}  P(t_0))^*J\dot \tr_{t_0}  P(t_0)\big) u_j, u_j\rangle_{\cH}\\
	&=\langle \dot V_{t_0}u_j, u_j\rangle_{\cH}-\omega(\tr_{t_0} u_j, \dot Q_{t_0}\tr_{t_0} u_j)-\omega( \dot \tr_{t_0}u_j,  \tr_{t_0} u_j)\\
	&=\langle \dot V_{t_0}u_j, u_j\rangle_{\cH}+\omega(\dot Q_{t_0}\tr_{t_0} u_j, \tr_{t_0} u_j)+\omega( \tr_{t_0}u_j, \dot \tr_{t_0} u_j)
	\end{split}
	\end{align}
	 which gives \eqref{2.36}. In the last step we used the inclusions \[\omega(\tr_{t_0} u_j, \dot Q_{t_0}\tr_{t_0} u_j)\in\bbR
	 \text{ and } \omega( \tr_{t_0}u_j, \dot \tr_{t_0} u_j)\in\bbR.\] The latter inclusion follows from $\omega( \tr_{t}u_j,  \tr_{t} u_j)=0$ after differentiating at $t=t_0$. To prove the former inclusion we use $JQ_t+Q_tJ=J$ to get $J\dot{Q}_{t_0}=-\dot{Q}_{t_0}J$ and write
	 \begin{align}
	 \begin{split}\lb{aub33}
	  &\omega(\tr_{t_0} u_j, \dot Q_{t_0}\tr_{t_0} u_j)=\langle J\tr_{t_0} u_j, \dot Q_{t_0}\tr_{t_0} u_j \rangle_{\bndra}\\
	  &\quad=-\langle J \dot Q_{t_0}\tr_{t_0} u_j, \tr_{t_0} u_j \rangle_{\bndra}\\
	 &\quad=-\omega( \dot Q_{t_0}\tr_{t_0} u_j,\tr_{t_0} u_j)=\overline{\omega(\tr_{t_0} u_j, \dot Q_{t_0}\tr_{t_0} u_j)},
	 \end{split}
	 \end{align} 
as claimed.\end{proof}

In PDE and quantum graph settings the Lagrangian planes are often defined by operators $[X, Y]$ as in \eqref{3241}--\eqref{inv} rather than by orthogonal projections onto these planes. It is therefore natural to restate \eqref{new1.43nn}, \eqref{2.36} in these terms which we do next.  Given families $t\mapsto X_t, Y_t\in\cB(\mathfrak H)$, we will now denote by $\cA_t$ the self-adjoint extension of $A$ with $\dom(\cA_t):=\{u\in\cD:\,
[X_t, Y_t]\tr_t u=0\}$, that is, we augment \eqref{3272} by requiring that
\begin{align}
\begin{split}\lb{3271}
&\overline{\tr_t(\dom(\cA_t))}=\ran (Q_t)=\ker([X_t, Y_t]), \\
&X_t,Y_t\in\cB(\mathfrak{H}); X_t Y_t^*=Y_tX_t^*, 0\not\in \spec(M^{X_t,Y_t}),
\end{split}
\end{align}
where $M^{X_t,Y_t}$ is defined in  \eqref{inv}. We recall formula \eqref{orproj} for the projection $Q_t$ onto $\ker([X_t. Y_t])$. A typical example of $X_t,Y_t$ are given by $X_t=I$ and $Y_t=-\Theta_t$ where $\Theta_t$ is an operator (in general, not local) entering the Robbin boundary condition.

\begin{theorem}\lb{prop5.1}  Under Hypothesis \ref{hyp1.3}, if $\cA_t$ satisfies \eqref{3271} then the following symplectic resolvent difference formula holds
for the resolvent $R_t(\zeta)=(H_t-\zeta)^{-1}$ of the operator $H_t=\cA_t+V_t$,
	\begin{align}
	\begin{split}
	\lb{XYKrein}
	R_{t}(\zeta)- R_{s}(\zeta)&=  R_t({\zeta})(V_s-V_t)R_s(\zeta)+(\tr_{t}  R_{t}(\overline{\zeta}))^*\, Z_{t,s}\tr_{s} \,   R_{s}(\zeta)\\
	&\quad+(\tr_{t}  R_t(\overline{\zeta}))^*J (\tr_t-\tr_{s})   R_{s}(\zeta),
	\end{split}
	\end{align}	
	where $\zeta \not \in(\spec(H_t)\cup \spec(H_s))$, $s,t\in[0,1]$, and the operator $Z_{t,s}\in\cB(\bndra)$ is given by formula \eqref{defw},
	\begin{equation}\label{ZXYKr}
	Z_{t,s}:=\big(W(X_{t}, Y_{t})\big)^*( X_{t}{Y}_{s}^*-Y_{t}{X}_{s}^*)\big(W(X_{s}, Y_{s})\big).
	\end{equation}
	Moreover, under Hypothesis \ref{hyp1.3i2}, if the mappings $t\mapsto \tr_t, V_t, X_t,Y_t$ are continuous at $t_0\in[0,1]$ in the respective spaces of operators, then the function $t\mapsto R_t(\zeta_0)$ is continuous at $t=t_0$ for
any	$\zeta_0 \not \in\spec(H_{t_0})$.
	Further, assume Hypothesis \ref{hyp1.3i3} and suppose that the mappings $t\mapsto \tr_t, V_t, X_t,Y_t$ are differentiable at $t_0\in[0,1]$. Then the function $t\mapsto R_t(\zeta_0)=(H_t-\zeta_0)^{-1}$  is differentiable at $t=t_0$ and satisfies the following Riccati equation,
	\begin{align}
	\begin{split}\lb{derRnew}
	&\dot{R}_{t_0}(\zeta_0)=-R_{t_0}(\zeta_0)\dot V_{t_0}R_{t_0}(\zeta_0)\\
	&+(\tr_{t_0}  R_{t_0}(\overline{\zeta_0}))^*\big(W(X_{t_0}, Y_{t_0})\big)^*(\dot X_{t_0}{Y}_{t_0}^*-\dot Y_{t_0}{X}_{t_0}^*)\big(W(X_{t_0}, Y_{t_0})\big)\times\\&\hskip3cm\times \tr_{t_0}    R_{t_0}(\zeta_0)\\
	&+(\tr_{t_0}  R_{t_0}(\overline{\zeta_0}))^*J\dot \tr_{t_0}  R_{t_0}(\zeta_0), \quad \zeta_0 \not \in\spec(H_{t_0}).
	\end{split}
	\end{align} 
Furthermore, if $\lambda(t_0)\in\spec(H_{t_0})$ is an isolated eigenvalue of multiplicity $m\ge1$ then there exists  a choice of orthonormal eigenfunctions $\{u_j\}_{j=1}^m\subset \ker(H_{t_0}-\lambda(t_0))$ and a labeling of the eigenvalues $\{\lambda_j(t)\}_{j=1}^m$  of $H_t$, for $t$ near $t_0$, such that the following Hadamard-type formula holds, 
	\begin{equation}\lb{5.10new}
	\dot{\lambda}_j(t_0)= \langle \dot V_{t_0}u_j, u_j\rangle_{\cH}+\big\langle (X_{t_0}\dot{Y}_{t_0}^*-Y_{t_0}\dot{X}_{t_0}^*) \phi _j, \phi_j\big\rangle_{\mathfrak{H}}+\omega( \tr_{t_0}u_j, \dot \tr_{t_0} u_j) ,\ 
	\end{equation}
	where we denote $\phi_{j}=W(X_{t_0}, Y_{t_0})\tr_{t_0} u_j$, $1\leq j\leq m$, with the operator $W$ defined in \eqref{defw}, or, equivalently, $\phi_j$ is a unique vector in $\mathfrak H$ satisfying
	\begin{equation}\lb{phij}
		\Gamma_0 u_j=-Y_{t_0}^*\phi_j \, \text{ and }\,	\Gamma_1 u_j=X^*_{t_0}\phi_j.
	\end{equation}
\end{theorem}
\begin{proof}  The resolvent difference formula \eqref{XYKrein} follows from \eqref{new5.14a} and the computation
	\begin{align}
	(\tr_{t}  R_t(\overline{\zeta}))^*(Q_t-Q_{s})J \tr_{s}   R_{s}(\zeta)&=(\tr_{t}  R_t(\overline{\zeta}))^*Q_tJQ_s\tr_{s}   R_{s}(\zeta)\\
	&= (\tr_{t}  R_{t}(\overline{\zeta}))^*\, Z_{t,s}\tr_{s} \,   R_{s}(\zeta).
	\end{align}
	Hypothesis \ref{hyp1.3i2} and \eqref{XYKrein} imply continuity of $t\mapsto R_t(\zeta)$ as in the proof of Theorem \ref{prop1.8new}.
	To prove \eqref{derRnew} we remark that
	$X_tY_s^*-Y_tX_s^*=(X_t-X_s)Y_s^*-(Y_t-Y_s)X_s^*$ by \eqref{com}. Plugging this in \eqref{ZXYKr}, using \eqref{XYKrein} at $s=t_0$, dividing by $(t-t_0)$ and passing to the limit as $t\to t_0$ yields \eqref{derRnew}. 
	Next we turn to \eqref{5.10new}. We recall that $u_j$ in Theorem \ref{thm2.2} are the eigenvectors in $\ran(P(t_0))$ such that $T^{(1)}u_j=\lambda_j^{(1)}u_j$. But since $\ran(P(t_0))=\ker(H_{t_0}-\lambda(t_0))$ the vectors $u_j$ are also eigenvectors of $H_{t_0}$ such that $H_{t_0}u_j=\lambda(t_0)u_j$. By \eqref{2.36} we only need to show  
	\begin{equation}\label{3.40new}
	\omega(\dot Q_{t_0}\tr_{t_0} u_j, \tr_{t_0} u_j)=\big\langle (X_{t_0}\dot{Y}_{t_0}^*-Y_{t_0}\dot{X}_{t_0}^*) \phi _j, \phi_j\big\rangle_{\mathfrak H}. 
	\end{equation}
	 Using \eqref{orproj} and differentiating $Q_t$	we infer
	 \begin{align} 
	&\omega(\dot Q_{t_0}\tr_{t_0} u_j, \tr_{t_0} u_j)\\
	&\quad=\omega\Big([-Y^*_{t_0}, X^*_{t_0}]^\top
	\big(\frac{\rm d}{{\rm d}t}\Big|_{{t=t_0}}W(X_t, Y_t)\big)\tr_{t_0} u_j, \tr_{t_0} u_j\Big)\\
	&\quad\quad+\omega\Big(\big(\frac{\rm d}{{\rm d}t}\Big|_{{t=t_0}}[-Y^*_{t_0}, X^*_{t_0}]^\top\big)
	W(X_{t_0},Y_{t_0})\tr_{t_0} u_j, \tr_{t_0} u_j\Big)\\
	&\quad=\Big\langle\Big(\frac{\rm d}{{\rm d}t}\Big|_{{t=t_0}}W(X_t, Y_t)\Big)\tr_{t_0} u_j, [X_{t_0}, Y_{t_0}]\tr_{t_0} u_j\Big\rangle_{\mathfrak{H}}\\
	&\quad\quad+\omega\Big(\Big(\frac{\rm d}{{\rm d}t}\Big|_{{t=t_0}}[-Y^*_{t_0}, X^*_{t_0}]^\top\Big)
	W(X_{t_0},Y_{t_0})\tr_{t_0} u_j, \tr_{t_0} u_j\Big)\\
	&\quad=\omega\Big(\Big(\frac{\bd}{\bd t}\Big|_{{t=t_0}}[-Y^*_{t_0}, X^*_{t_0}]^\top\Big)
	W(X_{t_0},Y_{t_0})\tr_{t_0} u_j, \tr_{t_0} u_j\Big). 
	\end{align}
	where we used $[X_{t_0}, Y_{t_0}]\tr_{t_0} u_j=0$. Finally, employing \eqref{orproj} and 
	\begin{equation}\label{phijpr}
	\tr_{t_0} u_j=Q_{t_0}\tr_{t_0} u_j= [-Y_{t_0}^*, X_{t_0}^*]^{\top}\phi_j,\, \,  \phi_{j}:=W(X_{t_0}, Y_{t_0})\tr_{t_0} u_j,
	\end{equation}
	we obtain
	\begin{align}
	\omega(\dot Q_{t_0}\tr_{t_0} u_j, \tr_{t_0} u_j)&=\Big\langle[
	\dot X_{t_0}^*,\dot Y_{t_0}^*
	]^\top\phi_j, [-Y_{t_0}^*, X_{t_0}^*]^{\top}\phi_j\Big\rangle_{\mathfrak H}\\&=\big\langle (X_{t_0}\dot{Y}_{t_0}^*-Y_{t_0}\dot{X}_{t_0}^*) \phi _j, \phi_j\big\rangle_{\mathfrak H},
	\end{align}
thus completing the proof of \eqref{5.10new}, while \eqref{phij} follows from \eqref{phijpr}.\end{proof}

\begin{remark}
We close with a remark that assertions proved in Theorem \ref{prop5.1} allow one to make conclusions regarding the behavior of the spectra of the operators $H_t$ as a function of $t$, see, e.g., \cite[Theorem VIII.23]{RS1}. Also, the results of this section can be used to study  various properties of strongly continuous semigroups generated by  the operators $-H_t$. For instance, the Trotter-Kato Approximation Theorem, see, e.g.,  \cite[Theorem III.4.8 ]{EnNa00}, implies that the semigroups are continuous with respect to the parameter $t$ as soon as the continuity of the resolvent of $H_t$ in Theorem \ref{prop5.1} is established, see Section \ref{ssHeatEq} for an example. 
\hfill$\Diamond$  \end{remark}

\section{\sel{Ordinary}  boundary triplets}\lb{abt}

\sel{Ordinary} boundary triplets have been intensively studied since probably \cite{Calkin1939,GG}, see the vast bibliography in \sel{\cite{Behrndt_2020, DMbook, DM, Schm}} and related papers \sel{\cite{BL2012, MR0415381, DM95, DHMS2006,DHMS2012,HS2012,MR0365218, W2012a}} and the bibliography therein.
In this section we revisit main results of Sections \ref{KreinFormulas} and \ref{section1} in the context of \sel{ordinary}  boundary triplets and present several applications.
The case of  boundary triplets is the one that is widely considered in the literature, and in this section we will see that for this case one may impose fewer assumptions to prove the same set of general results. Also, we will demonstrate that this case is sufficient to cover many interesting applications. In particular, we show that conclusions of Theorems \ref{prop1.8new}, \ref{thm2.2}, \ref{prop5.1} hold under a mere assumption that the mappings $t\mapsto Q_t$, $t\mapsto \tr_t$, $t\mapsto V_t$ are continuous (differentiable) with respect to $t$ and that $(\mathfrak{H}, \Gamma_{0, t}, \Gamma_{1, t})$ is an \sel{ordinary} boundary triplet. Utilizing this, we derive Hadamard-type formulas for quantum graphs, Schr\"odinger operators with singular potentials, and Robin realizations of the Laplace operator on bounded domains.

We recall the following widely used definition, cf.\ \sel{\cite[Section 2.1]{Behrndt_2020}, \cite{DMbook}, \cite{GG}, \cite[Section 14.2]{Schm}}.

\begin{definition}\lb{cond} Given a symmetric densely defined closed operator $A$ on a Hilbert space $\cH$ with equal deficiency indices, we equip $\cH_+=\dom(A^*)$ with the graph scalar product and consider linear operators $\Gamma_0$ and $\Gamma_1$ acting from $\cH_+$ to a (boundary) Hilbert space $\mathfrak{H}$. We say that 
  $(\mathfrak H, \Gamma_0, \Gamma_1)$ is a {\em \sel{ordinary} boundary triplet} if the operator $\tr:=(\Gamma_0, \Gamma_1): \cH_+\to\mathfrak{H}\times\mathfrak{H}$ is surjective
  and the following abstract Green identity holds,
	\begin{equation}\label{3.61biss}
	\langle A^*u,v\rangle_{\cH}-\langle u,A^*v\rangle_{\cH}=\langle\Gamma_1u,\Gamma_0v\rangle_{\mathfrak{H}}-\langle\Gamma_0u,  \Gamma_1v\rangle_{\mathfrak{H}} \text{ for all $u,v\in\cH_+$}.\end{equation}
\end{definition}

In other words,  $(\mathfrak H, \Gamma_0, \Gamma_1)$ is an \sel{ordinary} boundary triplet provided Hypothesis \ref{hyp3.6} holds with $\cD=\cH_+$ and surjective $\tr$. In this case, we have $\tr\in\cB(\cH_+, \bndra)$ by Proposition \ref{remark2.2} (2). 

\begin{remark}\lb{rem4.2}
The setting of  \sel{ordinary} boundary triplets gives a particularly simple illustration of Corollary \ref{LLSASALL}. Specifically, if $(\mathfrak H, \Gamma_0, \Gamma_1)$ is a  \sel{ordinary} boundary triplet associated with $A$ then $\cG\subset \mathfrak H\times\mathfrak H$ is Lagrangian if and only if $\cA:= A^*|_{\tr^{-1}(\cG)}$ is self-adjoint. In other words, the Lagrangian plane $\cG$ and the self-adjoint operator $\cA:= A^*|_{\tr^{-1}(\cG)}$  are automatically aligned in the sense of Definition \ref{defASSOC} as long as $(\mathfrak H, \Gamma_0, \Gamma_1)$ is a  \sel{ordinary} boundary triplet. In particular, if $\cA$ is  a self-adjoint extension of $A$ then the subspace $\tr(\dom(\cA))$ is closed, cf.\ \cite[Lemma 14.6(iii)]{Schm}.
\hfill$\Diamond$  \end{remark}

\subsection{Main results for the case of boundary triplets}\label{Sec4.2} 
In this section we discuss our main results, Theorems \ref{prop1.8new}, \ref{prop5.1}, in the context of boundary triplets. In Proposition 4.5 we verify that Hypothesis \ref{hyp1.3i3} (and, hence, Hypothesis \ref{hyp1.3i2}) holds automatically for  boundary triplets. This allows us to obtain the central result of the current section, Theorem \ref{thm:sec4.1}. The latter, in turn, gives a plethora of applications discussed in Sections \ref{ssLapLip}, \ref{ssQuaGr}, \ref{ssPerKrPen}, \ref{mastriples}. 

In the setting of  boundary triplets Hypothesis \ref{hyp1.3} should be naturally replaced by the following assumption.
\begin{hypothesis}\label{hyp1.3bis}
 Let \[\tr:[0,1]\rightarrow\cB(\cH_+, \mathfrak{H}\times\mathfrak{H}):t\mapsto\tr_t:=[\Gamma_{0t}, \Gamma_{1t}]^\top\] be a one-parameter family of trace operators. Suppose that  $(\mathfrak H, \Gamma_{0t}, \Gamma_{1t})$ is an \sel{ordinary} boundary triplet for each $t\in[0,1]$.   Let  $Q: [0,1]\rightarrow \cB(\mathfrak{H}\times\mathfrak{H}), t\mapsto Q_t$ be a one-parameter family of orthogonal projections. Suppose that $\ran  (Q_t) \in\Lambda(\mathfrak{H}\times\mathfrak{H})$ is a Lagrangian plane for each $t\in[0,1]$. Let $\cA_t$ be a family of self-adjoint extensions of $A$ satisfying 
 \begin{align}
 {\tr_t\big( \dom(\cA_{t})\big)}=\ran (Q_t).\lb{3272new-bis}
 \end{align}
 Let 
 $V: [0,1]\rightarrow \cB(\cH):t\mapsto V_t$ be a one-parameter family of	self-adjoint bounded operators. We denote $H_t:= \cA_t+V_t$ and $R_t(\zeta):= (H_t-\zeta)^{-1}\in\cB(\cH)$ for $\zeta\not\in \Sp(H_t)$ and $t\in[0,1]$. 
\end{hypothesis}

\begin{proposition}\lb{prop2.4h}
Suppose that Hypothesis \ref{hyp1.3bis} holds for the  \sel{ordinary} boundary triplet \\ $(\mathfrak{H}, {\Gamma_0}_t,{\Gamma_1}_t)$. If $Q$ and $\tr$ are continuous at a given $t_0\in[0,1]$ then 
	\begin{equation}\label{ABC}
	\big\|(\cA_t-\bfi)^{-1}-(\cA_{t_0}-\bfi)^{-1}\big\|_{\cB(\cH, \cH_+)}=o(1),\ t\rightarrow t_0. 
	\end{equation}
	In other words, Hypothesis \ref{hyp1.3i3} is automatically satisfied for the  boundary triplets.
\end{proposition}
\begin{proof}
	We claim that 
	\begin{align}\lb{3291ineq}
	\|(\cA_t-\bfi)^{-1}-(\cA_{t_0}-\bfi)^{-1}\|_{\cB(\cH, \cH_+)}\leq\sqrt2 	\|(\cA_t-\bfi)^{-1}-(\cA_{t_0}-\bfi)^{-1}\|_{\cB(\cH)}.
	\end{align}
	Indeed, using $\cA_t\subset A^*$, $\cA_{t_0}\subset A^*$ we get
	\begin{align}
	&\|(\cA_t-\bfi)^{-1}h-(\cA_{t_0}-\bfi)^{-1}h\|_{\cH_+}^2=\|(\cA_t-\bfi)^{-1}h-(\cA_{t_0}-\bfi)^{-1}h\|_{\cH}^2\\
	&\quad+\|A^*(\cA_t-\bfi)^{-1}h-A^*(\cA_{t_0}-\bfi)^{-1}h\|_{\cH}^2=2\|(\cA_t-\bfi)^{-1}h-(\cA_{t_0}-\bfi)^{-1}h\|_{\cH}^2.
	\end{align}
	Thus it is enough to prove that the right-hand side of \eqref{3291ineq} is $o(1)$ as $t\to t_0$. To this end, we first  note that, given $\cA_tu+\bfi u=f, u\in\dom(\cA_t)$, we have
	\begin{align}
	&\|(\cA_t+\bfi)^{-1}f\|^2_{\cH_+}=\|u\|^2_{\cH_+}= \|A^*u\|_{\cH}^2+\| u\|^2_{\cH}\\
	&\quad= \|\cA_tu\|_{\cH}^2+\| u\|^2_{\cH}=\|\cA_tu+\bfi u\|^2_{\cH}=\|f\|^2_{\cH}; 
	\end{align}
	hence,
	\begin{equation}\lb{2.4o}
	\|(\cA_t+\bfi)^{-1}\|_{\cB(\cH, \cH_+)}\leq 1.
	\end{equation}
	By resolvent difference formula \eqref{5.14a} we infer
	\begin{align}
	\begin{split}\lb{2.5o}
	&\|(\cA_t-\bfi)^{-1}-(\cA_{t_0}-\bfi)^{-1}\|_{\cB(\cH)}\\
	&\quad=\|(\tr_t (\cA_t+\bfi)^{-1})^*(Q_t-Q_{t_0})J\tr_t (\cA_{t_0}+\bfi)^{-1}\|_{\cB(\cH)}\\
	&\quad\le\|\tr_t \|_{\cB(\cH_+, \bndra)} \|(\cA_t+\bfi)^{-1}\|_{\cB(\cH, \cH_+)} \|(Q_t-Q_{t_0})\|_{\cB(\bndra)}\times\\
	&\hspace{5cm}\times  \|\tr_t\|_{\cB(\cH_+, \bndra)}  \|(\cA_{t_0}+\bfi)^{-1}\|_{\cB(\cH, \cH_+)}\\
	&\quad\leq c\|Q_t-Q_{t_0}\|_{\cB(\bndra)}\underset{t\rightarrow t_0}{=}o(1), c>0,
	\end{split}
	\end{align}
	where we used \eqref{2.4o}, and continuity of $Q$ and $\tr$ at $t_0$. Then \eqref{3291ineq}, \eqref{2.5o} yield \eqref{ABC} and so equation \eqref{2.2bo} in Hypothesis \ref{hyp1.3i3} holds. 
\end{proof}
We summarize our main results for the case of  boundary triplets as follows.

\begin{theorem}\lb{thm:sec4.1}  Assume Hypothesis \ref{hyp1.3bis}. If $\cA_t$ is defined as in \eqref{3271} and $H_t=\cA_t+V_t$ then for $R_t(\zeta)=(H_t-\zeta)^{-1}$ the following resolvent difference formula holds,
	\begin{align}
	\begin{split}\lb{aub28}
	R_{t}(\zeta)- R_{s}(\zeta)&=  R_t({\zeta})(V_s-V_t)R_s(\zeta)+(\tr_{t}  R_{t}(\overline{\zeta}))^*\, Z_{t,s}\tr_{s} \,   R_{s}(\zeta)\\
	&\quad+(\tr_{t}  R_t(\overline{\zeta}))^*J (\tr_t-\tr_{s})   R_{s}(\zeta),
	\end{split}
	\end{align}	
	where $\zeta \not \in(\spec(H_t)\cup \spec(H_s))$, $s,t\in[0,1]$ and 
	\begin{equation}
	Z_{t,s}:=\big(W(X_{t}, Y_{t})\big)^*( X_{t}{Y}_{s}^*-Y_{t}{X}_{s}^*)\big(W(X_{s}, Y_{s})\big),
	\end{equation}
	with the operator $W$ defined in \eqref{defw}. Moreover, if the mappings $t\mapsto \tr_t, V_t, X_t,Y_t$ are continuous at $t_0\in[0,1]$ in the respective spaces of operators, then the function $t\mapsto R_t(\zeta_0)$ is continuous at $t=t_0$ for
	any	$\zeta_0 \not \in\spec(H_{t_0})$.
	Further, if the mappings $t\mapsto \tr_t, V_t, X_t,Y_t$ are differentiable at $t_0\in[0,1]$, then the function $t\mapsto R_t(\zeta_0)=(H_t-\zeta_0)^{-1}$  is differentiable. In this case, the following two assertions hold:
	
	(1) The resolvent operators satisfy the following differential equation,
	\begin{align}\lb{5.10newnewa}
	\begin{split} 
	&\dot{R}_{t_0}(\zeta_0)=-R_{t_0}(\zeta_0)\dot V_{t_0}R_{t_0}(\zeta_0)\\
	&+(\tr_{t_0}  R_{t_0}(\overline{\zeta_0}))^*\big(W(X_{t_0}, Y_{t_0})\big)^*(\dot X_{t_0}{Y}_{t_0}^*-\dot Y_{t_0}{X}_{t_0}^*)\big(W(X_{t_0}, Y_{t_0})\big) \tr_{t_0}    R_{t_0}(\zeta_0)\\
	&+(\tr_{t_0}  R_{t_0}(\overline{\zeta_0}))^*J\dot \tr_{t_0}  R_{t_0}(\zeta_0), \quad \zeta_0 \not \in\spec(H_{t_0}).
	\end{split}
	\end{align} 

	(2) If  $\lambda(t_0)\in\spec(H_{t_0})$ is an isolated eigenvalue of multiplicity $m\ge1$ then there exists  a choice of orthonormal eigenfunctions $\{u_j\}_{j=1}^m\subset \ker(H_{t_0}-\lambda(t_0))$ and a labeling of the eigenvalues $\{\lambda_j(t)\}_{j=1}^m$  of $H_t$, for $t$ near $t_0$, such that 
	\begin{equation}\lb{5.10newnew}
	\dot{\lambda}_j(t_0)= \langle \dot V_{t_0}u_j, u_j\rangle_{\cH}+\big\langle (X_{t_0}\dot{Y}_{t_0}^*-Y_{t_0}\dot{X}_{t_0}^*) \phi _j, \phi_j\big\rangle_{\mathfrak{H}}+\omega( \tr_{t_0}u_j, \dot \tr_{t_0} u_j) ,\ 
	\end{equation}
	where $\phi_{j}=W(X_{t_0}, Y_{t_0})\tr_{t_0} u_j$, $1\leq j\leq m$, or, equivalently, $\phi_j$ is a unique vector in $\mathfrak H$ satisfying
	\begin{equation}\label{5.10-3new}
	\Gamma_0 u_j=-Y_{t_0}^*\phi_j \, \text{ and }\,	\Gamma_1 u_j=X^*_{t_0}\phi_j.
	\end{equation}
\end{theorem}

\begin{proof}
The resolvent difference formula \eqref{aub28} follows directly from \eqref{XYKrein}. The continuity of $t\mapsto R_t(\zeta_0)$ at $t_0$ follows from Theorem \ref{prop5.1} upon noticing that Hypothesis \ref{hyp1.3i2} holds in the setting of boundary triplets by Proposition \ref{prop2.4h}. Similarly, Proposition \ref{prop2.4h} combined with \eqref{derRnew}, \eqref{5.10new}, yield \eqref{5.10newnewa}, \eqref{5.10newnew}.
\end{proof}
\begin{remark} (1)\,  In the setting of Theorem \ref{thm:sec4.1}, the resolvent difference formula \eqref{aub28}  can be also rewritten as
	\begin{align}
\begin{split}\lb{aub103}
R_{t}(\zeta)- R_{s}(\zeta)&=  \cR_t({\zeta})(V_s-V_t)R_s(\zeta)+\cR_{t}({\zeta})\tr_{t}  ^*\, Z_{t,s}\tr_{s} \,   R_{s}(\zeta)\\
&\quad+\cR_{t}({\zeta})\tr_{t}  ^*J (\tr_t-\tr_{s})   R_{s}(\zeta),
\end{split}
\end{align}	
where in the RHS we have $\cR_{t}({\zeta})\in\cB(\cH_-, \cH)$, that is, as in Proposition \ref{rem:adjtrR} and Corollary \ref{Rem2.6}, we view $\cR_{t}({\zeta})\in\cB(\cH_-, \cH)$ as a unique extension of $R_{t}({\zeta})\in\cB(\cH)$ to an element of $\cB(\cH_-, \cH)$, while $\tr_t\in\cB(\cH_+, \bndra)$, $\tr_t^*\in\cB(\bndra, \cH_-)$. We note that, in a more general setting of Theorem \ref{prop5.1}, the trace operator $\tr_t$ is unbounded and one only has the inclusion $(\tr_{t}  R_{t}(\overline{\zeta}))^*\supseteq R_{t}({\zeta})(\tr_t)^*$. In this case, \eqref{aub103} holds provided $\ran(Z_{t,s}\tr_{s} \,   R_{s}(\zeta))\subseteq J\tr(\cD)$. 

(2)\, The resolvent difference formula derived in Theorem \ref{thm:sec4.1} yields  continuity of the  mapping $\cB(\mathfrak H)\times\cB(\mathfrak H)\ni (X, Y)\mapsto (\cA_{X,Y}-\bfi)^{-1}\in \cB(\cH)$; here, for an  \sel{ordinary} boundary triplet $(\mathfrak{H}, \Gamma_0, \Gamma_1)$, we denote by $\cA_{X,Y}$ the self-adjoint extension of $A$ such that $\tr(\dom(\cA_{X,Y}))=\ker([X, Y])$, cf.\ \eqref{3271}. 
\hfill$\Diamond$  \end{remark}

In sections \ref{ssLapLip}, \ref{ssQuaGr}, \ref{ssPerKrPen}, \ref{mastriples} below we will give applications of Theorem \ref{thm:sec4.1} for several important classes of problems that fit the framework of the boundary triplets. To give the simplest possible illustration of the setup described in Hypothesis \ref{hyp1.3bis} and of Theorem \ref{thm:sec4.1} we now consider the following two ODE examples where the conclusions of the theorem are  probably well-known, see, e.g., \cite{AGMST, ClarkGesMitr,ClarkGesNickZinch, K80} and the vast bibliography therein.
\begin{example}\label{ODEexm}
Let $Au=-u''$ be the minimal symmetric operator on $\cH=L^2(0,1)$ with domain $\dom(A)=H^2_0(0,1)$ so that $A^*u=-u''$ with $\dom(A^*)=\cH_+=H^2(0,1)$, set $\mathfrak{H}=\bbC^2$
and introduce the surjective trace operator $\tr=(\Gamma_0,\Gamma_1)\in\cB(\cH_+,\mathfrak{H}\times\mathfrak{H})$ using the Dirichlet and (inward) Neumann traces $\Gamma_0u=[u(0), u(1)]^\top$ and $\Gamma_1u=[u'(0), - u'(1)]^\top$. Integration by parts yields \eqref{3.61biss}, and thus $(\mathfrak{H},\Gamma_0,\Gamma_1)$ is an \sel{ordinary} boundary triplet, cf.\ \cite[Section 14.4]{Schm}. For $t\in[0,1]$ we let $\cA_t$ denote the self-adjoint extension of $A$ with the domain
\begin{equation}\label{domAT}
\dom(\cA_t)=\{u\in H^2(0,1): \cos({\pi}t/2)\Gamma_0 u-\sin({\pi}t/2)\Gamma_1 u=0\}=\ker([X_t,Y_t]),\end{equation}  where, cf.\ \eqref{3271}, 
\[X_t=\cos({\pi}t/2)I_2,\, Y_t=-\sin({\pi}t/2)I_2, \, Q_t=\begin{bmatrix}\sin^2(\pi t/2)& \frac12\sin(\pi t)\\\frac12\sin(\pi t)&\cos^2(\pi t/2)\end{bmatrix}. \]
Given a bounded real-valued potential $V$, we
let $H_tu=-u''+Vu$, $t\in[0,1]$, be the family of scalar Schr\"odinger operators on $L^2(0,1)$ equipped with the boundary conditions 
specified in \eqref{domAT} so that Hypothesis \ref{hyp1.3bis} holds.
In particular, $H_0$ is the Dirichlet and $H_1$ is the Neumann Schr\"odinger operator.
To apply Theorem \ref{thm:sec4.1} we first perform a standard calculation of the resolvent $R_t(\zeta)=(H_t-\zeta)^{-1}$, cf., e.g., \cite[Lemma 9.7]{Tes14}: For $t\in[0,1]$ and $\zeta\in\bbC$ we let $v_t(\cdot\,;\zeta)$,  $w_t(\cdot\,;\zeta)$ denote the solutions to the equation $-u''+Vu=\zeta u$ that satisfy the initial conditions 
\begin{align*}
(v_t(0;\zeta), v'_t(0,\zeta))&=\big(\sin({\pi}t/2), \cos({\pi}t/2)\big),\\ (w_t(1;\zeta), w'_t(1,\zeta))
&=\big(\sin({\pi}t/2), -\cos({\pi}t/2)\big),
\end{align*}
and let $\cW_t(\zeta)=v_t(x;\zeta)w'_t(x;\zeta)-v'_t(x;\zeta)w_t(x;\zeta)$ denote their Wronskian. Then for each $u\in L^2(0,1)$ the function $R_t(\zeta)u$ is given by the formula
\[\big(R_t(\zeta)u\big)(x)=(\cW_t(\zeta))^{-1}\Big(w_t(x;\zeta)\int_0^xv_t(y;\zeta)u(y) {\rm d } y
+v_t(x;\zeta)\int_x^1w_t(y;\zeta)u(y) {\rm d} y\Big),\]
$x\in[0,1]$.Using this, it is convenient to write $\tr R_t(\zeta)=K_t(\zeta)L_t(\zeta)$ where we temporarily introduced the $(4\times2)$ matrix $K_t(\zeta)$ and the operator $L_t(\zeta)$ by the formulas
\begin{align}
&K_t(\zeta)=(\cW_t(\zeta))^{-1}\big[\sin({\pi}t/2)I_2, \cos({\pi}t/2)I_2\big]^\top,\\
&L_t(\zeta)u=\big[\langle w_t(\cdot;\zeta), \overline{u}\rangle_{L^2}, \langle v_t(\cdot;\zeta), \overline{u}\rangle_{L^2}\big]^\top, L_t(\zeta)\in\cB(L^2(0,1), \bbC^2)
\end{align}
so that $(L_t(\zeta))^*$ maps $(z_1,z_2)\in\bbC^2$ into $ w_t(\cdot\,;\overline{\zeta})z_1+ v_t(\cdot\,;\overline{\zeta})z_2\in L^2(0,1)$. Theorem \ref{thm:sec4.1} and  a short calculation  now yield
\begin{align}
\big(&R_t(\zeta)-R_s(\zeta)\big)u=(\cW_t(\zeta)\cW_s(\zeta))^{-1}\sin(\pi(t-s)/2)\\
&\times\big(\langle w_s(\cdot;\zeta), \overline{u}\rangle_{L^2}w_t(\cdot;\zeta)+ \langle v_s(\cdot;\zeta), \overline{u}\rangle_{L^2}v_t(\cdot;\zeta)\big), \,\zeta\notin\spec(H_t)\cup\spec(H_s),
\end{align}
\begin{align}
\dot{R}_t(\zeta)u&=\frac{\pi}{2}(\cW_t(\zeta))^{-2}
\big(\langle w_t(\cdot;\zeta), \overline{u}\rangle_{L^2}w_t(\cdot;\zeta)+ \langle v_t(\cdot;\zeta), \overline{u}\rangle_{L^2}v_t(\cdot;\zeta)\big),\\&\hskip5cm\zeta\notin\spec(H_t), \\
\dot{\lambda}(t_0)&=-\frac{\pi}{2}
\big\|\sin(\pi t_0/2)\Gamma_0u_0+\cos(\pi t_0/2)\Gamma_1u_0\big\|^2_{\bbC^2},\,  t_0\in[0,1],
\end{align}
where $u_0$ is the normalized  eigenfunction corresponding to the eigenvalue $\lambda(t_0)\in\spec(H_{t_0})$.
\hfill$\Diamond$  \end{example}

\begin{example}\label{exm:rell}   As promised prior to Theorem \ref{thm2.2}, we now recall the classical Rellich's example, cf., e.g., \cite[Example V.4.14]{K80} which
shows the singularity at $t_0=0$ of the smallest eigenvalue $\lambda^{(0)}(t)$ of the operator $\cA_t=-\partial^2_{xx}$ in $L^2(0,1)$ equipped with the boundary conditions $u(0)=0$, $tu'(1)=u(1)$ for real $t$; meanwhile, the resolvent $t\mapsto (\cA_t-\bfi)^{-1}$ is continuous and all other eigenvalues $\lambda^{(k)}(t)$, $k=1,2,\dots$, are differentiable for each $t$ including $t=0$, see \cite[Fig. 1, p.292]{K80}. Indeed, letting $\Gamma_0u=(u(0), u(1))^\top$, $\Gamma_1u=(u'(0), -u'(1))^\top$ for $u\in\cH_+:=H^2(0,1)$, $\mathfrak{H}=\bbC^2$ and
\begin{equation}\label{DXQ}
X_t=\begin{bmatrix}1&0\\0&-1\end{bmatrix},\, Y_t=\begin{bmatrix}0&0\\0&-t\end{bmatrix},\,
Q_t=\begin{bmatrix}0&0&0&0\\0&t^2(1+t^2)^{-1}&0&-t(1+t^2)^{-1}\\
0&0&1&0\\0&-t(1+t^2)^{-1}&0&(1+t^2)^{-1}\end{bmatrix},\end{equation}
we notice that $\dom(\cA_t)=\ker([X_t \, Y_t])=\ran Q_t$. The maps $t\mapsto X_t,Y_t,Q_t$ are all differentiable at each $t\in\bbR$ and so Theorem \ref{thm:sec4.1} (or Theorem \ref{prop5.1}) applies. In particular, the resolvent operators of $\cA_t$ are differentiable at each $t$, and a short calculation using \eqref{5.10newnew}, \eqref{5.10-3new} and \eqref{DXQ} shows that if $u$ denotes the norm one eigenfunction with the eigenvalue $\lambda(t)\in\Sp(\cA_t)$  then $\dot{\lambda}(t)=|u'(1)|^2$ provided we know that $\lambda(t)$ is an eigenvalue of $\cA_t$ for a given $t\in\bbR$. Thus, each of the branches $\lambda^{(k)}(\cdot)$, $k\in\{0\}\cup\bbN$,  of the eigenvalues is monotone for all $t$ where it is defined.

We proceed with finding the actual location of $\lambda\in\Sp(\cA_t)$ and formulas for $u$ dealing with the two possible cases: (i)\, $\lambda=\kappa^2>0$, respectively, (ii)\, $\lambda=-\kappa^2<0$ for $\kappa=\kappa(t)\in\bbR$. 
Solving the equation $u''=0$ with the boundary conditions we note that $\lambda=0\in\Sp(\cA_1)$ with $u=x$, and that $0\notin \Sp (\cA_t)$ for all $t\neq 1$.
Plugging a linear combination of (i)\, $\cos(\kappa x)$ and $\sin(\kappa x)$, respectively, (ii)\, $\sinh (\kappa x)$ and $\cosh (\kappa x)$ into the boundary value problem $-u''=\lambda u$, $u(0)=0$, $tu'(1)=u(1)$ shows that nonzero $\kappa=\kappa(t)$ are the solutions to the equation (i)\,  $t\kappa=\tan\kappa$ with $u=a\sin(\kappa x)$, 
$a^{-2}=(1-t\cos^2\kappa)/2$, respectively, equation (ii)\, $t\kappa=\tanh\kappa$ with $u=a\sinh(\kappa x)$, $a^{-2}=(t\cosh^2\kappa-1)/2$.
By inspection of the graphs in the equations, in case (i), for each $t\in\bbR$ and $n\in\bbZ\setminus\{0\}$ there is a unique solution 
$\kappa\in(-\pi/2+\pi n, \pi/2+\pi n)$, for each $t>1$ there is a unique solution $\kappa=\kappa(t)\in(-\pi/2,\pi/2)$ with $\kappa(t)\to0^+$ as $t\to1^+$, and for any $t<1$ there are no solutions $\kappa\in(-\pi/2,\pi/2)$. In case (ii), for any $t\le0$ or $t>1$ there are no solutions $\kappa\in\bbR$, while for each $t\in(0,1]$ there exists a unique solution $\kappa=\kappa(t)\in\bbR$ such that $\kappa(t)\to0$ as $t\to1^-$ and $\kappa(t)\to+\infty$ as $t\to0^+$. By squaring $\kappa$ we obtain the branches $\lambda^{(0)}(t)<\lambda^{(1)}(t)<\dots$
of the eigenvalues of $\cA_t$ such that $\lambda^{(0)}(t)$ are defined for $t>0$ with $\lambda^{(0)}(t)\to-\infty$ as $t\to0^+$, is negative for $t\in(0,1)$ and positive for $t>1$ while $\lambda^{(k)}(t)$ for $k\in\bbN$ is defined and positive for all $t\in\bbR$, cf.\
\cite[Fig. 1 p.292]{K80}. Using $\dot{\lambda}(t)=|u'(1)|^2$ and the expressions for $u$ just given one obtains very particular formulas for $\dot{\lambda}^{(k)}(t)$ for all $t$ and $k$ except when $k=0$ and $t\le0$.
\hfill$\Diamond$  \end{example}

\subsection{Laplace operator on bounded domains via boundary triplets}\label{ssLapLip}
The main result of this section is Theorem \ref{theorem5.4} in which we derive the resolvent difference formula, Riccati equation, and Hadamard-type formula for a family of Robin-type Laplacians. To that end, we employ abstract results of Theorem \ref{thm:sec4.1} with an   {ordinary} boundary triplet specifically defined for the Laplace operator.  {The construction of such triplet for second order elliptic operators goes back to the work of M.I. Vi\v{s}ik \cite{MR0052655, Vi} who proposed the regularization of the Neumann trace by means of the Dirichlet-to-Neumann map, G. Grubb \cite{Gr} who investigated the case of higher-order operators building upon the trace theory of J. L. Lions, E. Magenes \cite{MR0247243}. We also note that the work of M. Malamud \cite{M2010} provides boundary triplets with dual party in $L^2(\partial\Omega)\times L^2(\partial\Omega)$ as well as important relation between the Weyl function and the Dirichlet-to-Neumann map. Another relevant construction of trace maps is offered in \cite{BMN2017} where a $B-$regularized boundary triplet was originally proposed. } 

Throughout this section
we assume that  $n\in\mathbb{N},n\geq2$ and  $\Om\subset\mathbb{R}^n$ is a bounded domain with $C^{1,r}$-boundary with $r>1/2$ (although this assumption can be considerably weakened, see Remark \ref{rem:LipD} below).
We define the maximal and minimal Laplace operators as follows,
\begin{align}
-\Delta_{\max}&:\dom(-\Delta_{\max})\subset L^2(\Omega)\rightarrow L^2(\Omega),\no\\
\dom(-\Delta_{\max})&=\big\{u\in L^{2}(\Om)\big|\  \Delta u\in L^{2}(\Om)\big\},\no\\
-\Delta_{\max}u&=-\Delta u\ \  (\text{in the sence of distribudtions}),\\\no
\dom(-\Delta_{\min})&=H^2_0(\Om),\,-\Delta_{\min}u=-\Delta u,\no
\end{align}
and remark that by \cite[Theorem 8.14]{GM10}\footnote{\sel{the description of the minimal domain in the case of $C^2$ boundary $\partial\Omega$ is a classical result, cf.  \cite{Gr2009, MR0247243}}} one has
	\begin{equation}\label{m=m*}\begin{split}
	\dom(-\Delta_{\min})&=H^2_0(\Om)=\{u\in L^{2}(\Om)|\  \Delta u\in L^{2}(\Om) ,\ \gd(u)=0,\ \gn(u)=0\},\\
	-\Delta_{\min}&=(-\Delta_{\max})^*,\,\,-\Delta_{\max}=(-\Delta_{\min})^*.\end{split}
	\end{equation}
	Here and below we use the following extensions of the  Dirichlet and (weak) Neumann
	traces,
	\begin{equation}\label{dfnbtrDN1}\begin{split}
	\gd&:\{u\in L^{2}(\Om)\,|\,\Delta u\in L^{2}(\Om)\}\to H^{-1/2}(\dOm),\\
	\gn&:\{u\in L^{2}(\Om)\,|\,\Delta u\in L^{2}(\Om)\}\to H^{-3/2}(\dOm),\end{split}
	\end{equation}
	and consider the map
	\begin{equation}\label{dfnbtrDN2}
	\tN:\{u\in L^2(\Omega)| \Delta u\in L^2(\Omega)\}\rightarrow H^{1/2}(\partial\Omega),\,
	\tN u:=\gn u+ M_{D,N}(\gd u),
	\end{equation}
where $M_{D,N}$ is the Dirichlet-to-Neumann map acting by the rule $M_{D,N}:g\mapsto -\gn u$ for 
 $u$ being the solution of the boundary value problem
 \begin{equation}\label{bvpg}
-\Delta u =0,\,
u\in L^2(\Om),\quad     
\gd u =g   \text{\ on\ } \partial\Omega. \end{equation}
More details regarding the definitions of $\gd, \gn$ and $\tau_{{}_N}$ and their properties are discussed in Appendix \ref{appA}, cf. Lemma \ref{wd}, \ref{A.5}, \ref{nn} taken from \cite{GM10}. In the sequel we use the Reisz isomorphism given by 
\begin{align}
\begin{split}\lb{aub27}
&\Phi: H^{-1/2}(\partial\Omega)\rightarrow H^{1/2}(\partial\Omega),\\
&H^{-1/2}(\partial\Omega)\ni\psi\mapsto \Phi_{\psi}\in H^{1/2}(\partial\Omega),\\
& \langle f,\psi  \rangle_{{-1/2}}:=\psi(f)=\langle f,\Phi_{\psi}  \rangle_{1/2}, f\in H^{1/2}(\partial\Omega), \psi\in H^{-1/2}(\partial\Omega),
\end{split}
\end{align}
in particular, for $f, \psi \in H^{1/2}(\partial\Omega)\hookrightarrow
L^2(\partial\Omega)\hookrightarrow H^{-1/2}(\partial\Omega)$ we have
\begin{equation}
\langle f,\psi  \rangle_{{-1/2}}=\langle f,\psi  \rangle_{L^2(\partial\Omega)}.
\end{equation}
We also note that $\Phi$ is a conjugate linear mapping. The next lemma is a well-known fact that goes back to \cite{Gr}.
\begin{lemma}\label{lem:BTR} Assume that $\Om\subset\mathbb{R}^n$ is a bounded domain with $C^{1,r}$-boundary, $r>1/2$, and the boundary traces $\gd$, $\tau_{{}_N}$ are as in \eqref{dfnbtrDN1}, \eqref{dfnbtrDN2}. Then
\begin{equation}\lb{int70}
(\mathfrak H, \Gamma_0, \Gamma_1):=(H^{1/2}(\partial\Omega), \tau_{{}_N}, \Phi\gd)
\end{equation}
is an \sel{ordinary} boundary triplet for $A=-\Delta_{\min}$.
\end{lemma}
\begin{proof} The trace operator $\tr:= [\tau_{{}_N}, \Phi\gd]^\top$ is defined on the space \[\cH_+:=\{u\in L^2(\Omega): \Delta u\in L^2(\Omega)\}\] with the norm  \[\|u\|_{\cH_+}=(\|u\|_{L^2(\Omega)}^2+\|\Delta u\|_{L^2(\Omega)}^2)^{1/2}. \] 
Recalling the Green identity \eqref{Green}
\begin{align}
&(-\Delta u,v)_{L^{2}(\Om)}-(u,-\Delta v)_{L^{2}(\Om)}\no\\
&=-{\lnoh \tN u, \gd v \rnohs}+\overline{\lnoh \tN v, \gd u \rnohs},
\end{align}
we rewrite it as
\begin{align}
&\langle A^* u,v\rangle_{\cH}-\langle u,A^* v\rangle_{\cH}=-\langle\Gamma_0u, \Gamma_1 v\rangle_{\mathfrak H}+\overline{\langle\Gamma_0v, \Gamma_1 u\rangle_{\mathfrak H}},\no\\
&=\langle\Gamma_1 u, \Gamma_0v\rangle_{\mathfrak H}-\langle\Gamma_0u, \Gamma_1 v\rangle_{\mathfrak H}, 
\end{align}
and thus check that \eqref{int70} satisfies the abstract Green identity. It remains to show that the map $\tr:\cH_+\to H^{1/2}(\partial\Omega)\times H^{1/2}(\partial\Omega)$ is onto. 
We fix a vector $(f,g)\in H^{1/2}(\partial\Omega)\times H^{-1/2}(\partial\Omega)$. 
By  \eqref{onto} there exists 
$u_0\in H^2(\Omega)\cap H_0^1(\Omega)$ such that $\tN u_0=f$. By \cite[Theorem 10.4]{GM10}, the boundary value problem \eqref{bvpg}
has a unique solution that we denote by $v_0$ (we note that zero is outside of the spectrum of the Dirichlet Laplacian). Employing \eqref{kernel} and $v_0\in \ker(\tN)$  yields
\begin{align}
\tr (u_0+v_0)=(\tN(u_0+v_0),\Phi \gd(u_0+v_0))=(\tN u_0 ,\Phi \gd v_0)=(f,\Phi g)
\end{align}
since $\gd u_0=\gaD u_0 =0$.\end{proof}

\begin{remark} In PDE literature, boundary value problems are often formulated in terms of the Dirichlet and Neumann traces defined by 
\begin{align}
&\gaD: \{u\in H^1(\Omega): \Delta u\in L^2(\Omega) \}\rightarrow H^{1/2}(\dOm), \gaD:=\gd\upharpoonright_ {\{u\in H^1(\Omega): \Delta u\in L^2(\Omega) \}},\\
&\gaN: \{u\in H^1(\Omega): \Delta u\in L^2(\Omega) \}\rightarrow H^{-1/2}(\dOm), \gaN:=\gn\upharpoonright_ {\{u\in H^1(\Omega): \Delta u\in L^2(\Omega) \}}.
\end{align}	
We note that $(-\Delta_{\max}, \gaD, \gaN)$ is not an  \sel{ordinary} boundary triplet. Firstly, $\tr:=(\gaD, \gaN)$ is not defined on the entire space $\dom(-\Delta_{\max})$. Secondly, $\tr$ is not onto, see \cite[Proposition 2.11]{LS1}. However, Hypothesis \ref{hyp3.6} is still satisfied with $\cD:=\{u\in H^1(\Omega): \Delta u\in L^2(\Omega) \}$ and equipped with the norm $(\|u\|_{H^1(\Omega)}^2+\|\Delta u\|_{L^2(\Omega)}^2)^{1/2}$. In fact, Hypothesis \ref{hyp2.2} is also satisfied for this choice of $\tr, \cD$. These facts serve as our main motivation for introducing Hypotheses \ref{hyp3.6}, \ref{hyp2.2}. We elaborate on this further in Section \ref{subsec1.1}. 
\hfill$\Diamond$  \end{remark}
Having constructed the   \sel{ordinary} boundary triplet for the Laplacian, we can now apply the abstract results from Theorem \ref{thm:sec4.1}.
\begin{theorem}\lb{theorem5.4} Let $\Om\subset\mathbb{R}^n$ be a bounded domain with $C^{1,r}$-boundary, $r>1/2$, and let
  $t\mapsto \Xi_t\in \cB(H^{1/2}(\partial\Omega))$, $t\in[0,1]$, be a differentiable family of self-adjoint operators. Then for $t\in[0,1]$ the linear operator
\begin{align}\lb{aub1}
&-\Delta_t: \dom(-\Delta_t)\subset L^2(\Omega)\rightarrow L^2(\Omega), -\Delta_t u=-\Delta u, \\
&u\in\dom(-\Delta_t):=\{ u\in \dom(\Delta_{\max}): \Phi\gd u+\Xi_t \tN u=0 \},
\end{align}
is self-adjoint. 
The following resolvent difference formula holds
\begin{align}\begin{split}\lb{deltaKrein}
(-\Delta_t-\zeta)^{-1}&-(-\Delta_s-\zeta)^{-1}\\
&= \big(\tN(-\Delta_t-\overline{\zeta})^{-1} \big)^* (\Xi_t-\Xi_s)\big(\tN(-\Delta_s-\zeta)^{-1}\big), 
\end{split}
\end{align}
for $t,s\in[0,1]$, $\zeta\not\in(\spec(-\Delta_t)\cup\spec(-\Delta_s))$. Moreover, for a fixed $t_0\in[0,1]$ the mapping 
\begin{equation}\lb{deltacont}
t\mapsto (-\Delta_t-\zeta)^{-1}\in \cB(L^2(\Omega))
\end{equation}
is well defined for $t$ near $t_0$ as long as $\zeta\not\in\spec(-\Delta_{t_0})$. This mapping is differentiable at $t_0$ and satisfies the following Riccati equaiton
\begin{align}
\begin{split}\lb{deltaRic}
\frac{\bd}{\bd t}\big|_{t=t_0}\big((-\Delta_t&-\zeta)^{-1}\big)\\
&= \big(\tN(-\Delta_{t_0}-\overline{\zeta})^{-1} \big)^* \big(\frac{\bd}{\bd t}\big|_{t=t_0}\Xi_t\big)\big(\tN(-\Delta_{t_0}-\zeta)^{-1}\big).
\end{split}
\end{align}
Finally, if $\lambda(t_0)$ is an eigenvalue of $-\Delta_{t_0}$ of multiplicity $m\geq 1$ then there exists  a choice of orthonormal eigenfunctions  $\{u_j\}_{j=1}^m\subset \ker(-\Delta_{t_0}-\lambda(t_0))$
and a labeling of eigenvalues $\{\lambda_j(t)\}_{j=1}^m$  of $-\Delta_{t}$, for $t$ near $t_0$, such that  
\begin{equation}\lb{thetder}
\dot{\lambda}_j(t_0)= -\langle\dot \Xi_{t_0} \tN u_j, \tN u_j\rangle_{L^2(\partial\Omega)}, 1\leq j\leq m.
\end{equation}
\end{theorem}
\begin{proof}
By Lemma \ref{lem:BTR}, $(H^{1/2}(\partial\Omega), \tN, \Phi\gd)$ is an  \sel{ordinary} boundary triplet. In order to check that $-\Delta_t$ is self-adjoint, it suffices to check conditions \eqref{com}, \eqref{inv} with $X:=\Xi_t$, $Y:=I$. Indeed, \eqref{com} holds since $\Xi_t$ is self-adjoint, \eqref{inv} holds since the operator $XX^*+YY^*$  given by $I+\Xi_t^2>0$ is invertible.  The fact that \eqref{deltacont} is well defined for $t$ near $t_0$ follows from continuity of $\Xi_t$ and Theorems \ref{thm:sec4.1} and \ref{prop1.8new} upon setting $\cA_t:=-\Delta_t$, $V_t:=0$, $\tr_t:=[\tN, \Phi\gd]^\top$. In order to prove \eqref{deltaKrein}, \eqref{deltaRic}, \eqref{thetder}, we use \eqref{XYKrein}, \eqref{derRnew}, \eqref{5.10new}, respectively, with
\begin{align}
&\big(W(\Xi_t, I)\big) \tr   R_{t}(\zeta)= (I+\Xi_t^2)^{-1}( -\Gamma_0R_{t}(\zeta)+\Xi_t\Gamma_1R_{t}(\zeta))\no\\
&\quad=(I+\Xi_t^2)^{-1}( -\Gamma_0R_{t}(\zeta)-\Xi_t^2\Gamma_0R_{t}(\zeta))=-\Gamma_0R_{t}(\zeta)=-\tau_N R_{t}(\zeta)
\end{align}
and $\phi_j=-\tN u_j$. 
\end{proof}
\begin{remark}\label{rem:LipD}
The assumption $\partial\Omega$ being $C^{1,r}$, $r>1/2$, imposed in this section could be replaced by  $\partial\Omega$ being Lipschitz and $\Omega$ quasi-convex, see \cite[Section 8]{GM10} for the definition. As proved  in \cite{GM10}, these weaker assumptions are sufficient for the domains of the Dirichlet and Neumann Laplacians to belong to $H^2(\Omega)$, which in turn is equivalent to \eqref{m=m*} to hold. Also, for the case of Lipschitz domains Lemma \ref{lem:BTR} and, as demonstrated in \cite{GM10}, leading to it Lemmas \ref{wd}, \ref{A.5}, \ref{nn} hold with the Sobolev spaces $H^{1/2}(\partial\Omega)$ and $H^{-1/2}(\partial\Omega)$ replaced by $N^{1/2}(\partial\Omega)$ and its adjoint $\big(N^{1/2}(\partial\Omega)\big)^*$, respectively, where the space $N^{1/2}(\partial\Omega)$ is defined as $\{f\in L^2(\partial\Omega): f\nu_j\in H^{1/2}(\partial\Omega)\}$, $\nu=(\nu_j)_{j=1}^n$, and is equal to $H^{1/2}(\partial\Omega)$ provided $\partial\Omega$ is $C^{1,r}$,  $r>1/2$. In the context of Lipschitz domains we also mention an important paper  \cite{BM}.
\hfill$\Diamond$  \end{remark}
\begin{remark}
Our motivation to consider the boundary condition in Theorem \ref{theorem5.4} stems from \cite{CJLS, GM08, LS17}. More generally, the boundary condition described in Theorem \ref{theorem5.4} can be replaced by  $ X_t \gn u+Y_t \tN u=0$ for $X_t, Y_t\in \cB(H^{1/2}(\partial\Omega))$ satisfying \eqref{com}, \eqref{inv}.  In this case, as in Theorem \ref{theorem5.4},  continuity of the mappings $t\mapsto X_t$, $t\mapsto Y_t$ yields continuity of the resolvent operator with respect to $t$. Moreover, differentiability of the mappings $t\mapsto X_t$, $t\mapsto Y_t$ yields differentiability of the resolvent operator with respect to $t$ as well as the Reccati equation and the formula for the slopes of the eigenvalue curves (both obtained by dropping the potential terms  $V_t$ in  \eqref{5.10newnewa}, \eqref{5.10newnew}, respectively).
\hfill$\Diamond$  \end{remark}

\subsection{Quantum graphs}\label{ssQuaGr} The main result of this section is Theorem \ref{theorem4.5} in which we derive the resolvent difference formula, Riccati equation, and Hadamard-type formula for Schr\"odinger operators on metric graphs. To that end, we employ the abstract results discussed in Theorem \ref{thm:sec4.1} with an  \sel{ordinary} boundary triplet specifically defined for quantum graphs. Examples \ref{ex48}  and   \ref{ex47} give two applications of Theorem \ref{theorem4.5}. Both examples concern monotonicity of eigenvalue curves of Schr\"odinger operators with respect to some natural parameter present in the boundary conditions.

We begin by discussing  differential operators on metric graphs. To set the stage, let us fix a discrete graph $(\cV,\cE)$ where $\cV$ and $\cE$ denote the set of vertices and edges respectively. We assume that $(\cV,\cE)$ consists of a finite number of vertices,  $|\cV|$, and a finite number of edges,  $|\cE|$. We assign to each edge $e\in\cE$ a positive and finite length $\ell_{e}\in(0,\infty)$. The corresponding metric graph is denoted by $\cg$. The boundary $\partial\cg$ of the metric graph is defined by
\begin{equation}
\partial\cg:=\cup_{e\in\cE} \{a_e,b_e\}, 	
\end{equation}
where $a_e, b_e$ denote the end points of the edge $e$.  It is convenient to treat $2|\cE|$ dimensional vectors as functions on the boundary $\partial\cg$, in particular,
$L^2(\partial\cg)\cong \bbC^{2|\cE|}$,
where the space $L^2(\partial\cg)=\bigoplus_{e\in\cE}\left( L^2(\{a_e\})\times L^2(\{b_e\})\right)$  corresponds to the discrete Dirac measure with support $\cup_{e\in \cE} \{a_e, b_e\}$. In addition to the space of functions on the boundary we consider the Sobolev spaces of functions on the graph $\cg$,
\begin{align}
&L^2(\cg):=\bigoplus_{e\in\cE}L^2(e),\  \hatt{H}^2(\cg):=\bigoplus_{e\in\cE}H^2(e)\no,
\end{align}  
where $H^2(e)$ is the standard $L^2$ based Sobolev space.   
As in the setting of  Laplace operators on bounded domains, the spaces ${L}^2(\cg)$ and $L^2(\partial \cg)$ are related via the trace maps. We define the trace operators $(\Gamma_0, \Gamma_1)$  by the formulas
\begin{align}
&\Gamma_0: \hatt{H}^2(\cg)\rightarrow L^2(\partial \cg),
\ \Gamma_0u:=u|_{\partial \cg}, u\in \hatt{H}^2(\cg),\\
&\Gamma_1: \hatt{H}^2(\cg)\rightarrow L^2(\partial \cg),
\ \Gamma_1 u:=\partial_{n}u|_{\partial \cg}, u\in \hatt{H}^2(\cg),
\end{align}
where $\partial_{n} u$ denotes the derivative of $u$ taken in the {\it inward} direction. The trace operator is a bounded, linear operator given by
\begin{equation}
\tr:=
[\Gamma_0, 
\Gamma_1]^\top,\, \tr: \hatt{H}^2(\cg)\rightarrow L^2(\partial \cg)\times L^2(\partial \cg)\cong \bbC^{4|\cE|}.
\end{equation}
The Sobolev space of functions vanishing on the boundary $\partial\cg$ together with their derivatives is denoted by
\begin{equation}\no
H^2_0(\cg):=\left\{u\in \hatt{H}^2(\cg): \tr u=0\right\}.
\end{equation}
Using our notation for the trace maps, the Green identity can be written as follows,
\begin{align}\label{vv5}
\int_{\cg} (-u'')\overline{v}-u\overline{(-v'')}
&=\int_{\partial\cg}
\partial_{n}u\overline{v}-u\overline{\partial_{n}v} \\
&=\langle [J\otimes I_{2|\cE|}]\tr u, \tr v\rangle_{\bbC^{4|\cE|}},\  u,v\in  \hatt{H}^2(\cg).\no
\end{align}
The right-hand side of the Green identity defines a symplectic form 
\begin{align} 
\label{vv9}
&\omega:\  \dL^2(\partial \cg)\times\,\dL^2(\partial
\cg)\rightarrow \bbC, \\
\label{eq:def_omega}
&\omega( (f_1, f_2), (g_1, g_2)):=\int_{\partial\cg}f_2 \overline {g_1}-f_1\overline{g_2},\\
& (f_1, f_2), (g_1, g_2)\in \dL^2(\partial\cg),\label{vv10}
\end{align}
where $\dL^2(\partial \cg):=L^2(\partial \cg)\times L^2(\partial \cg)$.

Next, we introduce the minimal Laplace operator $A_{min}$ and its adjoint $A_{max}$. The operator
\begin{equation}\label{b5}
A_{min}:=-\frac{\bd^2}{\bd x^2},\quad \dom(A_{min})=\hatt H^2_0(\cg),
\end{equation}
is symmetric in $L^2(\cg)$. Its adjoint $A_{max}:=A_{min}^*$ is given by 
\begin{equation}\label{b6}
A_{max}:=-\frac{\bd^2}{\bd x^2},\quad \dom(A_{max})=\hatt{H}^2(\cg).
\end{equation}
The dificiency indices of $A_{min}$ are finite and equal, that is,
\begin{equation}
0<\dim\ker(A_{max}-\bfi)=\dim\ker(A_{max}+\bfi)<\infty.
\end{equation}

\begin{theorem}\lb{theorem4.5}Assume that  
	\begin{align}
	&t\mapsto V_t \text{ is in } C^1([0,1], L^{\infty}(\cg)),\\
	&t\mapsto X_t, Y_t \text{ is in } C^1([0,1], \bbC^{2|\cE|\times 2|\cE|}),\  \det(X_tX^*_t+Y_tY_t^*)\not=0, X_tY_t^*=Y_t^*Y_t. 
	\end{align}
	Then the operator 
	\begin{align}
	&\cA_t: L^2(\cg)\rightarrow L^2(\cg), \dom(\cA_t):=\{u\in H^2(\cg): [X_t, Y_t]\tr u=0 \},\\
	&\cA_t u=-u'', u\in\dom(\cA_t),
	\end{align}
	is a self-adjoint extension of $A_{min}$. The operator-valued function  
	\begin{equation}
	t\mapsto R_t(\zeta_0):=(\cA_t+V_t-\zeta_0)^{-1} \text{ for all } \zeta_0\not\in\spec(\cA_t) 
	\end{equation}
	is in $C^1([0,1], \cB(L^2(\cg)))$ and for any $t_0\in[0,1]$ one has
	\begin{align}\lb{derRnewnew}
	\begin{split}
	&\dot{R}_{t_0}(\zeta_0)=-R_{t_0}(\zeta_0)\dot V_{t_0}R_{t_0}(\zeta_0)\\
	&+(\tr R_{t_0}(\overline{\zeta_0}))^*\big(W(X_{t_0}, Y_{t_0})\big)^*(\dot X_{t_0}{Y}_{t_0}^*-\dot Y_{t_0}{X}_{t_0}^*)\big(W(X_{t_0}, Y_{t_0})\big)\tr  R_{t_0}(\zeta_0),
	\end{split}
	\end{align} 
	where $W(X_{t_0}, Y_{t_0})$ is as in \eqref{defw}.
	Furthermore, if $\lambda(t_0)$ is an eigenvalue of $\cA_{t_0}+V_{t_0}$ of multiplicity $m\geq 1$ then there exist  a choice of orthonormal eigenfunctions 
	\begin{equation}\no
	\{u_j\}_{j=1}^m\subset \ker(\cA_{t_0}+V_{t_0}-\lambda(t_0))
	\end{equation}
	and a labeling of eigenvalues $\{\lambda_j(t)\}_{j=1}^m$  of $\cA_t+V_t$, for $t$ near $t_0$, such that  
	\begin{equation}\lb{5.10newnew24}
	\dot{\lambda}_j(t_0)= \langle \dot V_{t_0}u_j, u_j\rangle_{L^2(\cg)}+\big\langle (X_{t_0}\dot{Y}_{t_0}^*-Y_{t_0}\dot{X}_{t_0}^*) \phi _j, \phi_j\big\rangle_{L^2(\partial\cg)}, 
	\end{equation}
	where $\phi_{j}=W(X_{t_0}, Y_{t_0})\tr u_j$ is a unique $2|\cE|$ dimensional vector satisfying $\Gamma_0 u_j=-Y_{t_0}^*\phi_j$ and $\Gamma_1 u_j=X_{t_0}^*\phi_j$, $1\leq j\leq m$.
\end{theorem}
\begin{proof}
	Since $(L^2(\partial\cg), \Gamma_0, \Gamma_1)$ is an  \sel{ordinary} boundary triplet,  equations \eqref{5.10newnewa} and \eqref{5.10newnew} in Theorem \ref{thm:sec4.1} 
	 give \eqref{derRnewnew} and \eqref{5.10newnew24} respectively.  
\end{proof}
\begin{example}\lb{ex48}
	Consider the Schr\"odinger operator $\displaystyle{H_{t}=-\frac{\bd^2}{\bd x^2}+V}$ on a compact star graph $\cg=(\cE,\mathcal{V})$ with a bounded real-valued potential $V$ subject to arbitrary self-adjoint vertex conditions at the vertices of degree one and the following $\delta$-type condition at the center ${\rm v_c}\in\mathcal{V}$,
	\begin{equation}\label{BCnu}
	\sum_{e\sim {\rm v_c}}{\partial_{n} u_e({\rm v_c})}=t u({\rm v_c}),\ t\in\bbR.
	\end{equation}
	We recall that the spectrum of $H_t$ can be described via secular equations \cite{BK}. In this example, we will derive an Hadamard-type formula \eqref{BKHad} for the derivative of the eigenvalues of $H_t$. Such a formula is discussed in \cite[Proposition 3.1.6]{BK} for simple eigenvalues. The general case can be treated using \eqref{5.10newnew24} as follows.  The  boundary matrices describing the vertex conditions are given by $\wti X\times X_{t}$ and $\wti Y\times Y$ where 
	\begin{equation}\no
	X_{t}= \begin{bmatrix}
	1&-1&&\cdots &0\\
	0&1&-1&\cdots&0\\
	&&\ddots\\
	0&&&1&-1\\
	-t&0&\cdots&&0
	\end{bmatrix},\quad  
	Y= \begin{bmatrix}
	0&&\cdots &&0\\
	0&&\cdots&&0\\
	&&\ddots\\
	0&&\cdots &&0\\
	1&1&\cdots&&1
	\end{bmatrix},
	\end{equation}
	and the matrices $\wti X$ and $\wti Y$ correspond to the vertex conditions at $\cV\setminus\{{\rm v_c}\}$. A direct computation gives
	\begin{equation}\lb{3.33}
	X_{t}^*Y=Y^*X_{t}= \begin{bmatrix}
	0&0&\cdots &&0\\
	0&0&\cdots&&0\\
	&&\ddots\\
	0&0&\cdots&&-t
	\end{bmatrix}.\  
	\end{equation}
	For the eigenvalue $\lambda(t_0)$ of $H_{t_0}$ of multiplicity $m\in\bbN$  we use \eqref{5.10newnew24} to get
	\begin{equation}
	\dot{\lambda}_j(t_0)= \big\langle (X_{t_0}\dot{Y}_{t_0}^*-Y_{t_0}\dot{X}_{t_0}^*) \phi _j, \phi_j\big\rangle_{L^2(\partial\cg)}=|\phi_j({\rm v_c})|^2, 
	\end{equation}
	where $1\leq j\leq m$, $\phi_{j}=W(X_{t_0}, Y_{t_0})\tr u_j$, and $\{u_j\}_{j=1}^m$ are the eigenfunctions of $H_{t_0}$ corresponding to $\lambda(t_0)$. Furthermore, using \eqref{phij} we obtain $\phi_j({\rm v_c})=-u_j({\rm v_c})$, hence,
	\begin{equation}\lb{BKHad}
	\dot{\lambda}_j(t_0)=|u_j({\rm v_c})|^2,\, 1\leq j\leq m. 
	\end{equation}
\hfill$\Diamond$  \end{example}

\begin{example}\lb{ex47} This example concerns monotonicity of eigenvalue curves of a class of Schr\"odinger operators on a compact interval arising in the spectral theory of periodic Hamiltonians. Specifically, we consider the Schr\"odinger operators $H_\vartheta$ with a real valued potential $V\in L^{\infty}(0,1)$ which are parameterized by $\vartheta\in[0, 2\pi)$ and defined as follows,
	\begin{align}
	&H_\vartheta=\cA_\vartheta+V,\, \cA_{\vartheta}: L^2(0,1)\rightarrow L^2(0,1), \cA_{\vartheta} u=-u'', u\in\dom(\cA_{\vartheta}), \\
	&\dom(\cA_{\vartheta}):=\{u\in H^2(0,1): e^{\bfi \vartheta}u(0)=u(1), e^{\bfi \vartheta}u'(0)=u'(1)  \}.\lb{aub14}
	\end{align}
	Such operators are of interest, in particular, because their eigenvalues fill up the spectral bands of the Schr\"odinger operator 
	in $L^2(\bbR)$ with  the potential given by the periodic extension of $V$, see   \cite[Theorems  XIII.89, XIII.90]{RS78}. We claim that the eigenvalue curves satisfy
	\begin{equation}\label{formdernew}
\dot\lambda_j(\vartheta_0)=2\Im(u_j'(0)\overline{u_j(0)}) \text{ for all } \vartheta_0\in(0, 2\pi),
	\end{equation} 
	where, as usual, $u_j\in \ker(\cA_{\vartheta_0}-\lambda_j(\vartheta_0))$, $j=1,2$ (in fact all but, possibly, periodic and antiperiodic operators have simple spectra). We derive this formula from \eqref{5.10newnew} by defining trace operators appropriately.  It is well known that ordinary differential operators fit well into the scheme of boundary triplets, cf. e.g. \cite[Chapter 3]{GG}, however, for completeness we recall the setting. We set
	\begin{align}
	&\cH:=L^2(0,1), \cH_+:=H^2(0,1), A=-\frac{\bd^2}{\bd x^2}, \dom(A)=H^2_0(0,1)\\
	&\tr: H^2(0,1)\rightarrow \C^4, \Gamma_0u:=(u(0), u(1))^{\top}, \Gamma_1u:=(u'(0), -u'(1))^{\top}.  
	\end{align}
	Next, to utilize \eqref{5.10newnew} we first rewrite the boundary conditions in \eqref{aub14} as follows,
	\begin{equation}
	X_{\vartheta}\Gamma_0 u+Y_{\vartheta}\Gamma_1 u=0, \text{ where }
	X_{\vartheta}:=\begin{bmatrix}
	-\e^{\bfi\vartheta}& 1\\0&0
	\end{bmatrix}, Y_{\vartheta}:=\begin{bmatrix}
	0&0\\
	\e^{\bfi \vartheta}&1
	\end{bmatrix},
	\end{equation}
	and compute
	\begin{align}
	\phi_j=W(X_{\vartheta}, Y_{\vartheta})\tr u_j&= \frac{1}{2}(-Y_{\vartheta_0}\Gamma_0u_j+X_{\vartheta_0}\Gamma_1u_j)= -e^{\bfi \vartheta_0}(u_j'(0), u_j(0))^{\top},\\
	X_{\vartheta_0} \dot Y^*_{\vartheta_0}-Y_{\vartheta_0}X^*_{\vartheta_0} &=\begin{bmatrix}
	0&\bfi\\
	-\bfi&0
	\end{bmatrix}. 
	\end{align}
	Plugging this in \eqref{5.10newnew24} yields \eqref{formdernew}. Monotonicity of the eigenvalues follows from linear independence of $u_j, \overline{u_j}$ and the  formula 
	\begin{equation}
	2|\Im(u_j'(0)\overline{u_j(0)})|=|\cW(u_j, \overline{u_j})(0)|\not=0, \vartheta_0\in(0, 2\pi).
\end{equation}
	involving the Wronskian. \hfill$\Diamond$
  \end{example}

\subsection{Periodic Kronig--Penney model}\footnote{An alternative approach applicable to a very broad class of second-order operators is discussed in the upcoming work of D. Damanik, J. Fillman and the second author. See also \cite{MR3543796}. }\label{ssPerKrPen}  In  this section we give yet another application of Theorem \ref{thm:sec4.1} proving a version of B. Simon's theorem \cite{MR473321} which states that a certain open gap property (described below) of periodic Schr\"odinger operators is generic in the class of periodic $C^{\infty}(\bbR)$ potentials.  The main result of this section, Theorem  \ref{thm4.14}, states this assertion for singular $\delta-$type potentials. Its proof is based on a perturbative argument inspired by \cite{MR473321} and technically made available by Theorem \ref{thm:sec4.1}.

The spectrum of the Schr\"odinger operator with periodic potential on the line has a band-gap structure, that is, in general, it consists of closed segments, called {\em bands}, such that two adjacent bands can either have a common endpoint or be separated by an open interval, a {\em gap}, of the resolvent set; in the latter case we say that the gap is open.
We will now use Theorem \ref{thm:sec4.1} to prove that {\em all gaps} of a generic periodic Kronig--Penney model are open. The operators in question are the Schr\"odinger operators with $\delta-$type potentials which in physics literature are written as follows,
\begin{align}
&H_{{\alpha}}:=-\frac{\bd^2}{\bd x^2}+\sum_{k\in\bbZ}\alpha_k \delta(x-k),
\end{align}
and mathematically are defined by 
\begin{align}
&H_{\alpha}u:=-u'', u\in\dom(H_{\alpha}), H_{\alpha}:\dom(H_{\alpha})\subset L^2(\bbR)\rightarrow L^2(\bbR),\\
&\dom(H_{{\alpha}})= \{u\in \hatt H^2(\bbR\setminus\bbZ): u\ \text{satisfies}\  \eqref{6.8new}\text{\ for all}\ k\in\bbZ\},\\
&\hspace{3.3cm}u(k^+)=u(k^-),\ u'(k^+)-u'(k^-)=\alpha_ku(k),\lb{6.8new}
\end{align}
where $\alpha=\{\alpha_k\}_{k\in\bbZ}\in \ell^{\infty}(\bbZ; \bbR)$, $u(k^\pm)$ are the one-sided limits, and $\hatt{H}^2$ denotes the direct sum of the Sobolev spaces on respective intervals. The spectrum of $H_{\alpha}$ for the case of periodic sequence $\alpha$ has a band-gap structure, see \cite[Theorem 2.3.3]{MR2105735}. This was originally proved for $1-$periodic sequences but can be directly extended to any $p-$periodic ones. Specifically, given  a $p-$periodic sequence $\alpha=\{\alpha_k\}_{k\in\bbZ}\in \ell^{\infty}(\bbZ; \bbR)$, the operator $H_{\alpha}$ is unitary equivalent to the direct integral 
\begin{equation}
\int_{[0,2\pi)}^{\oplus} H_{{\alpha^{(p)}}, \vartheta}\frac{\bd\vartheta}{2\pi}, \, \text{ where we denote $\alpha^{(p)}:=\{\alpha_0, ..., \alpha_{p-1}\}\in\bbR^p$,}
\end{equation}
and $H_{{\alpha^{(p)}}, \vartheta}$ for $\vartheta\in[0,2\pi)$ is the operator defined in $ L^2(I_p)$ with $I_p:=(-1/2,p-1/2)$  by 
\begin{align}
&H_{{\alpha^{(p)}}, \vartheta}u:=-u'',\ H_{{\alpha^{(p)}}, \vartheta}: \dom(H_{\alpha^{(p)}, \vartheta})\subset L^2(I_p)\rightarrow L^2(I_p),\\
&\dom(H_{\alpha^{(p)}, \vartheta})= \Big\{u\in \hatt H^2(I_p\setminus\bbZ):  u\ \text{satisfies}\  \eqref{6.8new} \text{\ for $k\in I_p\cap\bbZ$ and\ } \eqref{6.9new}\Big\},\\
&u(-1/2^+)=e^{\bfi\vartheta}u((p-1/2)^-),\, u'(-1/2^+)=e^{\bfi\vartheta}u'((p-1/2)^-),\lb{6.9new}
\end{align}
where 
\begin{equation}
\hatt H^2(I_p\setminus\bbZ):=H^2(-1/2, 0)\oplus H^2(0, 1)\oplus...\oplus H^2(p-2, p-1)\oplus H^2(p-1, p-1/2). 
\end{equation}
Denoting the eigenvalues of $H_{\alpha^{(p)}, \vartheta}$ (ordered in non-decreasing order) by \[\lambda_j(\alpha^{(p)}, \vartheta), j=1,2,\dots,\]  we have 
\begin{align}
\lambda_1(\alpha^{(p)}, 0)&\leq \lambda_1(\alpha^{(p)}, \vartheta)\leq  \lambda_1(\alpha^{(p)}, \pi)\leq \lambda_2(\alpha^{(p)}, \pi)\leq  \lambda_2(\alpha^{(p)}, \vartheta)\leq\lambda_2(\alpha^{(p)}, 0)\\
&\leq \lambda_3(\alpha^{(p)}, 0)\leq \lambda_3(\alpha^{(p)}, \vartheta) \leq\lambda_3(\alpha^{(p)}, \pi)\leq\dots \text{ for $\vartheta\in[0,\pi]$.} 
\end{align}
Then the spectrum of $H_{\alpha}$ is given by 
\begin{align}
\spec(H_{\alpha})&=\bigcup\limits_{\vartheta\in[0,\pi]}\spec(H_{\alpha^{(p)}, \vartheta})\\
&=[\lambda_1(\alpha^{(p)}, 0),  \lambda_1(\alpha^{(p)}, \pi)]\cup[\lambda_2(\alpha^{(p)}, \pi), \lambda_2(\alpha^{(p)}, 0)]\cup\dots.
\end{align}
The intervals $[\lambda_1(\alpha^{(p)}, 0),  \lambda_1(\alpha^{(p)}, \pi)]$, $[\lambda_2(\alpha^{(p)}, \pi), \lambda_2(\alpha^{(p)}, 0)], \dots$  are called {\em bands}. The endpoints of two adjacent bands may coincide. In this case we say that the respective gap is closed; otherwise the respective {\em gap},  $\big(\lambda_1(\alpha^{(p)}, \pi), \lambda_2(\alpha^{(p)}, \pi)\big)$, $\big(\lambda_2(\alpha^{(p)}, 0),  \lambda_3(\alpha^{(p)}, 0)\big),\dots$ is said to be open. 
In the following theorem we show that all gaps are open for a generic periodic sequence $\alpha$. 
\begin{theorem}\lb{thm4.14}There is a dense $G_{\delta}-$set $\cS\subset \ell^{\infty}(\bbZ; \bbR)$ of sequences $\alpha$ such that for each $\alpha\in\cS$ all gaps in the spectrum of $H_{\alpha}$ are open. 
\end{theorem}
\begin{proof} We let
	\[\cS_n:=\{\alpha\in \ell^{\infty}(\bbZ;\bbR): \alpha \text{\ is $p-$periodic and the $n$-th gap of $H_{\alpha}$ is open} \}.\]  It is enough to prove that each $\cS_n$ is open and dense (then $\bigcap_{n\in\bbN}\cS_n$ gives the required dense $G_{\delta}$-set of potentials). To begin,
	let us rewrite $\dom(H_{\alpha^{(p)}, \vartheta})$ in terms of Lagrangian planes in $\Lambda(\bbC^{4(p+1)})$. For $u\in \hatt H^2(I_p\setminus\bbZ)$ we introduce the traces $\Gamma_0 u, \Gamma_1 u\in\bbC^{2(p+1)}$ by
	\begin{align}
	&\Gamma_0 u:=\{u(-1/2^+),u((p-1/2)^-),u(0^-),u(0^+),\dots, u(k^-),u(k^+),\dots,\\
	&\hspace{5cm}u((p-1)^-),u((p-1)^+) \}\in \bbC^{2(p+1)},\\
	&\Gamma_1 u:=\{u'(-1/2^+),-u'((p-1/2)^-),-u'(0^-),u'(0^+),\dots, -u'(k^-),u'(k^+),\dots,\\
	&\hspace{5cm}-u'((p-1)^-),u'((p-1)^+) \}\in \bbC^{2(p+1)}.
	\end{align}
	Also, let us introduce $2(p+1)\times 2(p+1)$ matrices 
	\begin{align}
	&X_{\alpha^{(p)}, \vartheta}:= 
	\begin{bmatrix}
	-e^{\bfi \vartheta}&1\\
	0&0
	\end{bmatrix}
	\oplus
	\begin{bmatrix}
	1&-1\\
	-\alpha_0&0
	\end{bmatrix}
	\oplus...\oplus
	\begin{bmatrix}
	1&-1\\
	-\alpha_{p-1}&0
	\end{bmatrix},\\
	&Y_{\alpha^{(p)}, \vartheta}:= 
	\begin{bmatrix}
	0&0\\
	e^{\bfi \vartheta}&1
	\end{bmatrix}
	\oplus
	\begin{bmatrix}
	0&0\\
	1&1
	\end{bmatrix}
	\oplus...\oplus
	\begin{bmatrix}
	0&0\\
	1&1
	\end{bmatrix}. 
	\end{align}
	Then one has
	\begin{equation}
	\dom(H_{\alpha^{(p)}, \vartheta})=\{u\in \hatt H^2(I_p\setminus\bbZ): X_{\alpha^{(p)}, \vartheta}\Gamma_0 u+Y_{\alpha^{(p)}, \vartheta}\Gamma_1 u=0\}. 
	\end{equation}
	That is, the Lagrangian plane corresponding to $H_{\alpha^{(p)}, \vartheta}$ is given by \[\ker[X_{\alpha^{(p)}, \vartheta}, Y_{\alpha^{(p)}, \vartheta}]. \] 
	
	In order to prove that $\cS_n$ is open, let us recall that the edges of the spectral gaps are given by consecutive eigenvalues of the periodic,  $H_{\alpha^{(p)}, 0}$, or anti-periodic, $H_{\alpha^{(p)}, \pi}$, operators. Suppose that $\alpha\in\cS_n$ and that the edges of the $n-$th  gap satisfy $\lambda_n(\alpha^{(p)}, \vartheta)<\lambda_{n+1}(\alpha^{(p)}, \vartheta)$ with either $\vartheta=0$ or $\vartheta=\pi$. We claim that this strict inequality holds for all $\wti{\alpha}^{(p)}\in \bbR^p$ near $\alpha^{(p)}$, i.e. that the gap is open under small perturbations of $\alpha^{(p)}$.    Indeed, since the mapping
	\begin{equation}
	\bbR^p\ni\alpha^{(p)}\mapsto  [X_{\alpha^{(p)}, \vartheta}, Y_{\alpha^{(p)}, \vartheta}] \text{ for $\vartheta=0$ or $\vartheta= \pi$}
	\end{equation}
	is continuous, Theorem \ref{thm:sec4.1} yields continuity of the mapping 
	\begin{equation}
	\bbR^p\ni\alpha^{(p)}\mapsto  (H_{\alpha^{(p)}, \vartheta}-\bfi)^{-1}\in \cB(L^2(I_p)) \text{ for $\vartheta=0$ or $\vartheta= \pi$};
	\end{equation}
	hence, the mappings
	\begin{equation}
	\alpha^{(p)}\mapsto \lambda_j(\alpha^{(p)}, \vartheta), \alpha^{(p)}\mapsto \lambda_{j+1}(\alpha^{(p)}, \vartheta), \text{ for $\vartheta=0$ or $\vartheta= \pi$}
	\end{equation}
	are also continuous, which implies the asserted strict  inequality
	\[\lambda_n(\wti{\alpha}^{(p)}, \vartheta)<\lambda_{n+1}(\wti{\alpha}^{(p)}, \vartheta)\] for all $\wti{\alpha}^{(p)}$ near ${\alpha}^{(p)}$.
	
	In order to prove that $\cS_n$ is dense we need to show that for both cases $\vartheta=0$ and $\vartheta= \pi$ the equality $\lambda_n(\alpha^{(p)}, \vartheta)=\lambda_{n+1}(\alpha^{(p)}, \vartheta)$  will not hold if $\alpha^{(p)}$ is replaced by its small perturbation. We will consider the case $\vartheta=0$, that is, we will assume that $\lambda_n(\alpha^{(p)}, 0)=\lambda_{n+1}(\alpha^{(p)}, 0)$; the case $\vartheta=\pi$ is treated analogously. For $t\in\bbR$, let us introduce the perturbation $\alpha^{(p)}(t):= \{t+\alpha_0,\alpha_1,\dots, \alpha_{p-1}\}$. We claim that for every $\varepsilon>0$ there is a $t_0\in(0, \varepsilon)$ with
	\begin{equation}\lb{eigcur}
	\lambda_n(\alpha^{(p)}(t_0), 0)<\lambda_{n+1}(\alpha^{(p)}(t_0), 0).
	\end{equation}
	When proven, this inequality shows that there exist arbitrarily close to $\alpha^{(p)}$  perturbations which open the closed gap. To prove the claim we utilize the Hadamard-type formula \eqref{5.10newnew24} for the boundary matrices $X_{\alpha^{(p)}(t), 0}$, $Y_{\alpha^{(p)}(t), 0}$.  We recall that $\lambda:=\lambda_n(\alpha^{(p)}, 0)=\lambda_{n+1}(\alpha^{(p)}, 0)$ is an eigenvalue of $H_{\alpha^{(p)}, 0}$ of multiplicity two. By Theorem \ref{theorem4.5}, there is a basis $\{u_1, u_2\}$ in $\ker(H_{\alpha^{(p)}, 0}-\lambda)$ such that 
	\begin{align}
	&\frac{\rm d}{{\rm d}t}\Big|_{t=0}\lambda_n(\alpha^{(p)}(t), 0)=|u_1(0)|^2,\lb{derone}\\
	&\frac{\rm d}{{\rm d}t}\Big|_{t=0}\lambda_{n+1}(\alpha^{(p)}(t), 0)=|u_2(0)|^2.\lb{dertwo}
	\end{align}
	Next, we will prove that the values of the derivatives in \eqref{derone}, \eqref{dertwo} are not equal to each other. This fact  implies that the eigenvalue curves $t\mapsto\lambda_n(\alpha^{(p)}(t), 0)$ and   $t\mapsto\lambda_{n+1}(\alpha^{(p)}(t), 0)$ do not coincide for $t$ near $t=0$, which, in turn, yields \eqref{eigcur} as needed. Starting the proof of the fact, we first remark that the eigenfunctions $u_1$ and $u_2$ are real valued because the boundary conditions for $\vartheta=0$ are real. Upon multiplying the eigenfunctions by appropriate constants 
	we may and will assume that $u_1(0)$ and $u_2(0)$ are non-negative. If  $u_1(0)\not=u_2(0)$ then the left-hand sides of \eqref{derone}, \eqref{dertwo} are not equal as required. If $u_1(0)=u_2(0)$ then for any $t\in\bbR$ the function $u_1-u_2$ satisfies the boundary condition at $x=0$ with $\alpha_0$ replaced by $t+\alpha_0$. Therefore,
	$
	u_1-u_2\in \ker (H_{\alpha^{(p)}(t), 0}-\lambda)\setminus\{0\}
	$
	and thus $\lambda$ is an eigenvalue of $H_{\alpha^{(p)}(t), 0}$ for all $t\in\bbR$. That is, either $\lambda_n(\alpha^{(p)}(t), 0)$ or $\lambda_{n+1}(\alpha^{(p)}(t), 0)$ should be identically equal to $\lambda$ for all $t$ near $0$. Hence, one of the derivatives in \eqref{derone}, \eqref{dertwo} vanishes, say, the first one. Then $u_1(0)=0$. But in this case $u_2(0)\not=0$ for otherwise $u_1$ and $u_2$ would be linearly dependent. Thus, the value of the derivative in \eqref{derone} is equal to zero while the value of the derivative in \eqref{dertwo} is not, as required. 
\end{proof}

\subsection{Maslov crossing form for abstract  boundary triplets}\lb{mastriples} 
In this section, we discuss an infinitesimal version of the formula equating the Maslov index and the spectral flow for the family of operators $H_t=\cA_t+V_t$ satisfying Hypothesis \ref{hyp1.3bis}, which is assumed throughout this section. Formulas relating these two quantities are quite classical, and we refer the reader to  the papers  \cite{BbF95, BZ3, BZ2, BZ1, CLM, CJLS, CJM1, F04, LS1, LSS, LS17, RoSa95} and the literature therein. Employing the abstract Hadamard-type formula obtained in Theorem \ref{thm2.2}, we prove in Theorem \ref{HadmardSimple}  that the signature of the Maslov crossing form defined in \eqref{dfnMcf} at an eigenvalue $\lambda$ of the operator $H_{t_0}$ is equal to the difference between the number of  monotonically decreasing and the number of monotonically increasing eigenvalue curves for $H_t$ bifurcating from $\lambda$. 

For $\lambda\in\bbR$ and $t\in[0,1]$ we introduce the following subspaces,
\begin{align}
&\bbK_{\lambda,t}:=\tr_{t}\big(\ker(A^*+
V_t-\lambda)\big)\subset\bndra,\\
&\cG_t:=\ran (Q_t)\subset\bndra, \\
&\Upsilon_{\lambda,t}:=\bbK_{\lambda,t}\oplus\cG_t\subset\big((\mathfrak{H}\times\mathfrak{H})\oplus(\mathfrak{H}\times\mathfrak{H})\big), \label{dfnUPS}\\
&\mathfrak{D}:=\{\bp=(p,p)^\top: p\in\mathfrak{H}\times\mathfrak{H}\}
\subset\big((\mathfrak{H}\times\mathfrak{H})\oplus(\mathfrak{H}\times\mathfrak{H})\big).
\end{align} 
Since $\tr_t(\dom(\cA_t))=\ran(Q_t)$ by Hypothesis \ref{hyp1.3bis},  the following assertions are equivalent
\begin{equation}\label{equivEV}
(i)\,\, \ker(H_t-\lambda)\not=\{0\}, \quad (ii)\,\, \bbK_{\lambda,t}\cap\cG_t\neq\{0\}, \quad (iii)\,\, \Upsilon_{\lambda,t}\cap\mathfrak{D}\neq\{0\}
\end{equation}
since $\mathfrak{D}$  is the diagonal subspace in $(\mathfrak{H}\times\mathfrak{H})\oplus(\mathfrak{H}\times\mathfrak{H})$. In fact, using a fundamental Proposition 3.5 in \cite{BbF95}, one can deduce deeper connections between the spectral information for $H_t$ and the behavior of Lagrangian planes under the following hypotheses. 
\begin{hypothesis}\lb{i17} Given $\lambda\in \bbR$ and $t_0\in[0,1]$,  we assume that
\begin{itemize}
\item[(i)] $\lambda\not\in\spec_{\rm ess}(H_{t_0})$.
\end{itemize}
Moreover, there exists an interval $\cJ\subset[0,1]$ centered at $t_0$ such that
\begin{itemize}
\item[(ii)] the mappings $t\mapsto \tr_t$, $t\mapsto V_t$, $t\mapsto Q_t$ are $C^1$ on $\cJ$,
	\item[(iii)] $\ker(A^*+V_t-\lambda)\cap\dom(A)=\{0\}$ for all $t\in\cJ$. 
	\end{itemize}
\end{hypothesis}

Hypothesis \ref{i17}  will be assumed through this section. 
Part (iii) of this hypothesis is an abstract version of the unique continuation principle for PDEs, and we refer to Theorems 3.2 and Hypothesis 5.9 in \cite{LS1} for a discussion of this connection.
Part (i) implies that the operator $H_{t_0}-\lambda$ is Fredholm.
Since $\ker(\tr)=\dom(A)$ by Proposition \ref{remark2.2}(1), parts (i) and (iii) of Hypothesis \ref{i17} imply that $\tr\big|_{\ker(H_{t_0}-\lambda)}$ is an isomorphism between $\ker(H_{t_0}-\lambda)$ and
$ \bbK_{\lambda,t}\cap\cG_t$, cf.\ \eqref{equivEV}. Moreover, 
the subspaces  $\bbK_{\lambda,t_0}$ and  $\cG_{t_0}$ form a Fredholm pair (i.e. their intersection is finite dimensional and their sum is closed and has finite co-dimension). The latter fact has been established in \cite[Proposition 3.5]{BbF95} in the setting of Lagrangian planes in $\dom(A^*)/\dom(A)$; using this one can readily deduce the Fredholm property of the pair in the present setting via the symplectomorphism introduced in \cite[Proposition 5.3]{LS1}. The subspace   $\cG_t$ is Lagrangian by Hypothesis \ref{hyp1.3bis}. The subspace $ \bbK_{\lambda,t}$ is also Lagrangian again by \cite[Proposition 3.5]{BbF95}.
Furthermore,  part (ii) of Hypothesis \ref{i17} yields continuity in $t$ of the resolvent operators for $H_t$ by Theorem \ref{prop1.8new}. This, together with part (i), shows that $\lambda\not\in\spec_{\rm ess}(H_{t})$ for $t$ near $t_0$, hence,  the subspaces $\bbK_{\lambda,t}, \cG_{t}$ form a Fredholm  pair of Lagrangian subspaces for each $t$ near $t_0$. Hence, $\big(\Upsilon_{\lambda,t}, \mathfrak{D}\big)$ is a Fredholm  pair of Lagrangian subspaces for each $t$ near $t_0$.
 
 Let $\Pi_{\lambda,t}$ be the orthogonal projection onto $\Upsilon_{\lambda,t}$ from \eqref{dfnUPS} so that the mapping $t\mapsto \Pi_{\lambda,t}$ is continuously differentiable on $[0,1]$ for each $\lambda\in\bbR$, see \cite[pp.480--481]{LS1}. Furthermore, for $\lambda\in\bbR$ and $t_0\in[0,1]$ satisfying Hypothesis \ref{i17} there is an interval $\cI\subseteq\cJ\subset [0,1]$ centered at $t_0$ and a family of operators $t\mapsto \cM_{\lambda,t}$, $t\in\cI$, which is in $C^1\big(\cI,\cB(\Upsilon_{\lambda, t_0},(\Upsilon_{\lambda, t_0})^\bot)\big)$ with  $\cM_{\lambda,t_0}=0$ such that
\begin{equation}\label{RT}
\Upsilon_{\lambda,t}=\big\{\bq+\cM_{\lambda,t}\bq\ \big|\, \bq\in\Upsilon_{\lambda, t_0}\big\}, t\in\cI,
\end{equation}
see,  e.g., \cite[Lemma 3.8]{CJLS}. We call $(\lambda,t_0)$  a conjugate point if $\ker(H_{t_0}-\lambda)\not=\{0\}$, or equivalently, if assertions (ii) and (iii) in \eqref{equivEV} hold for $t=t_0$.  The Maslov crossing form $\mathfrak{m}_{t_0}$ for $\Upsilon_{\lambda,t}$ relative to $\mathfrak{D}$ at the conjugate point ${(\lambda, t_0)}$ 
is defined on the finite dimensional  intersection $\Upsilon_{\lambda, t_0}\cap\mathfrak{D}$ of the Lagrangian subspaces by the formula
\begin{equation}\label{dfnMcf}
\mathfrak{m}_{t_0}(\bq,\bp):=\frac{\bd}{\bd t}\big|_{t=t_0}\widehat{\omega} (\bq,\cM_{\lambda, t}\bp)=\widehat{\omega} (\bq,\dot{\cM}_{\lambda, t_0}\bp),\ \bp,\bq\in \Upsilon_{\lambda, t_0}\cap\mathfrak{D},
\end{equation}
where $\widehat{\omega} ={\omega} \oplus(-{\omega} )$ is a symplectic form on $(\mathfrak{H}\times\mathfrak{H})\oplus(\mathfrak{H}\times\mathfrak{H})$ and, as usual, we abbreviate  $\displaystyle{\dot{\cM}_{\lambda, t_0}=\frac{\bd}{\bd t}\cM_{\lambda, t}\big|_{t=t_0}}$.


\begin{lemma}\label{wSelec} Let $(\lambda, t_0)$ be a conjugate point satisfying Hypothesis \ref{i17} and let $u\in\ker(H_{t_0}-\lambda)$. Then there exist an open interval $\cI\subseteq\cJ$ centered at $t_0$, a  family $t\mapsto w_t$  in $C^1(\cI, \cH_+)$, and 
	a family $t\mapsto g_t\in\ran(Q_t)$ in $C^1(\cI, \mathfrak{H}\times\mathfrak{H})$ such that 
	\begin{align}
	w_{{t_0}}&=u,\quad g_{{t_0}}=\tr_{t_0} u,\lb{2.43}\\
	 w_t&\in\ker(A^*+V_t-\lambda),\label{propwt0}\\
	\big(
	\tr_{t}w_t, g_t
	\big)^\top
	&=
	\big(
	\tr_{t_0} u, \tr_{t_0} u
	\big)^\top
	+
	\cM_{\lambda, t} \big(
	\tr_{t_0} u, \tr_{t_0} u
	\big)^\top,\ t\in\cI,\label{propwt}
	\end{align}
	where $\cM_{\lambda, t}$ is as in \eqref{RT}.
\end{lemma}
\begin{proof}
	The proof is similar to that of Lemma 2.6 in \cite[p.355]{LSS}.  For brevity, we denote $N_t:=\ker(A^*+V_t-\lambda)$, $q:= \tr_{t_0}u$, $\bq:=(q,q)$ and let $P_t$ be the orthogonal projections onto $\bbK_{\lambda, t}$. Then $P_t\in C^1\big(\cI,  \cB(\mathfrak{H}\times\mathfrak{H})\big)$ for some open interval $\cI\subseteq\cJ$ centered at $t_0$ (see, e.g., \cite[Theorem 3.9]{BbF95}, \cite[Theorem 5.10]{LS1}). We now consider the projections in $(\mathfrak{H}\times\mathfrak{H})\times(\mathfrak{H}\times\mathfrak{H})$ given by
	\begin{equation*}
	\widehat{P}_t:=\begin{bmatrix}P_t&0\\0&0\end{bmatrix}, \ 
	\widehat{Q}_t:=\begin{bmatrix}0&0\\0&Q_t\end{bmatrix},
	\end{equation*}
	so that $\widehat{P}_t+\widehat{Q}_t=\Pi_{\lambda,t}$, $\ran (\Pi_{\lambda, t})=\Upsilon_{\lambda, t}=\bbK_{\lambda,t}\oplus\cG_t$.
	Using the definition of $\Upsilon_{\lambda, t}$ and $\cM_{\lambda,t}$, see \eqref{dfnUPS} and \eqref{RT}, we define 
	\begin{equation}
	h_t\in\ran (P_t)\subset\mathfrak{H}\times\mathfrak{H},\ g_t\in\ran (Q_t)\subset\mathfrak{H}\times\mathfrak{H},
	\end{equation}
	such that  \begin{equation}\label{Dfngt}
	(h_t,0)^\top=\widehat{P}_t(\bq+\cM_{\lambda,t}\bq)\text{ and $(0,g_t)^\top=\widehat{Q}_t(\bq+\cM_{\lambda,t}\bq)$},\end{equation}
	and so $h_{{t_0}}=g_{{t_0}}=q$. Since $t\mapsto \cM_{\lambda,t}$, $t\mapsto P_t$ and $t\mapsto Q_t$ are $C^1$, we know that the maps $t\mapsto h_t$ and $t\mapsto g_t$ are $C^1$.
	As above, employing Hypothesis \ref{i17} and $\ker \tr_t= \dom(A)$, see  Proposition \ref{remark2.2} (1),   we conclude that the restriction 
	\begin{equation}
	\tr_{t}\upharpoonright_{N_t}: N_t\to\ran(P_t)\subset\mathfrak{H}\times\mathfrak{H}, 
	\end{equation}
	of $\tr_{t}$ to $N_t$ is a bijection. Therefore, there is a unique vector $w_t\in N_t$ satisfying $\tr_{t}w_t=h_t$. Assertions \eqref{propwt0}, \eqref{propwt} hold with this choice of $w_t$ and $g_t$ .

	It remains to show that the function $t\mapsto w_t$ is in $C^1(\cI, \cH_+)$. 
	Let $U_t$ denote the $C^1$ family of boundedly invertible transformation operators in $\cH_+$ that split the projections $\cP_{N_t}$ onto $N_t$ and $\cP_{N_{{t_0}}}$ onto $N_{{t_0}}$ so that the identity 
	$U_t\cP_{N_{{t_0}}}=\cP_{N_t}U_t$ holds, and $U_t: N_{{t_0}}\mapsto N_t$ are bijections for $t$ near $t_0$, cf.\ \cite[Remark 2.4]{LSS},
	\cite[Remark 3.5]{CJLS}, \cite[Section IV.1]{DK74}, \cite[Remark 6.11]{F04}. We temporarily  introduce  $v_t\in N_{{t_0}}$ by $v_t=U_t^{-1}w_t$ so that $\tr_{t}w_t=h_t$ yields $(\tr_{t}\circ\, U_t)v_t=h_t$.
	The map $\tr_{t}\circ \,U_t\big|_{N_{{t_0}}}:N_{{t_0}}\to\ran(P_t)$ is a bijection and 
	$t\mapsto\tr_{t}\circ\, U_t\big|_{N_{{t_0}}}$ is in $C^1\big(\cI, \cB(N_{{t_0}},\mathfrak{H}\times\mathfrak{H})\big)$  by the assumptions in the lemma. Since 
	$w_t=U_t\circ\big(\tr_{t}\circ\, U_t\big)^{-1}h_t$, the function $t\mapsto w_t$ is $C^1$  because each of the three terms in the composition is $C^1$.
\end{proof}

\begin{theorem}\label{HadmardSimple}   Under Hypothesis \ref{hyp1.3bis}, let $(\lambda, t_0)$ be a conjugate point satisfying Hypothesis \ref{i17}. Let $\{\lambda_j(t)\}_{j=1}^m$, with $\lambda=\lambda(t_0)$, $\{u_j\}_{j=1}^m$ be as in Theorem \ref{thm:sec4.1}, and let $\bq_j:=(\tr_{t_0}u_j, \tr_{t_0}u_j)^\top\in\Upsilon_{\lambda, t_0}\cap\mathfrak{D}$. Then the slope of the eigenvalue curves satisfies
	\begin{equation}\label{formdr}
	\dot\lambda_j(t_0)=\mathfrak{m}_{t_0}(\bq_j,\bq_j),\ 1\leq j\leq m,
	\end{equation}
	where $\mathfrak{m}_{t_0}$ is the Maslov form  introduced in \eqref{dfnMcf}.
\end{theorem}

\begin{proof}
	For a fixed  $j$, let $(w_t, g_t)$ be as in Lemma \ref{wSelec} with $u:=u_j$. Differentiating 
	\begin{equation}\lb{2.49}
	A^*w_t+V_tw_t-\lambda w_t=0,
	\end{equation}
	at $t_0$ and multiplying the result by  $w_{{t_0}}=u_j$ we get
	\begin{equation}
	\langle(A^*+V_{{t_0}}-\lambda) \dot{w}_{{t_0}},\, w_{{t_0}}\rangle_\cH+	\langle\dot V_{{t_0}} {w}_{{t_0}},\, w_{{t_0}}\rangle_\cH=0.
	\end{equation}
	Using the Green identity \eqref{3.61biss} with $u=\dot{w}_{{t_0}}$ and $v=w_{{t_0}}$ we obtain
	\begin{align}\lb{2.50}
	\begin{split}
	&\langle(A^*+V_{{t_0}}-\lambda) \dot{w}_{{t_0}},\, w_{{t_0}}\rangle_\cH= \langle\dot{w}_{{t_0}},\, (A^*+V_{{t_0}}-\lambda)w_{{t_0}}\rangle_\cH\\
	&\quad+\langle\Gamma_{1{t_0}}\dot{w}_{{t_0}},\,\Gamma_{0{t_0}}w_{{t_0}}\rangle_{\mathfrak{H}}-\langle\Gamma_{0{t_0}}\dot{w}_{{t_0}},\,\Gamma_{1{t_0}}w_{{t_0}}\rangle_{\mathfrak{H}}
	\end{split}
	\end{align}
	Combining \eqref{2.49} and \eqref{2.50} yields 
	\begin{equation}\label{df5}
	\omega\big(\tr_{{t_0}}\dot{w}_{{t_0}},\, \tr_{t_0}u_j\big)+\langle\dot{V}_{{t_0}}u_{j},\, u_{j}\rangle_\cH=0.
	\end{equation}
	Next,  \eqref{dfnMcf}  and \eqref{propwt}  yield
	\begin{align}\lb{2.53}
	\mathfrak{m}_{{t_0}}(\bq_j,\bq_j)
	=\omega\big(\tr_{t_0}u_j,\frac{\bd}{\bd t}\big|_{t={t_0}}(\tr_{t}w_t)\big) 
	-\omega(\tr_{t_0}u_j,\,\dot{g}_{{t_0}}).
	\end{align}
	Since $g_t=Q_tg_t$ we have \[\dot{g}_{{t_0}}=\dot{Q}_{{t_0}}g_{t_0}+Q_{{t_0}}\dot{g}_{{t_0}}=\dot{Q}_{{t_0}}\tr_{t_0} u_j+Q_{{t_0}}\dot{g}_{{t_0}}.\] Utilizing this, the fact that $\ran(Q_{{t_0}})$ is Lagrangian and $\tr u_j\in\ran(Q_{t_0})$ we get
	\begin{equation}\lb{2.54}
	\omega(\tr_{t_0} u_j,\,\dot{g}_{{t_0}})=	\omega(\tr_{t_0} u_j,\,\dot{Q}_{{t_0}}\tr_{t_0} u_j+Q_{{t_0}}\dot{g}_{{t_0}})=\omega(\tr_{t_0} u_j,\dot{Q}_{{t_0}}\tr_{t_0}u_j).
	\end{equation}
	Then \eqref{df5}, \eqref{2.53}, and \eqref{2.54} yield
	\begin{align}
	\begin{split}\lb{aub34}
	\mathfrak{m}_{{t_0}}(\bq_j,\bq_j)
	&=\omega\big(\tr_{t_0}u_j,\dot\tr_{t_0}u_{j}\big)
	+\omega\big(\tr_{t_0}u_j,\,\tr_{{t_0}}\dot{w}_{{t_0}}\big)\\
	&\hskip4.7cm-\omega(\tr_{t_0}u_j,\dot{Q}_{{t_0}}\tr_{t_0}u_j)\\
	&\quad=\omega\big(\tr_{t_0}u_j,\dot\tr_{t_0}u_{j}\big)
	+\langle\dot{V}_{{t_0}}u_{j},\, u_{j}\rangle_\cH\\
	&\hskip4.7cm+\omega(\dot{Q}_{{t_0}}\tr_{t_0}u_j,\tr_{t_0}u_j),
	\end{split}
	\end{align}
	where we used $\omega(\dot{Q}_{{t_0}}\tr_{t_0}u_j,\tr_{t_0}u_j)\in\bbR$, see \eqref{aub33}. Comparing \eqref{aub34} and \eqref{2.36} one infers \eqref{formdr} as required. 
\end{proof}

\begin{remark}\label{masindexspflow}
Formula \eqref{formdr} in Theorem \ref{HadmardSimple} yields a fundamental relation between the Maslov index and the spectral flow of the family of operators $H_t=\cA_t+V_t$ satisfying the condition $\tr_t(\dom(H_t))=\cG_t$ for a given family of Lagrangian subspaces $\cG_t$, $t\in[0,1]$. This relation goes back to the celebrated Atiyah--Patodi--Singer Theorem and it has been a subject of intensive research ever since, see, e.g.,  \cite{BbF95, BZ3, BZ2, BZ1, CLM, CJLS, RoSa95} and many more references therein. We will briefly comment on the equality of the Maslov index and the spectral flow. First, we recall the definition of the Maslov index via crossing forms.  For a fixed $\lambda=\lambda_0$ from now on we assume that Hypothesis \ref{i17} is satisfied for all $t=t_0\in[0,1]$. Then, given the subspaces defined in \eqref{dfnUPS}, and assuming that all conjugate points $(\lambda, t_0)$ for $t_0\in[0,1]$ are non-degenerate (in the sense that the quadratic form $\mathfrak{m}_{t_0}$ from \eqref{dfnMcf} is non-degenerate), one defines the Maslov index by the formula
\begin{equation}\label{defMasInd}
{\rm Mas } \big(\Upsilon_{\lambda_0,t}: t\in[0,1]\big)=
-m_-(0)+\sum_{0<t_0<1}\big(m_+(t_0)-m_-(t_0)\big)+m_+(1),
\end{equation}
where the summation is taken over all $t_0$ such that $(\lambda, t_0)$ is a conjugate point and we denote by $m_+(t_0)$, respectively, $m_-(t_0)$ the number of positive, respectively, negative squares of the quadratic form $\mathfrak{m}_{t_0}$ at the conjugate point. Next, we recall the definition of the spectral flow: The spectral flow ${\rm SpF }_{\lambda_0}(H_t: t\in[0,1])$ for the family of operators $H_t$ is the net count of the eigenvalues of $H_t$ passing through $\lambda_0$ as $t$ changes from $t=0$ to $t=1$ and is defined as follows, cf., e.g., \cite[Appendix]{BZ2}. Take a partition $0=t_0<t_1<\dots<t_N=1$ and $N$ intervals $[a_\ell, b_\ell]$ such that $a_\ell<\lambda_0<b_\ell$ and
$a_\ell, b_\ell\notin\spec(H_t)$ for all $t\in[t_{\ell-1},t_\ell]$, $1\le\ell\le N$. Then the spectral flow is defined by
\begin{equation}\label{dfnspfl}
{\rm SpF }_{\lambda_0}(H_t: t\in[0,1])=\sum_{\ell=1}^N\sum_{a_\ell\le\lambda<\lambda_0}\big(\dim\ker(H_{t_{\ell-1}}-\lambda)-\dim\ker(H_{t_{\ell}}-\lambda)\big).
\end{equation}
By our assumptions, due to part (i) in Hypothesis \ref{i17}, $\lambda_0$ does not belong to the essential spectrum of the operator $H_t$ for all $t\in[0,1]$. Moreover, let us assume, in addition, that for each $t_0\in[0,1]$ such that $\lambda_0\in\spec_{\rm disc}(H_{t_0})$
the inequality $\dot{\lambda}_j(t_0)\neq0$ holds for all $j=1,\dots,m$.
 Here,  $m=m(t_0)$ is the multiplicity of the  isolated eigenvalue $\lambda_0$ of $H_{t_0}$, and
$\{\lambda_j(t)\}$ are the eigenvalues of $H_t$ as in Theorem \ref{prop5.1}(2) and Theorem \ref{thm:sec4.1}(2) for $t\in[t_0',t_0'']$ near $t_0$. With no loss of generality $t=t_0$ could be assumed to be the only point in $[t_0',t_0'']$ such that $\lambda_0\in\spec(H_{t})$. By our assumptions and formula \eqref{formdr} in Theorem \ref{HadmardSimple} the quadratic
form $\mathfrak{m}_{t_0}$ defined in \eqref{dfnMcf} is non-degenerate and $m_+(t_0)$, respectively, $m_-(t_0)$ is equal to the number of $j$'s such that the eigenvalue $\lambda_j(t)$ moves through $\lambda_0$ in the positive, respectively, negative direction as $t$ changes from $t_0'$ to $t_0''$. Formulas \eqref{defMasInd} and \eqref{dfnspfl} now show that
${\rm Mas } \big(\Upsilon_{\lambda_0,t}: t\in[t_0',t_0'']\big)={\rm SpF }_{\lambda_0}(H_t: t\in[t_0',t_0''])$.
Passing to a partition of $[0,1]$ then gives
\begin{equation}\label{MISF}
{\rm Mas } \big(\Upsilon_{\lambda_0,t}: t\in[0,1]\big)={\rm SpF }_{\lambda_0}(H_t: t\in[0,1]),\end{equation}
the desired equality of the Maslov index and the spectral flow.
\hfill$\Diamond$  \end{remark}

\section{Hadamard-type formula for elliptic operators via Dirichlet and Neumann traces}\lb{subsec1.1} 

In this section concerns self-adjoint realizations of second order elliptic operators on bounded domains. We begin by discussing a resolvent difference formula, see Proposition \ref{prop1.7}, an Hadamard-type formula, \eqref{cLthetder}, and asymptotic resolvent expansions, Theorem \ref{theorem6.2},  for the elliptic operators \eqref{1.15b} posted on bounded domains with smooth boundary. We deduce all these results from Theorem \ref{prop5.1} by appropriately choosing the trace maps. The main technical issue is to validate Hypotheses \ref{hyp1.3i2bis} and \ref{hyp1.3i2}, which is done in Proposition \ref{prop5.2}. Next, these results are utilized to give simple and unified proofs of Friedlander's Theorem \cite[Theorem 1.1]{Fri91}, see Example \ref{FREX}, and Rohleder's Theorem \cite[Theorem 3.2]{Rohl14}, see Example \ref{Rohleder}. Furthermore, in Section \ref{ssHeatEq} we consider the heat equation with space-dependent diffusion coefficient equipped with Robin boundary conditions so that both the equation and the boundary conditions contain a physically relevant parameter, the thermal conductivity. The results in this section provide, in particular, a new proof of the fact  that the temperature of a non-homogeneous material immersed into a surrounding medium of constant temperature depends continuously on the thermal conductivity of the material.

\subsection{Elliptic operators}\label{SS5.1} On a $C^\infty$-smooth bounded domain $\Omega$  we consider the following differential expression, 
\begin{align}
\begin{split}
\cL:&=-\sum_{j,k=1}^{n} \partial_j \mathtt{a}_{jk}\partial_k + \sum_{j=1}^{n}\mathtt{a}_j\partial_j -\partial_j {\mathtt{a}_j}+\mathtt{q},\lb{1.15b} \\
&=-\div(\mathtt{A}\nabla)+ \mathtt{a}\cdot \nabla -\nabla \cdot \mathtt{a} +\mathtt{q},
\end{split}
\end{align}
with coefficients $\mathtt{A}=\{\mathtt{a}_{ij}\}_{1\leq i, j\leq n}$, $\mathtt{a}:=\{\mathtt{a}_i\}_{1\leq i\leq n}$ satisfying, for some $c=c(\cL)>0$, 
\begin{align}
& \sum\limits_{j,k=1}^n\mathtt{a}_{jk}(x)\xi_k\overline{\xi_j} \geq c \sum\limits_{j=1}^n|\xi_j|^2, x\in\overline{\Omega}, \xi=\{\xi_j\}_{j=1}^n\in\C^n,\lb{unifel}\\
&\mathtt{a}_{jk}, \mathtt{a}_j\in C^{\infty}(\overline{\Omega};\bbR), \mathtt{q}\in L^{\infty}(\Omega; \bbR), \mathtt{a}_{jk}(x)={\mathtt{a}_{kj}(x)}, 1\leq j,k\leq n.
\end{align}
Associated with $\cL$ is the following space of distributions,
\begin{align}\lb{1.15c}
& \cD^{s}(\Omega):=\{u\in H^s(\Omega): \cL u\in L^2(\Omega)\},\ s\geq 0, 
\end{align} 
equipped with the norm
\begin{equation} \lb{1.15d}
\|u\|_{s}:=\left(\|u\|_{H^s({\Omega} )}^2+\|\cL u\|_{L^2({\Omega })}^2\right)^{1/2}, 
\end{equation}
where $\cL u$ should be understood in the sense of distributions. Let us introduce two operators acting in $L^2(\Omega)$,
\begin{align}
&\cL_{min}f:=\cL f,\ f\in\dom(\cL_{min}):=H^2_0(\Omega), \\
&\cL_{max}f:=\cL f,\ f\in\dom(\cL_{max}):=\cD^0(\Omega). 
\end{align}
The operator $\cL_{min}$ is closed, symmetric, and  $(\cL_{min})^*=\cL_{max}$. Associated with $\cL$  is a first order trace operator $\gamma_{{}_{N, \cL}} \in\cB(\cD^1(\Omega), H^{-1/2}(\partial\Omega))$ which is a unique extension of the co-normal derivative 
\begin{equation}
{\gamma_{{}_{N,\cL}}}u:=\sum_{j,k=1}^{n} \mathtt{a}_{jk}\nu_j\gaD(\partial_k u)+\sum_{j=1}^{n}{\mathtt{a}_j}\nu_j \gaD u, u\in H^2(\Omega)\label{dfncnd}
\end{equation}
to the space $\cD^{1}(\Omega)$ (here, $(\nu_1,\cdots,\, \nu_n)$ is the outward unit normal on $\partial\Omega$). Then the following Green identity holds,
\begin{align}
\langle\cL u, v\rangle_{L^2(\Omega )}- \langle u,\cL v\rangle_{L^2(\Omega )}={\langle{\gaD u, \gamma_{{}_{N, \cL}}} v  \rangle}_{{-1/2}}-\overline{\langle{\gaD v, \gamma_{{}_{N, \cL}}} u \rangle_{{-1/2}}},\lb{1.2}
 \end{align}
for all $u,v\in\cD^1(\Omega)$. In order to rewrite this identity in a form compatible with \eqref{3.61} let $\Phi$ denote the Riesz isomorphism $\Phi\in\cB (H^{-1/2}(\partial \Omega), H^{1/2}(\partial \Omega))$ as in \eqref{aub27}
and define 
\begin{equation}\lb{1.60}
\Gamma_0:=\gaD\in \cB (\cD^{1}( \Omega), H^{1/2}(\partial \Omega)),\ \Gamma_1:=-\Phi\gamma_{N,\cL} \in \cB (\cD^{1}( \Omega), H^{1/2}(\partial \Omega)).
\end{equation}
Then we have, for all $u,v\in\cD^1(\Omega)$,
\begin{align}
\begin{split}\lb{1.23o}
\langle\cL_{max} u, v\rangle_{L^2(\Omega )}&- \langle u,\cL_{max} v\rangle_{L^2(\Omega )}\\
&={\langle\Gamma_1u, \Gamma_0v \rangle_{H^{1/2}(\partial \Omega)}}-\langle\Gamma_0u, \Gamma_1v \rangle_{H^{1/2}(\partial \Omega)}.
\end{split}
\end{align}
We claim that Hypotheses \ref{hyp3.6} and  \ref{hyp2.2} are satisfied for 
\begin{equation}\lb{aub11}
A=\cL_{min}, \cH_+=\cD^0(\Omega), \cD=\cD^1(\Omega), \Gamma_0=\gaD, \Gamma_1=-\Phi\gamma_{N,\cL}.  
\end{equation}
Since we already checked the Green identity, \eqref{1.23o}, 
to justify the claim it remains to show that $\tr(\cD)$ is dense in $H^{1/2}(\partial\Omega)\times H^{1/2}(\partial\Omega)$ and that $\cD^1(\Omega)$ is dense in $\cD^0(\Omega)$. By \cite[Proposition 2.1]{Gr}, \cite[Section 4.3]{BM} one has
\begin{equation}\no
(\gaD, \gamma_{{}_{N,\cL}})(H^2(\Omega))=H^{3/2}(\partial\Omega)\times H^{1/2}(\partial\Omega),
\end{equation}
and the right-hand side is dense in $H^{1/2}(\partial\Omega)\times H^{1/2}(\partial\Omega)$. By \cite[Theorem 3.2]{Gr}, $H^2(\Omega)$ is dense in $\cD^s(\Omega), s<2$, hence $\cD^1(\Omega)$ is dense in $\cD^0(\Omega)$. 
\begin{proposition}\lb{prop1.7} Under the assumptions on $\cL$ imposed in this section, for any two self-adjoint extensions $\cL_1, \cL_2$ of $\cL_{min}$ with domains containing in $\cD^1(\Omega)$ and $\zeta\not\in (\spec(\cL_1)\cup\spec(\cL_2))$, the following resolvent difference formula holds,
	\begin{align}
	(\cL_2-\zeta)^{-1}-(\cL_1-\zeta)^{-1}&=(\tr(\cL_2-\overline{\zeta})^{-1})^*\, J\tr (\cL_1-\zeta)^{-1}, 
	\end{align}
	where $\tr=[\Gamma_0, \Gamma_1]^\top$ is defined in \eqref{1.60}, and \[(\tr(\cL_2-\overline{\zeta})^{-1})^*\in\cB( H^{1/2}(\partial\Omega)\times  H^{1/2}(\partial\Omega), L^2(\Omega)).\]
\end{proposition}
\begin{proof}
	The results follows directly from \eqref{5.14aJ}. 
\end{proof}
\subsection{Hadamard-type formulas for Robin elliptic operators, L. Friedlander's and J. Rohleder's inequalities }\label{ssKKFREO}
In this section we obtain an Hadamard-type formula for a one-parameter family of differential operators $\cL_tu=\cL u$ as in \eqref{1.15b} for which the dependence on the parameter $t$ enters through the Robin boundary condition $\gamma_{{}_{N,\cL}} u=\Theta_t \gaD u$, see Theorem \ref{theorem6.2}. We will utilize Theorem \ref{prop5.1} by choosing the symmetric operator $A$, the function spaces $\cH, \cH_+, \mathfrak H$, and the trace operator $\tr$ as indicated in \eqref{aub11}. The main challenge is to check Hypothesis \ref{hyp1.3i3} which in this setting reads as follows,
\begin{equation}
\big\|(\cL_t-\bfi)^{-1}-(\cL_{t_0}-\bfi)^{-1}\big\|_{\cB(L^2(\Omega), \cD^1(\Omega))}=o(1),\ t\rightarrow t_0,
\end{equation}
and can be reduced to showing that for some constant $c>0$ one has
the inequality
	\begin{align}
&\|\nabla u\|_{L^2(\Omega)}^2\leq c\big( \|  \cL u\|^2_{L^2(\Omega)}+  \|u\|_{L^2(\Omega)}^2\big), u\in\dom(\cL_t), 
\end{align}
for $t$ near $t_0$. We discuss the reduction and give the proof of this inequality in Proposition \ref{prop5.2}. Throughout this section we will make use of the continuous embedding $\iota: H^{1/2}(\partial\Omega)\hookrightarrow	L^2(\Omega)$ and its adjoint $\iota^*: L^2(\Omega)\hookrightarrow H^{-1/2}(\partial\Omega)$.  

\begin{theorem}\lb{theorem6.2} Suppose that, in addition to  the assumptions on $\cL$ listed in Subsection \ref{SS5.1}, we are given a mapping $t\mapsto\Theta_t$  belonging to $C^1 ([0,1], L^{\infty}( \partial\Omega, \bbR))$. Then for $t\in[0,1]$ the Robbin elliptic operator $\cL_t$ defined by
	\begin{align}
	\begin{split}\lb{elt}
	&\cL_t: \dom(\cL_t)\subset L^2(\Omega)\rightarrow L^2(\Omega),\quad \cL_t u=\cL u, \\
	&u\in\dom(\cL_t)=\{ u\in \cD^1(\Omega): \gamma_{{}_{N,\cL}} u=\iota ^*\Theta_t \iota\gaD u \},
	\end{split}
	\end{align}
	is self-adjoint, where $\iota$ denotes the embedding of $H^{1/2}(\partial\Omega)$ into $L^2(\Omega)$.  The following resolvent difference formula holds,
	\begin{align}\lb{cLKrein}
	(\cL_t-\zeta)^{-1}-(\cL_s-\zeta)^{-1}= \big(\gaD(\cL_t-\overline{\zeta})^{-1} \big)^* (\Theta_t-\Theta_s)\big(\gaD(\cL_s-\zeta)^{-1}\big), 
	\end{align}
	for $t,s\in[0,1]$, $\zeta\not\in(\spec(\cL_t)\cup\spec(\cL_s))$. Moreover, the mapping 
	\begin{equation}\label{cLcont}
	t\mapsto (\cL_t-\zeta)^{-1}\in \cB(L^2(\Omega))
	\end{equation}
is well defined for $t$ near $t_0$ as long as $\zeta\not\in\spec(\cL_{t_0})$. This mapping is differentiable at $t_0$ and satisfies the following Riccati equaiton,
	\begin{align}\lb{cLRic}
	\frac{\bd}{\bd t}|_{t=t_0}\big((\cL_t-\zeta)^{-1}\big)=\big(\gaD(\cL_{t_0}-\overline{\zeta})^{-1}\big)^* \big(\frac{\bd}{\bd t}|_{t=t_0}\Theta_t\big)\big(\gaD(\cL_{t_0}-\zeta)^{-1}\big).
	\end{align}
	Finally, if $\lambda(t_0)$ is an isolated eigenvalue of $\cL_{t_0}$ of multiplicity $m\geq 1$ then there exist  a choice of orthonormal eigenfunctions  $\{u_j\}_{j=1}^m\subset \ker(\cL_{t_0}-\lambda(t_0))$
	and a labeling of eigenvalues $\{\lambda_j(t)\}_{j=1}^m$  of $\cL_{t}$, for $t$ near $t_0$, such that  
	\begin{equation}\lb{cLthetder}
	\dot{\lambda}_j(t_0)=- \langle\dot \Theta_{t_0} \gaD u_j, \gaD u_j\rangle_{L^2(\partial\Omega)}, 1\leq j\leq m.
	\end{equation}
\end{theorem}
\begin{proof}
	We will employ Theorem \ref{prop5.1}. The proof consists of two steps. First, we derive \eqref{cLKrein} from \eqref{XYKrein}. We can use \eqref{XYKrein} because Hypothesis \ref{hyp1.3} is trivially satisfied. Second,  we derive 
	\eqref{cLRic} and \eqref{cLthetder} from \eqref{derRnew} and \eqref{5.10new}. To apply \eqref{derRnew} and \eqref{5.10new} we need to verify Hypotheses \ref{hyp1.3i2} and \ref{hyp1.3i3}. They are satisfied by Proposition \ref{prop5.2} given next; the proof of this proposition uses formula \eqref{cLKrein} proved in the first step.
	
To proceed, we choose $ \cH_+, \cD, A$ as in \eqref{aub11} and rewrite the Robin condition $\gamma_{{}_{N, {\cL}}}u=\iota ^*\Theta_t\iota\gaD u$ in the definition of $\cL_t$ as $\Phi\gamma_{{}_{N, {\cL}}}u=\Phi\iota^* \Theta_t \iota\gaD u$
	\begin{equation}
	X_t\Gamma_0u+Y_t\Gamma_1u=0, \text{ where we set }\, X_t:=\Phi \iota^*\Theta_t\iota, Y_t:=I. 
	\end{equation}
It is worth noting that $X_t$ just defined is self-adjoint in $H^{1/2}(\partial\Omega)$ since for $\phi,\psi\in H^{1/2}(\partial\Omega)$ one has
\begin{align}
\langle \Phi \iota^*\Theta_t\iota \phi,\psi\rangle_{1/2}&=
\overline{\langle \psi,\Phi \iota^*\Theta_t\iota \phi,\psi\rangle_{1/2}}=
\overline{\langle \psi,\iota^*\Theta_t\iota \phi,\psi\rangle_{-1/2}}
\\&=\overline{\langle \iota\psi, \Theta_t\iota \phi\rangle_{L^2(\partial\Omega)}}
=\langle  \iota \phi,\Theta_t\iota\psi\rangle_{L^2(\partial\Omega)}\\&=\langle   \phi,\iota^*\Theta_t\iota\psi\rangle_{-1/2}=\langle   \phi,\Phi\iota^*\Theta_t\iota\psi\rangle_{1/2}.
\end{align}
Continuity of $\Theta_t$ with respect to $t$ and Theorem \ref{prop1.8new} with $\cA_t:=\cL_t$, $V_t:=0$, $\tr_t:=[\gaD, -\Phi \gamma_{N,\cL}]^\top$ yield that the map $t\mapsto R_t(\zeta):=(\cL_t-\zeta)^{-1}$ is well defined for $t$ near $t_0$.
Next, with $W$ defined in \eqref{defw}, we observe that  $R_t(\zeta)u\in\dom(\cA_t)$ yields
	\begin{equation}
	\big(W(X_t, I)\big) \tr  R_t(\zeta)u = -\Gamma_0 R_t(\zeta)u=-\gaD R_t(\zeta)u \text{ for all $u\in L^2(\Omega)$}.
	\end{equation}
This can be checked directly or by noting that $\phi=\big(W(X_t, I)\big) \tr    R_t(\zeta)u$ is the unique vector satisfying the relations
$\Gamma_0R_t(\zeta)u=-\phi$, $\Gamma_1R_t(\zeta)u=X_t\phi$, cf.\ \eqref{phij}. 
	This observation together with \eqref{XYKrein} yield \eqref{cLKrein}. 
We can now involve Proposition \ref{prop5.2} given next and verify 
Hypotheses \ref{hyp1.3i2} and \ref{hyp1.3i3} in the present setting. 
Thus, Theorem \ref{prop5.1} applies and therefore 
 \eqref{cLRic} and \eqref{cLthetder} follow from  \eqref{derRnew} and \eqref{5.10new} with $\phi_j=-\Gamma_0 u_j$.
\end{proof}
\begin{remark} It is worth comparing Theorems \ref{theorem5.4}  and \ref{theorem6.2} for the case $\cL=-\Delta$ where both theorems apply. The major difference is in the type of trace operators utilized in each theorem. In Theorem  \ref{theorem5.4} we use $\tr=[-\tau_{{}_N}, \Phi \gd]^\top$ which is defined on the entire space $\cH_+=\dom(-\Delta_{max})$ and is surjective, while in Theorem \ref{theorem6.2} we have $\tr=[\gaD, -\Phi \gamma_{{}_{N, \cL}}]^\top$ which is defined only on a dense subset $\cD=\cD^1(\Omega)$ of $\cH_+=\cD^0(\Omega)$. We note that the latter trace operator is local while the former is not. In addition, these trace maps do not match even on smooth functions on $\Omega$. Another major technical difference is that Hypotheses \ref{hyp1.3i2} and  \ref{hyp1.3i3} are automatically satisfied in one case but not in the other. 
\hfill$\Diamond$  \end{remark}
 
\begin{proposition}\lb{prop5.2} Under assumptions of Theorem \ref{theorem6.2} one has
	\begin{align}
	&\|(\cL_t-\bfi)^{-1}\|_{\cB(L^2(\Omega), \cD^1(\Omega))}=\cO(1),\ t\rightarrow t_0,\lb{5.7}\\
	&\big\|(\cL_t-\bfi)^{-1}-(\cL_{t_0}-\bfi)^{-1}\big\|_{\cB(L^2(\Omega), \cD^1(\Omega))}=o(1),\ t\rightarrow t_0, \lb{5.8}
	\end{align}
	for all $t_0\in[0,1]$. In other words, Hypotheses \ref{hyp1.3i2} and \ref{hyp1.3i3} hold for $\cA_t:=\cL_t$. 
\end{proposition}
\begin{proof} To prove \eqref{5.7} it is enough to show that there exists a constant $c>0$ such that  
	\begin{equation}
	\|u\|_{\cD^1(\Omega)}^2\leq c\|\cL u-\bfi u\|^2_{L^2(\Omega)}, u\in\dom(\cL_t),
	\end{equation}
	for all $t\in[0,1]$. By the definition of $\cD^1(\Omega)$-norm, see \eqref{1.15d}, we need to prove that
	\begin{equation}\lb{5.7n}
	\|\nabla u\|_{L^2(\Omega)}^2\leq c (\|\cL u\|^2_{L^2(\Omega)}+\| u\|^2_{L^2(\Omega)}), u\in\dom(\cL_t).
	\end{equation}
	To show this, we first notice that for $u\in\dom(\cL_t)$ one has
	\begin{equation}
	\langle \mathtt{A}\nabla u,\nabla u\rangle_{L^2(\Omega)}=\langle\cL u,u\rangle_{L^2(\Omega)}-\langle \mathtt{q} u, u\rangle_{L^2(\Omega)} -\langle \Theta_t \gaD u,\gaD u\rangle_{L^2(\partial\Omega)}.
	\end{equation}
	Using the Cauchy--Schwartz inequality and \eqref{unifel} we get
		\begin{equation}\lb{gradest}
	\| \nabla u\|_{L^2(\Omega)}^2\leq c( \|\cL u\|^2_{L^2(\Omega)}+\|u\|^2_{L^2(\Omega)}+\|\Theta_t\|_{L^{\infty}(\partial\Omega)} \|\gaD u\|_{L^2(\partial\Omega)}^2),
	\end{equation}
	for $c>0$ (which is $t-$ and $u-$independent). Let us recall from  \cite[Lemma 2.5]{GM08} the inequality 
	\begin{equation}
	\|\gaD u\|^2_{L^2(\Omega)}\leq \varepsilon \|\nabla u\|_{L^2(\Omega)}^2+\beta(\varepsilon)\|u\|^2_{L^2(\Omega)}, \text{ where $\varepsilon>0$ and $\beta(\varepsilon)\underset{\varepsilon\rightarrow 0}{=}\cO(\varepsilon^{-1})$}. 
	\end{equation} 
    Thus, continuing \eqref{gradest} we infer
	\begin{align}
	&\|\nabla u\|_{L^2(\Omega)}^2
	\leq  c \Big(\|  \cL u\|^2_{L^2(\Omega)}+  \|u\|_{L^2(\Omega)}^2+\varepsilon\|\Theta_t\|_{L^{\infty}(\partial\Omega)}\|\nabla u\|^2_{L^2(\Omega)}\\
	&\hskip4.5cm+\beta(\varepsilon)\|\Theta_t\|_{L^{\infty}(\partial\Omega)}\|u\|_{L^2(\Omega)}\Big)
	\end{align}
	for some $c>0$.  Taking $\varepsilon>0$ sufficiently small 
	 yields \eqref{5.7n} and thus \eqref{5.7}.
	
 	Starting the proof of \eqref{5.8}, we first show that
	\begin{equation}\lb{L2H1}
	\big\|(\cL_t-\bfi)^{-1}-(\cL_{t_0}-\bfi)^{-1}\big\|_{\cB(L^2(\Omega), H^1(\Omega))}=o(1),\ t\rightarrow t_0. 
	\end{equation}
	We denote $R(t):=(\cL_t-\bfi)^{-1}$ and
	 recall that we may use resolvent difference formula \eqref{cLKrein} already established in the first part of the proof of Theorem \ref{theorem6.2}. It yields
	\begin{align}\lb{thetagreen}
	&\langle R(t)u-R(t_0)u,v\rangle_{L^2(\Omega)}=\langle(\Theta_{t_0}-\Theta_{t})\gaD R(t)u, \gaD R(t_0)v\rangle_{L^2(\partial\Omega)}
	\end{align}
	 for all  $u, v\in L^2(\Omega)$.
	For $v\in(H^1(\Omega))^*=H^{-1}(\Omega)$ we view $w:=R(t_0)v\in H^1(\Omega)$ as the solution to the boundary value problem $(\cL-\bfi)w=v$, $ \gamma_{{}_{N,\cL}}w=\Theta_{t_0}\gaD w$. Using a well-known elliptic estimate  $\|w\|_{H^1(\Omega)}\le c\|v\|_{H^{-1}(\Omega)}$ from \cite[Theorem 4.11(i)]{Mc},	
	the operator $R(t_0)$ can be extended to an operator in $\cB((H^1(\Omega))^*, H^1(\Omega))$.
	So, \eqref{thetagreen} can be extended as follows,
	\begin{equation}
	_{H^1(\Omega))}\langle R(t)u-R(t_0)u,v\rangle_{(H^1(\Omega))^*}=\langle(\Theta_{t_0}-\Theta_{t})\gaD R(t)u, \gaD R(t_0)v\rangle_{L^2(\partial\Omega)},
	\end{equation}
	now for all $u\in L^2(\Omega)$ and  $v\in (H^1(\Omega))^*$. Hence, 
	\begin{align}
	| _{H^1(\Omega))}\langle R(t)u&-R(t_0)u,v\rangle_{(H^1(\Omega))^*}|\leq \|\Theta_{t_0}-\Theta_{t}\|_{L^{\infty}(\partial\Omega)}\|\gaD\|^2_{\cB(\cD^1(\Omega), H^{1/2}(\partial\Omega))}\\
	& \times \|R(t)\|_{\cB(L^2(\Omega), \cD^1(\Omega))}\|u\|_{L^2(\Omega)} \|R(t_0)\|_{\cB((H^1(\Omega))^*, H^1(\Omega))}\|v\|_{(H^1(\Omega))^*}. 
	\end{align}
	Since $\|R(t)\|_{\cB(L^2(\Omega), \cD^1(\Omega))}=\cO(1)$ by \eqref{5.7}, and $\|\Theta_{t_0}-\Theta_{t}\|_{L^{\infty}(\partial\Omega)}=o(1)$, $t\rightarrow t_0,$ the above inequality gives \eqref{L2H1}.
	 We now combine \eqref{L2H1} with the estimate
	\begin{align}
	\big\|(\cL_t-\bfi)^{-1}u&-(\cL_{t_0}-\bfi)^{-1}u\big\|_{\cD^1(\Omega)}^2=\big\|(\cL_t-\bfi)^{-1}u-(\cL_{t_0}-\bfi)^{-1}u\big\|_{H^1(\Omega)}^2\\
	&\qquad + \big\|\cL(\cL_t-\bfi)^{-1}u-\cL(\cL_{t_0}-\bfi)^{-1}u\big\|_{L^2(\Omega)}^2\\
	&\le2\big\|(\cL_t-\bfi)^{-1}u-(\cL_{t_0}-\bfi)^{-1}u\big\|_{H^1(\Omega)}^2, u\in L^2(\Omega),
	\end{align}
finishing the proof of \eqref{5.8}. \end{proof}

\begin{example}\label{FREX} Theorem \ref{theorem6.2} can be used in proving the celebrated  Friedlander Inequalities $\lambda_{D,k}\geq \lambda_{N, k+1}$, $k=1,2,\dots$, for the eigenvalues of the Dirichlet and Neumann Laplacians, see \cite{Fri91}, which was improved in \cite{Fil04} to state  that $\lambda_{D,k}>\lambda_{N, k+1}$, see also \cite{BRS18, FrLa10, GMjde09,Saf08} for further advances, detailed bibliography and a historical account of this beautiful subject. Also, we refer to Example \ref{rem:FrMorse} for connections to the Maslov index.
The proof of the Friedlander Inequalities consists of two major steps. First,  one proves that the counting functions of the 
 Dirichlet and Neumann boundary problems differ by a number of negative eigenvalues of the Dirichlet-to-Neumann operator, see \eqref{rem:FrMorse1} below. 
 Second, one proves the existence of a nonnegative eigenvalue of the latter. 
 The first step involves  a one-parameter family of Robin boundary value problems giving a homotopy of the Dirichlet to the Neumann boundary problem. 
The critical issue here is to show monotonicity of the eigenvalues of the Robin  problems
with respect to the parameter, and this is where the results of the current paper help. (In fact, monotonicity holds not merely for the Laplacian but for  general elliptic operators as described in Subsection \ref{SS5.1}).  
Indeed, formula \eqref{cLthetder} in Theorem \ref{theorem6.2}  with 
 $\cL=-\Delta$ and $\Theta_t=-\cot(\frac{\pi}{2}t)$  shows that the eigenvalues $\lambda=\lambda(t)$ of the Robin problem
	\begin{equation}\label{eq:RobEV}
		\begin{cases}
			\cL u=\lambda u \text{\ in\ }\ \Omega,\\
			\sin(\frac{\pi}{2}t)\gaN u+\cos(\frac{\pi}{2}t)\gaD u=0\text{\ on $\partial\Omega$ for $t\in[0,1]$,}
		\end{cases}
	\end{equation}
	are monotonically decreasing with respect to $t\in[0,1]$. We note that
	\begin{align*}
&\lambda_k(0)=\lambda_{D,k}\le\lambda_{D,k+1} =\lambda_{k+1}(0)\text{ and }\\
&\lambda_k(1)=\lambda_{N,k}\le\lambda_{N,k+1}=\lambda_{k+1}(1), k=1, 2, \ldots,
	\end{align*}
	are the Dirichlet and  Neumann eigenvalues. From this point on, the arguments given in \cite{Fri91} and \cite{Fil04} are as follows. Monotonicity in $t$ of the Robin eigenvalues $\lambda_k(t)$ just proved, and the standard inequalities $\lambda_{D,k}\ge\lambda_{N,k}$ show the strict inequalities $\lambda_{D,k}>\lambda_{N, k+1}$ provided we know the fact, cf.\ \cite[Lemma 1.3]{Fri91}, that for each $\lambda$ there is a $t\in[0,1]$ such that \eqref{eq:RobEV}  has a nontrivial solution. This fact is equivalent to the existence of a positive eigenvalue $\cot(\frac{\pi}{2}t)$ of the Dirichlet-to-Neumann operator when $\lambda\notin\Sp(-\Delta_D)$, and its proof  has been carried out in \cite{Fri91} and  \cite{Fil04} for the Laplacian using the minimax principle and infinitely many linearly independent explicit functions $e^{i\eta\cdot x}$, with $\eta\in\bbR^n$ such that $\|\eta\|^2_{\bbR^n}=\lambda$, that satisfy $-\Delta(e^{i\eta\cdot x})=\lambda e^{i\eta\cdot x}$.
\hfill$\Diamond$  \end{example}

\begin{example}\label{Rohleder}
	We will now derive from Theorem  \ref{theorem6.2} an elegant result in \cite[Theorem 3.2]{Rohl14} regarding monotonicity of Robin eigenvalues. Given $\Theta^{(\ell)}\in L^\infty(\Omega; \bbR)$,  $\ell=0,1$, we define the Robin operators $\cL^{(\ell)}u=\cL u$ such that \[\dom(\cL^{(\ell)})=\{u\in\cD^1(\Omega): \gamma_{{}_{N,\cL}}u=\Theta^{(\ell)}\gaD u\}\]  for the elliptic  differential expression in \eqref{1.15b}. We let $\lambda_1(\cL^{(\ell)})\le\lambda_2(\cL^{(\ell)})\le\dots$ denote the eigenvalues of
	$\cL^{(\ell)}$ counting multiplicities. Assume that $\Theta^{(0)}\le\Theta^{(1)}$. We will give a new proof of J. Rohleder's result stating that
	\begin{equation}\label{RohT}
	\text{if $\Theta^{(0)}<\Theta^{(1)}$ on a set of positive measure then $\lambda_k(\cL^{(0)})>\lambda_k(\cL^{(1)})$}
	\end{equation}
	for $k=1,2,\dots$. Denote $\Theta_t=\Theta^{(0)}+t(\Theta^{(1)}-\Theta^{(0)})$ for $t\in[0,1]$ and introduce operators $\cL_t$ as in Theorem \ref{theorem6.2} such that $\cL_0=\cL^{(0)}$ and $\cL_1=\cL^{(1)}$. Denoting by $\lambda_k(t):=\lambda_k(\cL_t)$ the eigenvalues of $\cL_t$ counting multiplicities and by $u_k$ the respective eigenfunctions, formula \eqref{cLthetder} implies
	\begin{equation}\lb{aub20}
	\frac{\bd\lambda_k(t)}{\bd t}=-\langle(\Theta^{(1)}-\Theta^{(0)})\gaD u_k,\gaD u_k\rangle_{L^2(\partial\Omega)}<0 \text{, $k=1,2,\dots$, $t\in[0,1]$}
	\end{equation}
	because $\Theta^{(0)}<\Theta^{(1)}$ on a set of positive measure, thus proving \eqref{RohT}. Let us elaborate on some additional consequences of monotonicity of eigenvalues.
As the eigenvalue curves $t\mapsto \lambda_k(t)$ are strictly monotone and continuous we obtain the following count for the eigenvalues, see Figure \ref{fig1},
	\begin{align}\lb{aub15}
	\begin{split}
		(\#\{k: \lambda_k(\cL^{(1)})<\lambda\})-(\#\{k: \lambda_k&(\cL^{(0)})<\lambda\})\\
		&=\sum_{t\in [0,1]}\dim\ker(\cL_t-\lambda).
	\end{split}
	\end{align}
	A weaker version of this counting formula 
	\begin{align}
	(\#\{k: \lambda_k(\cL^{(1)})<\lambda\})-(\#\{k: \lambda_k(\cL^{(0)})<\lambda\})\geq \dim\ker(\cL^{(0)}-\lambda),
	\end{align}
	was obtained by J. Rohleder \cite[(3.4)]{Rohl14} by variational methods. This is a key estimate in \cite{Rohl14} leading to \eqref{RohT} in the original proof. Now, \eqref{aub15} can be viewed as a prequel to Section \ref{section5.4}, where the left-hand side of \eqref{aub15} is treated as the spectral flow of the family $\{\cL_t\}_{t\in[0,1]}$ through $\lambda$ and the right-hand side is viewed as the Maslov index of a certain path of Lagrangian planes. The equality between the Maslov index and the spectral flow in a very general setting has been recently investigated in, for example, \cite{CJLS, CJM1, CJM2, LS1, LSS} and the vast literature cited therein.
	\begin{figure}
		\begin{picture}(100,110)(20,0)
		\put(78,12){\Tiny $\lambda$}
		\put(80,20){\line(0,1){60}}
		\put(10,20){\vector(0,1){95}}
		\put(10,20){\vector(1,0){95}}
		\put(100,12){\Tiny$ \lambda$}
		\put(15,105){\Tiny$ t$}
		\put(2,10){\Tiny$0$}
		\put(10,20){\line(1,0){70}}
		\put(10,80){\line(1,0){70}}
		\qbezier(20,80)(30,30)(65,20)
		\qbezier(80,50)(50,50)(40,80)
		\qbezier(80,70)(70,70)(60,80)
		\put(65,20){\circle*{4}}
		\put(50, 12){{\Tiny ${\ }_{\lambda_1(0)}$}}
		\put(80,50){\circle*{4}}
		\put(80,70){\circle*{4}}
		\put(20,80){\circle*{4}}
		\put(8,87){{\Tiny ${\ }_{\lambda_1(1)}$}}
		\put(40,80){\circle*{4}}
		\put(31,87){{\Tiny ${\ }_{\lambda_2(1)}$}}
		\put(60,80){\circle*{4}}
		\put(55,87){{\Tiny ${\ }_{\lambda_3(1)}$}}
		\put(0,80){{\tiny $1$}}
		\end{picture}
		\caption{Illustration of \eqref{aub20}, \eqref{aub15}}\lb{fig1}
	\end{figure} 
	
\hfill$\Diamond$  \end{example}

\subsection{Continuous dependence of solutions to heat equation on thermal conductivity}\label{ssHeatEq} In this section we apply our general results to give a new proof that solutions to the linear homogeneous heat equation depend continuously on a certain physically relevant parameter present in both the operator and the boundary condition. The assertions of this type  have a long and distinguished history, and have been resolved even for quite general Wentzell boundary conditions. We refer the reader to \cite{CFGGGOR,CFGGR} where one can also find further literature. We did not attempt to cover the case of Wentzell boundary conditions anywhere in this paper but remark parenthetically
that it is an interesting open area to develop a version of the asymptotic  perturbation theory for operators equipped with this type of dynamical boundary conditions. At the moment, as in \cite{MR2215623}, we consider the following heat equation,
\begin{equation}\lb{aub2}
\begin{cases}
&u_\mathtt{t}(\mathtt{t},x)=\kappa \rho(x) \Delta_x u (\mathtt{t},x), x\in\Omega,
\mathtt{t}\ge0,\\
&-\kappa \frac{\partial  u}{\partial n}= u,\text{\ on\ }\partial \Omega,
\end{cases}
\end{equation}
describing the temperature $u$ of a material in the region $\Omega\subset \bbR^3$ with thermal conductivity $\kappa$ immersed in a surrounding medium of zero temperature. Here, $1/\rho(x)$ is the product of the density of the material times its heat capacity.  The continuous dependence of the temperature $u$ on the thermal conductivity $\kappa$ with respect to $L^2(\Omega)$ norm follows from Theorem \ref{thm5.7} proved below, which is a version of Theorem \ref{theorem6.2}. To sketch the argument, we consider the self-adjoint operator  $\cL_{\kappa}:=-\kappa\Delta$, $\cL_{\kappa}:\dom(\cL_{\kappa})\subset L^2(\Omega)\rightarrow L^2(\Omega)$ with $\dom(\cL_{\kappa})=\{u\in\cD^1(\Omega): -\kappa \gaN u= \gaD u\}$. Then by Trotter--Kato Approximation Theorem \cite[Theorem III.4.8]{EnNa00}, the family of semigroups $\{e^{-\mathtt{t} \rho \cL_{\kappa}}\}_{\mathtt{t}\ge0}$ is strongly continuous in $\kappa $ uniformly for $\mathtt{t}$ from compact subsets whenever $\kappa\mapsto         (\rho \cL_{\kappa}-\zeta)^{-1}$ is continuous as a mapping from $(0,+\infty)$  to $\cB(L^2(\Omega))$ for some $\zeta\not\in \spec(\cL_{\kappa})$ (we note that $\rho \cL_{\kappa}$ is not necessarily self-adjoint). The next theorem gives a rigorous argument for the required continuity of the resolvent in a slightly more general form. (In the next theorem, to keep up with notation used in the rest of the paper, we denote the parameter with respect to which the continuity is established by $t$, not by $\kappa$; this is not to be confused with notation $\mathtt{t}$ for time used in \eqref{aub2}). 

\begin{theorem}\lb{thm5.7} 
	Let $\Omega\subset \bbR^d$ be a bounded open set with $ C^{\infty}$-smooth boundary $\partial\Omega$. We assume that $t\mapsto \alpha_t$, $t\mapsto \beta_t$ are mappings in $C([0,1], L^{\infty}(\partial\Omega; \bbR))$ such that $\alpha_t^2(x)+\beta_t^2(x)\not=0$ for  $x\in\partial\Omega$, $t\in[0,1]$, and  $t\mapsto\rho_t$ is a mapping in $C([0,1], C(\overline{\Omega};\bbR))$ such that $\inf\{\rho_t(x): t\in[0,1], x\in\overline{ \Omega}\}>0$. Recall the differential expression $\cL$ from \eqref{1.15b} and define the following operator acting in $L^2(\Omega)$,
	\begin{align}
	&\cL_{t,\rho}u:=\rho_t\cL u, u\in\dom(\cL_{t, \rho}), \\
	&\dom(\cL_{t,\rho}):=\{u\in \cD^1(\Omega): \alpha_t\gaD u+\beta_t \gamma_{{}_{N,\cL}} u=0\}.
	\end{align}
	Then the operator $\cL_{t,\rho}$ is sectorial and the mapping $t\mapsto (\cL_{t, \rho}-\zeta)^{-1}$ lies in $C([0,1], \cB(L^2(\Omega)))$ for all $\zeta\in\bbC\setminus\spec(\cL_{t, \rho})$.
\end{theorem}

\begin{proof} To prove that $\cL_{t,\rho}$ is sectorial we have to show the existence of such $\theta\in(0,\frac{\pi}{2})$ and $M=M(\theta)>0$ that \[\zeta\in\bbC\setminus\spec(\cL_{t, \rho})\,\text{ and }\, \|(\cL_{t,\rho}-\zeta)^{-1}\|_{\cB(L^2(\Omega))}\le M|\zeta|^{-1}\] provided $\zeta\neq0$ and $|\arg\zeta|\in(\theta,\pi]$.
	First, we introduce a self-adjoint operator $\cL_t$ acting in $L^2(\Omega)$ and defined by
	$\cL_{t}u:=\cL u$ for $u\in\dom(\cL_{t}):=\dom(\cL_{t, \rho})$
	so that $\cL_{t,\rho}=\rho_t\cL_t$.
	Since $\cL_{t}$ is bounded from below we may assume without loss of generality that $\cL_{t}\ge0$ and, given a $\theta\in(0,\frac{\pi}{2})$, use the estimate
	\begin{equation}\label{estrxi}
	\|(\cL_t-\xi)^{-1}\|_{\cB(L^2(\Omega))}\leq (|\xi|\sin\theta)^{-1}
	\text{ for all $\xi\in\bbC\setminus\{0\}$ such that $|\arg\xi|\in(\theta,\pi]$.}
	\end{equation}
	Indeed,  \eqref{estrxi} follows from the estimate \[\|(\cL_t-\xi)^{-1}\|_{\cB(L^2(\Omega))}\leq |\Im \xi|^{-1}\le(|\xi|\sin\theta)^{-1}\]
	provided $|\arg\xi|\in(\theta,\frac{\pi}{2}]$ and
	\[\|(\cL_t-\xi)^{-1}\|_{\cB(L^2(\Omega))}=\big(\dist(\xi,\spec(\cL_{t}))\big)^{-1}\le|\xi|^{-1}\le(|\xi|\sin\theta)^{-1}\] provided $|\arg\xi|\in(\frac{\pi}{2},\pi]$.
	
	Throughout the rest of this proof we take all $\inf$'s and $\sup$'s over $(t,x)\in[0,1]\times\overline{\Omega}$.  We pick $\theta\in(0,\frac{\pi}{2})$ such that
	\begin{equation}\label{chofth}
	(1-\sin^2\theta)\sup\rho_t(x)<\inf\rho_t(x)
	\end{equation}
	and fix any $\zeta\in\bbC\setminus\{0\}$ such that $|\arg\zeta|\in(\theta,\pi]$.
	Using \eqref{chofth}, we can choose $\xi\in\bbC$
	such that $\arg\xi=\arg\zeta$ with $|\xi|$ that satisfies the inequality
	\begin{equation}\label{xizeta}
	(1-\sin^2\theta)\sup\rho_t(x)<|\zeta| |\xi|^{-1}<\inf\rho_t(x).
	\end{equation}
	Dividing this by $\rho_t(x)$ we infer
	\begin{equation}\label{xizeta2}
	\sup\big|(|\zeta| (|\xi|\rho_t(x))^{-1}-1)\big|\le\sin^2\theta.
	\end{equation}
	Since  $\xi\in\bbC\setminus\spec(\cL_{t})$ we have
	\begin{equation}\lb{aub3}
	\rho_t\cL_t-\zeta=\rho_t(\cL_t-\xi)\big(I-(\cL_t-\xi)^{-1}(\zeta\rho_t^{-1}-\xi)\big). \end{equation}
	Combining \eqref{estrxi} and \eqref{xizeta2}
	we infer
	\begin{align}
	\begin{split}
	\|(\cL_t-\xi)^{-1}(\zeta\rho_t^{-1}-\xi)\|_{\cB(L^2(\Omega))}&\leq (|\xi|\sin\theta)^{-1} \sup\big|e^{\bfi\arg\zeta}(|\zeta|\rho_t(x)^{-1}-|\xi|)\big|\\
	&\leq\sin\theta<1,
	\end{split}
	\end{align}
	which by \eqref{aub3} gives $\lambda\in\bbC\setminus \spec(\rho_t\cL_t)$ and, using  the second inequality in \eqref{xizeta}, the required resolvent estimate $\|(\cL_{t,\rho}-\zeta)^{-1}\|_{\cB(L^2(\Omega))}\le M|\zeta|^{-1}$. Thus, $\cL_{t,\rho}$ is sectorial.
	
	It is enough to prove continuity of the resolvent mapping at any $\zeta\in\bbR$ in the resolvent set of $\cL_{t,\rho}$. We note that if
	$\zeta\in\bbR\setminus\spec(\cL_{t,\rho})$ then $0\in\bbC\setminus\spec(\cL_t-\rho_t^{-1}\zeta)$ and the identity $(\rho_t\cL_t-\zeta)^{-1}=(\cL_t-\rho_t^{-1}\zeta)^{-1}\rho_t^{-1}$ holds. Since the map $t\mapsto\rho_t^{-1}$ is continuous, it remains to prove continuity of the map $t\mapsto(\cL_t-\rho_t^{-1}\zeta)^{-1}$, that is, of the resolvent of the operator $H_t=\cL_t-\rho_t^{-1}\zeta$ at zero. This follows from Theorem \ref{prop5.1} with $\cA_t=\cL_t$, $V_t=-\rho_t^{-1}\zeta$,
	$\tr:=(\gaD, \gamma_{{}_{N,\cL}})\in \cB (\cD^1(\Omega), \bndr)$ and
	\begin{align}
	Z_{t,s}:=[W(\alpha_{t}, \beta_{t})]^*( \alpha_{t}\beta_{s}-\beta_{t}\alpha_{s})[W(\alpha_{s}, \beta_{s})]\rightarrow 0, s\rightarrow t.
	\end{align}
	To justify the use of Theorem \ref{prop5.1}, we note that Hypothesis \ref{hyp1.3i2} in the theorem is satisfied, that is, $(\cL_t-\bfi)^{-1}=\cO(1)$ as $t\rightarrow s$ in $\cB(L^2(\Omega), \cD^1(\Omega))$. The proof of this assertion is similar to that of \eqref{5.7} (one imposes Robin boundary condition with $\Theta_t(x):=-\alpha_t(x)\beta_t^{-1}(x)$ on the portion of the boundary where $\beta_t^{-1}(x)\not=0$ and the Dirichlet condition elsewhere).
\end{proof}

\subsection{The Hadamard formula for star-shaped domains}\label{HFSSD} In this section we show how 
to use Theorem \ref{prop5.1}
 to derive the classical Hadamard formula for the  Schr\"odinger operators subject to the Dirichlet boundary condition on variable star-shaped domains. 

Let $\Omega\subset\bbR^n$ be a smooth star-shaped domain centered at zero and $\Omega_t=\{tx: x\in\Omega\}$ be its variation for $t\in(0,1]$. We consider a smooth ($N\times N$)-matrix potential $V=V(x)$ for $x\in\overline{\Omega}$ taking symmetric values. Suppose that $\mu\in\bbR$ is such that $\dim\ker(-\Delta_{D, \Omega}+V-\mu)=m\geq 1$, where $-\Delta_{D, \Omega}$ denotes the Dirichlet Laplacian acting in $L^2(\Omega)$. We claim that there exists a choice of orthonormal eigenfunctions $\{u_j\}_{j=1}^m\subset (-\Delta_{D,\Omega}+V-\mu)$ and a labeling of the eigenvalues $\{\mu_j(t)\}_{j=1}^m$ of $-\Delta_{D, \Omega_t}+V\upharpoonright_{\Omega_t}$, for $t$ near $1$, such that $\mu_j(1)=\mu$ for each $j$, and that the following classical Rayleigh--Hadamard--Rellich formula holds,
cf.\ \cite[Chapter 5]{Henry},
\begin{equation}\label{clHF}
\dot \mu_j(1)=-\int_{\partial\Omega}(\nu\cdot x)(\nu\cdot\nabla u_j)^2 {\rm d} x, 1\leq j\leq m.
\end{equation}
Rescaling $\Omega\ni t\mapsto tx\in \Omega_t$
of the operator $\big(-\Delta_{D,\Omega_t}+V\big)\big|_{\Omega_t}$  back to $\Omega$ yields a one-parameter family of self-adjoint operators  $H_t=-\Delta_{D,\Omega}+t^2V(tx)$, $t\in(0,1]$  acting in the fixed space $L^2(\Omega)$. This family of operators fits the framework of Theorem \ref{prop5.1} with 
$\cA_t\equiv-\Delta_{\Omega}$, $V_t(x)=t^2V(tx)$, $\tr_t=[\gaD, -t^{-1}\Phi\gaN]^\top$, cf \eqref{1.60}, $t_0=1$, $\lambda(t_0)=\mu$ and $Q_t$ given by the $t$-independent projection onto the Dirichlet subspace   $\{(0,g): g\in H^{1/2}(\partial\Omega)\}$ for all $t$. All assumptions of Theorem \ref{prop5.1} are clearly satisfied in the present setting. By the theorem there exists a choice of orthonormal eigenfunctions $\{u_j\}_{j=1}^m\subset \ker (-\Delta_{D,\Omega}+V-\mu)$ and a labeling of the eigenvalues $\{\lambda_j(t)\}_{j=1}^m$ of $H_t$, for $t$ near $1$, such that
\begin{align}
\begin{split}
\label{l1}
\dot\lambda_j(1)&=\left\langle \frac{{\rm d} (t^2V(tx))}{{\rm d} t}\Big|_{t=1}u_j,u_j\right\rangle_{L^2(\Omega)}\\
&=2\langle Vu_j,u_j\rangle_{L^2(\Omega)}+\langle(\nabla V\cdot x)u_j,u_j \rangle_{L^2(\Omega)}, 1\leq j\leq m.
\end{split}
\end{align} 
By the same rescaling as above, the eigenvalues $\lambda_j(t)$ uniquely determine the eigenvalues $\mu_j(t)$ for $t$ near $1$, and one has $\lambda_j(t)=t^2\mu_j(t)$. Our next objective is to use this identity together with \eqref{l1} to derive \eqref{clHF}.

We pause to consider the case of the Laplace operator with no potential. If $V\equiv 0$ then the proof is essentially completed as $H_t$ does not depend on $t$ and $0=\dot\lambda_j(1)=2\mu_j(1)+\dot\mu_j(1)$. This yields \eqref{clHF} by the celebrated Rellich formula \cite{Rell} expressing the eigenvalues $\lambda_j(1)=\mu_j(1)$ of the Dirichlet Laplacian via the Neumann boundary values of the respective eigenfunctions (this formula in turn easily follows from the Pokhozaev--Rellich identity, see, e.g., \cite[p.201]{Bandle}, \cite[p.237]{Kes}, and formula \eqref{l4} below).

Returning to the general case of nonzero potential, to derive \eqref{clHF} from \eqref{l1} we will follow the strategy of \cite[Lemma 5.5]{CJLS}. Let us fix $j$ and denote, for brevity, $u:=u_j$ and $\lambda(t):=\lambda_j(t)$, $\mu(t)=\mu_j(t)$.  First, integration by parts for $\Omega\subseteq\bbR^n$ yields
\begin{equation}\label{l3}
\langle(\nabla V\cdot x)u,u \rangle_{L^2(\Omega)}
=-\langle Vu ,2(\nabla u\cdot x)+nu\rangle_{L^2(\Omega)}
\text{ and } \langle u,\nabla 
u\cdot x\rangle_{L^2(\Omega)}=-{n}/{2}.
\end{equation}
Using $-\Delta u+Vu=\lambda(1)u$ and replacing $Vu$ by $\Delta u+\lambda(1)u$ in \eqref{l1} and \eqref{l3}, 
a short calculation gives
\begin{equation}\label{l2}
\dot{\mu}(1)=\dot\lambda(1)-2\lambda(1)=(2-n)\langle\Delta u,u\rangle_{L^2(\Omega)}-2\langle\Delta u,\nabla u\cdot x\rangle_{L^2(\Omega)}.
\end{equation}
The standard Rellich's identity, see, e.g., \cite[p.201]{Bandle},  yields
\begin{align}\label{l4}
\langle\Delta u,\nabla u\cdot x\rangle_{L^2(\Omega)}
&=\int_{\partial\Omega}\big((\nu\cdot\nabla u)(x\cdot\nabla u)-\frac12(x\cdot\nu)\|\nabla u\|^2\big){\rm d}x\\
&\hskip4cm+\frac{n-2}{2}\int_\Omega\|\nabla u\|^2{\rm d}x.\nonumber
\end{align}
Since $u$ satisfies the Dirichlet condition, $\partial\Omega$ is a level curve, and thus $\nabla u$ and $\nu$ are parallel, that is, $\nabla u=(\nu\cdot \nabla u)\nu$. Using all this in \eqref{l2} yields \eqref{clHF} because
\[\dot\mu(1)=\int_{\partial\Omega}\big(-2(\nu\cdot\nabla u)(x\cdot\nabla u)
+(x\cdot\nu)\|\nabla u\|^2\big){\rm d}x=
-\int_{\partial\Omega}(\nu\cdot\nabla u)^2(\nu\cdot x){\rm d} x.\]

\subsection{Maslov crossing form for elliptic operators}\lb{section5.4} In this section  we continue the discussion began in Section \ref{mastriples} on the relation between the Maslov crossing form and the slopes of the eigenvalue curves bifurcating from a multiple eigenvalue of the unperturbed elliptic operator. Here,
we assume the setting of Theorem \ref{theorem6.2} and obtain a version of formula \eqref{formdr} for the Robin-type elliptic operators $\cL_t$, see Proposition \ref{HadmardSimple5} below. 
For $\lambda\in\bbR$ we let
\begin{align}
\begin{split}
&{\cK}_{\lambda}:=\tr\Big(\Big\{u\in H^1(\Omega):\sum_{j,k=1}^{n}\langle \mathtt{a}_{jk}\partial_ku,\partial_j \varphi\rangle_{L^2(\Omega)}+\sum_{j=1}^{n}\langle \mathtt{a}_j\partial_ju,\varphi\rangle_{L^2(\Omega)}\\
&\quad+\sum_{j=1}^{n}\langle u,\mathtt{a}_j\partial_j\varphi\rangle_{L^2(\Omega)}+\langle vu-\lambda u,\varphi\rangle_{L^2(\Omega)}=0,\ \varphi\in H^1_0(\Omega)\Big\}\Big), \no
\end{split}
\end{align}
where the trace  operator $\tr=[\Gamma_0,\Gamma_1]^\top$is as in \eqref{aub11}. This is a ``weak'' version of the set $\bbK_{\lambda,t}$ from Section \ref{mastriples}. The mapping $\lambda\mapsto \cK_{\lambda}$ is in $C^1(\bbR, \Lambda(H^{1/2}(\partial\Omega)\times H^{1/2}(\partial\Omega)))$ by \cite[Proposition 3.5]{CJM1}.  

Let $t\mapsto \cG_t:=\{(f, -\Theta_t f): f\in H^{1/2}(\partial\Omega)\}$, then for $t_0\in[0,1]$ there is an interval $\cI\subset [0,1]$ centered at $t_0$ and a family of operators $t\mapsto \cM_{t}, t\in\cI$, which is in $C^1\big(\cI,\cB(\cG_{t_0},\cG_{t_0}^\bot)\big)$ with  $\cM_{t_0}=0$ and 
\begin{equation}
\cG_t=\big\{\bq+\cM_{t}\bq\ \big|\, \bq\in\cG_{t_0}\big\}, t\in\cI,
\end{equation}
see,  e.g., \cite[Lemma 3.8]{CJLS}. In other words, $\cG_t$ can be written locally as the graph of the operator $\cM_t$, which is a replacement of $\cM_{\lambda,t}$ from Section \ref{mastriples}. We say that $(\lambda,t_0)$ is a conjugate point if $\cK_{\lambda}\cap \cG_{t_0}\not=\{0\}$ or, equivalently,  if $\ker(\cL_{t_0}-\lambda)\not=\{0\}$. 

We recall $\lambda(t_0)\in\spec_{\rm disc}(\cL_{t_0})$ from Theorem \ref{theorem6.2} and let $\lambda:=\lambda(t_0)$. Then $(\lambda, t_0)$ is a conjugate point at which the Maslov crossing form $\mathfrak{m}_{t_0}$ for the path $t\mapsto \cK_{\lambda}\oplus\cG_{t}$ relative to the diagonal subspace $\mathfrak{D}=\{\bp=(p,p): p\in H^{1/2}(\partial\Omega)\times H^{1/2}(\partial\Omega))\}$ 
is defined by the formula
\begin{equation}\label{dfnMcfnew}
\mathfrak{m}_{t_0}(\bq,\bp):=\frac{\bd}{\bd t}\big|_{t=t_0}\widehat{\omega} (\bq,\cM_{t}\bp)=\widehat{\omega} (\bq,\dot{\cM}_{t_0}\bp),\ \bp,\bq\in (\cK_{\lambda}\oplus\cG_{t_0})\cap\mathfrak{D},
\end{equation}
where $\widehat{\omega} ={\omega} \oplus(-{\omega} )$ and 
$\displaystyle{\dot{\cM}_{t_0}=\frac{\bd}{\bd t}\cM_{t}\big|_{t=t_0}}$. We stress that the pair of Lagrangian subspaces  $\big(\cK_{\lambda}, \cG_{t_0}\big)$ is Fredholm since $\lambda=\lambda(t_0)\not\in \spec_{\rm{ess}}(\cL_{t_0})$, see \cite[Theorem 3.2]{LS1}. Hence, $\dim\big((\cK_{\lambda}\oplus\cG_{t_0})\cap\mathfrak{D}\big)<\infty$ and $\mathfrak{m}_{t_0}$ is a finite dimensional bilinear form. In fact, 
the pair of Lagrangian subspaces $\big(\cK_{\lambda}, \cG_{t}\big)$ is Fredholm for $t$ near $t_0$ due to continuity of the path of the resolvent operators $t\mapsto (\cL_t-\bfi)^{-1}$. 


\begin{proposition}\label{HadmardSimple5} Let $\lambda(t_0)$, $\{\lambda_j(t)\}_{j=1}^m$ and $\{u_j\}_{j=1}^m$ be as in Theorem \ref{theorem6.2}, and denote $\bq_j:=(\tr u_j, \tr u_j)$. Then $\bq_j\in (\cK_{\lambda(t_0)}\oplus\cG_{t_0})\cap\mathfrak{D}$ and
	\begin{equation}\label{formdr5}
	\dot\lambda_j(t_0)=\mathfrak{m}_{t_0}(\bq_j,\bq_j),\ 1\leq j\leq m,
	\end{equation}
	where $\mathfrak{m}_{t_0}$ is the Maslov crossing form introduced in \eqref{dfnMcfnew}.
\end{proposition}

\begin{proof}
	The inclusion $\bq_j\in (\cK_{\lambda(t_0)}\oplus\cG_{t_0})\cap\mathfrak{D}$  holds since $u_j$ is an eigenfunction of $\cL_{t_0}$ corresponding to the eigenvalue $\lambda(t_0)$. For a fixed $j$
we abbreviate $\bq:=\bq_j=\tr u_j$ and introduce $g_t\in H^{1/2}(\partial\Omega)\times H^{1/2}(\partial\Omega)$ as in \eqref{Dfngt}
but with $\cM_{\lambda,t}$ replaced by $\cM_t$. In particular, $g_{t_0}=\tr u_j$ because $\cM_{t_0}=0$.
Since $g_t=Q_tg_t$ where $Q_t$ is the orthogonal projection onto $\cG_t$, we have \[\dot{g}_{{t_0}}=\dot{Q}_{{t_0}}g_{t_0}+Q_{{t_0}}\dot{g}_{{t_0}}=\dot{Q}_{{t_0}}\tr u_j+Q_{{t_0}}\dot{g}_{{t_0}}.\] This and  that $\ran(Q_{{t_0}})$ is Lagrangian yields, as in \eqref{2.54},
	\begin{equation}
	\omega(\tr u_j,\,\dot{g}_{{t_0}})=\omega(\tr u_j,\dot{Q}_{{t_0}}\tr_{t_0}u_j).
	\end{equation}
As in \eqref{aub34},  by definition of $\mathfrak{m}_{t_0}$ this implies
	 	\begin{equation}
	\mathfrak{m}_{{t_0}}(\bq_j,\bq_j)=-\omega(\tr u_j,\,\dot{g}_{{t_0}})=-\omega(\tr u_j,\dot{Q}_{{t_0}}\tr u_j)=\omega(\dot{Q}_{{t_0}}\tr u_j, \tr u_j).
	\end{equation}
	By formula \eqref{cLthetder} in Theorem \ref{theorem6.2} we have  
$\dot{\lambda}_j(t_0)=-\langle\dot{\Theta}_{t_0}\gaD u_j,\gaD u_j\rangle_{L^2(\partial\Omega)}$. Thus,	
it remains to show that \[\omega(\dot{Q}_{t_0}\tr u_j, \tr u_j )=-\langle\dot{\Theta}_{t_0} \gaD u_j, \gaD u_j\rangle_{L^2(\partial\Omega)}.\]
	The latter assertion follows from \eqref{3.40new} with $\phi_j=-\gaD u_j$ and $X_t=\Theta_t$, $Y_t=I$ as \[\cG_t=\gr(-\Theta_t)=\ker([X_t, Y_t])\] with this choice of $X_t$ and $Y_t$.
\end{proof}

\begin{remark}\label{rem:HadmardSimple5}
As discussed in Remark \ref{masindexspflow},  formula \eqref{formdr5}
relating the derivative of the eigenvalues of the elliptic operators $\cL_t$ with respect to the parameter $t$ and the value of the (Maslov) crossing form for the flow $t\mapsto\cK_{\lambda(t)}\oplus\cG_t$ of  Lagrangian planes could be viewed as an infinitesimal version of the fundamental relation between the spectral flow and the Maslov index. Indeed, as in Remark \ref{masindexspflow},  formula \eqref{formdr5} implies relation \eqref{MISF} with $H_t$ replaced by $\cL_t$ and $\Upsilon_{\lambda,t}$ replaced by $\cK_{\lambda(t)}\oplus\cG_t$.\hfill$\Diamond$  \end{remark}

\begin{example}\label{rem:FrMorse} We will now briefly  return to the Robin eigenvalue problem \eqref{eq:RobEV} related to the Friedlander Inequalities but at once for the general elliptic operator $\cL$  described in Subsection \ref{SS5.1}. 
We recall that for $\lambda\notin\spec(\cL_D)$ the Dirichlet-to-Neumann operator $M_{D,N}(\lambda)$ is defined by $f\mapsto-\gaN u$ (in the relevant papers \cite{CJM2,Fri91}, $M_{D,N}$ is defined by $f\mapsto\gaN u$) where $u$ is the solution to $\cL u=\lambda u$, $\gaD u=f$. It is easy to see that \eqref{eq:RobEV} has a nontrivial solution if and only if $\mu=\cot(\frac{\pi}{2}t)$ is an eigenvalue of  $M_{DN}(\lambda)$. Combining Remarks \ref{masindexspflow}, \ref{rem:HadmardSimple5}
and Example \ref{FREX} with Proposition \ref{HadmardSimple5} can be used to show the following formula relating the spectral counting functions 
of the Dirichlet and Neumann realizations $\cL_D$ and $\cL_N$ and the Dirichlet-to-Neuman map $M_{D,N}(0)$, 
\begin{equation}\label{rem:FrMorse1}\begin{split}
\#\{\lambda\in\Sp(\cL_N): \lambda<0\}&-\#\{\lambda\in\Sp(\cL_D): \lambda<0\}\\&=\#\{\mu\in\Sp(M_{D,N}(0)): \mu\geq0\},
\end{split}
\end{equation} 
see \cite{Fri91} and,
specifically,  \cite[Theorem 3]{CJM2} and the literature therein (in \cite{CJM2,Fri91} the RHS of \eqref{rem:FrMorse1} is given by the number of {\it negative} eigenvalues of $M_{D, N}(\lambda)$, this is due to sign discrepancy in the definition of $M_{D, N}(\lambda)$).
We omit details and just mention that the monotonicity of the eigenvalue curves $\lambda_k(t)$, $k=1,2,\dots$, established in Example \ref{FREX} and formula \eqref{formdr5} show that the Maslov crossing form is sign definite at each conjugate point on the vertical line through $\lambda$ 
when $t$ changes from $0$ to $1$ (Figure \ref{fig1} serves as a schematic illustration of this assertion). By a standard calculation, see, e.g.,  Step 1 in the proof of \cite[Theorem 3.3]{LS1}, the Maslov crossing form is also sign definite at each conjugate point on the horizontal lines through $t=0$ and $t=1$ when $\lambda$ is changing.
These two properties are sometimes referred to as the monotonicity of the Maslov index. Thus, cf.\ Remark \ref{masindexspflow}, the spectral flow through zero given by the LHS of \eqref{rem:FrMorse1} is equal to the Maslov index along the vertical line through $\lambda$ which, in turn, is equal to the RHS.
\hfill$\Diamond$  \end{example}

\section{Symplectic resolvent difference formulas for dual pairs}\label{sec:dualpairs}

In this section we give a  generalization of the resolvent difference formula \eqref{5.14} to the case of boundary triplets for an adjoint pair $A, \tA$, see, e.g., \cite{AGW, BGW, BMNW} and the literature cited therein. The theory of adjoint pairs goes back to \cite{LySt, MR0052655}, see also \cite{A12, BEHRNDT20181808, MHMNW, MM02}. It allows one to describe non-selfadjoint extensions for an adjoint pair of densely defined closed (but not neceserely symmetric) operators. A typical example of the adjoint pair, see, e.g., \cite{BGW, BMNW}, is furnished by a non-symmetric elliptic second order partial differential operator and its formal adjoint; this example is discussed in detail in the end of this section.

We follow \cite{BMNW} to recall the definition of the adjoint pair and its boundary triplet. Let $A,\tA$ be closed densely defined operators on a Hilbert space $\cH$ forming an adjoint pair, that is, we assume that $\tA\subseteq A^*$ and $A\subseteq(\tA)^*$. We denote by $\cH_+$, respectively, $\wti{\cH}_+$ the domain $\dom(A^*)$, respectively, $\dom((\tA)^*)$ equipped with the graph-scalar product and graph norm for $A^*$, respectively, $(\tA)^*$, cf.\ Section \ref{KreinFormulas}. Let $\mathfrak{H}$ and $\mathfrak{K}$ be some ``boundary'' Hilbert spaces and 
\begin{equation}\label{Gammas}
\Gamma_0:\wti{\cH}_+\to\mathfrak{H},\quad \Gamma_1:\wti{\cH}_+\to\mathfrak{K}, \quad \wti{\Gamma}_0: {\cH}_+\to\mathfrak{K}, \quad \wti{\Gamma}_1: {\cH}_+\to\mathfrak{H}
\end{equation}
be some ``boundary trace operators". The collection $\{\mathfrak{H},\mathfrak{K}, \Gamma_0, \Gamma_1,\wti{\Gamma}_0, \wti{\Gamma}_1\}$ is called a
{\em boundary triplet for the adjoint pair $A,\tA$} when the following hypothesis is satisfied.

\begin{hypothesis} Suppose that $A,\tA$ is an adjoint pair of densely defined closed operators such that $\tA\subseteq A^*$ and $A\subseteq(\tA)^*$. Consider linear operators, called the {\em trace operators},
	\begin{equation}\label{APeq1}
	\tr:=[\Gamma_0,\Gamma_1]^\top:\wti{\cH}_+\to\mathfrak{H}\times\mathfrak{K},\, 
	\wti{\tr}:=[\wti{\Gamma}_0,\wti{\Gamma}_1]^\top:\cH_+\to\mathfrak{K}\times\mathfrak{H}.\, 
	\end{equation}
	Assume that the operators $\tr$ and $\wti{\tr}$ are surjective and satisfy
	\begin{equation}\label{asgf}
	\langle(\tA)^* u, v\rangle_\cH-\langle u,A^*v\rangle_\cH=
	\langle\Gamma_1u,\wti{\Gamma}_0v\rangle_\mathfrak{K}
	-\langle\Gamma_0u,\wti{\Gamma}_1v\rangle_\mathfrak{H},
	\end{equation}
for all $u\in\wti{\cH}_+$ and $v\in\cH_+$.
\end{hypothesis}

The existence of a boundary triplet for every adjoint pair $A,\tA$ was proved in \cite{LySt}, where, in addition, it was shown that
\begin{equation}
\dom(A)=\dom((\tA)^*)\cap\ker\Gamma_0\cap\ker\Gamma_1,\,
\dom(\tA)=\dom(A^*)\cap\ker\wti{\Gamma}_0\cap\ker\wti{\Gamma}_1.
\end{equation} 
It is well-known that the operators $T, \wti T$ are bounded, cf. \cite{MM02}, \cite[Lemma 14.13]{Schm}.

The following resolvent difference formula is a direct generalization of Theorem \ref{thm1.7}. It gives the difference of the resolvent operators of any two (not necessarily sel-adjoint) extensions of the operator $A$ that are parts of $(\tA)^*$. 

\begin{theorem}\label{APthm1.7} Let
	$\{\mathfrak{H},\mathfrak{K}, \Gamma_0, \Gamma_1,\wti{\Gamma}_0, \wti{\Gamma}_1\}$ be a boundary triplet for an adjoint pair $A,\tA$, and let $\cA_j$ for $j=1,2$ be any two closed extensions of $A$ acting in $\cH$ and satisfying $A\subseteq\cA_j\subseteq(\tA)^*$. Suppose that $\zeta\in\bbC\setminus(\spec(\cA_1)\cup\spec(\cA_2))$ and denote $R_j(\zeta):=(\cA_j-\zeta)^{-1}$ for $j=1,2$. Then one has
	\begin{align}\label{AP5.14}
	R_2(\zeta)-R_1(\zeta)&=\big(\wti{\Gamma}_0R_2^*(\zeta)\big)^*\Gamma_1R_1(\zeta)-
	\big(\wti{\Gamma}_1R_2^*(\zeta)\big)^*\Gamma_0R_1(\zeta),\\
	\label{AP5.14a}
	R_2(\zeta)-R_1(\zeta)&=\big(\wti{\tr}R_2^*(\zeta)\big)^*Q_2JQ_1\big(\tr R_1(\zeta)\big),
	\end{align}
	where $R_2^*(\zeta)=((\cA_2)^*-\overline{\zeta})^{-1}$, 
	the operator $\wti{\tr}R_2^*(\zeta)=\big(\wti{\Gamma}_0R_2^*(\zeta),\wti{\Gamma}_1R_2^*(\zeta)\big)$ is considered as an operator in $\cB(\cH,\mathfrak{K}\times\mathfrak{H})$ and the adjoint operators in \eqref{AP5.14}, \eqref{AP5.14a} are defined correspondingly, $Q_1$, respectively, $Q_2$ denotes the orthogonal projection onto $\overline{\tr(\dom(\cA_1))}$ in the space $\mathfrak{H}\times\mathfrak{K}$, respectively, onto $\overline{\wti{\tr}(\dom((\cA_2)^*))}$ in the space $\mathfrak{K}\times\mathfrak{H}$, and the operator $J$ maps a pair $(f, g)$ from $\mathfrak{H}\times\mathfrak{K}$ into the pair
	$(g,-f)$ from $\mathfrak{K}\times\mathfrak{H}$.
\end{theorem}
\begin{proof} The inclusion $A\subseteq\cA_j\subseteq(\tA)^*$  yields
	$\tA\subseteq(\cA_j)^*\subseteq A^*$ for $j=1,2$ \cite[Section III.5.5]{K80}. The operator $R_2^*(\zeta)\in\cB(\cH)$ is also  bounded from $\cH$ onto $\dom((\cA_2)^*)\subseteq\cH_+=\dom(A^*)$, cf.\ \eqref{1.13}. Thus, the operator $\wti{\tr}R_2^*(\zeta)$ is well defined, and, analogously, the operator 
	$\tr R_1(\zeta)$ is well defined. Moreover, for all $u,v\in\cH$ one has
	\begin{equation}\label{domaa}
	(A^*-\overline{\zeta})R_2^*(\zeta)v=(\cA_2-\zeta)^*R_2^*(\zeta)v=v, \,
	\big((\tA)^*-\zeta\big)R_1(\zeta)u=(\cA_1-\zeta)R_1(\zeta)u=u.\end{equation}
	We also have $Q_2 \wti{\tr}R_2^*(\zeta)=\wti{\tr}R_2^*(\zeta)$ and $Q_1 \tr R_1(\zeta)=\tr R_1(\zeta)$ by the definition of the orthogonal projections $Q_2$ and $Q_1$. Thus, \eqref{AP5.14a} is just a reformulation of \eqref{AP5.14}. For the proof of \eqref{AP5.14}, we use \eqref{asgf} and \eqref{domaa} to write
	\begin{align*}
	\langle\big(R_2(\zeta)&-R_1(\zeta)\big)u,v\rangle_{\cH}=
	\langle R_2(\zeta)u-R_1(\zeta)u,(\cA_2-\overline{\zeta})^*R_2^*(\zeta)v\rangle_{\cH}\\
	&=\langle (\cA_2-\zeta)R_2(\zeta)u,R_2^*(\zeta)v\rangle_{\cH}-
	\langle R_1(\zeta)u,(A^*-\overline{\zeta})R_2^*(\zeta)v\rangle_{\cH}\\
	&=\langle u,R_2^*(\zeta)v\rangle_{\cH}-
	\langle \big((\tA)^*-\zeta\big)R_1(\zeta)u,R_2^*(\zeta)v\rangle_{\cH}\\
	&\qquad+\langle\Gamma_1R_1(\zeta)u,\wti{\Gamma}_0R_2^*(\zeta)v\rangle_\mathfrak{K}-\langle\Gamma_0R_1(\zeta)u,\wti{\Gamma}_1R_2^*(\zeta)v\rangle_\mathfrak{H}\\
	&=\langle\big(\wti{\Gamma}_0R_2^*(\zeta)\big)^*\Gamma_1R_1(\zeta)u,v\rangle_\cH
	-\langle\big(\wti{\Gamma}_1R_2^*(\zeta)\big)^*\Gamma_0R_1(\zeta)u,v\rangle_\cH,
	\end{align*} 
	for all $u,v\in\cH$, yielding \eqref{AP5.14}.
\end{proof}
In particular, for $j=1,2$,  given an operator 
$\Psi_j\in\cB(\mathfrak{H},\mathfrak{K})$ (not necessarily self-adjoint), we consider in $\cH$ the extension $\cA_j$ of $A$ satisfying $A\subseteq\cA_j\subseteq(\tA)^*$ and defined by the formulas
\[\cA_ju=(\tA)^*u\, \text{ for }\, u\in\dom(\cA_j):=\{u\in\wti{\cH}_+: \Gamma_1u=\Psi_j\Gamma_0u\}, \quad j=1,2.\]
\begin{corollary}\label{cor6.3} Under assumptions in Theorem \ref{APthm1.7} one has
	\begin{equation}
	R_2(\zeta)-R_1(\zeta)=\big(\wti{\Gamma}_0R_2^*(\zeta)\big)^*\big(\Psi_1-\Psi_2\big)\Gamma_0R_1(\zeta).
	\end{equation}
\end{corollary}
\begin{proof}\sel{Let us provide an elementary proof for the sake of completeness.}
	We claim that $\tA\subset(\cA_j)^*\subset A^*$ satisfies
	$\dom((\cA_j)^*)=\{v\in \cH_+: \wti{\Gamma}_1v=(\Psi_j)^*\wti{\Gamma}_0v\}$. Indeed, to see the inclusion ``$\subseteq$'' we note that for $u\in\dom(\cA_j)$ and $v\in\dom((\cA_j)^*)$ by \eqref{asgf} 
	\begin{align*}
	0&=\langle\cA_ju,v\rangle_\cH-\langle u, \cA_j^*v\rangle_\cH=\langle(\tA)^* u, v\rangle_\cH-\langle u,A^*v\rangle_\cH\\&=
	\langle\Gamma_1u,\wti{\Gamma}_0v\rangle_\mathfrak{K}
	-\langle\Gamma_0u,\wti{\Gamma}_1v\rangle_\mathfrak{H}
	=\langle\Psi_j\Gamma_0u,\wti{\Gamma}_0v\rangle_\mathfrak{K}
	-\langle\Gamma_0u,\wti{\Gamma}_1v\rangle_\mathfrak{H}
	\end{align*}
	yielding $\wti{\Gamma}_1v=(\Psi_j)^*\wti{\Gamma}_0v$ as $\Gamma_0$ is surjective, while the opposite inclusion follows by running this chain of equalities backwards, thus proving the claim. Next, we note that $\Gamma_1R_1(\zeta)=\Psi_1\Gamma_0R_1(\zeta)$ and 
	$\wti{\Gamma}_1R_2^*(\zeta)=(\Psi_2)^*\wti{\Gamma}_0R_2^*(\zeta)$ since the resolvents map $\cH$ into  the domains of respective operators.
	Now \eqref{AP5.14} yelds
	\begin{align*}
	R_2(\zeta)-R_1(\zeta)&=\big(\wti{\Gamma}_0R_2^*(\zeta)\big)^*\Gamma_1R_1(\zeta)-
	\big(\wti{\Gamma}_1R_2^*(\zeta)\big)^*\Gamma_0R_1(\zeta)\\&=\big(\wti{\Gamma}_0R_2^*(\zeta)\big)^*\Psi_1\Gamma_0R_1(\zeta)-
	\big((\Psi_2)^*\wti{\Gamma}_0R_2^*(\zeta)\big)^*\Gamma_0R_1(\zeta)\\&=
	\big(\wti{\Gamma}_0R_2^*(\zeta)\big)^*\big(\Psi_1-\Psi_2\big)\Gamma_0R_1(\zeta),
	\end{align*}
	as required.
\end{proof}

\sel{\begin{remark}\lb{rem6.4}
We note that both the Weyl function and the $\gamma-$field for an adjoint pair was originally introduced and studied in \cite{MM97}, where the Krein type formula written in terms of these objects was derived for the first time. 
\end{remark}}
 We conclude this section with an example of the abstract boundary triplet for the adjoint pair formed by an elliptic second order partial differential operator and its formal adjoint, cf.\, \cite{Behrndt_2020, BL, BM, BGW, BMNW}. In particular, resolvent difference formulas in Theorem \ref{APthm1.7} and Corollary \ref{cor6.3} hold for these operators. The discussion below regarding the boundary traces is based on the paper \cite{Gr}, see also \cite{ BMNW}, dealing with general elliptic operators on smooth domains and is related to the material in Appendix \ref{appA} taken from
\cite{GM08,GM10} where the case of the Laplacian on Lipschitz domains was considered. \sel{As we noted in the preamble of Section \ref{ssLapLip} the work of M.I. Vi\v{s}ik \cite{MR0052655, Vi} and  G. Grubb \cite{Gr}  are foundational in the theory of self-adjoint extensions of elliptic differential operators. Our next goal is to employ their construction of dual pair to illustrate abstract setting discussed above. 
}

We assume throughout that $\Omega\subset\bbR^n$ is a bounded domain with $C^\infty$-boundary, and consider the following
mutually formally adjoint differential expressions, cf.\ \eqref{1.15b},
\begin{align}
\begin{split}
\cL:&=-\sum_{j,k=1}^{n} \partial_j \mathtt{a}_{jk}\partial_k + \sum_{j=1}^{n}\big(\mathtt{a}_j\partial_j -\partial_j {\mathtt{b}_j}\big)+\mathtt{q}\\
&\hspace{3cm}= -\div(\mathtt{A}\nabla)+ \mathtt{a}\cdot \nabla -\nabla \cdot \mathtt{b} +\mathtt{q},\lb{1.15b-new}\\
\end{split}
\end{align}

\begin{align}
\begin{split}
\wti{\cL}:&=-\sum_{j,k=1}^{n} \partial_k\overline{\mathtt{a}}_{jk}\partial_j + \sum_{j=1}^{n}\big(\overline{\mathtt{b}}_j\partial_j -\partial_j \overline{\mathtt{a}}_j\big)+\overline{\mathtt{q}}\\
&\hspace{3cm}= -\div(\mathtt{A}^*\nabla)+ \mathtt{b}^*\cdot \nabla -\nabla \cdot \mathtt{a}^* +\overline{\mathtt{q}},\lb{1.15c-new}
\end{split}
\end{align}
where the bar stands for complex conjugation, with coefficients $\mathtt{A}=\{\mathtt{a}_{ij}\}_{1\leq i, j\leq n}$, $\mathtt{a}:=\{\mathtt{a}_j\}_{1\leq j\leq n}$,
$\mathtt{b}:=\{\mathtt{b}_j\}_{1\leq j\leq n}$ and $\mathtt{q}$ satisfying
$\mathtt{a}_{jk}, \mathtt{a}_j, \mathtt{b}_j, \mathtt{q}\in C^{\infty}(\overline{\Omega};\bbC)$, $1\leq j,k\leq n$. In addition,  we assume the ellipticity condition: For $c>0$ we have
\begin{align}
& \Re\Big(\sum\limits_{j,k=1}^n\mathtt{a}_{jk}(x)\xi_k\overline{\xi_j}\Big) \geq c \sum\limits_{j=1}^n|\xi_j|^2, x\in\overline{\Omega}, \xi=\{\xi_j\}_{j=1}^n\in\C^n.
\end{align}

As in \eqref{1.15c}, we associate with $\cL$ from
\eqref{1.15b-new} the space $\cD^s(\Omega)$ equipped with the $\cL$-graph norm \eqref{1.15d}, and with $\wti{\cL}$ from
\eqref{1.15c-new}
the space 
$\wti{\cD}^{s}(\Omega):=\{u\in H^s(\Omega): \wti{\cL} u\in L^2(\Omega)\}$
equipped with the  $\wti{\cL}$-graph norm
$\|u\|_{\wti{s}}:=\big(\|u\|_{H^s({\Omega} )}^2+\|\wti{\cL} u\|_{L^2({\Omega })}^2\big)^{1/2}$, $s\geq 0$, 
where $\wti{\cL}u$ should be understood in the sense of distributions. 
We introduce operators acting in $L^2(\Omega)$ by 
$\cL_0u=\cL u$ and $\wti{\cL}_0=\wti{\cL}u$ for $u\in\dom(\cL_0)=\dom(\wti{\cL}_0):=C^\infty_0(\Omega)$, the space of functions compactly supported in $\Omega$, and recall from \cite[Section 11.3]{Schm} that these operators are closable. We will denote by $\cL_{min}$ the closure of $\cL_0$ and by $\wti{\cL}_{min}$ the closure of $\wti{\cL}_0$.
Also, we will consider the maximal operators, $\cL_{max}u:=\cL u$ for $u\in\dom(\cL_{max}):=\cD^0(\Omega)$ and $\wti{\cL}_{max}u:=\wti{\cL} u$ for $u\in\dom(\wti{\cL}_{max}):=\wti{\cD}^0(\Omega)$. It is known, see, e.g., 
\cite[Proposition 1.14]{Schm}, that
\begin{equation}\label{incladj}
\cL_{min}\subset(\wti{\cL}_{min})^*=\cL_{max},
\quad \wti{\cL}_{min}\subset(\cL_{min})^*=\wti{\cL}_{max}.
\end{equation}
Thus, $A=\cL_{min}$ and $\wti{A}=\wti{\cL}_{min}$ is an adjoint pair of operators. We will now construct the boundary traces so that the Green identity \eqref{asgf} holds.

 It will be convenient to use the notation \[\gamma_{{}_{D,\cL}}=\gamma_{{}_{D,\wti{\cL}}}:=\gaD\in\cB(H^1(\Omega), H^{1/2}(\partial\Omega))\] for the usual Dirichlet trace. As in Subsection \ref{SS5.1}, we will associate with $\cL$ from \eqref{1.15b-new} the first order trace operator $\gamma_{{}_{N, \cL}} \in\cB(\cD^1(\Omega), H^{-1/2}(\partial\Omega))$ 
 which is a unique extension of the co-normal derivative $\gamma_{{}_{N, \cL}}\in\cB(H^2(\Omega), H^{1/2}(\Omega))$ defined in \eqref{dfncnd}. Analogously, associated with $\wti{\cL}$ from \eqref{1.15c-new} is the first order trace operator $\gamma_{{}_{N, \wti{\cL}}} \in\cB(\wti{\cD}^1(\Omega), H^{-1/2}(\partial\Omega))$ 
 which is a unique extension of the co-normal derivative $\gamma_{{}_{N, \wti{\cL}}}\in\cB(H^2(\Omega), H^{1/2}(\Omega))$ defined by the formula
\begin{align}
&{\gamma_{{}_{N,\wti{\cL}}}}u:=\sum_{j,k=1}^{n} \overline{\mathtt{a}}_{kj}\nu_j\gamma_{{}_{D,\wti{\cL}}}(\partial_k u)+\sum_{j=1}^{n}{\overline{\mathtt{b}}_j}\nu_j \gamma_{{}_{D,\wti{\cL}}} u,\quad u\in H^2(\Omega).
\end{align}
 Then the following Green identity,
\begin{align}
\langle\cL u, v\rangle_{L^2(\Omega )}
&- \langle u,\wti{\cL} v\rangle_{L^2(\Omega )} \lb{1.2.6}\\&={\langle\gamma_{{}_{D,\cL}} u, \gamma_{{}_{N, \wti{\cL}}} v  \rangle_{{-1/2}}}-\overline{\langle\gamma_{{}_{D,\wti{\cL}}} v, \gamma_{{}_{N, {\cL}}} u  \rangle_{{-1/2}}},\no
\end{align}
holds for all $u\in\cD^1(\Omega)$ and $v\in\wti{\cD}^1(\Omega)$. In order to rewrite this identity in a form compatible with \eqref{asgf} we will need to take four more steps. 

First, we will extend further the usual Dirichlet and weak Neumann trace operators \begin{align*}
 &\gamma_{{}_{D,\cL}} =\gamma_{{}_{D,\wti{\cL}}}\in\cB(H^1(\Omega), H^{1/2}(\partial\Omega)),\\
&\gamma_{{}_{N, \cL}} \in\cB(\cD^1(\Omega), H^{-1/2}(\partial\Omega)),\, \gamma_{{}_{N, \wti{\cL}}} \in\cB(\wti{\cD}^1(\Omega), H^{-1/2}(\partial\Omega))\end{align*}
defined on the spaces $H^1(\Omega)$, $\cD^1(\Omega)$, $\wti{\cD}^1(\Omega)$, respectively, to $\cD^0(\Omega)$ and $\wti{\cD}^0(\Omega)$, the domains of the maximal operators $\cL_{max}$ and $\wti{\cL}_{max}$, and obtain the bounded and surjective traces
\begin{align*}
&\widehat{\gamma}_{{}_{D,\cL}}\in\cB(\cD^0(\Omega),H^{-1/2}(\partial\Omega)),\,
\widehat{\gamma}_{{}_{D,\wti{\cL}}}\in\cB(\wti{\cD}^0(\Omega), H^{-1/2}(\partial\Omega)),\\
& \widehat{\gamma}_{{}_{N,{\cL}}}\in\cB(\cD^0(\Omega), H^{-3/2}(\partial\Omega)),\,\widehat{\gamma}_{{}_{N,\wti{\cL}}}\in\cB(\wti{\cD}^0(\Omega), H^{-3/2}(\partial\Omega)),\end{align*}
cf.\ Lemmas \ref{wd} and \ref{A.5},
that agree with the usual Dirichlet and Neumann trace operators on their respective domains. The existence of such extensions for smooth domains and general elliptic operators is given in \cite[Sections II.1, III.1]{Gr}, and for Lipschitz domains and the Laplacian is given in \cite{GM10}.

Armed with the trace operators defined on the domains $\cD^0(\Omega)$ and $\wti{\cD}^0(\Omega)$ of the maximal operators, we proceed, following \cite[Section III.1]{Gr}, with the second step that involves the Dirichlet-to-Neumann operators, $M_{D,N}$ and $\wti{M}_{D,N}$, associated with the operators $\cL$ from \eqref{1.15b-new} and $\wti{\cL}$ from \eqref{1.15c-new}, respectively. We define  $M_{D,N}$ next; one deals with $\wti{M}_{D,N}$ analogously. By adding to $\mathtt{q}$ a constant, if needed, we may and will assume that zero is not in the spectrum of the Dirichlet realizations of $\cL$ as defined in \cite[Section I]{Gr} or \cite[Section 11.3]{Schm}. Thus, for each $f\in H^{-1/2}(\partial\Omega)$ there is a unique solution $u=u_f\in\cD^0(\Omega)$ to the
boundary value problem $\cL u=0$, $\widehat{\gamma}_{{}_{D,\cL}}u=f$. We may now define $M_{D,N}f:=-\widehat{\gamma}_{{}_{N,{\cL}}}u_f$ as an operator acting from $H^{-1/2}(\partial\Omega)$ to $H^{-3/2}(\partial\Omega)$.

 Our third step is to introduce yet another Neumann trace,
 $\tau_{{}_{N,\cL}}$, defined by the formula
 $\tau_{{}_{N,\cL}}u:=\widehat{\gamma}_{{}_{N,{\cL}}}u+M_{D,N}\widehat{\gamma}_{{}_{D,\cL}}u$ for $u\in\cD^0(\Omega)$. The following remarkable property of $\tau_{{}_{N,\cL}}$ is a consequence of elliptic regularity of solutions to the Dirichlet problem, see \cite[Theorem III.1.2]{Gr} and \cite[Theorem 12.1]{GM10} (or Lemma  \ref{nn} below): Although both distributions $\widehat{\gamma}_{{}_{N,{\cL}}}u$ and $M_{D,N}\widehat{\gamma}_{{}_{D,\cL}}u$ belong to $H^{-3/2}(\partial\Omega)$, we claim that the sum $\tau_{{}_{N,\cL}}u$ of these two distributions is, in fact, a function from $H^{1/2}(\partial\Omega)$. Indeed, given a $u\in\cD^0(\Omega)$ and letting $f=\widehat{\gamma}_{{}_{D,\cL}}u$ we observe that 
 $\tau_{{}_{N,\cL}}u=\widehat{\gamma}_{{}_{N,{\cL}}}(u-u_f)$ with $u_f$ as in the definition of $M_{D,N}$ in step two above. But then 
 $\widehat{\gamma}_{{}_{D,\cL}}(u-u_f)=0$, which shows that $u-u_f$ is in the domain of the Dririclet realization of $\cL$. By elliptic regularity then
 $u-u_f\in H^2(\Omega)$, see, e.g., \cite[Theorem I.3.1]{Gr}, and thus $\tau_{{}_{N,\cL}}u=\widehat{\gamma}_{{}_{N,{\cL}}}(u-u_f)={\gamma}_{{}_{N,{\cL}}}(u-u_f)\in H^{1/2}(\partial\Omega)$ as claimed. Analogously, we define
  $\tau_{{}_{N,\wti{\cL}}}u:=\widehat{\gamma}_{{}_{N,\wti{\cL}}}u+\wti{M}_{D,N}\widehat{\gamma}_{{}_{D,\wti{\cL}}}u$ for $u\in\wti{\cD}^0(\Omega)$.
 Moreover, the operators $\tau_{{}_{N,\cL}}\in\cB(\cD^0(\Omega), H^{1/2}(\partial\Omega))$ and $\tau_{{}_{N,\wti{\cL}}}\in\cB(\wti{\cD}^0(\Omega), H^{1/2}(\partial\Omega))$ are surjective, and the following Green identity,
 \begin{align}\label{GF3}
	(\cL u,v)_{L^{2}(\Om)}&-(u,\wti{\cL} v)_{L^{2}(\Om)}\\
	&=
	\overline{\langle \tau_{{}_{N,\wti{\cL}}} v, \widehat{\gamma}_{{}_{{D,\cL}}} u\rangle_{{-1/2}}}-{{\langle \tau_{{}_{N, {\cL}}}
	u, \widehat{\gamma}_{{}_{D,\wti{\cL}}} v \rangle_{{-1/2}}}},\no
\end{align}
  holds for all $u\in\cD^0(\Omega)$, $v\in\wti{\cD}^0(\Omega)$ by \cite[Theorem III.1.2]{Gr}, see also \cite[Theorem 7.4]{BMNW} and \cite[Theorem 12.1]{GM10} or Lemma \ref{nn} below.

Finally, in the last step we use the Riesz isomorphism $\Phi: H^{-1/2}(\partial\Omega)\rightarrow H^{1/2}(\partial\Omega)$ defined in \eqref{aub27}. We are ready to define the boundary triplet in question:  the function spaces are given by  
\begin{equation}
\cH_+=\dom((\cL_{min})^*)=\wti{\cD}^0(\Omega),
\wti{\cH}_+=\dom((\wti{\cL}_{min})^*)={\cD}^0(\Omega), 
\mathfrak{H}=\mathfrak{K}=H^{1/2}(\partial\Omega)
\end{equation}
and the trace operators are given by 
\begin{equation}
\Gamma_0=-\Phi\widehat{\gamma}_{{}_{D,\cL}}, \Gamma_1=\tau_{{}_{N,\cL}}, \wti{\Gamma}_0=-\Phi\widehat{\gamma}_{{}_{D,\wti{\cL}}}, \wti{\Gamma}_1=\tau_{{}_{N,\wti{\cL}}}.
\end{equation}
Now \eqref{GF3} shows that this collection is indeed a boundary triplet for $A=\cL_{min}$ and $\wti{A}=\wti{\cL}_{min}$ as \eqref{asgf} readily holds.


\appendix

\section{Lagrangian planes and self-adjoint extensions}\label{AppA}

In this appendix we elaborate on the assumption of the second part of Theorem \ref{thm1.7} -- that the image of the domain of a self-adjoint extension is a Lagrangian plane.  It is well known that self-adjoint extensions of $A$ can be parameterized by Lagrangian planes, 
see, e.g., \cite[Theorem 3.1.6]{GG}, \cite{Harmer}, \cite{Pa}, \cite[Proposition 14.7]{Schm}.  Such parameterization depends on the choice of the trace operator  $\tr$ and the ``boundary" space $\mathfrak{H}$, see, e.g., \cite[Proposition 2.4]{BL} and \cite[Chapter 3]{GG}. Theorems \ref{LLSA}  and \ref{SALL} and Corollary \ref{LLSASALL}
below give yet another variant of the parameterization. 

To formulate these results we will need some elementary preliminaries.
Let $\cG$ be a subspace in $\bndra$ and $\tr^{-1}(\cG):=\{u\in\cD: \tr u\in\cG\}$ denote the preimage of $\cG$. 
Consider the linear operator $\cA:=A^*\big|_{\tr^{-1}(\cG)}$ acting in $\cH$ and given by
		\begin{equation}\label{dfncA}
		\cA u:=A^* u,\   u\in \dom(\cA):={\tr^{-1}(\cG)}.
		\end{equation}
Since $\cA$ is a part of the closed operator $A^*$, the operator $\cA$ is closable. We denote by $\overline{\cA}$ the closure of $\cA$, that is, we let \begin{align*}
&\dom(\overline{\cA})=\big\{u\in\cH:
\exists\, \{u_n\}_{n\in\bbZ}\subset\dom(\cA)\text{ such that
	$u_n\to u$ in $\cH$ }\\
&\hspace{4cm}\text{ and $\{\cA u_n\}$ converges to some $w\in\cH$}\big\}\\
&\overline{\cA}u:=w, u\in\dom(\overline{\cA}).
\end{align*}
 In particular, we have
  \begin{equation}\label{domAbar}
\overline{\cA}=A^*\big|_{\dom(\overline{\cA})}\text{ where
$\dom(\overline{\cA})=\overline{{\tr^{-1}(\cG)}}^{\cH_+}=\overline{\dom(\cA)}^{\cH_+}$}.\end{equation}
Using the general definition of the adjoint operator we record the following,
\begin{equation}\label{defadj}
\begin{split}
\dom\cA^*&=\{u\in\cH: \exists\, w\in\cH \text{ such that
$\langle w,v\rangle_\cH-\langle u,\cA v\rangle_\cH=0$
for all $v\in\dom\cA$\}},\\
\cA^*u&:=w \text{ for $u\in\dom\cA^*$}.
\end{split}
\end{equation}
Since $\dom(A)\subset\dom(\cA)$, due to $\dom(A)=\ker\tr\subset\tr^{-1}(\cG)$, from Proposition \ref{remark2.2} (1), taking $v\in\dom(A)$ in \eqref{defadj} shows that $u\in\dom(A^*)$ and $w=A^*u=\cA^*u$. Thus, using 
$\cA=A^*\big|_{\dom(\cA)}=A^*\big|_{\tr^{-1}(\cG)}$, we get
\begin{equation}\label{defAadj}
\cA^*=A^*\big|_{\dom(\cA^*)},\,
\dom(\cA^*)=\{u\in\cH_+: \langle A^*u,v\rangle_\cH-\langle u,A^*v\rangle_\cH=0 \, \forall\,
v\in\tr^{-1}(\cG)\}.\end{equation}

We are ready to present a result saying that pre-images 
of Lagrangian planes in $\bndra$ under the trace map give domains of self-adjoint extensions of $A$.

\begin{theorem}\label{LLSA} Assume Hypothesis \ref{hyp3.6} and that 
 $\cG\in\Lambda(\bndra)$ is a Lagrangian subspace in $\bndra$ such that 
\begin{align}\label{AA}
\text{$\cG\cap\tr(\cD)=\tr\big(\tr^{-1}(\cG)\big)$ is $(\bndra)$-dense in $\cG$}.
\end{align} 
Then the operator $\cA=A^*\big|_{\tr^{-1}(\cG)}$ defined in \eqref{dfncA} is essentially self-adjoint, that is, $\overline{\cA}=\cA^*$, if and only if
\begin{align}\label{AA1}
&\text{$\dom(\cA^*)\cap\cD$ is $(\cH_+)$-dense in $\dom(\cA^*)$}.
\end{align}
\end{theorem}
\begin{proof} Assume \eqref{AA1} .
We derive $\overline{\cA}=\cA^*$ in three steps. First,
we show $\dom(\cA)\subseteq\dom(\cA^*)$. If $u\in\dom(\cA)=\tr^{-1}(\cG)$ then for any $v\in\tr^{-1}(\cG)=\dom(\cA)$ the Green identity \eqref{3.61new}  gives
\begin{equation}\label{eqzero1}
\langle A^*u,v\rangle_\cH-\langle u,A^*v\rangle_\cH=\omega(\tr u,\tr v)=0\end{equation}
because both $\tr u$  and $\tr v$ are in $\cG$ and $\cG\subseteq\cG^\circ$ as $\cG$ is isotropic by the assumption. Now \eqref{defAadj} and \eqref{eqzero1} yield $u\in\dom(\cA^*)$ as required. Second, we show that $\dom(\cA^*)\cap\cD\subseteq\dom(\cA)$. If $u\in\dom(\cA^*)\cap\cD$ then for any $v\in\tr^{-1}(\cG)$ we have
\begin{equation}\label{eqzero2}
\omega(\tr u,\tr v)=\langle A^*u,v\rangle_\cH-\langle u,A^*v\rangle_\cH=0\end{equation}
because $u\in\dom(\cA^*)$ and $v\in\dom(\cA)$, see \eqref{defAadj}. We now claim that $\omega(\tr u,g)=0$ for any $g\in\cG$. Indeed, we use \eqref{AA} to approximate $g\in\cG$ by a sequence $g_n\in\cG\cap\tr(\cD)$. For each $n$ choose $v\in\tr^{-1}(\cG)$ such that $g_n=\tr v$. By \eqref{eqzero2} then $\omega(\tr u,g)=\lim_{n\to\infty}\omega(\tr u, g_n)=0$, thus proving the claim. Therefore, $\tr u\in\cG^\circ\subseteq\cG$ as $\cG$ is maximally isotropic by the assumption, and then $u\in\tr^{-1}(\cG)=\dom(\cA)$ as required.
Third, taking $(\cH_+)$-closures in the inclusions 
\[\dom(\cA^*)\cap\cD\subseteq\dom(\cA)\subseteq\dom(\cA^*)\] just proved and using \eqref{AA1} yield $\overline{\dom(\cA)}=\dom(\cA^*)$ and therefore $\overline{\cA}=\cA^*$, see \eqref{domAbar}.

Conversely, assume that $\overline{\cA}=\cA^*$. To show \eqref{AA1} we need to prove that $\dom(\overline{\cA})\cap\cD$ is dense in $\dom(\overline{\cA})$. By \eqref{domAbar} we know that $\dom(\overline{\cA})=\overline{\tr^{-1}(\cG)}$ and thus it remains to show that $\overline{\overline{\tr^{-1}(\cG)}\cap\cD}=\overline{\tr^{-1}(\cG)}$. The inclusion ``$\subseteq$'' follows from $\overline{\tr^{-1}(\cG)}\cap\cD\subseteq\overline{\tr^{-1}(\cG)}$. To prove ``$\supseteq$'', we take $u\in\overline{\tr^{-1}(\cG)}$ and a sequence $u_n\in\tr^{-1}(\cG)$ approximating $u$. Since $\tr^{-1}(\cG)\subseteq\cD$ we have $u_n\in\overline{\tr^{-1}(\cG)}\cap\cD$ and thus $u\in\overline{\overline{\tr^{-1}(\cG)}\cap\cD}$ as required.
\end{proof}

Next, we present a result saying that the traces of the domains of self-adjoint extensions of $A$ form Lagrangian planes in $\bndra$.
\begin{theorem}\label{SALL} Assume Hypothesis \ref{hyp3.6} and that there exists 
a self-adjoint restriction $\cA$ of $A^*$ on a subspace $\dom(\cA)\subset\cH_+$ such that 
\begin{equation}\label{BB}
\text{$\dom(\cA)\cap\cD$ is $(\cH_+)$-dense in $\dom(\cA)$}.
\end{equation}
 Then the $(\bndra)$-closure of the subspace $\cG$ defined by $\cG:=\tr(\dom(\cA)\cap\cD)$ is Lagrangian, that is,
 $\overline{\cG}=\cG^\circ$, if and only if
 \begin{align}\label{BB1}
 &\text{$\cG^\circ\cap\tr(\cD)$ is $(\bndra)$-dense in $\cG^\circ$}.
 \end{align}
\end{theorem}
\begin{proof}  Assume \eqref{BB1}. We derive $\overline{\cG}=\cG^\circ$ in three steps. First, we show that $\cG\subseteq\cG^\circ$. If $f\in\cG=\tr(\dom(\cA)\cap\cD)$ then $f=\tr u$ for some $u\in\dom(\cA)\cap\cD$. Since $\cA\subseteq\cA^*$ by the assumption, we conclude that $u\in\dom(\cA^*)$. Pick any $g\in\cG$ and let $v\in\dom(\cA)\cap\cD$ be such that $g=\tr v$. Then 
\begin{equation}\label{eqzero3}
\omega(f,g)=\omega(\tr u,\tr v)=\langle A^*u,v\rangle_\cH-\langle u,A^*v\rangle_\cH=0\end{equation}
because $u\in\dom(\cA^*)$ and $v\in\dom(\cA)$, see \eqref{defAadj}. But \eqref{eqzero3} yields $f\in\cG^\circ$, as required. Second, we show that $\cG^\circ\cap\tr(\cD)\subseteq\cG$. If $f\in\cG^\circ\cap\tr(\cD)$ then $f=\tr u$ for some $u\in\cD$ and $\omega(f,g)=0$ for all $g\in\cG$. In particular, if $v\in\dom(\cA)\cap\cD$ and $g=\tr v\in\cG$
 then
 \begin{equation}\label{eqzero4}
\langle A^*u,v\rangle_\cH-\langle u,A^*v\rangle_\cH=\omega(\tr u,\tr v)=0.\end{equation}
 Due to 
\eqref{BB} we conclude from \eqref{eqzero4} that $\langle A^*u,v\rangle_\cH-\langle u,A^*v\rangle_\cH=0$ for all $v\in\dom(\cA)$. Thus $u\in\dom(\cA^*)$ by \eqref{defAadj}.
Since $\cA^*\subseteq\cA$ by the assumption,  we have $u\in\dom(\cA)$ and so $f=\tr u\in\tr(\dom(\cA)\cap\cD)=\cG$ as required. Third, taking $(\bndra)$-closures in the inclusions
\[\cG^\circ\cap\tr(\cD)\subseteq\cG\subseteq\cG^\circ\]
just proved and using \eqref{BB1} yield $\overline{\cG}=\cG^\circ$.
 
 Conversely, assume that $\overline{\cG}=\cG^\circ$. To show \eqref{BB1} we need to prove that $\overline{\cG}\cap\tr(\cD)$ is dense in $\overline{\cG}$, that is, that
 $\overline{\overline{\cG}\cap\tr(\cD)}=\overline{\cG}$.
 Since $\cG\subseteq\tr(\cD)$, this follows analogously to the last part of the proof of Theorem \ref{LLSA}.
 \end{proof}

We note that conditions \eqref{AA}, \eqref{AA1}, \eqref{BB}, \eqref{BB1} automatically hold for all classes of PDE, ODE and quantum graphs operators and all examples that we know; these conditions trivially hold provided $\cD=\cH_+$ and $\tr(\cD)=\bndra$, that is, when $(\mathfrak{H},\Gamma_0,\Gamma_1)$ is an abstract  boundary triplet, see Section \ref{abt}. 
We recall Remark \ref{rem:exist} regarding the existence of self-adjoint extensions of $A$ under Hypothesis \ref{hyp3.6}.

	\begin{remark}\label{rem:A3}
		The density assumptions $\overline{\dom(A)}=\cH, \overline{\ran(\tr)}=\bndra$ introduced in Hypothesis \ref{hyp3.6} are absolutely critical for Theorems \ref{LLSA} and \ref{SALL} to hold. Indeed, \cite[Example 6.6]{DHMS2006} gives a scenario in which dropping the above mentioned density assumptions facilitates a Lagrangian plane in $\bndra$ whose preimage is equal to $\dom(A)$, which is evidently not a domain of self-adjoint extension of $A$. 
	\end{remark}

Assuming Hypothesis \ref{hyp3.6}, for the sake of brevity,
 in the sequel  we will use the following terminology.
 \begin{definition}\label{defASSOC}
 {\em (i)}\, Given a subspace $\cG$ in $\bndra$, we  call 
 $\cA=A^*\big|_{\tr^{-1}(\cG)}$ defined in \eqref{dfncA} the {\em operator  associated with $\cG$}.
 
  {\em (ii)}\, Given an operator $\cA$, we call $\cG=\tr\big(\dom(\cA)\cap\cD\big)$ the {\em subspace associated with $\cA$}.
  
   {\em (iii)}\, We say that a Lagrangian subspace $\cG\in\Lambda(\bndra)$ is {\em $(\tr, \cD)$-aligned} or, when there is no confusion, simply {\em aligned} if \eqref{AA} holds and the adjoint to the associated with $\cG$ operator $\cA$ satisfies \eqref{AA1}.
   
  {\em (iv)}\, We say that a self-adjoint restriction $\cA$ of $A^*$  is {\em $(\tr, \cD)$-aligned} or, when there is no confusion, simply {\em aligned} if 
  \eqref{BB} holds and the annihilator of  the associated with $\cA$ subspace $\cG$ satisfies \eqref{BB1}.
 \end{definition}
 
 Definition \ref{defASSOC} yields the following short rephrasing of Theorems \ref{LLSA} and \ref{SALL}. 
 
 \begin{corollary}\label{LLSASALL}
  If $\cG$  is an aligned Lagrangian subspace
  then the operator $\cA$ associated with $\cG$ 
  is essentially self-adjoint and its closure $\overline{\cA}$ is aligned; in particular, the closure of the subspace associated with $\overline{\cA}$ is equal to $\cG$. 
  
 Conversely, if $\cA$ is an aligned self-adjoint restriction of $A^*$  then
the closure $\overline{\cG}$ of the subspace $\cG$ associated with $\cA$ is an aligned Lagrangian subspace; in particular, the closure of the operator associated with $\overline{\cG}$ is equal to $\cA$.
 \end{corollary}
 \begin{proof} Let $\cG$ be an aligned Lagrangian plane. Then \eqref{AA} and \eqref{AA1} hold and imply $\overline{\cA}=\cA^*$ by Theorem \ref{LLSA}. Let us consider the subspace $\cG'=\tr(\dom(\overline{\cA})\cap\cD)$ associated with the self-adjoint operator $\overline{\cA}$. To show that $\overline{\cA}$ is aligned we will have to prove that (a) $\dom(\overline{\cA})\cap\cD$ is dense in $\dom(\overline{\cA})$ and that (b) $\cG'^\circ\cap\tr(\cD)$ is dense in $\cG'^\circ$.   
 Assertion (a) follows from \eqref{AA1} since $\dom(\overline{\cA})=\dom(\cA^*)$. By Theorem \ref{SALL} applied to the operator $\overline{\cA}$ assertion (b) is equivalent to the fact that $\overline{\cG'}$ is Lagrangian. Thus it remains to show that $\overline{\cG'}=\cG$. To begin the proof of the latter equality we first recall from 
 \eqref{domAbar} that $\dom(\overline{\cA})=\overline{\tr^{-1}(\cG)}$. Since $\overline{\tr^{-1}(\cG)}\cap\cD\supseteq\tr^{-1}(\cD)$ we infer
 \[\cG'=\tr(\dom(\overline{\cA})\cap\cD)=\tr\big(\overline{\tr^{-1}(\cG)}\cap\cD\big)\supseteq
 \tr\big(\tr^{-1}(\cD)\big)=\cG\cap\tr(\cD),\]
 where the last equality is checked directly. Taking closure and using \eqref{AA1} yields $\overline{\cG'}\supseteq\cG$.
 It remains to show that $\overline{\cG'}\subseteq\cG$.
 We claim that $\cG'$ is isotropic, that is, $\cG'\subseteq\cG'^\circ$. To show this, we take any $f,g\in\cG'$ so that $f=\tr u$, $g=\tr v$ for  some $u,v\in
 \overline{\tr^{-1}(\cG)}\cap\cD$ and pick sequences $u_n,v_n\in\tr^{-1}(\cG)$ such that $u_n\to u$ and $v_n\to v$ as $n\to\infty$. Since $\tr u_n,\tr v_n\in\cG$ and $\cG\subseteq\cG^\circ$ by the assumption, we conclude that $\omega(f,g)=\omega(\tr u,\tr v)=\lim_{n\to\infty}\omega(\tr u_n,\tr v_n)=0$ as claimed. It follows from $\overline{\cG'}\supseteq\cG$ and the claim that $\cG\subseteq\overline{\cG'}\subseteq\cG'^\circ$ and therefore
 that $\cG'\subseteq\cG^\circ=\cG$ as $\cG$ is Lagrangian. 
 Hence, $\overline{\cG'}\subseteq\cG$ as required.
 
 To begin the proof of the second part of the corollary, let $\cA=\cA^*$ be an aligned restriction of $A^*$ and denote $\cG=\tr(\dom(\cA)\cap\cD)$. Then 
 \eqref{BB} and \eqref{BB1} hold and imply that $\overline{\cG}$ is Lagrangian by Theorem \ref{SALL}. To check that $\overline{\cG}$ is aligned we need to show that (a) $\overline{\cG}\cap\cD$ is dense in $\overline{\cG}$ and that (b) $\dom(\cA_{\tr^{-1}(\overline{\cG})}^*)\cap\cD$ is dense in $\dom(\cA_{\tr^{-1}(\overline{\cG})}^*)$. Since $\overline{\cG}=(\overline{\cG})^\circ=\cG^\circ$, assertion (a) is the same as  \eqref{BB1} and therefore holds. By Theorem \ref{LLSA} for $\overline{\cG}$ assertion (b) is equivalent to the fact that 
 the closure of the operator $\cA_{\tr^{-1}(\overline{\cG})}$ associated with $\overline{\cG}$ is self-adjoint. So, to complete the proof it suffices to show that the closure of $\cA_{\tr^{-1}(\overline{\cG})}$ is equal to $\cA$ or that $\dom(\cA)=\dom(\overline{\cA_{\tr^{-1}(\overline{\cG})}})$.
 In other words, see \eqref{domAbar}, we want to check the equality 
 \begin{equation}\label{EQDOM}
 \dom(\cA)=
 \overline{\tr^{-1}(\overline{\tr(\dom(\cA)\cap\cD)})}.
 \end{equation} The inclusion ``$\subseteq$'' in \eqref{EQDOM} follows from \eqref{BB} by taking closure in 
 \[ \dom(\cA)\cap\cD= \tr^{-1}(\tr(\dom(\cA)\cap\cD))\subseteq\tr^{-1}(\overline{\tr(\dom(\cA)\cap\cD)}).\]
Thus, it remains to prove the inclusion ``$\supseteq$'' in \eqref{EQDOM}. Take a $u$ from the RHS of \eqref{EQDOM} and select a sequence $u_n\in\tr^{-1}(\overline{\tr(\dom(\cA)\cap\cD)})$ such that $u_n\to u$ in $\cH_+$ as $n\to\infty$.
Since $\tr u_n\in\overline{\cG}\subseteq\cG^\circ$ as $\overline{\cG}$ is Lagrangian, $\omega(\tr u_n,g)=0$ for any $g\in\cG=\tr(\dom(\cA)\cap\cD)$. In particular, for all $v\in\dom(\cA)\cap\cD$ we have
\[\langle A^*u_n, v\rangle_\cH-\langle u_n, A^*v\rangle_\cH
=\omega(\tr u_n,\tr v)=0.\]
Using \eqref{BB} we then conclude that $\langle A^*u_n, v\rangle_\cH-\langle u_n, A^*v\rangle_\cH=0
$ for all $v\in\dom(\cA)$. 
This shows that $u_n\in\dom(\cA^*)$ and therefore $u=\lim_{n\to\infty}u_n\in\dom(\cA^*)=\dom(\cA)$ thus completing the proof of 
the inclusion ``$\supseteq$'' in \eqref{EQDOM}.
 \end{proof}
A particularly transparent and widely studied scenario of aligned Lagrangian subspaces and self-adjoint operators is discussed in Section \ref{abt}, see, in particular, Remark \ref{rem4.2}.

\section{The Krein--Naimark resolvent formula revisited}\label{Sec4.1} In this appendix, we revisit the classical Krein--Naimark \eqref{1.23} and Krein \eqref{fdi-krf} formulas for the difference of resolvents of two self-adjoint extensions of an abstract symmetric operator, see, e.g, \cite[Section 14.6]{Schm}. As we demonstrate in the proof of Proposition \ref{prop:KrNaF} the Krein--Naimark formula \eqref{1.23} can be naturally derived from formula \eqref{5.14} in Theorem \ref{thm1.7} by specializing it to the case of  boundary triplets.  Conversely,  in Remark \ref{Rem4.4} we show how to derive \eqref{5.14} from \eqref{1.23}. The poof of Krein's resolvent formula for the case of finite deficiency indices is given in Proposition \ref{prop:KrFormFDI}. 

Let $(\mathfrak{H}, \Gamma_0, \Gamma_1)$ be a boundary triplet as described in Definition \ref{cond}. Following common convention we define  one of the two self-adjoint extensions of $A$ in the Krein-Naimark formula  by 
\begin{equation}\lb{1.35a}
\cA_0:=A^*\upharpoonright_{\ker(\Gamma_0)},
\end{equation}
and subtract from its resolvent the resolvent of yet another, arbitrary, self-adjoint extension. 

First, we recall some known facts, see, e.g., \cite[Section 14]{Schm}. Since \[\dom(A^*)=\dom(A_0)\dot{+}\ker(A^*-\zeta) \text{ for } \zeta\in\bbC\setminus\bbR,\] the map
$\Gamma_0\upharpoonright_{\ker(A^*-\zeta)}: \ker(A^*-\zeta)\to\mathfrak{H}$ is bijective
and thus we define $\gamma(\zeta):=(\Gamma_0\upharpoonright_{\ker(A^*-\zeta)})^{-1}$ and notice that $\gamma(\zeta)\in\cB(\mathfrak H, \cH)$ and $\Gamma_0\gamma(\zeta)h=h$ for any $h\in\mathfrak{H}$. In particular, $\gamma(\zeta)$ is injective.  We will use the well-known Derkach-Malamud Lemma saying that $\gamma^*(\overline{\zeta})=\Gamma_1(\cA_0-\zeta)^{-1}$, see
\cite[Lemma 1]{DM} or \cite[Proposition 14.14(i)]{Schm}.  The operator-valued function $\gamma(\cdot)$ can be extended analytically to $\bbC\setminus\spec (\cA_0)$ giving rise to the abstract Weyl function $M(\zeta):=\Gamma_1\gamma(\zeta)$, $\zeta\in \bbC\setminus\spec(\cA_0)$. 

Next, let $\cA$ be an arbitrary self-adjoint extension of $A$, and let 
$\cG\in\Lambda(\mathfrak{H}\times\mathfrak{H})$ be the Lagrangian subspace such that $\cG=\tr(\dom(\cA))$, cf.\ Theorems \ref{LLSA}, \ref{SALL} and Remark \ref{rem4.2}. We will treat $\cG$ as a linear relation, 
see, e.g., \cite[Section 14.1]{Schm}.
Slightly abusing notation we do not distinguish between the operator $M(\zeta)$ and its graph, in particular, we write $\cG-M(\zeta):=\cG-\gr(M(\zeta))$ and treat both terms in the right-hand side as linear relations. The linear relation $\cG-M(\zeta)$ is called {\em invertible} whenever 
\begin{align}\label{inj}
\ker(\cG-M(\zeta))&:=\{f\in\mathfrak{H}: (f,0)\in(\cG-M(\zeta))\}=\{0\},\text{ and }\\
\ran(\cG-M(\zeta))&:=\{g\in\mathfrak H:\exists f\in\mathfrak{H}\  {s.t.}\  (f,g)\in(\cG-M(\zeta)) \}=\mathfrak H.\label{surj}\end{align}
In this case there exists an operator in $\cB(\mathfrak{H})$ whose graph is given by 
\begin{equation}
\{(g,f) \in \bndra: (f,g)\in(\cG-M(\zeta)) \};
\end{equation}
this operator is denoted by   $(\cG-M(\zeta))^{-1}$.
\begin{proposition}\label{prop:KrNaF} Let  $(\mathfrak{H}, \Gamma_0, \Gamma_1)$ be a boundary triplet for the symmetric operator $A$, see Definition \ref{cond}, let $\cA_0$ be the self-adjoint extension of $A$ from \eqref{1.35a}, let $\cA$ be an arbitrary self-adjoint extension of $A$ and $\cG=\tr(\dom(\cA))$.
	Then $\cG-M(\zeta)$ is invertible and
	\begin{equation}\label{1.23}
	(\cA-\zeta)^{-1}=(\cA_0-\zeta)^{-1}+\gamma(\zeta)(\cG-M(\zeta))^{-1}\gamma^*(\overline{\zeta})
	\text{ for $\zeta\not\in\Sp(\cA_0)\cup \Sp(\cA)$. }
	\end{equation}
\end{proposition}
\begin{proof}
	We denote $R_0(\zeta):=(\cA_0-\zeta)^{-1}$ and 
	$R(\zeta)=(\cA-\zeta)^{-1}$. Since $\Gamma_0 R_0(\overline{\zeta})=0$ by \eqref{1.35a}, the resolvent difference formula from Theorem \ref{thm1.7} and the Derkach-Malamud Lemma above yield
	\begin{equation*}
	\label{1.25new}
	R_{0}(\zeta)- R(\zeta)= (\Gamma_0R_0(\overline{\zeta}))^*\Gamma_1 R(\zeta)-(\Gamma_1R_0(\overline{\zeta}))^*\Gamma_0 R(\zeta)=-\gamma(\zeta)\Gamma_0 R(\zeta).
	\end{equation*}
	It remains to prove \eqref{inj}, \eqref{surj}, and that 
	\begin{equation}\label{g0R}
	\Gamma_0R(\zeta)=(\cG-M(\zeta))^{-1}\gamma^*(\overline{\zeta}). 
	\end{equation}
	The main identity needed for the proofs is that
	\begin{equation}\label{1.32}
	\gamma^*(\overline{\zeta})u=\Gamma_1R_0(\zeta)u= \Gamma_1R(\zeta)u-M(\zeta) \Gamma_0 R(\zeta)u \,
	\text{ for all  $u\in\cH$}.
	\end{equation}
	To justify the second equality in \eqref{1.32}, we use  $(A^*-\zeta)\gamma(\zeta)=0$
	and $\Gamma_0(I_\cH-\gamma(\zeta)\Gamma_0)=0$, yielding
	$\ran(I_\cH-\gamma(\zeta)\Gamma_0)\subset\dom(\cA_0)$, and write
	\begin{align*}
	\Gamma_1R_0(\zeta)&=\Gamma_1R_0(\zeta)(\cA-\zeta)R(\zeta)=\Gamma_1R_0(\zeta)(A^*-\zeta)R(\zeta)\\
	&=\Gamma_1R_0(\zeta)(A^*-\zeta)(I_\cH-\gamma(\zeta)\Gamma_0)R(\zeta)\\
	&=\Gamma_1R_0(\zeta)(\cA_0-\zeta)(I_\cH-\gamma(\zeta)\Gamma_0)R(\zeta)\\
	&=\Gamma_1(I_\cH-\gamma(\zeta)\Gamma_0)R(\zeta)=
	\Gamma_1R(\zeta)-M(\zeta) \Gamma_0 R(\zeta),
	\end{align*}
	thus proving \eqref{1.32}.
	Since $R(\zeta)$ is a bijection of $\cH$ onto $\dom(\cA)$, we have $\cG=\{(\Gamma_0R(\zeta)u, \Gamma_1R(\zeta)u): u\in\cH\}$. This and \eqref{1.32} yield
	\begin{align}\label{linre}
	\cG-M(\zeta)&=\{(f,g-M(\zeta)f): (f,g)\in\cG\}\\&=\big\{\big(\Gamma_0R(\zeta)u, \Gamma_1R(\zeta)u-M(\zeta) \Gamma_0 R(\zeta)u\big):u\in\cH\big\}\\
	&=\big\{\big(\Gamma_0R(\zeta)u, \gamma^*(\overline{\zeta})u\big):u\in\cH\big\}.\nonumber
	\end{align}
	Since $\tr$ is surjective, \eqref{surj} follows from \eqref{linre}. Indeed, for any $g\in\mathfrak{H}$ there is some $v\in\dom(A^*)$ such that $\Gamma_0v=0$ and $\Gamma_1v=g$. Since $v\in\dom(\cA_0)$, there is some $u\in\cH$ such that $v=R_0(\zeta)u$ and so $g=\Gamma_1R_0(\zeta)u\in\ran(\cG-M(\zeta))$ by \eqref{linre} and \eqref{1.32}. To begin the proof of \eqref{inj}, we first notice that
	$\gamma(\zeta)\ker(\cG-M(\zeta))\subset\dom(\cA)$. Indeed,
	by \eqref{linre} and \eqref{1.32} we have $\ker(\cG-M(\zeta))=\big\{\Gamma_0R(\zeta)u: \Gamma_1R(\zeta)u=M(\zeta)\Gamma_0R(\zeta)u, u\in\cH\big\}$ and thus
	\begin{align}
	\tr&\gamma(\zeta)\ker(\cG-M(\zeta))
	\\&=\{\big(\Gamma_0\gamma(\zeta)\Gamma_0R(\zeta)u,
	\Gamma_1\gamma(\zeta)\Gamma_0R(\zeta)u\big): \Gamma_1R(\zeta)u=M(\zeta)\Gamma_0R(\zeta)u, u\in\cH\}
	\\&=\{\big(\Gamma_0R(\zeta)u,
	M(\zeta)\Gamma_0R(\zeta)u\big): \Gamma_1R(\zeta)u=M(\zeta)\Gamma_0R(\zeta)u, u\in\cH\}\\
	&=\cG\cap {\rm graph } (M(\zeta)).
	\end{align}
	Therefore, $(\cA-\zeta)\gamma(\zeta)\ker(\cG-M(\zeta))=
	(A^*-\zeta)\gamma(\zeta)\ker(\cG-M(\zeta))=\{0\}$ yields
	the inclusion $\gamma(\zeta)\ker(\cG-M(\zeta))\subset\ker(\cA-\zeta)=\{0\}$
	and thus $\ker(\cG-M(\zeta))=\{0\}$ because $\gamma(\zeta)$ is injective, thus finishing the proof of \eqref{inj}.
	Finally, using \eqref{linre} again,
	\begin{align*} 
	{\rm graph } (\cG-M(\zeta))^{-1}&=\big\{ (g,f)\in\mathfrak{H}\times\mathfrak{H}: (f,g)\in(\cG-M(\zeta))\big\}
	\\&=\big\{\big(\gamma^*(\overline{\zeta})u,\Gamma_0R(\zeta)u\big): u\in\cH \big\}
	\end{align*}
	yielding $(\cG-M(\zeta))^{-1}\gamma^*(\overline{\zeta})
	=\Gamma_0R(\zeta)$,
	as required to finish the proof of \eqref{g0R} and thus \eqref{1.23}.
\end{proof}
\begin{remark}\label{Rem4.4}
	In the course of proof of the Krein-Naimark formula \eqref{1.23} we established relation \eqref{g0R}. Using this relation we now show how to derive formula \eqref{5.14} in Theorem \ref{thm1.7} from formula \eqref{1.23}, cf.\ the proofs of Theorem 2 and Corollary 4 in \cite{DM}. For any two self-adjoint extensions $\cA_1$ and $\cA_2$ and the extension $\cA_0$ given by \eqref{1.35a} we denote
	$R_j(\zeta)=(\cA_j-\zeta)^{-1}$ for any $\zeta$ which is not in the spectrum of $\cA_j$, $j=0,1,2$. Applying \eqref{1.23} and using \eqref{g0R} for $\cA_1$ and $\cA_2$ yields
	\begin{equation}\label{firstres}
	R_1(\zeta)=R_0(\zeta)+\gamma(\zeta)\Gamma_0R_1(\zeta),\quad
	R_2(\zeta)=R_0(\zeta)+\gamma(\zeta)\Gamma_0R_2(\zeta).
	\end{equation}
	Multiplying \eqref{firstres} by $\Gamma_1$ and using formulas $\gamma^*(\overline{\zeta})=\Gamma_1R_0(\zeta)$ and $M(\zeta)=\Gamma_1\gamma(\zeta)$ gives
	\[\Gamma_1R_1(\zeta)=\gamma^*(\overline{\zeta})+M(\zeta)\Gamma_0 R_1(\zeta), \quad 
	\Gamma_1R_2(\overline{\zeta})=
	\gamma^*({\zeta})+M(\overline{\zeta})\Gamma_0 R_2(\overline{\zeta}).\]
	Plugging this in the RHS of formula \eqref{5.14} and using
	the property  $M^*(\zeta)=M(\overline{\zeta})$ of the Weyl function, see, e.g., \cite[Proposition 14.15(ii)]{Schm}, yileds
	\begin{align*}
	\big(&\Gamma_0 R_2(\overline{\zeta})\big)^*\Gamma_1 R_1({\zeta})-\big(\Gamma_1 R_2(\overline{\zeta})\big)^*\Gamma_0 R_1({\zeta})\\&=\big(\Gamma_0 R_2(\overline{\zeta})\big)^*\big(\gamma^*(\overline{\zeta})+M(\zeta)\Gamma_0 R_1(\zeta)\big)
	-\big(\gamma^*({\zeta})+M(\overline{\zeta})\Gamma_0 R_2(\overline{\zeta})\big)^*\Gamma_0 R_1({\zeta})\\
	&=\big(\gamma(\overline{\zeta})\Gamma_0 R_2(\overline{\zeta})\big)^*-\big(\gamma(\zeta)\Gamma_0R_1(\zeta)\big)+\big(\Gamma_0 R_2(\overline{\zeta})\big)^*\big(M(\zeta)-M^*(\overline{\zeta})\big)\Gamma_0 R_1({\zeta})\\&=\big(R_2(\overline{\zeta})-R_0(\overline{\zeta})\big)^*-\big(R_1(\zeta)-R_0(\zeta)\big)=
	R_2(\zeta)-R_1(\zeta),
	\end{align*}
	where, to pass to the last line, we used \eqref{firstres} again.
	This proves \eqref{5.14} as required.
\hfill$\Diamond$  \end{remark}

We will conclude this section by deriving from formula \eqref{5.14aJ} in Theorem \ref{thm1.7} yet another classical Krein's resolvent formula
\eqref{fdi-krf}  valid under the temporary assumption that the equal deficiency indices of $A$ are finite and for which we refer to the classical text 
\cite[Section VIII.106]{AhkGlazman} and a very nice newer exposition in \cite[Appendix A]{ClarkGesNickZinch}.

\begin{proposition}\label{prop:KrFormFDI} Let  $(\mathfrak{H}, \Gamma_0, \Gamma_1)$ be a boundary triplet for the symmetric operator $A$ with equal and finite deficiency indices, let $\cA_1$ and $\cA_2$ be any two self-adjoint extensions of $A$, let $(m,m)$ denote the deficiency indices of the operator  $A_0=A^*\big|_{\dom(A_0)}$  defined by the equality $\dom(A_0):=\dom(\cA_1)\cap\dom(\cA_2)$, and let $\{u_k(\zeta)\}_{1\le k\le m}$ be any basis in the subspace $\ker(A_0^*-\zeta)$.  Then there exists a unique non-singular matrix $\cP(\zeta)=\big(p_{lj}(\zeta)\big)_{1\le l,j\le m}$, cf. \eqref{for1}, such that the resolvents of the operators $\cA_1$ and $\cA_2$ for each $u\in\cH$ satisfy 
\begin{equation}\label{fdi-krf} 
\big(R_1(\zeta)-R_2(\zeta)\big)u=\sum_{l,j}p_{lj}(\zeta)\langle  u, u_j(\overline{\zeta})\rangle_{{}_\cH} u_l(\zeta)\,
\end{equation}
for all $\zeta\in\bbC\setminus(\spec(\cA_1)\cup\spec(\cA_2))$.
\end{proposition}
\begin{proof} We temporarily denote by $R(\zeta)$ the RHS of  \eqref{5.14aJ} in Theorem \ref{thm1.7}, that is, we set $R(\zeta)=(\tr R_2(\overline{\zeta}))^*J (\tr R_1(\zeta))$. Thus, our objective is to prove that $R(\zeta)$ is equal to the RHS of \eqref{fdi-krf}. 
First, we will use the fact that the subspaces
 $\cG_1=\tr(\dom(\cA_1))$ and $\cG_2=\tr(\dom(\cA_2))$ are Lagrangian in $\mathfrak{H}\times\mathfrak{H}$, cf.\ Remark \ref{rem4.2}, and prove the following elementary assertions:
 \begin{equation}\label{assrts}
 {\rm (i)}\,\ \ker(R(\zeta))=\ran(A_0-\zeta);\quad {\rm (ii)}\,\
 \ran(R(\zeta))=\ker(A_0^*-\zeta).
 \end{equation} 
 To begin the proof we notice that for any $u,v\in\cH$ by \eqref{5.3} one has
 \begin{equation}\label{Romega}\begin{split}
 \langle R(\zeta)u,v\rangle_{{}_\cH}&= \langle  (\tr R_2(\overline{\zeta}))^*J (\tr R_1(\zeta))u,v\rangle_{{}_\cH}\\
 &=\langle  J (\tr R_1(\zeta))u, \tr R_2(\overline{\zeta})v\rangle_{{}_{\mathfrak{H}\times\mathfrak{H}}}=\omega\big(\tr R_1(\zeta)u, \tr R_2(\overline{\zeta})v \big).
  \end{split}\end{equation}
If $u=(A_0-\zeta)w$ for some $w\in\dom(A_0)=\dom(\cA_1)\cap\dom(\cA_2)$ then \[R_1(\zeta)u=R_1(\zeta)(A_0-\zeta)w=R_1(\zeta)(\cA_1-\zeta)w=w\]
because $w\in\dom(\cA_1)$ and thus $\tr R_1(\zeta)u=\tr w\in\cG_2$ because  $w\in\dom(\cA_2)$. Since $\cG_2$ is isotropic, \eqref{Romega} yields $R(\zeta)u=0$ and thus $\overline{\ran(A_0-z)}\subseteq\ker(R(\zeta))$. On the other hand, if $R(\zeta)u=0$ then $\tr R_1(\zeta)u\in\cG_2$ by \eqref{Romega} since $\cG_2=\tr(\dom(\cA_2))$ is co-isotropic. Then $w:=R_1(\zeta)u\in\dom(\cA_1)\cap\dom(\cA_2)$ and thus $u=(A_0-\zeta)w$ yielding $\ker(R(\zeta))\subseteq\ran(A_0-z)$. This proves (i) in \eqref{assrts}. In particular, $\ran(A_0-\zeta)$ is closed and thus $\cH=\ran(A_0-\zeta)\oplus\ker(A_0^*-\overline{\zeta})$ where the sum is orthogonal. To show that $R(\zeta)u\in\ker(A_0^*-\zeta)=\big(\ran(A_0-\overline{\zeta})\big)^\perp$ for each $u\in\cH$, we pick any vector $v=(A_0-\overline{\zeta})w\in\ran(A_0-\overline{\zeta})$ with some $w\in\dom(A_0):=\dom(\cA_1)\cap\dom(\cA_2)$. As above,
 \[R_2(\overline{\zeta})v=R_2(\overline{\zeta})(A_0-\overline{\zeta})w=R_2(\overline{\zeta})(\cA_2-\overline{\zeta})w=w\]
because $w\in\dom(\cA_2)$ and thus $\tr R_2(\overline{\zeta})v=\tr w\in\cG_1$ because  $w\in\dom(\cA_1)$. Since $\cG_1$ is isotropic, \eqref{Romega} yields $\langle R(\zeta)u, v\rangle_{{}_\cH}=0$ and so
the inclusion $\ran(R(\zeta))\subseteq\ker(A_0^*-\zeta)$ in assertion (ii) of \eqref{assrts} does hold. What we have proved so far shows that the finite dimensional operator 
\begin{equation}\label{Risom}
R(\zeta)\big|_{\ker(A_0^*-\overline{\zeta})}:
\ker(A_0^*-\overline{\zeta})\to \ker(A_0^*-{\zeta})
 \text{ is an isomorphism}
 \end{equation}
 as it is injective by 
assertion (i) of \eqref{assrts} and $\ker(A_0^*-\overline{\zeta})\cap\ran(A_0-\zeta)=\{0\}$. This implies assertion (ii)
and finishes the proof of  \eqref{assrts}.

The rest easily follows by representing the isomorphism in \eqref{Risom} as an $(m\times m)$ matrix ${\mathbf r}(\zeta)$ using the bases $(u_k(\overline{\zeta}))_{1\le k\le m}$ and $(u_k(\zeta))_{1\le k\le m}$ in $\ker(A_0^*-\overline{\zeta})$ and $\ker(A_0^*-\zeta)$, respectively. Indeed, let ${\mathbf r}(\zeta)=\big(r_{lk}(\zeta)\big)_{1\le l,k\le m}$ be chosen such that $R(\zeta)u_k(\overline{\zeta})=\sum_{l=1}^mr_{lk}(\zeta)u_l(\zeta)$, and let
$G(\zeta)=\big(\langle u_k(\zeta), u_l(\zeta)\rangle_{{}_\cH}\big)_{1\le k,l\le m}$ denote the Gramm matrix so that $G(\zeta)^\top=\overline{G(z)}$ while its inverse will be written as $G(\zeta)^{-1}=\big(G^{-1}_{kj}(\zeta)\big)_{1\le k,j\le m}$. If $u=\sum_{k=1}^mc_ku_k(\overline{\zeta})\in\ker(A_0^*-\overline{\zeta})$ then $\overline{G(\overline{\zeta})}(c_k)_{1\le k\le m}=\big(\langle u, u_j(\overline{\zeta})\rangle_{{}_\cH}\big)_{1\le j\le m}$ as vectors  in $\bbC^m$ and therefore
\begin{equation}\begin{split}
R(\zeta)u&=\sum_{k=1}^mc_kR(\zeta)u_k(\overline{\zeta})=\sum_{k=1}^m\big(\sum_{j=1}^m\overline{G^{-1}_{kj}(\overline{\zeta})}\langle u, u_j(\overline{\zeta})\rangle_{{}_\cH}\big)\big(\sum_{l=1}^mr_{lk}(\zeta)u_l(\zeta)\big)\\
&=\sum_{l,j}\big(\sum_{k=1}^mr_{lk}(\zeta)\overline{G^{-1}_{kj}(\overline{\zeta})}\big)\langle u, u_j(\overline{\zeta})\rangle_{{}_\cH}u_l(\zeta).
\end{split}\end{equation}
We now define ${\mathcal P}(\zeta)=(p_{lj}(\zeta))_{1\le l,j\le m}$ by the formula 
\begin{equation}\lb{for1}
{\mathcal P}(\zeta):={\mathbf r}(\zeta)\overline{G(\overline{\zeta})^{-1}}
\end{equation}
 and obtain equation 
\eqref{fdi-krf}  for 
$u\in\ker(A_0^*-\overline{\zeta})$. By $\cH=\ran(A_0-\zeta)\oplus\ker(A_0^*-\overline{\zeta})$ and assertion (i) in \eqref{assrts}
it also holds for all $u\in\cH$.
\end{proof}

\section{Dirichlet and Neumann trace operators}\label{appA}

In this appendix we recall  definitions and some facts about
various types of Dirichlet and Neumann trace operators which are discussed in detail in \cite{GM08}, \cite{GM10}.
\begin{hypothesis}\lb{6.1L}
	Let $n\in\mathbb{N},n\geq2$, and  $\Om\subset\mathbb{R}^n$ be a bounded domain with $C^{1,r}$, $r>1/2$, boundary.
\end{hypothesis}
First, we define the strong trace operators.
Let us introduce the boundary trace operator $\gaD^0$ (the Dirichlet
trace) by
\begin{equation}\lb{2.4}
\gaD^0\colon C^0(\ol{\Om})\to C^0(\dOm), \quad
\gaD  ^0 u = u|_\dOm .
\end{equation}
By the standard trace theorem, see, e.g., \cite[Proposition 4.4.5]{T11}, there exists a bounded, surjective Dirichlet
trace operator 
\begin{equation}
\gaD  \colon H^{s}(\Om)\to H^{s-1/2}(\dOm) \hookrightarrow
L^2(\dOm), \quad 1/2<s<3/2.
\lb{6.1}
\end{equation}

Next, retaining Hypothesis \ref{6.1L}, we introduce the Neumann trace operator $\gaN$ by
\begin{equation}\lb{Nstrong}
\gaN=\nu\cdot\gaD\nabla  \colon H^{s+1}(\Om)\to L^2(\dOm), \quad 1/2<s<3/2,
\end{equation}
where $\nu$ denotes the outward pointing normal unit vector to $\dOm$.
Furthermore, one can extend $\gaN$ to the weak Neumann trace operator still denoted by $\gaN$ such that
\[\gaN:\{u\in H^{1}(\Om)\,|\,\Delta u\in L^2(\Om)\}\to H^{-1/2}(\dOm).\]

\begin{lemma}[\cite{GM10}, Lemma 6.3]
	Assume Hypothesis \ref{6.1L}. Then the Neumann trace operator $\gaN$ considered in the context
	\begin{equation}\lb{gaN0}
	\gaN:H^2(\Om)\cap H^1_0(\Om)\to H^{1/2}(\dOm),
	\end{equation}
	is well-defined, linear, bounded, onto, and with a linear bounded right-inverse. 
	In addition, the null space of $\gaN$ in \eqref{gaN0}
	is $H^2_0(\Om)$, the closure of $C^{\infty}_0(\Om)$ in $H^2(\Om)$.
\end{lemma}

\begin{lemma}[\cite{GM10}, Corollary 6.6]\lb{wd}
	Assume Hypothesis \ref{6.1L}. Then there exists a unique linear bounded operator 
	\[
	\gd:\{u\in L^{2}(\Om)\,|\,\Delta u\in L^{2}(\Om)\}\to H^{-1/2}(\dOm),
	\]
	which is compatible with the Dirichlet
	trace introduced in \eqref{6.1}. This extension of the Dirichlet trace operator has dense range and allows for the following integration by 
	parts formula,
	\begin{equation}
	\lnoh \gaN w, \gd u\rnohs= (\Delta w,u)_{L^{2}(\Om)}-(w,\Delta u)_{L^{2}(\Om)}, 
	\end{equation}
	valid for every $u\in L^{2}(\Om)$ with $\Delta u\in L^{2}(\Om)$ and every $w\in H^{2}(\Om)\cap H^{1}_0(\Om)$.
\end{lemma}

\begin{lemma}[\cite{GM10}, Corollary 6.11]\lb{A.5}
	Assume Hypothesis \ref{6.1L}. Then there exists a unique linear bounded operator 
	\begin{equation}
	\gn:\{u\in L^{2}(\Om)\,|\,\Delta u\in L^{2}(\Om)\}\to H^{-3/2}(\dOm),
	\end{equation}
	which is compatable with the Neumann
	trace, introduced in \eqref{Nstrong}. This extension of the Neumann trace operator has dense range and allows for the following integration by parts formula,
	\begin{equation}
	{}_{H^{3/2}(\dOm)}\langle \gaD w, \gn u\rangle_{H^{-3/2}(\dOm)}= (w,\Delta u)_{L^{2}(\Om)}-(\Delta w,u)_{L^{2}(\Om)}, 
	\end{equation}
	valid for every $u\in L^{2}(\Om)$ with $\Delta u\in L^{2}(\Om)$ and every $w\in H^{2}(\Om)$ with $\gaN w=0$.
\end{lemma}
Next, we introduce the Dirichet-to-Neumann map $M_{D,N}$ associated with $-\Delta$ on $\Om$ as 
\begin{equation}
M_{D,N}:
H^{-1/2}(\dOm)\to H^{-3/2}(\dOm):
g\mapsto-\gn(u_D),  
\end{equation}
where $u_D$ is the unique solution of the boundary value problem
\begin{equation}
-\Delta u=0\,\,\hbox{in}\,\,\Om,\,\,\,\,u\in L^{2}(\Om),\,\,\gd u=g\,\,\hbox{in}\,\,\dOm.
\end{equation}

\begin{lemma}[\cite{GM10}, Theorem 12.1]\lb{nn}
	Assume Hypothesis \ref{6.1L}. Then the map
	\begin{align}
	\tN:\{u\in L^2(\Omega)| \Delta u\in L^2(\Omega)\}\rightarrow H^{1/2}(\partial\Omega),\,
	\tN u:=\gn u+ M_{D,N}(\gd u),
	\end{align}
	is bounded when the space 
	$\{u\in L^2(\Omega)| \Delta u\in L^2(\Omega)\}=\dom(-\Delta_{\max})$ is equipped with the natural graph norm 
	$(\|u\|_{L^{2}(\Om)}^2+\|\Delta u\|_{L^{2}(\Om)}^2)^{1/2}$. Moreover, this operator is onto. In fact,
	\begin{equation}\lb{onto}
	\tN(H^2(\Om)\cap H^1_0(\Om))=H^{1/2}(\partial\Omega).
	\end{equation}
	Also, the null space of the map $\tN$ is given by
	\begin{equation}\lb{kernel}
	\ker(\tN)=H_0^2(\Omega)\dot{+}\{u\in L^2(\Om),\,\, -\Delta u=0 \}.
	\end{equation}
	Finally, the following Green formula holds for every $u,v\in\dom(-\Delta_{\max})$,
	\begin{align}\lb{Green}
	&(-\Delta u,v)_{L^{2}(\Om)}-(u,-\Delta v)_{L^{2}(\Om)}\no\\
	&=-{\lnoh \tN u, \gd v \rnohs}+\overline{\lnoh \tN v, \gd u \rnohs}.
	\end{align}
\end{lemma}

\bibliography{mybib}
\bibliographystyle{siam}

\end{document}